\documentclass[a4paper,10pt,reqno]{amsart} %
\usepackage[usenames,dvipsnames]{color}
\usepackage{amsmath}
\usepackage{amsthm}
\usepackage{amssymb}
\usepackage[mathcal]{euscript}
\usepackage{mathrsfs,paralist}
\usepackage{verbatim}

\usepackage{cite}

\definecolor{viola}{rgb}{0.3,0,0.7}
\definecolor{ciclamino}{rgb}{0.5,0,0.5}

\newcommand{\bele}{\begin{lemm}}
\newcommand{\enle}{\end{lemm}}
\newcommand{\bedef}{\begin{defi}}
\newcommand{\bete}{\begin{teor}}
\newcommand{\eddef}{\end{defi}}
\newcommand{\ente}{\end{teor}}
\newcommand{\beos}{\begin{osse}}
\newcommand{\eddos}{\end{osse}}
\newcommand{\bepr}{\begin{prop}}
\newcommand{\empr}{\end{prop}}
\newcommand{\bepro}{\begin{prob}}
\newcommand{\empro}{\end{prob}}
\newcommand{\bede}{\begin{defin}}
\newcommand{\edde}{\end{defin}}
\newcommand{\beco}{\begin{coro}}
\newcommand{\enco}{\end{coro}}

\newcommand{\MC}{\color{black}}
\newcommand{\RRR}{\color{black}}
\newcommand{\RCOMM}{\color{black}}

\newcommand{\EEE}{\color{black}}

\newcommand{\beeq}[1]{\begin{equation}   \label{#1}}
\newcommand{\eddeq}{\end{equation}}
\newcommand{\beeqa}[1]{\begin{eqnarray}  \label{#1}}
\newcommand{\eddeqa}{\end{eqnarray}}
\newcommand{\beal}[1]{\begin{align}      \label{#1}}
\newcommand{\eddal}{\end{align}}
\newcommand{\bespl}[1]{\begin{split}     \label{#1}}
\newcommand{\edspl}{\end{split}}
\newcommand{\bega}[1]{\begin{gather}     \label{#1}}
\newcommand{\edga}{\end{gather}}
\newcommand{\beeqax}{\begin{eqnarray*}}
\newcommand{\eddeqax}{\end{eqnarray*}}

\newcommand{\tensore}{\varepsilon(\uu)}

\newcommand{\tensoret}{\varepsilon({\mathbf{u}_t})}

\newcommand{\teta}{\theta}

\newcommand{\nn}{{\bf n}}
\newcommand{\uu}{{\bf u}}

\newcommand{\vv}{{\bf v}}

\newcommand{\eeta}{{\mbox{\boldmath$\eta$}}}

\newcommand{\itt}{\int_0^t}
\newcommand{\io}{\int_\Omega}

\newcommand{\eps}{\varepsilon}

\newcommand{\weak}{\rightharpoonup}

\newlength{\dhatheight}
\newcommand{\doublehhat}[1]{%
    \settoheight{\dhatheight}{\ensuremath{\widehat{#1}}}%
    \addtolength{\dhatheight}{-0.35ex}%
    \widehat{\vphantom{\rule{1pt}{\dhatheight}}%
    \smash{\widehat{#1}}}}

    \newcommand{\hhat}{\widehat}

\let\TeXchi\chi
\def\chi{{\setbox0 \hbox{\mathsurround0pt
$\TeXchi$}\hbox{\raise\dp0 \copy0 }}}

\renewcommand{\part}{\partial_t}

\newcommand{\weaksto}{{\rightharpoonup^*}}
\newcommand{\weakto}{\rightharpoonup}
\newcommand{\debole}{\,\weak\,}

\newcommand{\pairing}[4]{ \sideset{_{#1 }}{_{ #2}}  {\mathop{\langle #3 , #4  \rangle}}}

 \def\fin{\hfill
         \trait .3 5 0
         \trait 5 .3 0
         \kern-5pt
         \trait 5 5 -4.7
         \trait 0.3 5 0
 \medskip}
 \def\trait #1 #2 #3 {\vrule width #1pt height #2pt depth #3pt}

\newcommand{\foraa}{\text{for a.a.}}
\newcommand{\aein}{\text{a.e.\ in}}
\newcommand{\down}{\downarrow}
\newcommand{\up}{\to}
\newcommand{\GC}{\Gamma_{\mathrm{C}}}
\newcommand{\GD}{\Gamma_{\mathrm{D}}}
\newcommand{\GN}{\Gamma_{\mathrm{N}}}
\newcommand{\overlineGC}{\overline{\Gamma}_{\mathrm{C}}}

\newcommand{\Dir}{{\mathrm{Dir}}}

\newcommand{\R}{\mathbb{R}}
\newcommand{\N}{\mathbb{N}}

\newcommand{\DDD}[2]{\begin{array}[t]{c}#1\vspace*{-1em}\\_{#2}\end{array}}
\newcommand{\ddd}[2]{\DDD{\begin{array}[t]{c}\underbrace{#1}\vspace*{.6em}\end{array}}{\text{\footnotesize #2}}}

\newcommand{\dd}{\,\mathrm{d}}

\newcommand{\tetas}{\theta_{\mathrm{s}}}
\newcommand{\tetaso}{\theta_{\mathrm{s}}^0}

\newcommand{\bsY}{\mathbf{Y}}

\newcommand{\bsVD}{H^1_{\GD}(\Omega;\R^3)}
\newcommand{\bsVDn}{H^1(\Omega)}

\newcommand{\calE}{\mathcal{E}}

\newcommand{\thetas}{\theta_{\mathrm{s}}}
\newcommand{\thetasq}{\theta_{\mathrm{s}}^2}
\newcommand{\balpha}{{\mbox{\boldmath$\eta$}}}
\newcommand{\bwalpha}{{\mbox{\boldmath$\widehat{\eta}$}}}

\newcommand{\Gdir}{\GD} 
\newcommand{\Gnew}{\GN}
\newcommand{\overlineGD}{\overline{\Gamma}_{\mathrm{D}}}
\newcommand{\overlineGN}{\overline{\Gamma}_{\mathrm{N}}}

\newcommand{\kr}{j}

\newcommand{\nlocs}[2]{\mathcal{J}[#1](#2)}
\newcommand{\nlocss}[1]{\mathcal{J}[#1]}
\newcommand{\nlname}{\mathcal{J}}
\newcommand{\ab}{\alpha}
\newcommand{\as}{\alpha}
\newcommand{\etens}{\mathbb{E}}
\newcommand{\vtens}{\mathbb{V}}

\newcommand{\VC}{H^1(\GC)}

\newcommand{\HC}{L^2(\GC)}

\newcommand{\pwM}[2]{\widetilde{#1}_{\kern-1pt#2}}
\def\trait #1 #2 #3 {\vrule width #1pt height #2pt depth #3pt}
\newcommand{\pwN}[2]{#1_{\kern-1pt#2}}
\def\trait #1 #2 #3 {\vrule width #1pt height #2pt depth #3pt}
\newcommand{\ds}[3]{{#1}_{#2}^{#3}}
\newcommand{\dss}[2]{\theta_{\mathrm{s},#1}^{#2}}
\newcommand{\sft}{\mathsf{t}}

\newcommand{\efm}{\mathsf{e}}
\newcommand{\vfm}{\mathsf{v}}

\newcommand{\zzeta}{{\mbox{\boldmath$\zeta$}}}
\newcommand{\xxi}{{\mbox{\boldmath$\xi$}}}

\newtheorem{maintheorem}{Theorem}
\newtheorem{theorem}{Theorem}[section]

\newtheorem{lemma}{Lemma}[section]

\newtheorem{proposition}[lemma]{Proposition}
\newtheorem{definition}[lemma]{Definition}%
\newtheorem{remark}[lemma]{Remark}%

\newtheorem{problem}[lemma]{Problem}
\newtheorem{notation}[lemma]{Notation}

\newcommand{\tetak}{\pwN {\theta^k}{\tau}}
\newcommand{\tetakmu}{\pwN {\theta^{k-1}}{\tau}}
\newcommand{\tetask}{\pwN {\theta^k}{s,\tau}}
\newcommand{\tetaskmu}{\pwN {\theta^{k-1}}{s,\tau}}
\newcommand{\hk}{\pwN {h^k}{\tau}}
\newcommand{\lk}{\pwN {\ell^k}{\tau}}
\newcommand{\fk}{\pwN {\mathbf{f}^k}{\tau}}
\newcommand{\gk}{\pwN {\mathbf{g}^k}{\tau}}

\newcommand{\chikmu}{\pwN {\chi^{k-1}}{\tau}}

\newcommand{\tk}{\pwN {t^k}{\tau}}
\newcommand{\tkmu}{\pwN {t^{k-1}}{\tau}}
\newcommand{\chik}{\pwN {\chi^k}{\tau}}
\newcommand{\chikp}{\RCNEW (\chik)^+\EEE}
\newcommand{\chikmup}{\RCNEW (\chikmu)^+\EEE}
\newcommand{\xik}{\RCNEW \pwN {\sigma^k}{\tau} \EEE}
\newcommand{\xiek}{\RCNEW \ds \sigma {\epsilon} k \EEE}

\newcommand{\fff}{\mathbf{f}}

\newcommand{\uuk}{\pwN {\uu^k}{\tau}}
\newcommand{\uukmu}{\pwN {\uu^{k-1}}{\tau}}

\newcommand{\tetaink}{\pwN {\teta^0}{\tau}}
\newcommand{\tetasink}{\dss \tau 0}
\newcommand{\tensoretk}{\varepsilon\bigg(\frac{\uuk - \uukmu}{\tau} \bigg)}

\newcommand{\accauno}{H^1(\Omega)}
\newcommand{\FF}{{\bf F}}

\newcommand{\piecewiseConstant}[2]{\overline{#1}_{\kern-1pt#2}}
\newcommand{\pwc}{\piecewiseConstant}
\newcommand{\piecewiseLinear}[2]{\widehat{#1}_{\kern-1pt#2}}
\newcommand{\pwl}{\piecewiseLinear}
\newcommand{\upiecewiseConstant}[2]{\underline{#1}_{\kern-1pt#2}}
\newcommand{\upwc}{\upiecewiseConstant}
\newcommand{\piecewiseConstants}[2]{\overline{#1}_{\kern-1pt{\mathrm{s},#2}}}
\newcommand{\pwcs}{\piecewiseConstants}
\newcommand{\piecewiseLinears}[2]{\widehat{#1}_{\kern-1pt{\mathrm{s},#2}}}
\newcommand{\pwls}{\piecewiseLinears}
\newcommand{\upiecewiseConstants}[2]{\underline{#1}_{\kern-1pt{\mathrm{s},#2}}}
\newcommand{\upwcs}{\upiecewiseConstants}

\newcommand{\RNEW}{\color{black}}
\newcommand{\RCOMMN}{\color{black}}{\color{red}}

\newcommand{\RCNEW}{\color{black}}
\newcommand{\RCORR}{\color{black}}
\newcommand{\RRRN}{\color{red}}

\newcommand{\QED}{\mbox{}\hfill\rule{5pt}{5pt}\medskip\par}

\newcommand{\tetaj}{\teta_{\rho_j}}
\newcommand{\tetasj}{\teta_{\mathrm{s},\rho_j}}
\newcommand{\uuj}{\uu_{\rho_j}}
\newcommand{\chij}{\chi_{\rho_j}}
\newcommand{\chijp}{\RCNEW(\chi_{\rho_j})^+\EEE}
\newcommand{\chijpt}{\RCNEW(\chi_{\rho_j}(t))^+\EEE}
\newcommand{\zzetaj}{\zzeta_{\rho_j}}
\newcommand{\xij}{\xi_{\rho_j}}

\newcommand{\betar}{\beta_{\varsigma}}
\newcommand{\wbetar}{\widehat{\beta}_{\varsigma}}
\newcommand{\etar}{\eta_{\varsigma}}
\newcommand{\wetar}{\widehat{\eta}_{\varsigma}}

\newcommand{\tetan}{\theta_{\varsigma_n}}
\newcommand{\tetasn}{\theta_{\mathrm{s},\varsigma_n}}
\newcommand{\uun}{\uu_{\varsigma_n}}
\newcommand{\chin}{\chi_{\varsigma_n}}
\newcommand{\xin}{\xi_{\varsigma_n}}
\newcommand{\zzetan}{\zzeta_{\varsigma_n}}

\newcommand{\tetasrhos}{\teta_{\mathrm{s},\rho,\varsigma}}
\newcommand{\xirs}{\xi_{\rho,\varsigma}}
\newcommand{\zzetars}{\zzeta_{\rho,\varsigma}}
\newcommand{\tetars}{\teta_{\rho,\varsigma}}
\newcommand{\chirs}{\chi_{\rho,\varsigma}}
\newcommand{\uurs}{\uu_{\rho,\varsigma}}
\newcommand{\uprs}{\RCNEW \sigma_{\rho,\varsigma}\EEE}
\newcommand{\upj}{\RCNEW \sigma_{\rho_j}\EEE}
\newcommand{\upin}{\RCNEW \sigma_{\varsigma_n}\EEE}

\newcommand{\chip}{\RCNEW (\chi)^+\EEE}
\newcommand{\chipt}{\RCNEW (\chi(t))^+\EEE}
\newcommand{\chinp}{\RCNEW (\chin)^+\EEE}

\newcommand{\RF}{\color{black}}

\begin{document}
\date{March 31st, 2021}

   \title{Global existence for a highly nonlinear temperature-dependent system modeling  nonlocal adhesive contact}

\author{Giovanna Bonfanti}
\address{G.\ Bonfanti, Sezione di Matematica, DICATAM, Universit\`a degli studi di Brescia. Via Valotti 9, I--25133 Brescia -- ITALY}
\email{giovanna.bonfanti\,@\,unibs.it}

\author{Michele Colturato}
\address{M.\ Colturato, Sezione di Matematica, DICATAM,  Universit\`a degli studi di Brescia. Via Valotti 9, I--25133 Brescia -- ITALY}
\email{michele.colturato\,@\,unibs.it}

\author{Riccarda Rossi}
\address{R.\ Rossi, DIMI, Universit\`a degli studi di Brescia. Via Branze 38, I--25133 Brescia -- ITALY}
\email{riccarda.rossi\,@\,unibs.it}

\maketitle

\numberwithin{equation}{section}

\begin{abstract}
In this paper
we \RRR analyze   a new temperature-dependent model for adhesive contact that
encompasses \EEE nonlocal adhesive forces and damage effects, as well as nonlocal heat flux contributions on the contact surface.
 The related PDE system combines heat equations,  in the bulk domain and on the contact surface, with   mechanical force balances, including micro-forces, \RRR  that result in the equation for the displacements and in the flow rule for the  damage-type internal variable describing the state of the adhesive bonds.
 \EEE 
 Nonlocal \RRR effects \EEE  are accounted for by  terms \RRR featuring integral operators \EEE  on the contact surface.
 \par
 The analysis of this system   \EEE  
  poses several difficulties due to its overall highly nonlinear character, and in particular to the presence of quadratic terms, in the rates of the strain tensor and of the internal variable, that feature in the bulk and surface heat equations. Another major challenge is related to proving \emph{strict} positivity for  the bulk and surface temperatures.
 \par
We tackle these issues by
very careful  estimates that enable us
to prove the existence of global-in-time solutions
and could be useful in other contexts.
All calculations are rigorously rendered on an accurately  devised
time discretization scheme in which the limit passage is carried out via variational techniques.  \EEE

\end{abstract}
\vskip3mm \noindent {\bf Key words:}  Contact; adhesion;  nonlocal effects; \RRR temperature; weak solvability; existence results; time discretization. \EEE
\vskip3mm \noindent {\bf AMS (MOS) Subject Classification: 35K55,
35Q72,  74A15, 74M15.}

\pagestyle{myheadings} \markright{\it Bonfanti, Colturato, Rossi / Nonlocal  adhesion with temperature}

\section{Introduction}
In this paper we investigate a PDE system describing adhesive contact between a thermoviscoelastic body and a rigid support, in the presence of nonlocal thermo-mechanical effects. Its overall highly nonlinear character
is in particular manifest in the heat equations in the bulk domain and on the contact surface. To prove the existence of global-in-time solutions
 we develop some techniques that could be of interest for the analysis of  other thermodynamically consistent systems in solid mechanics.
\subsection{The model and the PDE system}
The study of adhesive contact and delamination phenomena is of applicative interest   due to the extensive presence of 
\RCOMMN layered \EEE structures  in  several industrial contexts,
 cf.\ e.g.
\cite{RKMPVZ, Shillor-Sofonea-Telega} and the references therein.  In this paper we address  a temperature-dependent  model
for adhesive contact that originates   from the theory  by \textsc{M.\ Fr\'emond} \cite{fre},
in the broader framework of the  theory of generalized standard materials \cite{HalNgu75MSG}  (see \cite{BCNS} for a (partial) survey of adhesive contact and delamination models pertaining to this cadre).
Our model includes  \emph{nonlocal adhesive and damage forces} as originally proposed, in the \emph{isothermal} case,  in \cite{freddi-fremond}, as well as
\emph{nonlocal heat sources} on the contact surface. Its  derivation, based on the principle of virtual power also encompassing `microscopic movements' as in the approach from \cite{fre}, was carried out in \cite{BCNS}.
 \par
%
%
%
%
%
%
%
%
%
%
%
\par
The motivation for including  nonlocal effects  in adhesive contact modeling  stems from experiments showing that elongation, i.e.\ a variation of the distance of two distinct points on the contact surface, \RCOMMN may have damaging effects
\RF on the substance gluing the body to the support along such surface. \EEE
\EEE 
This is thoroughly illustrated in \cite{freddi-fremond} also via numerical experiments. The analysis of the isothermal model from \cite{freddi-fremond} was first carried out in \cite{BBRnonloc1}.
In the model derived in \cite{BCNS} we have additionally encompassed
a nonlocal interaction between the body and the adhesive substance as far as it concerns  heat exchange on the contact surface.
\par
More precisely,
during a time interval $(0,T),\, T>0$, we consider a thermoviscoelastic body  located in a smooth and bounded domain $\Omega\subset{\R}^3$ and lying on a rigid support on a part of its boundary, on which some adhesive substance is present. Hence,   $\partial\Omega= \overlineGD \cup \overlineGN \cup \overlineGC$
with
(1) $\GC$
 the contact surface, hereafter assumed   \emph{flat} and identified with a subset of  \RCORR $\R^2$, \EEE (2)
 $\GD$ the Dirichlet part of the boundary, with positive measure,  on which homogeneous boundary conditions are prescribed,  and (3) $\GN$ the Neumann part, on which a traction is applied.
%
%
%
%
The state variables
in the bulk domain $\Omega$
are the   absolute temperature $\teta$ of the body and its displacement $\uu$ 
\RCORR (at small strains); \EEE
the state variables defined on the contact surface $\GC$ are the absolute temperature
$\tetas$ of the adhesive substance,  and a   surface damage-type parameter $\chi$ representing
%
the fraction of fully effective links in the bonding. As such, $\chi$  takes values in $[0,1]$, with $\chi=0$ for completely damaged bonds, $\chi=1$ for fully intact bonds, and $\chi \in (0,1)$ for the intermediate states.
\RCORR The distinction \EEE between the temperature $\teta$ of the bulk domain and the temperature $\tetas$ of the adhesive substance is typical of
\textsc{Fr\'emond}'s approach \RCORR to the modeling of thermal effects in rate-dependent adhesive contact,  cf.\ e.g.\ \cite{BBR3, BBR8}. Nonetheless, it  has also been discussed  in the context of  the rate-independent modeling of delamination,  see \cite[Sec.\ 5.3.3.3]{Mielke-Roub-book}. \EEE
\par
The evolution of the variables $(\teta,\uu,\tetas,\chi)$ during the time interval $(0,T)$ is governed by a system of \RCOMMN PDEs \EEE  in the bulk domain and on the contact surface derived
from  the general laws of Thermomechanics and from suitable  choices for the  free energy and pseudo-potential of dissipation that also account for nonlocal interactions between the body and its support.
Indeed,
\begin{subequations}
	\label{PDE-INTRO}
the principle of virtual power leads to the quasistatic momentum balance for the macroscopic movements
\begin{equation}
\label{mom-balance}
	-\mathrm{div}(\etens \tensore + \vtens \tensoret + \teta{\mathbb I})={\bf f} \qquad  \textrm{in } \Omega\times (0,T),
\end{equation}
supplemented by the following boundary conditions
\begin{align} 		
&
	\label{Dir}
		\mathbf{u}= {\bf 0} &&   \text{in } \GD\times (0,T),
		\\
		&
		\label{condIi}
 (\etens \tensore+\vtens \tensoret +\theta{\mathbb I}){\bf n}={\bf g} &&  \text{in } \Gamma_{\mathrm{N}}\times (0,T),
\\ & 	 	\label{condIii}	
		(\etens\tensore + \vtens\tensoret   + \theta{\mathbb I}){\bf n}  +\chi{\bf u}+
		\partial I_{(-\infty,0]}(\uu\cdot {\bf n}){\bf n}  	
		 +\int_{\GC}\mu(|x{-}y|) \uu(x) \chi(x)\chi(y)
		\dd y 		\ni{\bf 0}   && \text{in }  \GC \times (0,T).
\end{align}
Here,  $\etens $ and $\vtens$ denote the elasticity and the viscosity tensors, $\mathbb{I}$ the identity matrix, $\bf{n}$ the outward unit normal vector to the boundary $\partial \Omega$.
Moreover, $\bf{f}$ is a volume force while $\bf{g}$  is a traction applied on $\GN$.
Condition
\eqref{condIii},
coupling the evolution of $\uu$ (hereafter, we shall denote its trace on $\GC$ by the same symbol) to that
of $\chi$, can be understood as a generalization of the classical \emph{Signorini contact} conditions. Indeed, it features a selection
$\xxi \in
\partial I_{(-\infty,0]}(\uu {\cdot}\mathbf{n}) \mathbf{n}$, with
$\partial I_{(-\infty,0]}$
the subdifferential of the indicator function of the interval $(-\infty,0]$, that represents a reaction force activated when the non-interpenetration constraint $\uu {\cdot}\mathbf{n}\leq 0$ on $\GC$ holds as  \RCORR an equality,
i.e.\ when \EEE
$\uu {\cdot}\mathbf{n}= 0$.
The normal reaction on the boundary condition described by \eqref{condIii} also includes
 the  nonlocal contribution
\[
\int_{\GC}\mu(|x{-}y|) \uu(x) \chi(x)\chi(y)
		\dd y \,,
\]
\RCOMMN where the positive function $\mu$ \EEE
accounts for
the attenuation of nonlocal interactions as the  distance $|x-y|$ between two points $x$ and $y$ on the contact surface increases.
Observe that in \eqref{condIii} (and in the forthcoming \eqref{eqII}, \RCOMMN \eqref{eqI-cont} \EEE and \eqref{eq-tetas}),  we have written explicitly the dependence of the unknows $(\teta,\uu, \tetas, \chi)$ on the variable $x \in \GC$ only in the nonlocal terms involving integrals (with respect to the spatial variable $y \in \GC$).
\par
As customary in \textsc{Fr\'emond}'s approach, the principle of virtual power leads to  a micro-force balance on the contact surface that results  in  the \RCORR following \EEE  flow rule for the damage-like parameter $\chi$
\begin{equation} \label{eqII}
	\begin{aligned}
		&
	\chi_t -\Delta\chi  +\partial I_{[0,1]}(\chi)   +
	\gamma'(\chi)+\lambda'(\chi)(\thetas-\theta_*)
	\\
	&
	\qquad  \ni -\frac 1 2\vert{\uu }\vert^2-\frac 1 2\int_{\GC} \mu(|x{-}y|)  (|\uu(x)|^2 + |\uu(y)|^2 )\chi(y)
	\dd y  \qquad \text{in } \GC \times (0,T),
		\end{aligned}
\end{equation}
 supplemented by the  no-flux boundary condition
\begin{equation}\label{bord-chi}
\partial_{\nn_\mathrm{s}} \chi=0 \qquad   \text{in } \partial \GC \times		(0,T).
\end{equation}
In \eqref{eqII}--\eqref{bord-chi},  $\nn_\mathrm{s}$ denotes the outward unit normal vector to $\partial \GC$, $\theta_* >0$ is a phase transition temperature, the function $\lambda$ is related to the latent heat while $\gamma$ describes possible non-monotone dynamics for $\chi$  (it may model some cohesion in the material).
Moreover, the subdifferential term $\partial I_{[0,1]}(\chi)$ ($I_{[0,1]}(\cdot)$ denoting the indicator function of the interval $[0,1]$) enforces the physical constraint that $\chi$ takes values in $[0,1]$.
The source of damage on the right-hand side of \eqref{eqII} features local and nonlocal terms and, in particular, it may differ from zero even in the case in which  $\uu={\bf 0}$, due to the nonlocal,  integral contribution
 that renders the damaging effects of  elongation.
\par
The equations for the bulk and surface  temperature variables $\teta$ and $\tetas$
 are recovered from the first principle of Thermodynamics. The internal energy balance equation written
 in
the bulk domain 
  reads
\begin{equation} 			\label{eqI}
	\theta_t-\theta\mathrm{div}(\uu_t)-\mathrm{div}(\ab(\theta)\nabla\theta)
	= h + \tensoret  \vtens  \tensoret, \qquad  \textrm{in } \Omega\times (0,T),
\end{equation}
with prescribed boundary conditions
\begin{align}
 			\label{eqI-bdry}
			&
\ab(\theta)\nabla\theta\cdot{\bf n}=0 &&  \text{in } (\GD{\cup}\Gnew)\times (0,T),
\\
&			\label{eqI-cont}
	\ab(\theta)\nabla\theta \cdot{\bf n}
	=-\theta\left(k(\chi)(\theta{-}\thetas){+}\int_{\GC} \mu(|x{-}y|) 	(\theta(x){-}\thetas(y))
	\chi(x) \chi(y)\dd y\right) &&   \text{in } \GC \times (0,T).
\end{align}
It is coupled to the  internal energy balance equation on the contact surface
\begin{align}
			\label{eq-tetas}
			&
\begin{aligned}
	&
	\partial_t {\thetas}-\thetas \lambda'(\chi)\chi_t-\mathrm{div}(\as(\thetas)\nabla\thetas)\\
	&=\ell + |\chi_t|^2+\thetas\left(k(\chi)(\theta{-}\thetas) {+} \int_{\GC}\mu(|x{-}y|)
	(\theta(y){-}\thetas(x)) \chi(x)\chi(y)\dd y\right)    \qquad \qquad  \text{in } \GC\times (0,T),
\end{aligned}
\\
&
\label{eq-tetas-b}
	\as(\thetas)\nabla\thetas\cdot\RCOMMN {\nn_\mathrm{s}}\EEE=0  \qquad \qquad  \qquad \qquad  \qquad \qquad  \qquad \qquad  \qquad \qquad    \qquad \qquad  \qquad \quad \text{in } \partial\GC\times (0,T).
\end{align}
\end{subequations}
Here, the positive function $\ab$ \RCOMMN represents the \EEE  heat conductivity coefficient both  in the bulk domain and on the contact surface, $k$ is a surface thermal diffusion coefficient, \RCOMMN while $h$ and $l$ are volume and surface heat sources. \EEE
The evolutions of the bulk and surface temperatures $\teta$ and $\tetas$ are coupled by
\RCORR the boundary condition \EEE
 \eqref{eqI-cont}, featuring two  distinct contributions. 
 The first one has a `local' character, as it depends on the quantity  $(\teta{-}\tetas)$ evaluated at the \emph{same point} $x \in \GC$ (again, we keep on denoting by $\teta$ the trace of the absolute temperature on $\GC$). Instead, the second term on the right-hand side of \eqref{eqI-cont} \RCOMMN has a nonlocal character as it \EEE
involves
 the quantity 
 \[
\int_{\GC} \mu(|x{-}y|) 	(\theta(x){-}\thetas(y))
	\chi(x) \chi(y)\dd y
 \]
featuring  the thermal gap between \emph{different} points $x$ and $y$ on the contact surface.
\par
Actually,  in what follows we will tackle the analysis of a
 generalized version of system  \eqref{PDE-INTRO},
  in which  the nonlocal integral terms
  are replaced by more  general nonlocal operators and where the various, concrete, subdifferential operators are replaced by general  maximal monotone nonlinearities.
  \RCORR Nonetheless, throughout this Introduction we shall confine the discussion to system \eqref{PDE-INTRO}. \EEE
  \subsection{Analytical difficulties}
 The major difficulties related to the analysis of
 system \eqref{PDE-INTRO} are that
\begin{enumerate}
\item  it encompasses both  bulk and surface equations. In particular, the evolutions of  the displacement variable $\uu$ and  of the adhesion parameter $\chi$ are coupled through the Robin-type boundary condition \eqref{condIii}. This  prevents us from  applying  regularity results for elliptic systems
 that would lead to  enhanced spatial (e.g., $H^2$-) regularity for $\uu$ and $\uu_t$. In turn, such regularity would be handy, for instance, to better control the right-hand side of the heat equation \eqref{eqI}, since,
  indeed,
  \medskip
 \item the bulk temperature  and displacement equation (the surface temperature equation and the flow rule, respectively) are coupled by the quadratic term $\eps(\uu_t) \vtens \eps (\uu_t)$ (by the term $|\chi_t|^2$, \RCOMMN resp.\EEE) that is just in  $L^1 (\Omega{\times}(0,T))$ \RCOMMN (in $L^1 (\GC{\times}(0,T))$, resp.) \EEE once the basic energy estimates on system \eqref{PDE-INTRO} are performed. Other nonlinear coupling terms between bulk and surface equations occur in \eqref{condIii}, \eqref{eqII}, \eqref{eqI-cont}, and \eqref{eq-tetas}, but the $L^1$-character of the right-hand side of the heat equations poses the most prominent challenge, together with
 \medskip
 \item proving that the temperature variables $\teta$ and $\tetas$ are \emph{strictly} positive. As a matter of fact, because of the nonlocal terms in \eqref{eqI-cont} and \eqref{eq-tetas},
 well-established techniques \RCORR for proving strict positivity of  $\teta$ and $\tetas$, \EEE
based on comparison arguments,
  fail.
\end{enumerate}
\par
The presence of quadratic terms  in the rate of the internal variable (and in the rate of the strain tensor, when the momentum balance is also included in the system) is typical of  \emph{thermodynamically consistent} models; \emph{strict} positivity of the temperature is also a key ingredient for their compliance \RCORR with \EEE the laws of Thermodynamics. In fact, the challenges  in items (2) \& (3)  of the above list transcend the specific problem examined here, and have stimulated the development of a variety of techniques over the last two  decades.
\par
The first existence result for the `full' model by \textsc{Fr\'emond} for  solid-liquid phase transitions
(here `full' refers to the fact the quadratic term, in the rate of the phase-field parameter, on the right-hand side of the heat equation is not neglected), dates back to  \cite{Lut-Ste1,Lut-Ste2}, for (spatially) \emph{one-dimensional} systems. To our knowledge, the analysis of  a  `full' model in the three-dimensional case was first addressed in \cite{Bonetti-Bonfanti-1}, tackling a  thermodynamically consistent PDE system for damage in thermo-visco-elastic materials. Therein, the heat equation featured the quadratic terms $|\chi_t|^2$ and $|\nabla\chi_t|^2$ (with $\chi$ the damage parameter), as well as $\eps(\uu_t) {:} \eps(\uu_t)$, on its right-hand side, while the heat conduction coefficient $\alpha = \alpha(\teta) $ was assumed to be constant. In that framework, only a local-in-time existence result was obtained. Ever since, in  most of the papers addressing
the analysis of thermo-(visco-)elastic  models with $L^1$ right-hand sides in the heat equation, \emph{global-in-time} existence results have been obtained under suitable growth conditions, either on the non-constant heat conductivity, or on a non-constant heat capacity coefficient.
\par
The latter course has been  pursued in a series of papers by \textsc{T.\ Roub\'{\i}\v{c}ek}, starting from \cite{tr1} \RCORR that address \EEE the analysis of a broad  class of thermomechanical, thermodynamically consistent, \emph{rate-independent} processes.    In \cite{tr1} and in several subsequent papers covering a wide range of applications (cf., e.g.,
\cite{tr1-bis, tr2, tr3, Mie_Rou20}; see also \cite{trr,BBR8} for applications to adhesive contact and delamination)
\textsc{Roub\'{\i}\v{c}ek} switches to an alternative thermal variable, the `enthalpy', defined in terms of a primitive of the heat-capacity. In this way, the nonlinear character of the heat equation is partially `tamed'; its $L^1$ right-hand side (featuring quadratic terms in the rates of the strain tensor and of the internal variables of the model) is dealt with by means of  \textsc{Boccardo-Gallou\"et} type estimates \cite{Boccardo},
as adapted in \cite{FM}. Such estimates  yield   a limited spatial regularity for the enthalpy/temperature variable, \RCORR which is \EEE estimated in the space $W^{1,r}(\Omega)$ for some specific $r\in (1,2)$. Hence, in  the aforementioned papers (global-in-time) existence results are typically obtained for a formulation of the heat/enthalpy equation with spatially smooth test functions.
\par
In turn, \cite{FPR} pioneered an alternative approach to the analysis of the heat equation with a $L^1$-right-hand side in the `full' model for solid-liquid phase transitions \RCORR by \textsc{Fr\'emond}. \EEE The core assumption there is some suitable growth condition on the heat conductivity $\alpha$.
 This leads to a $H^1$-spatial regularity for the temperature variable, albeit in the context of a quite weak formulation for thermal evolution. Specifically,  in \cite{FPR, FPR-errata} the heat equation is formulated, consistently with the laws of Thermodynamics, in terms of an entropy inequality, involving smooth test functions, and of a total energy balance. The `entropic' solution concept
 advanced in \cite{FPR} has proved to be remarkably flexible. It
  has been extended to various contexts, from
the evolution of non-isothermal nematic liquid crystals \cite{FFRS, FRSZ}, to models for damage and phase separation in thermo-visco-elastic solids in $\R^d$, $d\in \{2,3\}$,  cf.\  \cite{RocRos, HKRR}.
In the latter papers the existence of `entropic' solutions was proved
under the condition that
%
%
\begin{equation}
\label{growth-alpha-RR}
\exists\, c_0,\, c_1> 0 \ \ \exists\, \mu>1 \ \  \forall\, \teta \in \R^+ \, : \quad   c_0(1{+}\teta^\mu) \leq \alpha(\teta) \leq c_1(1{+}\teta^\mu) 
\end{equation}
\RF (cf., e.g., \cite{zr} for examples of nonlinear heat conduction). \EEE
Under the more restrictive condition that
\begin{equation}
\label{more-restrictive}
\text{the exponent $\mu$ in \eqref{growth-alpha-RR} satisfies }  \mu \begin{cases} \in (1,2),
\\
\in \left(1, \tfrac 53 \right)
\end{cases}
 \text{ if the space dimension } d = \begin{cases}
 2,
 \\
 3
 \end{cases}
\end{equation}
 \cite{RocRos} showed
the existence of `conventional' weak solutions to  the PDE system coupling the momentum balance, the flow rule for the damage parameter, and the heat equation,  which was
formulated in a variational way, with suitable test functions.
\par
Finally, we would like to mention the analysis of  a (still thermodynamically consistent) PDE system for thermo-visco-plasticity at small strains from \cite{HMS}. Via  maximal parabolic regularity arguments,
the authors succeeded in proving  the existence  of
global-in-time solutions to a suitable weak formulation of the system without resorting to growth conditions on
the heat conductity $\alpha(\teta) \equiv 1$.
\subsection{Our results}
With \underline{\textbf{our main result, Theorem \ref{thm:1}}} ahead, we are going to prove  the existence of  global-in-time, weak solutions to system \eqref{PDE-INTRO}, under the \emph{more general condition} \eqref{growth-alpha-RR}: in particular, 
\RF we are not going to restrict the range of \EEE
 the exponent $\mu$.
We highlight, in the similar contexts of
\cite{RocRos, HKRR}, \eqref{growth-alpha-RR}
 previously granted the existence of `entropic' solutions, only, with the heat equation  formulated via \RCORR an \EEE entropy inequality and an overall energy balance. Therefore here,  under the same conditions as in \cite{RocRos, HKRR},  we succeed in bypassing entropic solutions and directly conclude the existence  of `conventionally'
 weak solutions, \RCORR that we will term \emph{weak energy solutions}. \EEE
 We are also going to obtain the \emph{strict} positivity properties
 \begin{equation}
 \label{strict-positivity-intro}
 \exists\, \overline{\theta}, \,  \overline{\theta}_{\mathrm{s}} > 0\, : \qquad
 \theta \geq \overline{\theta}  \ \, \text{a.e. in $\Omega \times (0,T)$,} \quad
\thetas \geq \overline{\theta}_{\mathrm{s}} \ \,  \text{a.e. in $\GC \times (0,T)$},
 \end{equation}
 for which the comparison arguments often used in the literature
 are not applicable due to the nonlocal terms in the heat equations and related boundary conditions.
 The cornerstone of our existence proof for weak solutions, under the \emph{sole}
\eqref{growth-alpha-RR},
is a suitable  estimate for the temperature variables, akin to the estimate that
  lies at the core of the proof of \eqref{strict-positivity-intro}.
 \par
Indeed,  for proving \eqref{strict-positivity-intro}    we will revisit a powerful technique, advanced in
\cite{SS},
 that consists in testing the heat equations \eqref{eqI} and \eqref{eq-tetas} by the \emph{negative} powers  $-\teta^{-p}$ and $-\tetas^{-p}$, respectively, with  \RCORR an \emph{arbitrary} $p>2$.  \EEE As it will be shown in Section \ref{ss:3.2},
 this leads
 to an estimate for $\frac1\teta$ and $\frac1{\tetas}$ in $L^\infty (0,T;L^{p-1}(\Omega))$ and
 $L^\infty (0,T;L^{p-1}(\GC))$, respectively. Letting $p\to\infty$ leads to an estimate for the reciprocal temperatures in $L^\infty(\Omega{\times}(0,T))$ and  $L^\infty(\GC{\times}(0,T))$, which gives
\eqref{strict-positivity-intro}.
\par
It turns out that a closely related idea will   also allow  us to `tame' the $L^1$ right-hand sides of the heat equations. For that, the key issue is estimating the spatial gradient of the  temperatures $\teta$ and $\tetas$. This will result from testing  \eqref{eqI} and \eqref{eq-tetas} by
$ \theta^{\nu -1}$
and $ \thetas^{\nu-1}$, respectively, for an arbitrary $\nu \in (0,1)$ 
(cf.\ Section\  \ref{sss:3.3.3} ahead). 
This will  lead
to the bounds
\begin{equation} \label{3rd-estimate-INTRO}
\| \theta^{(\mu{+}\nu)/2} \|_{L^{2}(0,T;H^1(\Omega))}
+\| \thetas^{(\mu{+}\nu)/2} \|_{L^{2}(0,T;H^1(\GC))} \RCOMMN \leq C \EEE 
\end{equation}
\RCORR (recall that the exponent $\mu$ featured in the growth condition \eqref{growth-alpha-RR}). \EEE
In turn, via  interpolation arguments, \eqref{3rd-estimate-INTRO} shall  bring to
 higher integrability estimates  for the temperature variables, which will  allow us to estimate their derivatives $\teta_t$ and $\partial_t\tetas$ in $L^1(0,T;W^{1,3+\epsilon}(\Omega)^*)$ and $L^1(0,T;W^{1,2+\epsilon}(\GC)^*)$ for all $\epsilon>0$, respectively.
Clearly,   from
the estimates of  the  gradients and the time derivatives of $\teta$ and $\tetas$ we will
extract all the compactness information necessary for dealing with the heat equations. The analysis of the momentum balance and of the flow rule  will follow more standard paths.
\par
\RCORR The estimates described above will be formally developed in  Sections\  \ref{ss:3.2} and \ref{s:3-calculations}. 
Making them rigorous \EEE
  in the frame of a time discretization scheme for system \eqref{PDE-INTRO}, which might be conducive to its numerical analysis,
has been a challenging issue by itself.  First of all, in devising the approximation scheme for \eqref{PDE-INTRO}
we have had to carefully balance the terms to be kept implicit with those to be kept explicit. In this way,  we have ensured
 the validity of a discrete form of the total energy balance associated with \eqref{PDE-INTRO}, whence all the basic energy estimates stem. Secondly, we have had to combine time discretization with an additional regularization obtained by (1) adding the higher order terms $-\rho \mathrm{div}(|\tensoret |^{\omega-2} \tensoret )$ and $ \rho |\chi_t|^{\omega-2}\chi_t\,,  \omega>4,  $ to the momentum balance and to the flow rule for $\chi$;
 (2) replacing the maximal monotone operators in the flow rule for $\chi$ and in the boundary condition for $\uu$ on $\GC$ by their Yosida \RCORR regularizations.  \EEE The reason for this
   \emph{threefold} approximation procedure essentially
 resides in the fact that, on the time-discrete level,
we shall not be able to
fully carry out the
arguments from \cite{SS}, leading to a uniform, in space and time, estimate for the reciprocal temperatures.
Namely,
 for the discrete bulk and surface temperatures, we shall only  prove a  strict positivity
 property, but not a lower bound by a positive constant as in
 \eqref{strict-positivity-intro}. Therefore, in order to rigorously perform the test of
  the temperature equations by negative powers of $\teta$ and $\tetas$
  that leads to \eqref{3rd-estimate-INTRO}, we will need to work on the  regularized version of system \eqref{PDE-INTRO} described in the above lines, cf.\ system \eqref{PDE-regul} ahead.
  \par
 We believe that the formal estimates from  Sections  \ref{ss:3.2} and \ref{s:3-calculations}, 
 as well as  the technical machinery rigorously supporting them developed in Sections \ref{s:5} and \ref{s:6},  are robust enough to be applied to other thermodynamically consistent models in solid mechanics. In particular, the analysis of a PDE system for damage in thermo-visco-elastic materials will be carried out in \RCOMMN future work, \EEE 
 in which the issues related to the unidirectionality of the evolution of the damage parameter will also be addressed.
 \paragraph{\bf Plan of the paper.}
 \RCORR In Section \ref{s:3}, after settling some preliminary results and all our conditions on the constitutive functions of the model, and on the problem data, we will consider a generalized version of system \eqref{PDE-INTRO} and introduce
  our notion of `weak energy solution' to the associated Cauchy problem. We will then state our main existence result, Theorem \ref{thm:1}.
Throughout   Section \ref{s:4} we will  carry out in a formal way all the calculations that provide the strict positivity properties \eqref{strict-positivity-intro}, 
 and all the a priori estimates at the core of the proof of
Theorem \ref{thm:1}.   In Section \ref{s:3.4} we will then introduce the regularized 
 system
 \eqref{PDE-regul}   on which  all estimates will be rigorously performed. The existence of solutions to \RCORR the Cauchy problem for system \eqref{PDE-regul}  will be proved, via a careful time discretization procedure, throughout Sections \ref{s:5}--\ref{s:6}. Finally, in Section \ref{s:7} we will take the limit of the regularized system in two steps, and thus conclude the proof of Theorem \ref{thm:1}.  \EEE


\section{The main result}
\label{s:3}
Let us  fix some general notation that will be used throughout the paper.
\begin{notation}
\upshape
 For a given a Banach space $X$, we will  denote by
 $\pairing{}{X}{\cdot}{\cdot}$ the duality pairing
between   $X'$ and $X$; to avoid overburdening notation, we shall write   $\Vert\cdot\Vert_X$
both the norm in $X$ and in any power of it.
\par
We will work with  the space
\[
\begin{aligned}
\bsVD:=\{{\bf v}\in H^1 (\Omega;\R^3)\, : \ {\bf v}={\bf 0}\hbox{ a.e. on }\GD \}\,,
\end{aligned}
\]
endowed with the natural norm induced by
$H^1 (\Omega;\R^3)$, and denote  the Laplace operator with homogeneous boundary conditions 
by
\begin{equation}
\nonumber
A:H^1(\GC) \to H^1(\GC)^* \qquad \pairing{}{H^1(\GC)}{A\chi}{w}:= \int_{\GC} \nabla
\chi \nabla w \dd x \  \text{ for all }\chi, \, w \in H^1(\GC).
\end{equation}
Moreover, we shall
use  special  notation for the following  function space
\[
\begin{aligned}
 \bsY:
= H^{1/2}_{00,\GD}(\GC;\R^3) =
\Big\{ \mathbf{w} \in H^{1/2}(\GC;\R^3)\, :  \ \exists\,
\widetilde{\mathbf{w}}\in H^{1/2}(\partial\Omega;\R^3) \text{ with }
\widetilde{\mathbf{w}}=\mathbf{w}  \text{ in $\GC$,} \
\widetilde{\mathbf{w}}=\mathbf{0} \text{ in $\GD$} \Big\}\,.
\end{aligned}
\]
\end{notation}
\paragraph{\bf Preliminary results.}
Throughout the paper, we will also use that
\begin{equation}
\label{embedding}
\bsVD \subset L^4(\GC) \text{ continuously, and } \bsVD \Subset L^{4-s}(\GC) \text{ compactly for all  }  s \in (0,3],
\end{equation}
where the above embeddings have to be understood in the sense of traces.
\par
\RNEW Finally,
we shall resort to the following \emph{nonlinear}  Poincar\'{e}-type inequality
 (cf.\  e.g.\ \cite[Lemma 2.2]{gmrs})
  \begin{equation}
 \label{poincare-type}
 \forall\, q>0 \quad \exists\, C_q >0 \quad \forall\, w \in H^1(\Omega)\, : \qquad
 \| |w|^{q} w \|_{H^1(\Omega)} \leq C_q (\| \nabla (|w|^{q} w )\|_{L^2(\Omega)} + |m(w)|^{q+1})\,,
\end{equation}
\RCORR (with  $m(w)$ the mean value of $w$), \EEE
and to the well-known interpolation formula for Lebesgue spaces, holding for every measurable $O \subset \R^d$, $d\geq 1$:
\begin{equation}
\label{interpolation-Lebesgue}
L^r(0,T;L^s(O)) \cap L^p(0,T;L^q(O)) \subset L^a(0,T;L^b(O)) \quad \text{with } \quad \begin{cases}
\frac1a = \frac{\vartheta}r + \frac{1-\vartheta}p,
\\
\frac1b = \frac{\vartheta}s + \frac{1-\vartheta}q,
\end{cases}
\text{ for some } \vartheta \in (0,1)
\end{equation}
with a continuous embedding.
\EEE

\subsection{Setup and assumptions}
\label{ss:2.1}
We start by detailing our conditions on the reference configuration:
 \begin{equation}
 \label{ass:domain}
 \begin{gathered}
 \text{
 $\Omega$ is a
bounded   Lipschitz  domain in $\R^3$, with }
\\
\partial\Omega= \overlineGD \cup \overlineGN \cup \overlineGC, \ \text{ $\Gdir$, $\Gnew$, $\GC$,
 open disjoint subsets in the relative topology of $\partial\Omega$,
such that  }
\\
\mathcal{H}^{2}(\Gdir), \,  \mathcal{H}^{2}(\GC)>0,
\text{ and ${\GC} \subset \R^2$ a
\emph{flat} surface,}
\end{gathered}
\end{equation}
which means that $\GC$ is a subset of a hyperplane of
$\R^3$ and on $\GC$  the Lebesgue and Hausdorff measures $\mathcal{L}^2 $ and $
\mathcal{H}^2$ coincide.
\par
Let us now fix
\begin{enumerate}
\item
the properties of the elasticity and viscosity tensors: 
 we assume that the fourth-order tensors $\etens=(e_{ijkh})$ and  $\vtens=(v_{ijkh})$,
satisfy the classical symmetry and ellipticity conditions
\begin{subequations}
\label{ass-K}
\begin{equation}
\label{ass-K-sym-ell}
\begin{aligned}
& e_{ijkh}=e_{jikh}=e_{khij}\,,\quad  \quad
v_{ijkh}=v_{jikh}=v_{khij}\,,
  \quad   i,j,k,h=1,2,3,
\\
&  \exists \, \eps_0>0 \quad  \forall\, \xi_{ij}\colon \xi_{ij}=
\xi_{ji}\,,\quad i,j=1,2,3 \,:  \qquad e_{ijkh} \xi_{ij}\xi_{kh}\geq
\eps_0\xi_{ij}\xi_{ij}\,,
\\
& \exists \, \nu_0>0 \quad \forall\, \xi_{ij}\colon \xi_{ij}=
\xi_{ji}\,,\quad i,j=1,2,3 \,:  \qquad v_{ijkh} \xi_{ij}\xi_{kh}\geq
\nu_0\xi_{ij}\xi_{ij}\,,
\end{aligned}
\end{equation}
where the usual summation convention is used.
 Moreover, we require that
\begin{equation}
\label{ass-K-bdd}
e_{ijkh}, v_{ijkh} \in L^{\infty}(\Omega)\,, \quad  i,j,k,h=1,2,3.
\end{equation}
\end{subequations}
Observe that conditions \eqref{ass-K} are compatible with the properties  of an anisotropic and
inhomogeneous material. They ensure that
 the  associated  bilinear forms $\efm, \vfm : H^1 (\Omega;\R^3) \times H^1 (\Omega;\R^3)
\to \R$,   defined by
$$
\begin{aligned}
&
\efm({\bf u},{\bf v}):=\int_{\Omega} e_{ijkh} \varepsilon_{kh}({\bf u})
\varepsilon_{ij}({\bf v}) \dd x   &&   \text{for all } \uu, \vv \in H^1 (\Omega;\R^3),
\\
&
\vfm({\bf u},{\bf v}):=\int_{\Omega} v_{ijkh} \varepsilon_{kh}({\bf u})
\varepsilon_{ij}({\bf v}) \dd x   &&   \text{for all }  \uu, \vv \in H^1 (\Omega;\R^3)
\end{aligned}
$$
are continuous and symmetric, i.e.\
\begin{equation}
\label{continuity} \exists \, M >0: \ |\efm(\uu, \vv)| +   |\vfm(\uu,
\vv)| \leq M \| \uu\|_{H^1 (\Omega)} \| \vv\|_{H^1 (\Omega)} \quad \text{for all }
\uu, \vv \in H^1 (\Omega;\R^3).
\end{equation}
Furthermore, since $\GD$ has positive measure,
 by Korn's inequality we deduce that  \RCORR the forms  $\efm(\cdot,\cdot)$ and
$\vfm(\cdot,\cdot)$ are $H^1 (\Omega;\R^3)$-elliptic on $\bsVD\times \bsVD$, i.e.\ \EEE there exist $C_{\mathrm{e}},
C_{\mathrm{v}}>0 $ such that
\begin{align}
\label{korn}
 \efm({\bf u},{\bf u})\geq C_{\mathrm{e}}\Vert{\bf u}\Vert^2_{H^1 (\Omega)},
 \qquad \vfm({\bf u},{\bf u})\geq C_{\mathrm{v}}\Vert{\bf u}\Vert^2_{H^1 (\Omega)} \qquad
\text{for all }\uu\in \bsVD.
\end{align}
\end{enumerate}
We will in fact deal with a extended version of system  \eqref{PDE-INTRO}, where
the subdifferentials $\partial I_{(-\infty,0]}$ and $\partial I_{[0,1]}$ will be replaced by general maximal monotone operators.
Namely,
\begin{enumerate}
 \setcounter{enumi}{1}
\item we     consider  a function
\begin{equation}
\label{hyp-alpha-eta}
\widehat{\eta} : \R \to [0,+\infty] \quad \text{proper, convex, and lower semicontinuous, with } \widehat{\eta}(0)=0
\end{equation}
 (note that, if  $0\in \mathrm{dom}(\widehat\eta)$, we can always reduce to the case
$\widehat{\eta}(0)=0$ by a translation).
 Then, we  introduce the proper, convex and lower semicontinuous functional
\[
\bwalpha:  \bsY \to [0,+\infty] \quad \text{ defined by } \quad \bwalpha(\uu): = \begin{cases}
\int_{\GC}\widehat{\eta}(\uu \cdot \mathbf{n}) \,\dd x & \text{if } \widehat{\eta}(\uu\cdot\mathbf{n}) \in L^1(\GC),
\\
+\infty &\text{otherwise}.
\end{cases}
\]
We set
$
\balpha: = \partial \bwalpha: \bsY\rightrightarrows \bsY^*.
$
It follows from \eqref{hyp-alpha-eta} that $\mathbf{0} \in \balpha(\mathbf{0})$.
The subdifferential $\balpha(\uu)$ shall replace the term
 $ \partial I_{(-\infty,0]}(\uu{\cdot} \mathbf{n}) \mathbf{n}
$  in the boundary condition  \eqref{condIii}. Observe that the impenetrability condition $\uu \cdot \mathbf{n}\leq 0$ a.e.\ on $\GC$ is rendered as soon as $\mathrm{dom}(\widehat\eta)\subset (-\infty,0] $.
\item
We also generalize the subdifferential $\partial I_{[0,1]}$ to the subdiffential of a function
 \begin{equation}
\label{hyp-beta}
\widehat{\beta} : \R \to [0,+\infty]  \text{ proper, convex, lower semicontinuous, with }
\mathrm{dom}(\widehat{\beta})\subset [0,1] \text{ and } \widehat{\beta}(0)=0,
\end{equation}
and set $\beta: = \partial\widehat{\beta}:\R \rightrightarrows \R$.
 \end{enumerate}
Observe that the integral terms encompassing nonlocal effects
  in \eqref{eqI-cont}, \eqref{condIii}, \eqref{eq-tetas}, and
\eqref{eqII} can be rewritten as
\[
\begin{aligned}
&
\begin{cases}
-\RCOMMN \theta^2(x)\EEE \chi(x) \int_{\GC} \kr(x,y) \chi(y) \dd y + \RCOMMN \theta (x)\EEE \chi(x)  \int_{\GC} \kr(x,y) \thetas(y) \chi(y) \dd y,
\\
 \uu(x) \chi(x) \int_{\GC} \kr(x,y) \chi(y)\dd y,
 \\
 \RCOMMN \thetas(x)\EEE \chi(x)  \int_{\GC} \kr(x,y)\theta(y) \chi(y) \dd y - \chi(x) \RCOMMN \thetas^2(x)\EEE  \int_{\GC} \kr(x,y) \chi(y) \dd y,
 \\
-\frac12 |\uu(x)|^2 \int_{\GC} \kr(x,y) \chi(y) \dd y -\frac12  \int_{\GC} \kr(x,y)  |\uu(y)|^2 \chi(y) \dd y
\end{cases} \qquad \qquad
\\ &  \RCORR  \text{ with  } \RCOMMN  \kr(x,y) : = \mu(|x-y|)\,. \EEE 
\end{aligned}
\]
It is thus natural to
generalize these terms by considering
\begin{enumerate}
 \setcounter{enumi}{3}
 \item
 a kernel
 \begin{equation}
 \label{hyp-k}
 \kr:\GC \times \GC \to [0,+\infty) \text{  symmetric, positive,  with } \kr \in L^\infty((\GC{\times}\GC);\R^+)
 \end{equation}
\end{enumerate}
and introducing the associated nonlocal operator
 \begin{equation}
 \label{def-nonloc-operator}
 \nlname: L^1(\GC)\to L^\infty(\GC) \qquad \nlocs w x: = \int_{\GC} \kr(x,y)w(y) \dd y \quad \text{for all } w \in L^1(\GC)\,.
 \end{equation}
Lemma  \ref{lemmaK} ahead  will provide  some key properties of the operator
$\nlname$.
\par
With the above outlined generalizations,
system \eqref{PDE-INTRO} turns into the PDE system
\begin{subequations}
	\label{PDE-true}
	\begin{align}
		\label{eqI-true}
		&
		\begin{aligned}
		\theta_t-\theta\mathrm{div}(\uu_t)-\mathrm{div}(\ab(\theta)\nabla\theta)
		= h + \tensoret \, \vtens \, \tensoret \quad  \text{ in }  \Omega\times (0,T),
		\end{aligned}
				\\
			\label{eqI-bdry-true}
		&\ab(\theta)\nabla\theta\cdot{\bf n}=0  \quad  \text{ in }   \GD\cup\Gnew\times (0,T),
		\\
		&
			\label{eqI-cont-true}
		\begin{aligned}
		\ab(\theta)\nabla\theta \cdot{\bf n}
		 =
		 -k(\chi)\teta(\theta{-}\thetas) - \nlocss {\chi} \chi \teta^2  + \nlocss {\chi\tetas} \chi \teta  \quad  \text{ in }     \GC \times (0,T),
		\end{aligned}
		\\
		&
		\label{mom-balance-true}
		-\mathrm{div}(\etens \tensore + \vtens \tensoret + \theta{\mathbb I})={\bf f}  \quad  \text{ in }  \Omega\times (0,T),\\
		\label{Dir-true}
		&\mathbf{u}= {\bf 0}  \quad  \text{ in }   \GD\times (0,T),
		\\
		&
		(\etens \tensore+\vtens \tensoret +\theta{\mathbb I}){\bf n}={\bf g}  \quad  \text{ in } \Gamma_{\mathrm{N}}\times (0,T),\label{condIi-true}\\
		&
		\begin{aligned}
	 (\etens\tensore + \vtens\tensoret   + \theta{\mathbb I}){\bf n}  +\chi{\bf u}+
		\balpha(\uu)  +\nlocss{\chi} \chi \uu \ni  {\bf 0}  \quad  \text{ in }    \GC \times (0,T),
		\end{aligned}
		\label{condIii-true}
		\\
		 &
		 \begin{aligned}
		 \label{eq-tetas-true}
 &  \partial_t {\thetas}-\thetas \lambda'(\chi)\chi_t-\mathrm{div}(\as(\thetas)\nabla\thetas)
 \\ & \quad
		=\ell+|\chi_t|^2+k(\chi)(\theta{-}\thetas)\tetas +
		\nlocss{\chi\teta}\chi \tetas-\nlocss{\chi}\chi\tetas^2  \quad  \text{ in }  \GC\times (0,T),
		\end{aligned}
		\\  \label{eq-tetas-b-true}
		& \as(\thetas)\nabla\thetas\cdot\RCOMMN {\nn_\mathrm{s}} \EEE=0  \quad  \text{ in }  \partial\GC\times (0,T),
		\\
		&
		\begin{aligned}
		\chi_t -\Delta\chi +\beta(\chi)   +
		\gamma'(\chi)+\lambda'(\chi)\thetas
		 \ni -\frac 1 2\vert{\uu }\vert^2-\frac 12\nlocss{\chi}|\uu|^2 - \frac 1 2\nlocss{\chi |\uu|^2}  \quad  \text{ in }   \GC
		\times (0,T),
		\end{aligned}
		\label{eqII-true}
				\\
&\partial_{\nn_\mathrm{s}} \chi=0   \quad  \text{ in }  \partial \GC \times
		(0,T),\label{bord-chi-true}
\end{align}
		\end{subequations}
 \RCORR that \EEE will be studied in the sequel (note that, here  in \eqref{eqII-true},
we have incorporated the term
			$-\lambda'(\chi)\teta_*$, \RCORR featuring \EEE  on the left-hand side of the former \eqref{eqII} into the function $\gamma'$).
Let us finally specify
 \begin{enumerate}
 \setcounter{enumi}{4}
 \item
 our requirements on the heat conductivity:  the function
$\ab:  [0,+\infty) \to \R^+$ is continuous and fulfills
\begin{equation}
\label{hyp-alpha}
\exists \, c_0,\, c_1>0 \quad \RCORR \exists\, \mu>1 \EEE \quad  \forall\, \teta\in  [0,+\infty) \, :  \qquad \qquad  c_0 (1+\theta^\mu) \leq \ab(\theta) \leq c_1 (1+\theta^\mu)\,.
\end{equation}
\RF We will work with its primitive  $\widehat{\alpha} : \R^+ \to \R^+ $  defined by \EEE
\begin{equation}
\label{def-hat-alpha}
\hhat{\alpha}(r): = \int_0^r \alpha(s) \dd s\,;
\end{equation}
 \item our conditions on the  nonlinear functions $k$, $\gamma$, and
 $\lambda$:
 \begin{align}
 &
 \label{hyp-kappa}
 \begin{aligned}
 &
 k: \R \to [0,+\infty) \quad \text{ has polynomial growth, i.e.}
 \\
 &
 \exists\, s>1 \,\, \exists\, C_{k}>0 \  \forall\, x\in \R \, : \quad
 \RCNEW
 k(x)  \EEE \leq C_{k}(|x|^s+1)\,;
 \end{aligned}
 \\ \label{hyp-lambda}
 &
 \lambda: \R \to \R \quad \text{ is Lipschitz continuous and $\delta$-concave for some  $\delta \in \R$};
\\
&
\begin{aligned}
&
\gamma\in \mathrm{C}^1(\R)  \text{ with } \gamma' :\R \to \R \text{ Lipschitz continuous,
$\nu$-convex for some $\nu \in \R$,
and such that }
\\ \label{hyp-W}
& \quad
\exists\, C_W>0 \ \forall\, x \in \R\, :  W(x): = \widehat{\beta}(x) + \gamma(x) \geq -C_W\,;
\end{aligned}
 \end{align}
 \item
 our conditions on the heat sources
$h$ and $\ell$ and on the forces
$\mathbf{f}$ and $ \mathbf{g}$:
\begin{subequations}
\label{cond-data}
\begin{align}
&
\label{cond-h}
h\in L^1(0,T;L^1(\Omega)) \cap L^2(0,T;H^1(\Omega)^*), \qquad h \geq 0 \ \aein \  \Omega \times (0,T),
\\
&
\label{cond-ell}
\ell \in L^1(0,T;L^1(\GC)) \cap L^2(0,T;H^1(\GC)^*), \qquad \ell \geq 0 \ \aein \  \GC \times (0,T),
\\
&
\label{cond-bf-f}
\mathbf{f}\in H^1(0,T;\bsVD^*),
\\
&
\label{cond-bf-h}
\mathbf{g}\in H^1(0,T;\bsY^*).
\end{align}
\end{subequations}
We then introduce the function
\begin{equation}
 \label{cond-F}
\mathbf{F}\in H^1(0,T;\bsVD^*), \qquad \pairing{}{\bsVD}{\mathbf{F}(t)}{\vv}: =  \pairing{}{\bsVD}{\mathbf{f}(t)}{\vv} +
 \pairing{}{\bsY}{\mathbf{g}(t)}{\vv}  \quad \foraa\, t \in (0,T).   \EEE
\end{equation}
\item
As for the initial data $\teta_0, \, \tetaso, \, \uu_0,\, \chi_0$  we suppose that
\begin{subequations}
\label{cond-init}
\begin{align}
&
\label{cond-teta0}
\teta_0\in L^1(\Omega) \qquad \text{with } \quad \inf_{x\in\Omega}\teta_0(x) \geq \theta^*>0,
\\
&
\label{cond-tetaso}
\tetaso\in L^1(\GC) \qquad \text{with } \quad \inf_{x\in\GC}\tetaso(x) \geq \thetas^*>0,
\\
&
\label{cond-u0}
\uu_0\in \bsVD, \qquad  \RCOMMN \uu_0\in \mathrm{dom} (\widehat{\balpha}), \EEE
\\
&
\label{cond-chi0}
 \chi_0\in H^1(\GC), \qquad   \widehat{\beta}(\chi_0) \in L^1(\GC). 
\end{align}
\end{subequations}
\end{enumerate}
 \begin{remark}
 \label{rmk:on-assumptions}
 \upshape
 Observe that the $\delta$-concavity and $\nu$-convexity requirements for $\lambda$
 and $\gamma$ mean that
 \[
\text{ \RCORR the function \EEE } \quad \begin{cases}
r\mapsto \lambda(r) -\tfrac \delta2 r^2 \text{ is  concave;}
\\
r\mapsto \gamma(r) +\tfrac \nu 2 r^2 \text{ is convex.}
\end{cases}
 \]
 These properties will be used for devising a time-discretization scheme of system \eqref{PDE-true} such that the validity of a discrete form of the total energy inequality is ensured, cf.\ Remark
 \ref{rmk:comm-discretiz} ahead.
 \par
 Also the growth condition for $k$ from  \eqref{hyp-kappa} is functional to our approximation scheme, or rather serves to the purpose of simplifying it, cf.\ Remark \ref{rmk:could-be-dispensed} ahead.
 \end{remark}

\subsection{Our existence result}
\label{ss:2.2}
We will prove the existence of weak solutions in the sense specified by Definition \ref{def:weak-sol} below. We mention in advance that
our notion of `weak energy solution' to (the Cauchy problem for)
system \eqref{PDE-true}  features
\begin{itemize}
\item the weak formulation of the   heat equations  \eqref{eqI} and  \eqref{eq-tetas}
with test functions $v\in W^{1,3+\epsilon}(\Omega)$ and $w\in W^{1,2+\epsilon}(\GC)$ for any $\epsilon>0$;
\item the standard weak formulation of the displacement equation \eqref{mom-balance-true}, with test functions in
$\bsVD$;
\item the pointwise formulation  (a.e.\ in $\GC \times (0,T)$) of the flow rule  for the adhesion parameter; 
\item
the \emph{total energy balance}
\begin{equation}
\label{total-enbal}
\begin{aligned}
& \calE(\teta(t),\tetas(t),\uu(t),\chi(t))
  +\int_s^t \int_{\GC} k(\chi(x,r))(\teta(x,r){-}\tetas(x,r))^2 \dd x \dd r
 \\
 &
 \qquad +
\int_s^t \RCOMMN \iint_{\GC\times \GC} \EEE \kr(x,y)\chi(x,r)\chi(y,r)(\teta(x,r){-}\tetas(y,r))^2 \dd x \dd y  \dd r
\\
&
=
\calE(\teta(s),\tetas(s),\uu(s),\chi(s)) + \int_s^t \int_\Omega h \dd x \dd r + \int_s^t \int_{\GC} \ell \dd x \dd r +\int_s^t \pairing{}{\bsVD}{\mathbf{F}}{\uu_t} \dd r 
\end{aligned}
\end{equation}
\RCORR for every $0 \leq s \leq t \leq T$, \EEE
featuring
 the stored energy of the system
\begin{equation}
\label{stored-energy}
\begin{aligned}
\calE(\teta,\tetas,\uu,\chi): =  & \int_\Omega \teta \dd x   + \int_{\GC}\tetas \dd x
\\ &  \
+\frac12 \efm(\uu,\uu) + \bwalpha(\uu)
+\frac12 \int_{\GC} \left( \chi|\uu|^2 {+}\chi|\uu|^2 \nlocss\chi \right) \dd x
+\int_{\GC} \left( \frac12|\nabla\chi|^2 {+}W(\chi)\right) \dd x.
\end{aligned}
\end{equation}
\end{itemize}
\begin{definition}
\label{def:weak-sol}
Given  initial data $(\teta_0,\tetaso,\uu_0,\chi_0)$ fulfilling \eqref{cond-init}, we call
 a quadruple $(\teta,\tetas,\uu,\chi)$ a \emph{weak energy solution} to the Cauchy problem for system
  \eqref{PDE-true} if
   \begin{subequations}
   \label{reg-entropic}
   \begin{align}
   &
   \label{reg-teta}
   \teta \in L^2(0,T;H^1(\Omega)) \cap L^\infty (0,T;L^1(\Omega)) \cap W^{1,1}(0,T;W^{1,3+\epsilon}(\Omega)^*),
   \\
   &
      \label{reg-hatalpha-teta}
 \widehat{\alpha}(\teta) \in L^1(0,T;W^{1,(3+\epsilon)/(2+\epsilon)}(\Omega)),    && \text{for all } \epsilon>0,
   \\
      \label{reg-tetas}
   & \tetas \in L^2(0,T;\VC) \cap L^\infty (0,T;L^1(\GC))  \cap W^{1,1}(0,T;W^{1,2+\epsilon}(\GC)^*), 
   \\
   &
      \label{reg-hatalpha-tetas}
    \widehat{\alpha}(\tetas) \in L^1(0,T;W^{1,(2+\epsilon)/(1+\epsilon)}(\GC)), && \text{for all } \epsilon>0,
   \\
   &
   \label{reg-uu}
   \uu \in H^1(0,T;\bsVD), &&
   \\
   &
   \label{reg-chi}
   \chi \in L^2(0,T;H^2(\GC)) \cap L^\infty(0,T;\VC) \cap H^1(0,T;\HC), &&
   \end{align}
   \end{subequations}
   the quadruple $(\teta,\uu, \tetas, \chi)$ comply with the initial conditions
   \begin{equation}
   \label{init-cond}
   \begin{aligned}
   &
  \teta(x,0) = \RCOMMN \teta_0(x), \quad  \uu(x,0) = \uu_0(x)\quad  && \RCOMMN \foraa\, x \in \Omega,
   \\
   & \RCOMMN \tetas(x,0) =\tetaso(x), \quad \chi(x,0)=\chi_0(x) \quad  && \RCOMMN \foraa\, x \in \GC, \EEE
   \end{aligned}
   \end{equation}
   and with the \emph{positivity} properties
      \begin{equation}
   \label{teta-pos}
   \begin{aligned}
   &
   \teta(x,t) >0  &&   \foraa\, (x,t) \in \Omega \times (0,T)\,,
    \\
&
 \tetas(x,t)>0  &&   \foraa\, (x,t) \in \GC \times (0,T)\,,
    \end{aligned}
   \end{equation}
and there exist
\begin{equation}
\label{reg-selections}
\zzeta \in L^2(0,T;\bsY^*), \qquad \xi \in L^2(0,T;\HC)
\end{equation}
such that the functions $(\teta,\uu,\tetas,\chi,\zzeta,\xi)$ fulfill
\begin{itemize}
   \item the weak formulation of the bulk heat equation
   \begin{equation}
 \label{weak_theta}
\begin{aligned}
  &  \pairing{}{W^{1,3+\epsilon}(\Omega)}{\theta_t}{v} 
- \io \theta\mathrm{div}(\uu_t) v \dd x
+ \io \nabla(\widehat{\alpha}(\theta)) \cdot  \nabla v \dd x
+ \int_{\GC} k(\chi) \theta (\theta - \thetas) v \dd x
\\
& \qquad
+ \int_{\GC} \nlocss{\chi} \chi \theta^2 v \dd x
- \int_{\GC} \nlocss{\chi \thetas} \chi \theta v \dd x
\\
& = \io \tensoret \mathbb{V} \tensoret v  \dd x + \io h v \dd x \qquad \text{for all   } v \in W^{1,3+\epsilon}(\Omega), \ \epsilon>0,  \qquad \aein\, (0,T);
\end{aligned}
\end{equation}
   \item the weak formulation of the surface heat equation
\begin{equation} \label{weak_thetas}
\begin{aligned}
&
 \pairing{}{W^{1,2+\epsilon}(\GC)}{\partial_t {\thetas}}{w} 
- \int_{\GC} \thetas \lambda'(\chi) \chi_t w \dd x
+ \int_{\GC}  \nabla(\widehat{\alpha}(\thetas)) \cdot  \nabla w  \dd x
\\
& =
\int_{\GC} \ell w \dd x +
  \int_{\GC} |\chi_t|^2 w \dd x
+ \int_{\GC}  k(\chi)\thetas(\theta - \thetas) w \dd x
+ \int_{\GC} \nlocss{\chi\theta} \chi \thetas w \dd x
- \int_{\GC} \nlocss{\chi} \chi \thetasq  w \dd x
\\
& \qquad \qquad
 \qquad \text{for all  } w \in W^{1,2+\epsilon}(\GC), \ \epsilon>0, \qquad \aein\, (0,T);
 \end{aligned}
\end{equation}
\item the weak formulation of the displacement equation
\begin{subequations}
\label{weak-U}
\begin{align}
&
\label{weak-displ}
\vfm(\uu_t,\mathbf{v}) + \efm(\uu,\mathbf{v})
+\int_{\Omega} \teta \mathrm{div}(\vv) \dd x +
 \int_{\GC} \chi\uu \mathbf{v} \dd x +
 \langle \zzeta,\mathbf{v} \rangle_{\bsY}
+  \int_{\GC} \chi\uu \nlocss\chi \mathbf{v} \dd x =
 \pairing{}{\bsVD}{\mathbf{F}}{\vv}
 \\
 \intertext{for all $ \vv \in \bsVD$, with }
 \label{zeta}
 &  \zzeta\in L^2(0,T;\bsY^*) \text{ fulfilling }  \zzeta(t) \in \balpha(\uu(t))  \text{ in } \bsY^* \ \foraa\, t \in (0,T);
\end{align}
\end{subequations}
\item \RNEW the pointwise formulation of the flow rule for the adhesion parameter \EEE
   \item the \emph{total energy balance}  \eqref{total-enbal}.
%
\end{itemize}
\end{definition}
\par
We are now in a position to state the main result of the paper: \RCNEW for technical reasons related to our approximation scheme,  we will 
  prove  the existence of a weak energy
solution  such that  the pointwise flow rule for the adhesion parameter  holds with an additional measurable coefficient $\sigma = \sigma(x,t)\in [0,1]$  for the terms on its right-hand side. However, we point out that the function $\sigma$  can take values different from $1$ only on the
set $\{ \chi =0 \}$. \EEE
\begin{maintheorem}[Global existence of weak  \RCORR energy \EEE solutions]
\label{thm:1} 
Assume
\eqref{ass:domain}--\eqref{hyp-k} and \eqref{hyp-alpha}--\eqref{cond-data}.
Then, for every \RCORR quadruple of initial data \EEE  $(\theta_0,\tetaso,\uu_0,\chi_0)$ as in \eqref{cond-init} there exists  a  \emph{weak energy solution} $(\theta,\thetas,\uu,\chi)$ to the Cauchy problem for system
\eqref{PDE-true}, 
\RF with an associated selection $\zzeta$ fulfilling \eqref{weak-displ}--\eqref{zeta}, and a pair $(\xi,\sigma)
\in L^2(0,T;L^2(\GC)) \times  L^\infty(\GC{\times}(0,T))$ \EEE such that
 the pointwise formulation of the flow rule for $\chi$ holds in the following form:
\begin{subequations}
\label{weak-chi-sigma}
 \begin{align}
&  \chi_t+ A\chi +\xi + \gamma'(\chi) +\lambda'(\chi)\tetas= -\frac12 |\uu|^2 \sigma -\frac12 \nlocss{\chi}\, |\uu|^2 \sigma -\frac12 \nlocss{\chi|\uu|^2}\, \sigma \qquad \aein\, \GC \times (0,T),
 \\
 & \text{with } \xi \in \beta(\chi) \qquad \aein\  \GC\times (0,T),
 \\
 & \text{and } \sigma \begin{cases}
  \equiv 1 & \text{on } \{(x,t)\in \GC\times (0,T)\, : \ \chi(x,t)>0\},
 \\
 \in [0,1] & \text{on } \{(x,t)\in \GC\times (0,T)\, : \ \chi(x,t)=0\}.
 \end{cases}
 \end{align}
 \end{subequations} \EEE
 In addition, $\theta$ and
$\thetas$ comply with the positivity properties
\begin{equation}
\label{teta-strict-pos}
\theta \geq \overline{\theta}> 0  \quad \text{a.e. in $\Omega \times (0,T)$}, \qquad
\thetas \geq \overline{\theta}_{\mathrm{s}} > 0 \ \quad \text{a.e. in $\GC \times (0,T)$}
\end{equation}
for some positive constants $\overline\theta$ and $\overline{\theta}_{\mathrm{s}}$.
\end{maintheorem}


\EEE
\section{Formal a priori estimates and strategy of the proof of Theorem \ref{thm:1}} \label{s:4}
In this Section we derive the basic a priori estimates on the solutions
to system \eqref{PDE-true},
that are at the core of our definition of \emph{weak  energy  solution}, by carrying out a series of  formal calculations in Section \ref{s:3-calculations} \RCORR ahead. \EEE
Prior to that, we will fix some preliminary results in Sec.\ \ref{ss:3.1} and, again formally, prove the strict positivity of the temperature variables in Sec.\ \ref{ss:3.2}.
 All  \RCORR the calculations in Sec.\ \ref{ss:3.2} and \ref{s:3-calculations} \EEE will be rendered rigorously in the context of  (the time discretization scheme for) a suitable approximation of system  \eqref{PDE-true}, set forth in Sec.\ \ref{s:3.4}. Therein, we will also outline the scheme of the proof of Theorem \ref{thm:1}.
 \par
In what follows, we shall work under the  assumptions listed in
Section \ref{ss:2.1};  in particular,
we  will omit to explicitly invoke them in the statements of Lemmas \ref{lemmaK} and \ref{LemmaE}.
\par
Finally, let us point out that
throughout the paper,
we will use the symbols $c, c', C, C',\ldots$, with meaning that  possibly varies in the same line,  to denote several positive constants only depending on known quantities. Analogously, with the symbols $I_1,I_2,\ldots$ we will denote several integral terms appearing in the estimates.
\subsection{Preliminaries}
\label{ss:3.1}
We will extensively use the following result
\RCORR (cf., e.g., \cite{BBRnonloc1}), \EEE
 collecting key properties of the nonlocal operator $\nlname$ from \eqref{def-nonloc-operator}.
\begin{lemma}\label{lemmaK}
The operator
$\nlname: L^1(\GC) \to L^\infty(\GC)$  is well defined, linear and bounded,
with
\begin{equation}
\label{stima-puntuale}
\RCNEW  \| \nlocss w\|_{L^\infty(\GC)} \leq \| j \|_{L^\infty(\GC\times \GC)}  \|w\|_{L^1(\GC)} \quad \text{for all } w \in L^1(\GC); \EEE
\end{equation}
$\nlname$ also enjoys the positivity property
\begin{equation}
\label{positivity-J}
w \geq 0 \quad \aein\  \GC \ \ \Rightarrow \
\nlocss w \geq  0 \qquad \aein\, \ \GC\,.
\end{equation}
 Furthermore,  for every $1\leq p<\infty$  the operator $\nlname$ is
 continuous from $L^1(\GC)$, equipped with the weak topology, to  $L^p(\GC)$  with the strong topology, namely
    if $w_n\debole w$ in $L^1(\GC)$ then $\nlocss{w_n}\rightarrow \nlocss{w}$ in $L^p(\GC)$.  Finally, there holds
\begin{equation}
\label{nl-symm}
\int_{\GC} \nlocs {w_1}x \, w_2(x) \dd  x = \int_{\GC} \nlocs {w_2}x \, w_1(x) \dd x \qquad \text{for all } w_1,\, w_2 \in L^1(\GC)\,.
\end{equation}
\end{lemma}
\paragraph{\bf Variational formulations \RCORR of the heat equations.}
In the following  calculations, we shall (formally) use
the variational formulation of the \RCORR boundary-value problem  \eqref{eqI-true}--\eqref{eqI-cont-true}  \EEE for the bulk heat equation, namely
\begin{equation}
 \label{form_debole_theta}
\begin{aligned}
  & \io \theta_t v \dd x
- \io \theta\mathrm{div}(\uu_t) v \dd x
+ \io \ab(\theta) \nabla\theta \nabla v \dd x
+ \int_{\GC} k(\chi) \theta (\theta - \thetas) v \dd x
\\
& \qquad
+ \int_{\GC} \nlocss{\chi} \chi \theta^2 v \dd x
- \int_{\GC} \nlocss{\chi \thetas} \chi \theta v \dd x
\\
& = \io \tensoret \mathbb{V} \tensoret v  \dd x + \io h v \dd x \qquad \text{for all suitable test functions } v, \qquad \aein\, (0,T)
\end{aligned}
\end{equation}
and of   \RCORR the boundary value problem \eqref{eq-tetas-true}--\eqref{eq-tetas-b-true} for the \EEE surface heat equation, namely
\begin{equation} \label{form_debole_thetas}
\begin{aligned}
&
  \int_{\GC} \partial_t {\thetas} w \dd x
- \int_{\GC} \thetas \lambda'(\chi) \chi_t w \dd x
+ \int_{\GC} \as(\thetas) \nabla\thetas \nabla w  \dd x
\\
& =
\int_{\GC} \ell w \dd x +
  \int_{\GC} |\chi_t|^2 w \dd x
+ \int_{\GC}  k(\chi)\thetas(\theta - \thetas) w \dd x
+ \int_{\GC} \nlocss{\chi\theta} \chi \thetas w \dd x
- \int_{\GC} \nlocss{\chi} \chi \thetasq  w \dd x
\\
& \qquad \qquad
 \qquad \text{for all suitable test functions } w, \qquad \aein\, (0,T).
 \end{aligned}
\end{equation}
We have been purposefully imprecise in
\eqref{form_debole_theta}
and \eqref{form_debole_thetas} since, in any case,
the choices  of the test functions  that we will
make in the  calculations carried out in \RCORR  Sections \ref{ss:3.2} and \ref{s:3-calculations} \EEE
 will be only formal.
\paragraph{\bf Derivation of the total energy balance \eqref{total-enbal}.}
We test the bulk heat equation
\eqref{form_debole_theta}
by $1$, 
 the displacement equation
 \eqref{mom-balance-true}
  by $\uu_t$, 
 the surface heat equation
 \eqref{form_debole_thetas}
 by $1$, and the flow rule
 \eqref{eqII-true}
 for $\chi$ by $\chi_t$.
 Adding up the resulting relations, observing the cancellation of some terms, and
  integrating on  a time interval $(s,t)\subset(0,T)$,  we obtain
  \begin{equation}
  \label{energy-test}
  \begin{aligned}
  &
  \int_s^t \int_\Omega \teta_t \dd x \dd r
  + \ddd{\int_s^t \int_{\GC} k(\chi)(\theta{-}\tetas)\theta \dd x \dd r}{$\doteq I_1$}
  + \ddd{\int_s^t \iint_{\GC\times\GC} \kr(x,y) \chi(x)\theta(x) \chi(y)(\teta(x){-}\tetas(y))  \dd x \dd y \dd r}{$\doteq I_2$}
\\
&  \quad +\int_s^t \efm(\uu,\uu_t) \dd r
  +\ddd{\int_s^t  \langle \balpha(\uu), \uu_t \rangle_{\bsY} \dd r }{$\doteq I_3$}
  +\ddd{\int_s^t \int_{\GC}\chi \uu \uu_t \dd x \dd r}{$\doteq I_4$}
  \\
   & \quad + \ddd{\int_s^t \iint_{\GC\times\GC} \kr(x,y) \chi(x)\uu(x) \uu_t(x) \chi(y)   \dd x \dd y \dd r}{$\doteq I_5$}
+  \int_s^t \int_{\GC} \partial_t \tetas \dd x \dd r
  -  \ddd{\int_s^t \int_{\GC} k(\chi)(\theta{-}\tetas)\tetas \dd x \dd r}{$\doteq I_6$}
  \\
  &\quad
 -  \ddd{\int_s^t \iint_{\GC} \kr(x,y) \chi(x)\thetas(x) \chi(y)(\teta(y){-}\tetas(x))  \dd x \dd y \dd r }{$\doteq I_7$}
 +\int_s^t\int_{\GC} \nabla\chi \cdot \nabla\chi_t \dd x \dd r
 \\
 &\quad
  +\ddd{\int_s^t \int_{\GC} \beta(\chi)\chi_t  \dd x \dd r}{$\doteq I_8$}
 +\ddd{\int_s^t \int_{\GC} \gamma'(\chi)\chi_t  \dd x \dd r}{$\doteq I_{9}$}
 +\ddd{\int_s^t \int_{\GC} \frac12 |\uu|^2\chi_t \dd x \dd r}{$\doteq I_{10}$}
 \\
 & \quad
  +\ddd{\int_s^t \iint_{\GC\times\GC} \frac12 \kr(x,y) \chi_t(x)|\uu|^2(x) \chi(y) \dd x \dd y \dd r}{$\doteq I_{11}$}
 +\ddd{\int_s^t \iint_{\GC} \frac12 \kr(x,y) \chi_t(x)|\uu|^2(y) \chi(y) \dd x \dd y \dd r }{$\doteq I_{12}$}
 \\
 &
 =\int_s^t \int_\Omega h \dd x \dd r + \int_s^t \int_{\GC} \ell \dd x \dd r +\int_s^t \pairing{}{\bsVD}{\mathbf{F}}{\uu_t} \dd r,
 \end{aligned}
  \end{equation}
 where we have formally written the subdifferentials $\balpha(\uu)$ and $\beta(\chi)$ as
  if  singletons.
  We then observe that
  \begin{align}
  &
  I_1-I_6= \int_s^t \int_{\GC} k(\chi)(\theta{-}\tetas)^2 \dd x \dd r,
  \nonumber
  \\
  &
  \label{key-posit-conto}
  \begin{aligned}
  I_2-I_7 & =
  \int_s^t \iint_{\GC\times\GC} \kr(x,y) \chi(x)\theta(x) \chi(y)(\teta(x){-}\tetas(y))  \dd x \dd y \dd r
  \\
  & \quad
  -\int_s^t \iint_{\GC\times\GC} \kr(x,y) \chi(y)\thetas(y) \chi(x)(\teta(x){-}\tetas(y))  \dd x \dd y \dd r
 \\ &  \stackrel{(1)}{=}   \int_s^t \iint_{\GC} \kr(x,y) \chi(x) \chi(y)(\teta(x){-}\tetas(y))^2   \dd x \dd y \dd r
  \end{aligned}
  \\
  &
  \nonumber
  I_3 \stackrel{(2)}{=}  \widehat{\balpha}(\uu(t)) -  \widehat{\balpha}(\uu(s)), \quad
 I_8+I_{9} \stackrel{(3)}{=}  \int_{\GC} W(\chi(t)) \dd x -  \int_{\GC} W(\chi(s)) \dd x,
 \\
 &
 \nonumber
 I_4+I_{10} = \int_{\GC} \frac12\chi(t)|\uu(t)|^2 \dd x -   \int_{\GC} \frac12\chi(s)|\uu(s)|^2 \dd x,
 \\
 &
 I_5+I_{11}+I_{12} =\int_s^t \int_{\GC}
  \frac{\dd}{\dd t} \left(\frac12 \chi|\uu|^2 \nlocss{\chi} \right)  \dd x \dd r =
   \int_{\GC} \frac12\chi(t)|\uu(t)|^2 \nlocs\chi t   \dd x -   \int_{\GC} \frac12\chi(s)|\uu(s)|^2
    \nlocs\chi s   \dd x,
    \nonumber
  \end{align}
  where for  (1)  we have used that $\kr$ is symmetric, for   (2)
  and   (3) we have applied  the chain rule for the subdifferential operators
  $\balpha$ and $\beta$. All in all, we conclude \eqref{total-enbal}.
\paragraph{\bf Coercivity properties of the energy functional $\calE$.}
For our first a priori estimate we  will indeed start from the energy balance  \eqref{total-enbal} and
derive the \RCORR energy \EEE bound $\sup_{t\in (0,T)} |\calE(\theta(t),\thetas(t), \uu(t),\chi(t))| \leq C$ that will be combined with
Lemma \ref{LemmaE} below to derive a series of uniform-in-time estimates for the solutions.
\begin{lemma}\label{LemmaE}
There exist two constants $C_1$, $C_2 >0$ such that
for all $\theta\in L^1(\Omega;\R^+)$, $\thetas \in L^1(\GC;\R^+)$,
$\uu \in \bsVD$, and $\chi \in \VC$, there holds
\begin{equation}
\label{coerc-ene}
\calE(\theta,\thetas, \uu,\chi) \geq
 C_1 \Big( \| \theta \|_{L^1(\Omega)} + \| \thetas \|_{L^1(\GC)} + \| \uu \|^2_{H^1 (\Omega)}
+ \| \chi \|_{\VC \cap L^{\infty}(\GC)}^2 \Big) - C_2.
\end{equation}
\end{lemma}
\begin{proof}
First of all, we may suppose that $\calE(\theta,\thetas, \uu,\chi)<+\infty$,
otherwise  estimate \eqref{coerc-ene} is trivial.
Hence, from
$\int_{\GC}W(\chi)\dd x <+\infty$ we infer that $\chi \in [0,1]$ a.e.\ in $\GC$.
Recalling the definition of the stored energy $\calE$ stated by \eqref{stored-energy}
 and taking into account that $\teta$ and
 $\thetas$ are positive functions,
 we have
\[
\begin{aligned}
\calE(\theta,\thetas, \uu,\chi)  & =
\| \teta \|_{L^1(\Omega)}
+ \| \tetas \|_{L^1(\GC)}
+\frac12 \efm(\uu,\uu)
+ \bwalpha(\uu)
\\
& \quad
+\frac12 \int_{\GC} \left( \chi|\uu|^2 +\chi|\uu|^2 \nlocss{\chi} \right)  \dd x
+\int_{\GC} \left( \frac12|\nabla\chi|^2 +W(\chi)\right) \dd x \,.
\end{aligned}
\]
By Korn's inequality \eqref{korn}  we have that
\begin{equation} \label{lemma-stima1}
\frac12 \efm(\uu,\uu) \geq \frac12{C_{\mathrm{e}}} \| \uu \|^2_{H^1(\Omega)}.
\end{equation}
Since  \RCORR $\chi \geq 0$ and $\nlocss{\chi} \geq 0$ \EEE  a.e.\ in $\GC$ by \eqref{positivity-J}, we find that
 \[
  \int_{\GC} \left( \chi|\uu|^2 +\chi|\uu|^2 \nlocss{\chi} \right)  \dd x \geq 0.
  \]
  We also have $\bwalpha(\uu)\geq 0$
while,
according to \eqref{hyp-W},  we have
\begin{equation}
\int_{\GC} W(\chi(t)) \geq -C_{W} |\GC|.
\end{equation}
Finally, since $\calE(\theta,\thetas, \uu,\chi) $ estimates
$\int_{\GC} |\nabla\chi|^2  \dd x $ and
$\int_{\GC} \widehat{\beta}(\chi) \dd x$, and
\RCORR taking into account that \EEE $\mathrm{dom}(\widehat{\beta}) \subset [0,1]$,
we readily conclude the bound for $\|\chi\|_{L^\infty(\GC)}$, and \RCOMMN  for
$\|\chi\|_{\VC}$ \EEE via the Poincar\'e inequality. Hence,  \eqref{coerc-ene} follows.
\end{proof}

\subsection{Strict positivity of $\theta$ and $\thetas$}
\label{ss:3.2}
  In the following calculations we will resort to monotonicity arguments that will be repeatedly used throughout the paper.
\par
\noindent
 We test \eqref{form_debole_theta} and \eqref{form_debole_thetas} by $-\theta^{-p}$ and
$-\thetas^{-p}$, respectively, with
$p>2$.
\par
This choice of  \RCORR these test functions \EEE is only formal for a two-fold reason:
the powers  $-\theta^{-p}$ and $-\thetas^{-p}$, with $p>2$ an arbitrary \emph{real} exponent, are well defined
only if $\theta$ and $\thetas$ are strictly positive. Furthermore,
\RCORR $-\theta^{-p}$ and $-\thetas^{-p}$
 lack  sufficient  spatial regularity to be admissible test functions for the heat equations. \EEE
Anyhow, these issues will be fixed when performing this estimate on a suitably regularized time-discretization scheme for  system \eqref{PDE-true}.
\par
 Adding up,  integrating over $(0,t)$, \RNEW and recalling \eqref{korn},  \EEE we obtain that
\begin{equation} \label{posit1}
\begin{aligned}
&
\frac{1}{p-1} \io \theta^{1-p}(t)  \dd x
+ \frac{1}{p-1} \int_{\GC} \thetas^{1-p}(t)  \dd x
+ p \itt \io  \ab(\theta) \theta^{-(1+p)} | \nabla\theta |^2 \dd x  \dd s
+ \RNEW C_{\mathrm{v}} \EEE \itt \io \frac{|\tensoret|^2}{\theta^p} \dd x \dd s
\\
& \quad
 + \itt \int_{\GC} \frac{|\chi_t|^2}{\thetas^p}  \dd x \dd s
+ p \itt \int_{\GC} \as(\thetas) \thetas^{-(1+p)} | \nabla\thetas|^2 \dd x \dd s
\\
& \quad
- \itt \int_{\GC}\nlocss{\chi} \chi\teta^{2-p} \dd x \dd s +\itt \int_{\GC} \nlocss{\chi\thetas} \chi\teta^{1-p} \dd x \dd s
+\itt \int_{\GC} \chi\thetas^{1-p}  \nlocss{\chi\theta} \dd x \dd s - \int_0^t \int_{\GC} \chi\thetas^{2-p} \nlocss{\chi} \dd x \dd s
\\
& \quad
+ \itt \int_{\GC} k(\chi)(\theta-\thetas) \Big(-\theta^{1-p} - (-\thetas^{1-p})\Big)   \dd x \dd s + \itt \io h \theta^{-p} \dd x \dd s + \int_0^t \int_{\GC} \ell \thetas^{-p} \dd x \dd s
\\
& \RNEW \leq \EEE
\frac{1}{p-1} \io \theta_0^{1-p} \dd x
+ \frac{1}{p-1} \int_{\GC} (\tetaso)^{1-p} \dd x
- \itt \io \theta^{1-p} \mathrm{div}(\uu_t) \dd x \dd s
- \itt \int_{\GC} \thetas^{1-p} \lambda'(\chi) \chi_t  \dd x  \dd s\,.
\end{aligned}
\end{equation}
Due to \eqref{cond-teta0}--\eqref{cond-tetaso}, for the first two terms on the right hand side of \eqref{posit1} it holds
\begin{equation} \label{posult1}
 \io \theta_0^{1-p} +  \int_{\GC} (\tetaso)^{1-p}
\leq \bigg(\frac{1}{\theta^*} \bigg)^{p-1} |\Omega| + \bigg(\frac{1}{\thetas^*} \bigg)^{p-1} |\GC|,
\end{equation}
(recall that $0 < \theta^* \leq \inf_{\Omega} \theta_0$ and $0 <\thetas^* \leq \inf_{\GC}\tetaso$).
Moreover, the third and the fourth term on the right-hand side of \eqref{posit1} can be estimated as follows:
\RCORR  we have \EEE
\begin{equation}
\label{posult2}
\begin{aligned}
- \itt \io \theta^{1-p} \mathrm{div}(\uu_t) \dd x \dd s
&\stackrel{(1)}{\leq}&& \RNEW  \frac{C_{\mathrm{v}}}{2} \EEE  \itt \io \frac{|\tensoret|^2}{\theta^p} \dd x \dd s
+ c \itt \io \theta^{2-p} \dd x \dd s  \\
&\stackrel{(2)} {\leq}&&  \RNEW  \frac{C_{\mathrm{v}}}{2} \EEE \itt \io \frac{|\tensoret|^2}{\theta^p} \dd x \dd s + C \Bigg( \frac{1}{p-1} + \frac{p-2}{p-1} \itt \io \theta^{1-p}
\dd x \dd s\Bigg)   \, \\
&=&&  \RNEW  \frac{C_{\mathrm{v}}}{2} \EEE \itt \io \frac{|\tensoret|^2}{\theta^p} \dd x \dd s + C \left( \frac{1}{p-1} + \frac{p-2}{p-1} \itt \bigg\| \frac{1}{\theta} (s)\bigg\|_{L^{p-1}(\Omega)}^{p-1}  \dd s\right),
\end{aligned}
\end{equation}
where
for (1)
we have used that
\begin{equation}
\label{divut}
|\mathrm{div}(\uu_t)| \leq c_d |\tensoret|,
\end{equation}
 with $c_d>0$ a constant only depending on the space dimension $d=3$. Moreover
we have resorted to the Young inequality
\begin{equation}
\label{Young}
ab \leq \frac{(\delta a)^q}{q} + \frac{1}{q'} \bigg( \frac{b}{\delta} \bigg)^{q'},
\end{equation}
for all $a,b\in\R^+\,, \delta>0$ and $q,q'>1$, such that
$\frac{1}{q}+\frac{1}{q'}=1$. In particular for (2) we have used \eqref{Young} with the choices $b=\delta=1$, $q=(p-1)/(p-2)$, $q'=p-1$ and $a=\theta^{2-p}$.
Furthermore,
\begin{equation}
\label{posult3}
\begin{aligned}
- \itt \int_{\GC} \thetas^{1-p} \lambda'(\chi)\chi_t
&\stackrel{(1)}{\leq}&&
 \frac{1}{2} \itt \int_{\GC} \frac{|\chi_t|^2}{\thetas^p} \dd x \dd s  + c \itt \int_{\GC} \thetas^{2-p} \dd x \dd s  \\
&\stackrel{(2)}{\leq}&&  \frac{1}{2} \itt \int_{\GC} \frac{|\chi_t|^2}{\thetas^p} \dd x \dd s  + C \left( \frac{1}{p-1} + \frac{p-2}{p-1} \itt \int_{\GC} \thetas^{1-p} \dd x \dd s  \right) \\
&=&& \frac{1}{2} \itt \int_{\GC} \frac{|\chi_t|^2}{\thetas^p}  \dd x \dd s+ C \Bigg( \frac{1}{p-1} + \frac{p-2}{p-1}
\itt \left\| \frac{1}{\thetas} (s)\bigg\|_{L^{p-1}(\GC)}^{p-1} \dd s\right),
\end{aligned}
\end{equation}
where (1) follows from
the Lipschitz continuity of $\lambda$ stated by \eqref{hyp-lambda} and (2), again, \RCORR from \EEE \eqref{Young}.
\par
As for  the left-hand side of  \eqref{posit1},
we observe that all terms from the first to the sixth are positive.  We rewrite the sum of the seventh, eighth, ninth and tenth terms as
\begin{equation}
\label{posult4}
\begin{aligned}
&
-\itt \int_{\GC} \theta^{1-p} \left(\nlocss{\chi} \chi\theta - \nlocss{\chi\thetas}\chi\right) \dd x \dd s +\itt\int_{\GC} \thetas^{1-p} \left( \nlocss{\chi\theta}\chi - \nlocss{\chi}\chi\thetas\right) \dd x \dd s
\\
&
\stackrel{(1)}{=} -\itt \iint_{\GC\times\GC}  j(x,y)\chi(x)\chi(y)\theta^{1-p}(x) \Big( \theta(x)-\thetas(y)\Big) \dd x \dd y \dd s
 \\
 & \qquad
 + \itt \iint_{\GC\times \GC} j(x,y)\chi(x)\chi(y)\thetas^{1-p}(y) \Big( \theta(x)-\thetas(y)\Big) \dd x \dd y \dd s
 \\
 & = -  \itt \iint_{\GC \times \GC} j(x,y)\chi(x)\chi(y) \Big(\theta^{1-p}(x) {-}  \thetas^{1-p}  (y)\Big) \Big(\theta(x){-}\thetas(y)\Big) \dd x \dd y \dd s  \stackrel{(2)}{\geq} 0
\end{aligned}
\end{equation}
where for (1) we have exchanged $x$ and $y$ in the second integral and used that the kernel $j$ is symmetric (cf.\ also \eqref{key-posit-conto}), and inequality (2)  follows from the fact that the function
$(0,+\infty) \ni r\mapsto -r^{1-p}$ is strictly increasing (since $p>2$) and from the positivity of the kernel $j$ and of $\chi$.
By the same monotonicity argument we also have that
\begin{equation}
\label{posult5}
\itt \int_{\GC} k(\chi)(\theta-\thetas) \big(-\theta^{1-p} - (-\thetas^{1-p})\big) \dd x \dd s  \geq 0.
\end{equation}
Finally, due to \eqref{cond-h} and  \eqref{cond-ell}, we have that
\begin{equation} \label{posult6}
\itt \io h \theta^{-p} \dd x \dd s \geq 0\,, \qquad \int_0^t \int_{\GC} \ell \thetas^{-p} \dd x \dd s \geq 0\,.
\end{equation}
%
Combining \eqref{posit1} with \eqref{posult1}--\eqref{posult2} and \eqref{posult3}--\eqref{posult6} we infer that
\begin{equation}
\label{posit1b}
\begin{aligned}
&
\bigg\| \frac{1}{\theta} (t)\bigg\|_{L^{p-1}(\Omega)}^{p-1}
+ p (p-1) \itt \io  \ab(\theta) \theta^{-(1+p)} | \nabla\theta |^2 \dd x \dd s+ \frac{\RNEW C_{\mathrm{v}} \EEE (p-1)}{2} \itt \io \frac{|\tensoret|^2}{\theta^p} \dd x \dd s
\\
&
+ \bigg\| \frac{1}{\thetas} (t)\bigg\|_{L^{p-1}(\GC)}^{p-1}+ p(p-1) \itt \int_{\GC} \as(\thetas) \thetas^{-(1+p)} | \nabla\thetas|^2 \dd x \dd s
+ \frac{(p-1)}{2}\itt \int_{\GC} \frac{|\chi_t|^2}{\thetas^p} \dd x \dd s
\\
&
\leq
 \bigg(\frac{1}{\theta^*} \bigg)^{p-1} |\Omega| + \bigg(\frac{1}{\thetas^*} \bigg)^{p-1} |\GC|
+ C \left(1+(p-2) \itt \bigg( \bigg\| \frac{1}{\theta} (s)\bigg\|_{L^{p-1}(\Omega)}^{p-1}
+ \bigg\| \frac{1}{\thetas} (s)\bigg\|_{L^{p-1}(\GC)}^{p-1}  \bigg) \dd s\right)
\end{aligned}
\end{equation}
\RCORR with the constant $C$ on the right-hand side of \eqref{posit1b} independent of $p$.
Thus, \EEE
\begin{equation}
\label{posit2b}
\begin{aligned}
&
\bigg\| \frac{1}{\theta} (t)\bigg\|_{L^{p-1}(\Omega)}^{p-1}
+ \bigg\| \frac{1}{\thetas} (t)\bigg\|_{L^{p-1}(\GC)}^{p-1}\\
&\leq
|\Omega| \bigg(\frac{1}{\theta^*} \bigg)^{p-1}  +  |\GC|\bigg(\frac{1}{\thetas^*} \bigg)^{p-1}
+ C \left(1+(p-2) \itt \bigg( \bigg\| \frac{1}{\theta} (s)\bigg\|_{L^{p-1}(\Omega)}^{p-1}
+ \bigg\| \frac{1}{\thetas} (s)\bigg\|_{L^{p-1}(\GC)}^{p-1}  \bigg) \dd s\right).
\end{aligned}
\end{equation}
Applying the Gronwall Lemma, we \RCORR therefore \EEE obtain that
\[
\bigg\| \frac{1}{\theta} (t)\bigg\|_{L^{p-1}(\Omega)}^{p-1}
+ \bigg\| \frac{1}{\thetas} (t)\bigg\|_{L^{p-1}(\GC)}^{p-1}
\leq \left( |\Omega| \bigg(\frac{1}{\theta^*} \bigg)^{p-1} +  |\GC| \bigg(\frac{1}{\thetas^*} \bigg)^{p-1}+C \right) \exp (C T (p-2))\,.
\]
Therefore,
\begin{equation}
\label{quoted-discrete}
\begin{aligned}
&\max\left\{ \bigg\| \frac{1}{\theta} (t)\bigg\|_{L^{p-1}(\Omega)},   \bigg\| \frac{1}{\thetas} (t)\bigg\|_{L^{p-1}(\GC)} \right \}
 \leq \left(  |\Omega|\bigg(\frac{1}{\theta^*} \bigg)^{p-1} +  |\GC| \bigg(\frac{1}{\thetas^*} \bigg)^{p-1} +C\right)^{1/(p-1)}
\exp \left(C T \frac{(p-2)}{(p-1)}\right)
\\ &
 \leq \left( |\Omega|^{1/(p-1)}\bigg(\frac{1}{\theta^*} \bigg)  + |\GC|^{1/(p-1)}\bigg(\frac{1}{\thetas^*} \bigg) + C^{1/(p-1)} \right)\exp (C T) \leq \overline{C},
 \end{aligned}
\end{equation}
where for the last estimate we have used that $ |\Omega|^{1/(p-1)} \leq |\Omega| +1$ and analogously for $|\GC|^{1/(p-1)}$ and $C^{1/(p-1)}$.
Since  the positive constant $\overline{C}$ is independent of $p$,
we are allowed to conclude that the above estimate holds for arbitrary $p$.
All in all, we find
\begin{equation}
\bigg\| \frac{1}{\theta} \bigg\|_{L^{\infty}(\Omega \times (0,T))} + \bigg\| \frac{1}{\thetas} \bigg\|_{L^{\infty}(\GC \times (0,T))} \leq C.
\end{equation}
Consequently, we infer that the positivity properties \eqref{teta-strict-pos} hold.
\subsection{A priori estimates}
\label{s:3-calculations} We are now in a position to (formally) derive all of our a priori estimates on the solutions to system \eqref{PDE-true}.
\subsubsection{\bf First a priori estimate}
\label{sss:3.3.1}
We consider the
total energy balance
\eqref{total-enbal} on a generic interval $(0,t)$, $t\in (0,T)$. Taking into account the positivity of the second and third terms on the left-hand side, we infer
\begin{equation}
\label{E0}
\begin{aligned}
 \calE(\teta(t),\tetas(t),\uu(t),\chi(t))
 &  \leq
\calE(\teta_0,\tetaso,\uu_0,\chi_0) +
\int_0^t \int_\Omega h \dd x \dd r
+\int_0^t \int_{\GC} \ell \dd x \dd r
 +\int_0^t \pairing{}{\bsVD}{\mathbf{F}}{\uu_t} \dd r
 \\
 &
\doteq I_0+I_1+I_2 \RCORR + I_3 \EEE \,.
\end{aligned}
\end{equation}
Now, by \eqref{cond-init} and \eqref{cond-h}--\eqref{cond-ell}  we have $I_0 + I_1 \RCORR +I_2\EEE \leq C$.
Integrating by parts in time, we have
\begin{equation}
\label{E1}
\begin{aligned}
&
\RCORR I_3 = \EEE
\int_0^t \pairing{}{\bsVD}{\mathbf{F}}{\uu_t} \dd r =  \pairing{}{\bsVD}{ \mathbf{F}(t)}{\mathbf{u}(t)}
-  \pairing{}{\bsVD}{ \mathbf{F}(0)}{\mathbf{u}_0}
- \int_0^t   \pairing{}{\bsVD}{ \mathbf{F}_t}{\mathbf{u}} \dd r
\\
&
\leq \frac{C_1}{2} \| \uu(t) \|_{\bsVDn}^2
+ c \bigg( \| \mathbf{F} \|_{L^{\infty}(0,T; H_{\GD}^1(\Omega)^*)}^2 + \| \mathbf{F}_t \|_{L^2(0,T; H_{\GD}^1(\Omega)^*)}^2 + \| \mathbf{u}_0 \|^2_{\bsVDn}
+  \itt \| \mathbf{u}(s) \|^2_{\bsVDn} \dd s \bigg),
\end{aligned}
\end{equation}
with $C_1>0$ the constant from  the coercivity estimate \eqref{coerc-ene}.  We combine \eqref{E0} and \eqref{E1}; taking into account \eqref{coerc-ene}, we may absorb the term  $\frac{C_1}{2} \| \uu(t) \|^2_{H^1(\Omega)}$,  into the left-hand side of \eqref{E0}. Applying the Gronwall Lemma we conclude that
$\| \uu \|_{L^\infty(0,T;H^1(\Omega;\R^3)} \leq C$.
  Then,
a fortiori, the term \RCOMMN $I_3$ \EEE on the right-hand side of \eqref{E0} is estimated by a constant. All in all, also taking into account that $\calE$ is bounded from below,  we conclude that
\[
\sup_{t\in (0,T)} | \calE(\teta(t),\tetas(t),\uu(t),\chi(t))| \leq C
\]
and then,
by \eqref{coerc-ene}, we find
that
\begin{equation} \label{posit-completa}
\| \theta \|_{L^{\infty}(0,T;L^1(\Omega))}
+ \| \thetas \|_{L^{\infty}(0,T;L^1(\GC))}
+ \| \uu \|_{L^{\infty}(0,T; {H^1(\Omega)})}
+ \| \chi \|_{L^{\infty}(0,T; L^{\infty}(\GC)\cap \VC)} \leq C.
\end{equation}
\begin{remark}
\label{rmk:expl-Linfty}
\upshape
In the calculations for the following a priori estimates we shall not use the
$L^{\infty}(0,T; L^{\infty}(\GC))$-bound for $\chi$. The reason is that
these computations will be rendered rigorously once performed on a suitable approximation of
system
\eqref{PDE-true}, cf.\ system \eqref{PDE-regul} ahead, in which, in particular, the maximal monotone operator $\beta$ is replaced by its Yosida regularization $\betar$, with primitive $\wbetar$. Therefore, on the approximate level the bound  for
$\sup_{t\in (0,T)} \int_{\GC} \wbetar(\chi) \dd x$ \RCORR will no longer yield \EEE the information that
$\chi$ takes values in the interval $[0,1]$ a.e.\ in $\GC$; \RCNEW in particular, some technical adjustments in devising system \eqref{PDE-regul} will be necessary to cope with the lack of positivity of $\chi$. \EEE
\par
In any case, the following calculations can be carried out without resorting to the $L^{\infty}(0,T; L^{\infty}(\GC))$-bound for $\chi$. In this way, they will be immediately translated in the context of system  \eqref{PDE-regul}.
\end{remark}
\subsubsection{\bf Second a priori estimate}
\label{sss:3.3.2}
We first carry out the calculations in the case $\mu\in (1,2),$ then address the cases $\mu>2$ and $\mu=2$.
\paragraph{\bf Case $\mu \in (1,2)$.}
We introduce the function
\begin{equation} \label{Fausiliaria}
F(v):= v^{\nu}/\nu, \quad \quad F'(v):= v^{\nu-1}, \quad \text{with } \nu = 2-\mu \in (0,1). 
\end{equation}
Then, we test \eqref{form_debole_theta} by $F'(\theta)= \theta^{\nu -1}$
and \eqref{form_debole_thetas} by $F'(\thetas)= \thetas^{\nu-1}$, respectively.
Integrating over $(0,t)$ and adding the corresponding equations,
with easy calculations \RNEW (and again recalling \eqref{korn}) \EEE
we obtain that
\begin{equation}
\label{est-teta-1}
\begin{aligned}
&
\itt \int_{\GC} |\chi_t|^2 F'(\thetas) \dd x \dd r
+\RNEW C_{\mathrm{v}} \EEE \itt \io |\tensoret|^2 F'(\theta)\dd x \dd r
\\
&
\quad
+ \ddd{\itt \int_{\GC} \nlocss{\chi \thetas} \chi \theta F'(\theta) \dd x \dd r }{$I_1$}
+ \ddd{\itt \int_{\GC} \nlocss{\chi\theta} \chi \thetas F'(\thetas) \dd x \dd r}{$I_2$}
\\
& \quad
-\ddd{ \itt \io \ab(\theta) \nabla\theta \nabla(F'(\theta)) \dd x \dd r}{$I_3$}
- \ddd{\itt \int_{\GC} \as(\thetas) \nabla\thetas \nabla (F'(\thetas)) \dd x \dd r}{$I_4$}
\\
&
\RNEW\leq \EEE
\ddd{\itt \int_{\GC} k(\chi)  (\theta - \thetas)(\theta F'(\theta)
-\thetas F'(\thetas)) \dd x \dd r}{$I_5$}
+ \ddd{\itt \int_{\GC} \nlocss{\chi} \chi \theta^2 F'(\theta) \dd x \dd r}{$I_6$}
\\
& \quad
+\ddd{ \itt \int_{\GC} \nlocss{\chi} \chi \thetas^2 F'(\thetas) \dd x \dd r}{$I_7$}
+ \ddd{\itt \io \theta_t F'(\theta) \dd x \dd r }{$I_8$}
\\
& \quad
+\ddd{
\itt \int_{\GC} \partial_t {\thetas} F'(\thetas) \dd x \dd r
}{$I_9$}
- \ddd{\itt \io \theta \mathrm{div}(\uu_t) F'(\theta) \dd x \dd r}{$I_{10}$}
\\
&\quad
- \ddd{\itt \int_{\GC} \thetas \lambda'(\chi) \chi_t F'(\thetas) \dd x \dd r}{$I_{11}$}
- \ddd{\itt \io h F'(\theta) \dd x \dd r}{$I_{12}$}- \ddd{\itt \int_{\GC} \ell F'(\thetas) \dd x \dd r}{$I_{13}$}\,.
\end{aligned}
\end{equation}
Now, by the previously proved positivity of $\teta$ and $\tetas$, it is immediate to see that
$I_1\geq 0$ and $I_2\geq 0$.  Recalling \RNEW the growth properties of $\ab$ (cf.\ \eqref{hyp-alpha}), \EEE
we have that
\begin{equation}
\label{termI3}
-I_3 = - \itt \int_{\Omega} \ab(\theta) \nabla\theta \nabla (F'(\theta)) \dd x \dd r
=  (1-\nu) \itt \int_{\Omega} \ab(\theta) |\nabla\theta|^2 \theta^{\nu-2}
\dd x \dd r \geq c \int_0^t  \int_{\Omega} |\nabla\theta|^2  \dd x \dd r
\end{equation}
since $\nu = 2-\mu \RCORR <1 $. \EEE  Analogously, we find that
\begin{equation}
\label{termI4}
-I_4 = - \itt \int_{\GC} \as(\thetas) \nabla\thetas \nabla(F'(\thetas)) \dd x \dd r  \geq  c \itt \int_{\GC} | \nabla\tetas|^2 \dd x \dd r.
\end{equation}
As for  the terms on the right-hand side of \eqref{est-teta-1}, we have that
\begin{eqnarray} \nonumber
I_5 &\leq &
 \itt 	\int_{\GC}  |k(\chi)| (\thetas + \theta) (\thetas^{\nu} + \theta^{\nu} ) \dd x \dd r  \\
& = & C_k \itt \int_{\GC} \RCOMMN (|\chi|^s+1) \EEE (\theta^{\nu+1} +\theta \thetas^{\nu} +\thetas \theta^{\nu} + \thetas^{\nu +1} ) \dd x \dd r
\leq C \itt 	\int_{\GC} (|\chi|^s+1) \Big( \theta^{\nu+1} + \thetas^{\nu+1} \Big) \dd x \dd r,
\nonumber
\end{eqnarray}
where we have used the polynomial growth  of $k$ (cf.\ \eqref{hyp-kappa}), and
 and the previously obtained positivity of $\teta$ and $\tetas$.
We have that
\[
I_6 \leq \itt \int_{\GC} |\nlocss{\chi} |\chi \theta^2 F'(\theta) \dd x \dd r  \leq C \itt   \int_{\GC}  (|\chi|^s+1)\theta^{\nu+1} \dd x \dd r,
\]
where we have again used $\eqref{hyp-kappa}$ and the fact that $\| \chi \|_{L^\infty(0,T;L^1(\GC))} \leq C$, so  that
 $ \|\nlocss{\chi}\|_{L^\infty(\GC{\times}(0,T))} \leq C$ by Lemma \ref{lemmaK}.
Analogously, we have
\[
I_7 \leq  C \itt  \int_{\GC}(|\chi|^s+1)\thetas^{\nu+1} \dd x \dd r.
\]
We clearly have
\[
\begin{aligned}
&
I_8 = \itt \io \theta_t F'(\theta)  \dd x \dd r =\int_\Omega \left(  F(\theta(t)) - F(\theta_0) \right) \dd x = \frac{1}{\nu} \int_\Omega \theta^{\nu}(t) \dd x -  \frac{1}{\nu} \int_\Omega \theta_0^{\nu} \dd x ,\\
&
I_9 = \itt \RCOMMN\int_{\GC} \EEE \partial_t {\thetas} F'(\thetas) \dd x \dd r = \int_{\GC} \left( F(\thetas(t)) - F(\tetaso) \right) \dd x  = \frac{1}{\nu} \int_{\GC}\thetas^{\nu}(t) \dd x -\frac1{\nu}\int_{\GC} (\tetaso)^{\nu} \dd x.
\end{aligned}
\]
Applying  Young's inequality and recalling \eqref{divut},
we obtain that
\[
I_{10} \leq \itt \io |\theta \mathrm{div}(\uu_t) F'(\theta)| \dd x \dd r
\leq \frac{\RCORR C_{\mathrm{v}}\EEE}2 \itt \io |\tensoret|^2  F'(\theta) \dd x \dd r  + C \itt \io \theta^{\nu + 1} \dd x \dd r
\]
while, due to the Lipschitz continuity of $\lambda$,
 we have that
\[
\begin{aligned}
I_{11}
\leq \itt \int_{\GC} \thetas| \lambda'(\chi) ||\chi_t | F'(\thetas) \dd x \dd r
\leq \frac12 \itt \int_{\GC} | \chi_t |^2 F'(\thetas)  \dd x \dd r + C \itt \int_{\GC}  \thetas^{\nu+1} \dd x \dd r.
\end{aligned}
\]
Finally,
by the  positivity  assumptions in \eqref{cond-h} and \eqref{cond-ell},
 we have that
\[
I_{12} = - \itt \io h F'(\theta)  \dd x \dd r \leq 0 \,, \qquad
I_{13} = - \itt \int_{\GC} \ell F'(\thetas)  \dd x \dd r \leq 0 \,.
\]
\par
Collecting all of the above estimates,
we arrive at
\begin{equation}
\label{very-last-calculation}
\begin{aligned}
 & \frac12\itt \int_{\GC} |\chi_t|^2 F'(\thetas) \dd x \dd r
+ \RNEW \frac{C_{\mathrm{v}}}2 \EEE \itt \io |\tensoret|^2 F'(\theta)\dd x \dd r
+c\int_0^t \int_\Omega |\nabla\theta|^2 \dd x \dd r
+c\int_0^t \int_{\GC} |\nabla\thetas|^2 \dd x \dd r
\\
& \quad +  \frac{1}{\nu}\left(  \int_\Omega \theta_0^{\nu} \dd x+ \int_{\GC}  (\tetaso)^{\nu}  \dd x \right)
\\
&
\leq \frac1\nu\left(  \int_{\Omega} \theta^\nu(t) \dd x + \int_{\GC} \thetas^\nu(t) \dd x \right)
+C \int_0^t  \int_{\GC} (|\chi|^s+1) \theta^{\nu+1} \dd x \dd r  + C \itt \int_{\GC} (|\chi|^s+1)  \thetas^{\nu+1} \dd x \dd r
\\
&
\doteq \RCOMMN I_{14}+I_{15}+I_{16}+I_{17} \EEE
\end{aligned}
\end{equation}
Now, since $\nu<1$, we clearly have
\begin{equation}
\label{va-citata}
I_{14} \leq C \left( \|\theta(t)\|_{L^1(\Omega)}^{\nu} {+}1 \right) \leq C, \qquad I_{15} \leq C  \left( \|\thetas(t)\|_{L^1(\GC)}^{\nu} {+}1 \right)  \leq C,
\end{equation}
where the last estimates are due to the previously obtained \eqref{posit-completa}.
Furthermore,  we have
\begin{equation}
\label{thanks-polynomial-growth}
\begin{aligned}
I_{16} \stackrel{(1)}{\leq} C
\int_0^t ( \| \chi^s\|_{L^{\varrho'}(\GC)}+1)\|\teta^{\nu+1}\|_{L^{\varrho}(\GC)} \dd r
 & =  C \int_0^t ( \| \chi\|_{L^{s \varrho'}(\GC)}^s +1) \|\teta\|_{L^{(\nu+1)\varrho}(\GC)}^{\nu+1} \dd r
\\ &  \stackrel{(2)}{\leq} C \int_0^t \|\teta\|_{H^1(\Omega)}^{\nu+1} \dd r
 \\ & \stackrel{(3)}{\leq} C\left(\itt \|\nabla\theta\|_{L^2(\Omega)}^{\nu+1} \dd r + \|\theta\|_{L^\infty(0,T;L^1(\Omega))}^{\nu+1}\right)
 \\ &
 \stackrel{(4)}{\leq} \frac c2 \itt \int_{\Omega} |\nabla\theta|^2 \dd x \dd r + C
 \end{aligned}
\end{equation}
where for (1) we have used  H\"older's inequality with some  \RCORR $\varrho>1$,  \EEE
chosen in such a way that
$(\nu+1)\varrho \leq 4$ so that, by Sobolev embeddings and trace theorems, we have that
$ \|\teta\|_{L^{(\nu+1)\varrho}(\GC)} \leq C \| \teta \|_{H^1(\Omega)}$. Then,
taking into account the previously proved estimate for $\chi$ in $L^\infty(0,T;H^1(\GC))$ and, a fortiori, in $L^\infty(0,T;L^q(\GC))$ for all $1\leq q<\infty$, we conclude (2).
 Estimate
 (3) follows from  the Poincar\'e inequality, and (4) from Young's inequality (since $\nu+1<2$) and, again, \eqref{posit-completa}. In this way, the term
$\RCOMMN \frac c2 \EEE \itt \int_{\Omega} |\nabla\theta|^2 \dd x \dd r$ can be absorbed into the left-hand side of \eqref{very-last-calculation}. The term $\itt \int_{\GC} (|\chi|^s+1)  \thetas^{\nu+1} \dd x \dd r
$ can be treated in a completely analogous way. Hence, from \eqref{very-last-calculation} we  conclude
\begin{equation} \label{2nd-estimate}
\| \theta \|_{L^{2}(0,T;H^1(\Omega))}
+\| \thetas \|_{L^{2}(0,T;H^1(\GC))} \leq C.
\end{equation}
\paragraph{\bf Case $\mu > 2$.}  We test  \eqref{form_debole_theta} and \eqref{form_debole_thetas} by $-\theta^{-q}$ and
$-\thetas^{-q}$, respectively, with $q = \mu -1$.
Adding the resulting relations and integrating over $(0,t)$
we obtain the analogue of  \eqref{posit1}, \RCORR with $q$ in place of $p$. \EEE
We observe that the first two terms on the right-hand side of  \eqref{posit1} can be \RCORR  estimated \EEE as in \eqref{posult1}, while the last seven terms on the left-hand side of \eqref{posit1} can be handled
by monotonicity arguments  as in \eqref{posult4}--\eqref{posult6}.
Since $q=\mu -1$, \RNEW in view of the growth properties of $\alpha$, cf.\ \eqref{hyp-alpha},  \EEE
we have that
\begin{equation} \label{dx1}
q \itt \io  \ab(\theta) \theta^{-(1+q)} | \nabla\theta |^2 \dd x  \dd s
\geq c_0 q \itt \io   \theta^{\mu-(1+q)} | \nabla\theta |^2 \dd x  \dd s = c_0 q \itt \io | \nabla\theta |^2 \dd x  \dd s,
\end{equation}
\begin{equation}\label{dx2}
q \itt \int_{\GC} \as(\thetas) \thetas^{-(1+q)} | \nabla\thetas|^2 \dd x \dd s
\geq  c_0 q \itt \int_{\GC} \thetas^{\mu-(1+q)} | \nabla\thetas|^2 \dd x \dd s
=  c_0  q \itt \int_{\GC} | \nabla\thetas|^2 \dd x \dd s.
\end{equation}
Besides, using \eqref{divut} and the Young inequality \eqref{Young},
the third term on the right-hand side of \eqref{posit1} can be estimated as follows:
\begin{equation} \label{gm}
- \itt \io \theta^{1-q} \mathrm{div}(\uu_t) \dd x \dd s
\leq  \RNEW \frac{C_{\mathrm{v}}}{2} \EEE \itt \io \frac{|\tensoret|^2}{\theta^q} \dd x \dd s  + c \Bigg( \itt \io \theta^{2-q} \dd x \dd s \Bigg).
\end{equation}
Since $\theta\geq\theta^*>0$ \RNEW a.e.\ in $\Omega\times (0,T)$, \EEE we have that
\begin{equation} \label{1aint}
\itt \io \theta^{2-q} \dd x \dd s \leq  \itt \io (\theta^*)^{2-q} \dd x \dd s \leq c \quad \textrm{whenever $q \geq 2$},
\end{equation}
while
\begin{equation} \label{2aint}
\itt \io \theta^{2-q} \dd x \dd s \leq  \itt \io \theta \dd x \dd s + |\Omega| T  \quad \textrm{whenever $1<q<2$}.
\end{equation}
Combining \eqref{gm} with \eqref{1aint}--\eqref{2aint} and recalling that
$
\| \theta \|_{L^{\infty}(0,T;L^1(\Omega))} \leq c,
$ by \eqref{posit-completa},
we conclude that
\begin{equation}
\label{chiuso1}
- \itt \io \theta^{1-q} \mathrm{div}(\uu_t) \dd x \dd s
\leq   \frac{C_{\mathrm{v}}}{2}   \itt \io \frac{|\tensoret|^2}{\theta^q} \dd x \dd s  + c \,.
\end{equation}
Arguing in a similar way and recalling the bound
$\| \thetas \|_{L^{\infty}(0,T;L^1(\GC))} \leq c $,
we infer that
\begin{equation}
\label{chiuso2}
\begin{aligned}
- \itt \int_{\GC} \thetas^{1-q} \lambda'(\chi)\chi_t \dd x \dd s
&\leq \frac{1}{2} \itt \int_{\GC} \frac{|\chi_t|^2}{\thetas^q} \dd x \dd s  + c \itt \int_{\GC} \thetas^{2-q} \dd x \dd s  \\
&\leq \frac{1}{2} \itt \int_{\GC} \frac{|\chi_t|^2}{\thetas^q} \dd x \dd s  + c,
\end{aligned}
\end{equation}
where we have taken into account the Lipschitz continuity of $\lambda$ stated by \eqref{hyp-lambda}.
Combining
 the analogue of  \eqref{posit1}
 with \eqref{dx1}--\eqref{dx2}, \eqref{chiuso1}, \RCORR \eqref{chiuso2}, \RCOMMN
we obtain that \EEE
\[
\begin{aligned}
&
\bigg\| \frac{1}{\theta} (t) \bigg\|^{q-1}_{L^{q-1}(\Omega)}
+ \bigg\| \frac{1}{\thetas} (t)  \bigg\|^{q-1}_{L^{q-1}(\GC)}
+ 
\itt \io \frac{|\tensoret|^2}{\theta^q} \dd x \dd s
\RCOMMN + \EEE \itt \int_{\GC} \frac{|\chi_t|^2}{\thetas^q}  \dd x \dd s
\\
& \quad
+ \itt \io  | \nabla\theta |^2 \dd x  \dd s
+  \itt \int_{\GC} | \nabla\thetas|^2 \dd x \dd s
\leq  c,
\end{aligned}
\]
whence 
estimate \eqref{2nd-estimate} \RCORR follows. \EEE

\paragraph{\bf Case $\mu=2$.}
We test \eqref{form_debole_theta} and \eqref{form_debole_thetas} by $-\theta^{-1}$ and
$-\thetas^{-1}$, respectively. Adding the resulting relations, integrating over $(0,t)$
using \eqref{hyp-alpha}, \eqref{posit-completa}, recalling  that the kernel $j$ is symmetric and that $h \geq 0$ a.e.\ in $\Omega \times (0,T)$ and $\ell \geq 0$ a.e.\ in
$\GC \times (0,T)$,
exploiting cancellations we obtain
\begin{equation}
 \label{mic1}
\begin{aligned}
&
c_0  \itt \io  | \nabla\theta |^2 \dd x  \dd s
+ c_0  \itt \int_{\GC} | \nabla\thetas|^2 \dd x \dd s
+\RNEW C_{\mathrm{v}} \EEE \itt \io \frac{|\tensoret|^2}{\theta} \dd x \dd s
+ \itt \int_{\GC} \frac{|\chi_t|^2}{\thetas}  \dd x \dd s
\\
&
\leq - \io \ln(\theta^0) \dd x  - \int_{\GC} \ln(\thetas^0) \dd x
+ \io \ln(\theta(t)) \dd x + \int_{\GC} \ln(\thetas(t)) \dd x
- \itt \io \mathrm{div}(\uu_t) \dd x \dd s
\\
& \quad
- \itt \int_{\GC} \lambda'(\chi) \chi_t  \dd x  \dd s.
\end{aligned}
\end{equation}
The first two terms on the  \RCORR right-hand \EEE side of \eqref{mic1} are bounded, due to  \eqref{cond-teta0}--\eqref{cond-tetaso}.
Since $\ln(r) \leq 0$ whenever $0<r\leq 1$ and $\ln(r) \leq r$ for every $r>1$,
estimates \eqref{posit-completa} 
ensure that the second and the third term on the right hand side of \eqref{mic1} can be estimated as follows:
\begin{equation} \label{m1}
\io \ln(\theta(t)) \dd x \leq \int_{\Omega \cap \{\theta >1 \}} \ln(\theta(t)) \dd x \leq \int_{\Omega} \theta(t) \dd x \leq c,
\end{equation}
\begin{equation}
\int_{\GC} \ln(\thetas(t)) \dd x \leq \int_{\GC \cap \{\thetas >1 \}} \ln(\thetas(t)) \dd x \leq \int_{\GC} \thetas(t) \dd x \leq c.
\end{equation}
Finally, the last two terms on the \RCORR  right-hand side \EEE can be estimated using \eqref{divut}, \eqref{posit-completa}, 
and the Young inequality:
\begin{eqnarray} \nonumber
- \itt \io \mathrm{div}(\uu_t) \dd x \dd s
\leq c_d \itt \io \frac{|\tensoret|}{\theta^{1/2}} \theta^{1/2} \dd x \dd s
&\leq& \RNEW \frac{C_{\mathrm{v}}}{2}  \EEE \itt \io \frac{|\tensoret|^2}{\theta} \dd x \dd s + c \itt \io \theta  \dd x \dd s \\
&\leq& \RNEW  \frac{C_{\mathrm{v}}}{2} \EEE \itt \io \frac{|\tensoret|^2}{\theta} \dd x \dd s + c, \\ \nonumber
- \itt \int_{\GC} \lambda'(\chi) \chi_t  \dd x  \dd s
\leq c \itt \int_{\GC} \frac{|\chi_t|}{\thetas^{1/2}} \thetas^{1/2} \dd x  \dd s
&\leq& \frac{1}{2} \itt \int_{\GC} \RCORR \frac{|\chi_t|^2}{\thetas} \EEE  \dd x  \dd s + c \itt \int_{\GC} \thetas \dd x  \dd s \\ \label{m2}
&\leq& \frac{1}{2} \itt \int_{\GC} \RCORR \frac{|\chi_t|^2}{\thetas} \EEE  \dd x  \dd s + c,
\end{eqnarray}
again using  the Lipschitz continuity of $\lambda$ ensured by \eqref{hyp-lambda}.
Combing \eqref{mic1} with \eqref{m1}--\eqref{m2}, we infer that
\begin{equation} \nonumber
c_0  \itt \io  | \nabla\theta |^2 \dd x  \dd s
+ \RCOMMN c_0 \EEE \itt \int_{\GC} | \nabla\thetas|^2 \dd x \dd s
+ \frac{\RNEW C_{\mathrm{v}} \EEE}{2}\itt \io \frac{|\tensoret|^2}{\theta} \dd x \dd s
+ \frac{1}{2} \itt \int_{\GC} \frac{|\chi_t|^2}{\thetas}  \dd x \dd s
\leq c,
\end{equation}
whence, recalling \eqref{posit-completa}, 
we have
\eqref{2nd-estimate}.

\subsubsection{\bf Third a priori estimate}
\label{sss:3.3.3}
  We enhance estimate \eqref{2nd-estimate} by  testing \eqref{form_debole_theta} by $F'(\theta)= \theta^{\nu -1}$
and \eqref{form_debole_thetas} by $F'(\thetas)= \thetas^{\nu-1}$, where $\nu \in (0,1)$ is now arbitrary.
Hence, for the terms $I_3$ and $I_4$
from \eqref{termI3} \& \eqref{termI4}
 contributing to
\eqref{est-teta-1} we now find
\[
-I_3 \geq
c  \itt \int_{\Omega} \theta^{\mu+\nu-2} |\nabla\theta|^2 
\dd x \dd r = c \int_0^t  \int_{\Omega} |\nabla(\theta^{(\mu{+}\nu)/2})|^2  \dd x \dd r
\]
and, in the same way,
\[
-I_4  \geq  c \itt \int_{\GC} | \nabla(\tetas^{(\mu{+}\nu)/2})|^2 \dd x \dd r.
\]
With the very same calculations that lead to \eqref{very-last-calculation},
and recalling  \eqref{va-citata},
  \begin{equation}
\label{very-last-calculationMC}
\begin{aligned}
 &
 \int_0^t  \int_{\Omega} |\nabla(\theta^{(\mu{+}\nu)/2})|^2  \dd x \dd r + \int_0^t  \int_{\Omega} |\nabla(\theta^{(\mu{+}\nu)/2})|^2  \dd x \dd r
 \\
 &
 \leq
C
+C \int_0^t  \int_{\GC} (|\chi|^s+1)  \theta^{\nu+1} \dd x \dd r  + C \itt \int_{\GC}  (|\chi|^s+1)\thetas^{\nu+1} \dd x \dd r \,.
\end{aligned}
\end{equation}
Now, as  in \eqref{thanks-polynomial-growth}
we control
$\int_0^t  \int_{\GC} (|\chi|^s+1)  \theta^{\nu+1} \dd x \dd r$
by
 means of $C\|\theta\|_{L^2(0,T; H^1(\Omega))}^2 + C $.
 We proceed analogously for the last term on the right-hand side of \eqref{very-last-calculationMC}. In
 view of the previously proved
estimate  \eqref{2nd-estimate},
we thus
 conclude that \RCOMMN for all  $\mu>1$ and $\nu \in (0,1)$ there exists a positive constant $C$ such that \EEE
\begin{equation} \label{3rd-estimate}
\| \theta^{(\mu{+}\nu)/2} \|_{L^{2}(0,T;H^1(\Omega))}
+\| \thetas^{(\mu{+}\nu)/2} \|_{L^{2}(0,T;H^1(\GC))} \leq C. 
\end{equation}
Notice that
the estimate for the full $H^1$-norm of $\theta^{(\mu{+}\nu)/2}$ (of $\thetas^{(\mu{+}\nu)/2}$, respectively) follows
 by the fact that $\| \teta\|_{L^\infty(0,T;L^1(\Omega))} \leq C$ and
 $\| \tetas\|_{L^\infty(0,T;L^1(\GC))} \leq C$
\RNEW (cf.\ \eqref{posit-completa}) via, e.g., the \RCORR
 Poincar\'e-type \EEE inequality from \eqref{poincare-type}. \EEE
\subsubsection{\bf Fourth a priori estimate}
\label{sss:3.3.4}
We test the weak formulation \eqref{weak-U} of the displacement equation by $\uu_t$,
the weak formulation  of
the flow rule for the adhesion parameter
\begin{equation}
\label{weak-chi}
\chi_t+ A\chi +\xi + \gamma'(\chi) +\lambda'(\chi)\tetas= -\frac12 |\uu|^2 -\frac12 \nlocss{\chi}\, |\uu|^2 -\frac12 \nlocss{\chi|\uu|^2} \qquad \aein\, \GC \times (0,T)
\end{equation}
(with $\xi \in \beta(\chi) $ a.e.\ in $\GC\times (0,T)$),
by
 $\chi_t$, add the resulting equations and integrate in time.
We thus arrive at
\begin{equation}
\label{mechanical-energy-bala}
\begin{aligned}
&
\int_0^t \vfm(\uu_t,\uu_t) \dd r  + \frac12\efm(\uu(t),\uu(t))
+\widehat{\balpha}(\uu(t))  +\frac12\int_{\GC} \chi(t) |\uu(t)|^2 \dd x + \frac12 \int_{\GC}\chi(t)|\uu(t)|^2  \nlocss{\chi(t)} \dd x  \\
&
 + \int_0^t \int_{\GC} |\chi_t|^2 \dd x \dd r
+ \int_{\GC} \left( \frac12 |\nabla \chi(t)|^2 + W (\chi(t)) \right) \dd x
\qquad \qquad
\\
&
\RCNEW = \EEE
\frac12\efm(\uu_0,\uu_0) +  \widehat{\balpha}(\uu_0)  +\frac12\int_{\GC} \chi_0 |\uu_0|^2 \dd x + \frac12 \int_{\GC}\chi_0|\uu_0|^2  \nlocss{\chi_0} \dd x  %
\RCNEW + \int_{\GC} \left( \frac12 |\nabla \chi_0|^2 + W (\chi_0) \right) \dd x
 \EEE
 \\
 & + \int_0^t  \pairing{}{\bsVD}{\mathbf{F}}{\uu_t} \dd r -\int_0^t \int_\Omega \theta \mathrm{div}(\uu_t) \dd x \dd r -\int_0^t \int_{\GC} \lambda'(\chi) \thetas \chi_t \dd x \dd r\,.
 \end{aligned}
 \end{equation}
Now, the first five terms on the right-hand side of \eqref{mechanical-energy-bala} are estimated by a constant in view of  conditions \eqref{cond-init}, also taking into account Lemma \ref{lemmaK}. Furthermore, we find
\[
\begin{aligned}
&
\left|  \int_0^t  \pairing{}{\bsVD}{\mathbf{F}}{\uu_t} \dd r  \right| \stackrel{(1)}{\leq C} \| \mathbf{F}\|_{L^2(0,T;H_{\GD}^1(\Omega)^*)}^2 +\frac14 \int_0^t \vfm(\uu_t,\uu_t) \dd r
\\
&
\left| \int_0^t \int_\Omega \theta \mathrm{div}(\uu_t) \dd x \dd r \right| \stackrel{(2)}{\leq} C \| \theta\|_{L^2(0,T;L^2(\Omega))}^2 +\frac14 \int_0^t \vfm(\uu_t,\uu_t) \dd r
\\
& \left|\int_0^t \int_{\GC} \lambda'(\chi) \thetas \chi_t \dd x \dd r \right|  \stackrel{(3)}{\leq} C    \| \thetas\|_{L^2(0,T;L^2(\GC))}^2 +\frac14 \int_0^t  \int_{\GC}  |\chi_t|^2 \dd x  \dd r
\end{aligned}
\]
where (1) \& (2) follow from Korn's inequality (cf.\ \eqref{korn}), while (3) is due to \eqref{hyp-lambda}. Combining the above estimates with \eqref{mechanical-energy-bala} we deduce
\begin{equation} \label{4th-estimate}
\| \uu \|_{H^{1}(0,T;H_{\GD}^1(\Omega))}  + \|\chi\|_{H^1(0,T;L^2(\GC))} \leq C.
\end{equation}
\subsubsection{\bf Fifth a priori estimate}
\label{sss:3.3.5}
Taking into account the previously obtained \eqref{posit-completa}, \eqref{2nd-estimate}, and \eqref{4th-estimate},
 it is immediate to see, arguing by comparison in the flow rule for the adhesion parameter, that
\[
 \| A\chi +\xi \|_{L^2(0,T;L^2(\GC))} \leq C.
\]
 Hence, well-known arguments from theory of maximal monotone operators
yield  a separate estimate for  $A\chi$ and $\xi$, namely
 \begin{equation}
 \label{comparison-flow-rule}
  \| A\chi \|_{L^2(0,T;L^2(\GC))}  +\|\xi \|_{L^2(0,T;L^2(\GC))} \leq C
  \end{equation}
  so that, by elliptic regularity, we infer that
  \begin{equation}
  \label{L2H2-est}
  \|\chi \|_{L^2(0,T;H^2(\GC))} \leq C\,.
  \end{equation}
\subsubsection{\bf Sixth a priori estimate}
\label{sss:3.3.6}
It follows from \eqref{3rd-estimate} that
$\theta^{(\mu{+}\nu)/2}  $ is estimated in $L^2(0,T;L^6(\Omega))$, namely that
\begin{equation}
\label{inter-b}
\RCOMMN \forall\, \mu >1 \ \forall\, \nu \in (0,1)\quad  \exists C>0\,: \EEE \qquad
  \|\theta\|_{L^{\mu+\nu} (0,T;L^{3(\mu+\nu)}(\Omega))}  \leq C\,.
  \end{equation}
We combine this with the previously found estimate for $\| \theta\|_{L^\infty(0,T;L^1(\Omega))} $: by \eqref{interpolation-Lebesgue} we have the continuous embedding
\[
\begin{gathered}
L^{\mu+\nu} (0,T;L^{3(\mu+\nu)}(\Omega)) \cap L^\infty(0,T;L^1(\Omega)) \subset L^{a}(0,T;L^b(\Omega))
\\
 \text{
 with } \quad  \begin{cases}
\frac1a = \frac{\vartheta}{\mu+\nu},
\\
\frac1b = \frac1{3a} +1-\vartheta
\end{cases}
 \text{ and } \vartheta  = \frac{\mu+\nu}{\mu-\nu+2}
 \end{gathered}
\]
(observe that, with such a choice \RCORR one has \EEE $\vartheta \in (0,1) $ since $\nu \in (0,1)$).
Therefore, we obtain $a = \mu-\nu+2$ and $b = \frac{3(\mu-\nu+2)}{7-6\nu}$, so that
we conclude, from \RCORR \eqref{posit-completa} and \eqref{inter-b}, the bound \EEE
\begin{equation}
\label{genialata-gc}
\RCOMMN \forall\, \mu >1 \  \forall\, \nu \in (0,1)\quad  \exists C>0\,: \EEE \qquad
  \|\theta\|_{L^{\mu-\nu+2} (0,T;L^{3(\mu-\nu+2)/(7-6\nu)}(\Omega))}  \leq C\,.
\end{equation}
\par
Analogously,
due
 to
\eqref{3rd-estimate} and the continuous embedding
$H^1(\GC) \subset L^q(\GC)$ for all
$q<\infty$, \RCOMMN we have \EEE that
$\thetas^{(\mu{+}\nu)/2}  $ is estimated in $L^2(0,T;L^q(\GC))$ for every $q\in [1,\infty)$. Thus
\begin{equation}
\label{interp-s}
\RCOMMN \forall\, \mu >1 \ \forall\,  \nu \in (0,1) \  \forall\, q \in [1,\infty)\quad  \exists C>0\,: \EEE \qquad
  \|\thetas\|_{L^{\mu+\nu} (0,T;L^{q}(\GC))}  \leq C\,.
  \end{equation}
We combine this with the previously found estimate for $\| \thetas\|_{L^\infty(0,T;L^1(\GC))} $. Indeed,  again resorting to
 by \eqref{interpolation-Lebesgue} we observe that the continuous embedding
\[
\begin{gathered}
L^{\mu+\nu} (0,T;L^{q}(\GC)) \cap L^\infty(0,T;L^1(\GC)) \subset L^{a}(0,T;L^b(\GC))
\\
 \text{
 with } \quad  \begin{cases}
\frac1a = \frac{\vartheta}{\mu+\nu},
\\
\frac1b = \frac\vartheta{q} +1-\vartheta
\end{cases}
 \text{ and } \vartheta  = \frac{\mu+\nu}{\mu-\nu+2}
 \end{gathered}
\]
holds.
Now, since $q\in [1,\infty)$ is arbitrary,  the exponent $b$ can be chosen
arbitrarily close to
$\frac{\mu-\nu+2}{2-2\nu}$, \RCOMMN while $a=\mu-\nu+2$.   \EEE
Therefore, \RCORR from \eqref{posit-completa}  and
\eqref{interp-s} \EEE
we conclude that
\begin{equation}
\label{genialata-gc-2}
\RCOMMN \forall\, \mu >1 \ \forall\, \nu \in (0,1)  \ \forall\, l \in \left(1, \frac{\mu-\nu+2}{2-2\nu} \right)\, \quad  \exists C>0\ : \EEE
 \qquad
  \|\thetas \|_{L^{\mu-\nu+2} (0,T;L^{l}(\GC))}  \leq C\,.
\end{equation}

\subsubsection{\bf Seventh estimate on the bulk heat equation}
\label{sss:3.3.7}
In the weak formulation \eqref{form_debole_theta}
of the bulk heat equation we (formally)  choose
 a test function
$v \in W^{1,3+\epsilon}(\Omega) \subset \mathrm{C}^0(\overline\Omega)$, with $\epsilon>0$.
By comparison, we have that
\begin{equation} \label{stim6-mic}
\bigg| \io \theta_t v \dd x \bigg|
\leq
\bigg| \io \mathcal{L}_1 v \dd x  \bigg|
+ \bigg| \int_{\GC} \mathcal{L}_2 v \dd x \bigg|
+ \bigg| \io \ab(\theta) \nabla\theta \nabla v \dd x \bigg|
\doteq I_1 +I_2 +I_3,
\end{equation}
where
\begin{equation} \nonumber
\mathcal{L}_1 := \theta\mathrm{div}(\uu_t) + \RNEW \tensoret \mathbb{V} \tensoret \EEE + h, \quad \quad
\mathcal{L}_2 := - k(\chi) \theta (\theta - \thetas) - \nlocss{\chi} \chi \theta^2 + \nlocss{\chi \thetas} \chi \theta.
\end{equation}
It follows from \RCOMMN \eqref{cond-h}, \eqref{2nd-estimate} \EEE and \eqref{4th-estimate} that
$\| \mathcal{L}_1\|_{L^1(0,T;L^1(\Omega))} \leq C$.
In order to estimate
$\mathcal{L}_2$,
 recalling \eqref{hyp-kappa}
 we observe that
 \begin{equation}
 \label{ancora-kappa}
 \begin{aligned}
 \| k(\chi) \theta (\theta - \thetas)\|_{L^1(\GC)}   &  \leq C \int_{\GC} (|\chi|^s+1) \theta (\theta-\thetas) \dd x
 \\
 &  \leq C(\| \chi\|^s_{L^{s\varrho'}(\GC)}+1) (\| \theta\|_{L^{2\varrho}(\GC)}^{2} + \| \theta\|_{L^{2\varrho}(\GC)}\| \thetas\|_{L^{2\varrho}(\GC)})
 \\
 &
 \leq C (  \|\theta\|_{H^1(\Omega)}^2 +  \|\theta\|_{H^1(\GC)}^2) \,,
 \end{aligned}
 \end{equation}
 where the exponent $\varrho$ is chosen  such  that $\varrho\leq 2$, so that
 $\RCOMMN \| \theta\|_{L^{2\varrho}(\GC)} \EEE \leq C  \|\theta\|_{H^1(\Omega)}$.
In order to
control the terms  $\nlocss{\chi} \chi \theta^2$ and $ \nlocss{\chi \thetas} \chi \theta$ we  resort to an analogous  H\"older estimate,
also taking into account
 that
$\| \nlocss{\chi}  \|_{L^\infty (0,T;L^\infty(\GC))} \leq C$
 and $\|\nlocss{\chi \thetas} \|_{L^\infty (0,T;L^\infty(\GC))} \leq C$  thanks to  \RCOMMN \eqref{posit-completa} \EEE 
 and Lemma \ref{lemmaK}.
 All in all, \RCOMMN thanks again to \eqref{2nd-estimate}, \EEE we conclude that
$\| \mathcal{L}_2\|_{L^1(0,T;L^1(\GC))} \leq C$.
Therefore, denoting by $L_1:= \| \mathcal{L}_1(t) \|_{L^1(\Omega)}$ and $L_2:= \| \mathcal{L}_2(t) \|_{L^1(\GC)}$,
we obtain
\begin{align} \label{6stim1}
I_1 &\leq L_1(t)  \| v \|_{L^{\infty}(\Omega)} \quad \quad \  \textrm{with $L_1 \in L^1(0,T)$}, \\ \label{6stim2}
I_2 &\leq L_2(t)  \| v \|_{L^{\infty}(\GC)} \quad \quad  \textrm{with $L_2 \in L^1(0,T)$}.
\end{align}
\par
Now, by the growth condition on $\alpha$ we have that
\[
I_3 \leq C  \int_\Omega  (1+\teta^\mu) |\nabla \theta| |\nabla v| \dd x \doteq I_{3,1}+I_{3,2}\,.
\]
Clearly,
\[
I_{3,1} \leq  C  \| \nabla \theta\|_{L^2(\Omega)} \| \nabla v \|_{L^2(\Omega)}\,.
\]
\par
In order to estimate the integral term $I_{3,2}$ we resort to
estimate \eqref{genialata-gc}, which yields the bound
\begin{equation}
\label{genialata-omega-2}
\| \teta^{(\mu-\nu+2)/2}\|_{L^2(0,T;L^{6/(7-6\nu)}(\Omega))} \RCOMMN \leq C\,.\EEE
\end{equation}
Therefore,
\begin{equation} \label{6stim3}
\begin{aligned}
I_{3,2} &  \leq
  C  \| \theta^{(\mu-\nu+2)/2} \|_{L^{6/(7-6\nu)}(\Omega)}
\|\RCOMMN \theta^{(\mu+\nu-2)/2} \EEE \nabla \theta\|_{L^2(\Omega)} \| \nabla v \|_{L^{3+\epsilon}(\Omega)} \\ &  =
 C   \| \theta^{(\mu-\nu+2)/2} \|_{L^{6/(7-6\nu)}(\Omega)}
\| \nabla (\theta^{(\mu+\nu)/2} )\|_{L^2(\Omega)} \| \nabla v \|_{L^{3+\epsilon}(\Omega)}
\end{aligned}
\end{equation}
where we have applied H\"older's inequality, choosing
$\nu \in (0,1)$ such that
\[
\frac{7-6\nu}6+ \frac12+\frac1{3+\epsilon} =1.
\]
Therefore, taking into account the previously obtained \eqref{3rd-estimate} and \eqref{genialata-omega-2}, we conclude  that
\[
I_3\leq L_3(t)\| \nabla v \|_{L^{3+\epsilon}(\Omega)} \quad \quad  \textrm{with $L_3 \in L^1(0,T)$}.
\]
All in all, we conclude that
\begin{equation}
\RCOMMN \forall\, \epsilon>0 \quad \exists\, C>0 \,: \EEE \quad
\|  \theta_t \|_{L^1(0,T; W^{1,3+\epsilon}(\Omega)^*)}  \leq C\,.
\end{equation}
%
%
%
\subsubsection{\bf Seventh estimate on the surface heat equation}
\label{sss:3.3.8}
We (formally) test  the surface heat equation \eqref{form_debole_thetas}
by a function $w \in W^{1,2+\epsilon}(\GC) \subset \mathrm{C}^0(\overlineGC)$, with $\epsilon>0$.
By comparison, we have that
\begin{equation} \label{stim6-mics}
\bigg| \int_{\GC} \partial_t {\thetas} w \dd x  \bigg|
\leq
\bigg| \int_{\GC} \mathcal{F} w \dd x \bigg|
+ \bigg| \int_{\GC} \as(\thetas) \nabla\thetas \nabla w  \dd x  \bigg| \doteq I_1 + I_2,
\end{equation}
where
\begin{equation} \nonumber
\mathcal{F}:= \thetas \lambda'(\chi) \chi_t +\ell + |\chi_t|^2
+ k(\chi)\thetas(\theta - \thetas) + \nlocss{\chi\theta} \chi \thetas
- \nlocss{\chi} \chi \thetasq.
\end{equation}
Thanks to   \eqref{2nd-estimate}, \eqref{posit-completa}, \eqref{3rd-estimate}, \eqref{4th-estimate},
and arguing as for \eqref{ancora-kappa},
 we obtain that $\|\RCORR \mathcal{F} \EEE \|_{L^1(0,T;L^1(\GC))} \leq C$.
Therefore,
\begin{equation}
I_1 \leq F(t)  \| w \|_{L^{\infty}(\GC)}, \quad \quad  \textrm{where $F(t):= \| \mathcal{F}(t)\|_{L^1(\GC)}$ and $F\in L^1(0,T)$}.
\end{equation}
In analogy with the calculations in the previous paragraph, we estimate
\[
I_2 \leq C  \int_{\GC}  (1+\thetas^\mu) |\nabla \thetas| |\nabla w| \dd x \doteq I_{2,1}+I_{2,2}\,,
\]
where, again, we trivially estimate
\[
I_{2,1} \leq   C \| \nabla \thetas\|_{L^2(\GC)} \| \nabla w \|_{L^2(\GC)}\,.
\]
In turn, as in \eqref{6stim3} we have
\begin{equation} \label{6stim3-thetas}
\begin{aligned}
I_{2,2} &  \leq
C    \| \thetas^{(\mu-\nu+2)/2} \|_{L^{l}(\GC)}
\| \nabla (\thetas^{(\mu+\nu)/2} )\|_{L^2(\GC)} \| \nabla w \|_{L^{2+\epsilon}(\GC)}
\end{aligned}
\end{equation}
where we have applied H\"older's inequality \RCORR and  chosen \EEE
$\nu \in (0,1)$ such that the exponent $l $ from \eqref{genialata-gc-2} fulfills
\[
\frac{1}l+ \frac12+\frac1{2+\epsilon} =1
\]
\RCORR Hence, we have that \EEE
\[
I_2\leq 
\RCORR F_2(t) \EEE
\| \nabla w \|_{L^{2+\epsilon}(\GC)} \quad \quad  \textrm{with $
\RCORR F_2 \EEE
\in L^1(0,T)$}.
\]
All in all, we conclude that
\begin{equation}
\RCOMMN \forall\, \epsilon>0 \quad \exists\, C>0 \,: \EEE \quad
\|  \partial_t \thetas \|_{L^1(0,T; W^{1,2+\epsilon}(\GC)^*)}  \leq C\,.
\end{equation}

\subsection{Outline of the proof of Theorem \ref{thm:1}}
\label{s:3.4}
As already mentioned, we will   rigorously render  the calculations
in Sec.\ \ref{s:3-calculations}, and thus the resulting a priori estimates,
 by working on a carefully devised time discretization scheme featuring
 \begin{enumerate}
 \item additional regularizing terms for
 system \eqref{PDE-true}, \RCORR modulated by a paramer $\rho>0$, \EEE
 \item
  the Yosida  regularizations \RCORR $\beta_\varsigma$ and $\eta_\varsigma$, $\varsigma>0$,  \EEE of the maximal monotone operators
in the flow rule for $\chi$, and in the boundary condition for $\uu$ on $\GC$,
\end{enumerate}
and where
 \begin{enumerate}
 \setcounter{enumi}{2}
 \item
 \RCNEW several occurrences of the term $\chi$ have been replaced by its positive part $(\chi)^+$  and the right-hand side of the flow rule for $\chi$ has been modulated by a selection in the subdifferential of the positive part of $\chi$ (see \eqref{eqII-regul} and \eqref{subdiff-pos-part} below). 
 \end{enumerate}
 \par
\RCNEW On the one hand,  the latter changes are motivated by the fact that, since the
 subdifferential operator $\beta :\R \rightrightarrows \R$, with domain in  $[0,1]$,  has been replaced by its Yosida regularization, we can no longer exploit the information that $\chi\geq 0$ a.e.\ $\GC\times (0,T)$ which, in turn, would be crucial to estimate from below several integral terms in the subsequent estimates. Clearly, upon passing to the limit as $\varsigma \down 0$ we shall recover positivity  of $\chi$. Correspondingly, the choice to replace
 $\chi$ by its positive part has led to the presence of the subdifferential of the positive part of $\chi$ in the flow rule (for a modelling justification, see, e.g., \cite{BBR4}).
 \par
On the other hand,
 the reason for this threefold approximation procedure
 (the time discretization of system \eqref{PDE-regul} combined with the double-parameter approximation), and
 hence for a threefold  passage to the limit, \EEE resides in the fact that
we shall not be able to obtain
 the positivity
  estimates \eqref{teta-strict-pos} \RCORR for the temperatures \EEE  on the time-discrete level.   Namely,
 for the discrete bulk and surface temperatures, we shall only  prove a  strict positivity
 property (cf.\ \eqref{discrete-positivity} ahead), but not a lower bound by a positive constant as in \eqref{teta-strict-pos}; we postpone to
 Remark \ref{rmk:failure} later on a thorough explanation for this.
  In turn, recall that the
   Second a priori estimate  involves  testing  the temperature equations by negative powers of $\teta$ and $\tetas$: in order to carry it out rigorously, such powers should be in $H^1(\Omega)$/$H^1(\GC)$, respectively. 
   \RCORR By the lack of the \emph{uniform positivity} estimates  \eqref{teta-strict-pos}, \EEE
    we will not have such information at disposal for the discrete bulk and surface temperatures. Hence, we will not be able to replicate the Second estimate (and, a fortiori, the Third, Fifth, and Sixth estimates) on the time discrete scheme.
   \par
   We shall be able to rigorously perform the Second estimate only
   on the time-continuous level, by working with  (a weak formulation of) 
 the following regularized system 
  \begin{subequations}
	\label{PDE-regul}
	\begin{align}
		\label{eqI-regul}
		&
		\begin{aligned}
		\theta_t-\theta\mathrm{div}(\uu_t)-\mathrm{div}(\ab(\theta)\nabla\theta)
		=\tensoret \, \vtens \, \tensoret +h \quad  \text{ in }  \Omega\times (0,T),
		\end{aligned}
				\\
			\label{eqI-bdry-regul}
		&\ab(\theta)\nabla\theta\cdot{\bf n}=0  \quad  \text{ in }   \GD\cup\Gnew\times (0,T),
		\\
		&
			\label{eqI-cont-regul}
		\begin{aligned}
		\ab(\theta)\nabla\theta \cdot{\bf n}
		 =
		 -k(\chi)\teta(\theta{-}\thetas) - \nlocss {\chip}\, \chip \teta^2  + \nlocss {\chip \tetas}\,  \chip \teta  \quad  \text{ in }     \GC \times (0,T),
		\end{aligned}
		\\
		&
		\label{mom-balance-regul}
		-\mathrm{div}(\etens \tensore + \vtens \tensoret + \theta{\mathbb I}) - \rho \mathrm{div}(|\tensoret |^{\omega-2} \tensoret )={\bf f}  \quad  \text{ in }  \Omega\times (0,T),
		\qquad \omega>4,
		\\
		\label{Dir-regul}
		&\mathbf{u}= {\bf 0}  \quad  \text{ in }   \GD\times (0,T),
		\\
		&
		(\etens \tensore+\vtens \tensoret +\theta{\mathbb I}){\bf n}={\bf g}  \quad  \text{ in } \Gamma_{\mathrm{N}}\times (0,T),\label{condIi-regul}\\
		&
		\begin{aligned}
	 (\etens\tensore + \vtens\tensoret   + \theta{\mathbb I}){\bf n}  +\chip{\bf u}+
		\zzeta_\varsigma  +\nlocss{\chip}\, \RCOMMN \chip \EEE \uu =  {\bf 0}   \quad  \text{ in }    \GC \times (0,T),
		\end{aligned}
		\label{condIii-regul}
		\\
		 &
		 \begin{aligned}
		 \label{eq-tetas-regul}
 &\partial_t {\thetas}-\thetas \lambda'(\chi)\chi_t-\mathrm{div}(\as(\thetas)\nabla\thetas)
\\ &
		 \quad
= \ell + |\chi_t|^2+k(\chi)(\theta{-}\thetas)\tetas +
		\nlocss{\chip\teta}\, \RCOMMN\chip \EEE \tetas-\nlocss{\chip}\,\RCOMMN \chip \EEE\tetas^2  \quad  \text{ in }  \GC\times (0,T),
		\end{aligned}
		\\  \label{eq-tetas-b-regul}
		& \as(\thetas)\nabla\thetas\cdot{\RCOMMN \nn_\mathrm{s}} \EEE=0  \quad  \text{ in }  \partial\GC\times (0,T),
		\\
		&
		\begin{aligned}
		& \chi_t + \rho |\chi_t|^{\omega-2}\chi_t  -\Delta\chi +\betar(\chi)   +
		\gamma'(\chi)+\lambda'(\chi)\thetas
		\\ &
		 \quad  =  -\frac 1 2\vert{\uu }\vert^2  \sigma -\frac 12\nlocss{\chip}\,|\uu|^2  \sigma- \frac 1 2\nlocss{\chip |\uu|^2} \sigma  \quad  \text{ in }   \GC
		\times (0,T), \qquad \omega>4,
		\end{aligned}
		\label{eqII-regul}
				\\
&
 \text{with } \sigma\in \partial \varphi(\chi) \quad  \text{ in }  \GC\times (0,T), 
\\
&\partial_{\nn_\mathrm{s}} \chi=0   \quad  \text{ in }  \partial \GC \times
		(0,T),\label{bord-chi-regul}
\end{align}
		\end{subequations}
 where
$\varphi: \R \to [0,+\infty)$ is defined by $\varphi(x) := (x)^+ \ \text{ for all } x \in \R\,$ and hence
\begin{equation}
\label{subdiff-pos-part}
\partial\varphi(x) = \begin{cases}
\{0\} &\text{if } x <0,
\\
[0,1] & \text{if } x=0,
\\
\{1\} & \text{if } x>0.
\end{cases}
\end{equation}
\par
System \eqref{PDE-regul}
features  two parameters   $\rho,\, \varsigma>0$, where:
  \begin{enumerate}
  \item
   the higher order terms $-\rho \mathrm{div}(|\tensoret |^{\omega-2} \tensoret )$ and $ \rho |\chi_t|^{\omega-2}\chi_t$,
   have been added
   to the left-hand sides of the momentum balance and of the flow rule for the adhesion parameter
   in order to compensate  the quadratic terms on the right-hand sides of the bulk and surface heat equations.
   This will pave the way to further estimates, and enhanced regularity, for the temperature variables which, in turn,  will enables us to rigorously perform the estimates from Sec.\ \ref{s:3-calculations} on system \eqref{PDE-regul};
   \item in place of a selection $\zzeta \in \balpha (\uu)$,
   the boundary condition \eqref{condIii-regul} features the term
   \begin{equation}
   \label{zeta-rho}
   \zzeta_\varsigma = \etar(\uu \cdot \mathbf{n}) \mathbf{n}
   \end{equation}
   where $\etar$ is the Yosida regularization of the subdifferential $\RCORR \eta = \EEE \partial\widehat\eta: \R \rightrightarrows \R$ of $\widehat \eta$ from  \eqref{hyp-alpha-eta};
   \item in the flow rule \eqref{eqII-regul}  we have considered the Yosida regularization
   $\betar$
   of the subdifferential  operator $\beta:\R \rightrightarrows \R$.
   \end{enumerate}
   The Yosida regularizations  in the momentum balance equation and in the flow rule for $\chi$ are motivated by the   presence of the nonlinear terms in $\eps(\uu_t)$ and $\chi_t$ in the approximate momentum balance and flow rules.    The parameter $\rho$ is kept distinct from the parameter of the Yosida regularization
   because, for technical reasons 
    it will be necessary to perform two different limit passages, in system \eqref{PDE-regul}. First, we shall let $\rho \down 0$ with fixed $
   \varsigma>0$, while
    the identification of the maximal monotone operators $\balpha$ and $\beta$
    in the momentum balance and in the flow rule
    will be performed  in the limit passage as $\varsigma \down 0$.
   \begin{remark}
\upshape
\label{rmk:could-be-dispensed}
As we have pointed out in Remark \ref{rmk:expl-Linfty},  by replacing the operator $\beta$, with domain in $[0,1]$,
by its Yosida regularization in the flow rule \eqref{eqII-regul}
we can no longer deduce a uniform-in-time bound for $\| \chi \|_{L^\infty(\GC)}$  from the First a priori estimate.
This is the reason why we need to impose
the  growth condition \eqref{hyp-kappa} for the function $k$,
that has been indeed used in Sections \ref{sss:3.3.2},  \ref{sss:3.3.3},  \ref{sss:3.3.7} and \ref{sss:3.3.8}.
\par
A close perusal at those calculations reveal that condition
 \eqref{hyp-kappa}
 could be dispensed with at the price of adding  an
additional approximation
 to system \eqref{PDE-regul}. Namely, it should be necessary to truncate the term $k(\chi)$, and remove the truncation in the limit as  $\varsigma \down 0$. 
  However, to avoid overburdening the analysis we have chosen not to do so.
\end{remark}
   \par
   We will supplement system \eqref{PDE-regul} with initial data
   \begin{subequations}
\label{approx-initial-data}
\begin{align}
&
\begin{aligned}
&
(\teta_\rho^0)_{\rho} \subset L^{\mu+2}(\Omega), \ (\dss \rho0)_{\rho} \subset L^{\mu+2}(\GC)  \text{ fulfilling \eqref{cond-teta0}--\eqref{cond-tetaso}, and such that }
\\
&
\begin{cases}
\teta_\rho^0 \to \teta_0 \text{ in } L^1(\Omega),
\\
\dss  \rho 0 \to \tetaso \text{ in } L^1(\GC)
\end{cases}
\qquad \text{ \RCORR as $\rho \down 0$,}
\end{aligned}
\intertext{with $\mu$ from \eqref{hyp-alpha},}
&
(\ds \uu \rho 0)_\rho \subset  W^{1,\omega}_{\mathrm{D}}(\Omega;\R^3) \text{ and such that } \ds \uu \rho 0\to \uu_0 \text{ in } \bsVD \quad \text{ as $\rho \down 0$,} 
\end{align}
\end{subequations}
where we have used the notation $W^{1,\omega}_{\mathrm{D}}(\Omega;\R^3) := W^{1,\omega}(\Omega;\R^3) \cap \bsVD$.
\par
Our strategy for proving Theorem  \ref{thm:1} is the following:
   \begin{enumerate}
   \item in
   Section \ref{s:5} we will devise a careful time discretization scheme for \RCORR system \eqref{PDE-regul}, \EEE and show that it admits a solution (cf.\ Proposition
   \ref{prop:exists-discrete});  
   \item in Section  \ref{s:6} we will derive a series of a priori estimates on the discrete solutions, and prove that, as the time step vanishes, they converge to a
  (weak) solution of system \eqref{PDE-regul}, cf.\ Theorem \ref{thm:exist-regul} ahead, \RCORR such that the temperature variables $\teta$ and $\tetas$   enjoy the positivity properties \eqref{teta-strict-pos}. This information will enable us to perform  the a priori estimates, formally carried out  in Section \ref{s:3-calculations},   in a rigorous way on the solutions  to system \eqref{PDE-regul}; \EEE
   \item in Section \ref{s:7} we will then address the limit passage in system \eqref{PDE-regul},
   first as $\rho \down 0$ and then as $\varsigma\down 0$. In this way, we shall  obtain the existence of
   \RCORR \emph{weak energy solutions} \EEE  
    to
   system \eqref{PDE-true} \RCORR and conclude the proof of Theorem  \ref{thm:1}. \EEE
   \end{enumerate}

\section{Time discretization}
\label{s:5}
%
%
%
Given a time step $\tau >0$ and an equidistant partition of $[0,T]$
with nodes $\tk := k \tau$, $k=0, \ldots, K_{\tau}$, we approximate the data
$\fff$, $\mathbf{g}$, $h$, and $\ell $ by local means, namely \RCORR we set for $k = 1, \ldots, K_\tau$ \EEE
\begin{equation}
\begin{aligned}
&
\fk := \frac{1}{\tau} \int_{\tkmu}^{\tk} \fff (s) \dd s, \quad \quad
\gk := \frac{1}{\tau} \int_{\tkmu}^{\tk} \mathbf{g} (s) \dd s,
\\
&
\hk := \frac{1}{\tau} \int_{\tkmu}^{\tk} h (s) \dd s,
\quad \quad
\lk := \frac{1}{\tau} \int_{\tkmu}^{\tk} \ell (s) \dd s.
\end{aligned}
\end{equation}
Accordingly, we will also consider the local means $( \ds {\mathbf{F}}\tau k)_{k=1}^{K_\tau}$
 of the function $\mathbf{F}$ from \eqref{cond-F}.
 \par
We shall construct discrete solutions to system \eqref{PDE-regul} by recursively solving an elliptic system, \eqref{discr-syst} below. In particular, for the discrete version of the flow rule for the adhesion parameter we shall use that, thanks to \eqref{hyp-lambda} and \eqref{hyp-W},
 the functions $\lambda$ and $\gamma$ decompose as
\begin{equation}
\label{concave/convex-decomp}
\begin{aligned}
&
\lambda(r) =  \lambda(r) -\frac\delta 2 r^2+ \frac\delta 2 r^2 \doteq \lambda_\delta(r) + \frac\delta 2 r^2  && \text{with } \lambda_\delta &&\text{concave},
\\
&
\gamma(r) =  \gamma(r) +\frac\nu 2 r^2- \frac\nu 2 r^2 \doteq \gamma_\nu (r) - \frac\nu 2 r^2  && \text{with } \gamma_\nu &&\text{convex}.
\end{aligned}
\end{equation}
We will look for the temperature components $\ds \teta \tau k$ and $\dss \tau k$ of  \RCORR the  solutions
to the discrete system \eqref{discr-syst} below \EEE
 in the spaces \RF (recall that $\widehat\alpha$ is the primitive of $\alpha$  null in $0$), \EEE
 \begin{equation}
\label{spaces-X}
\begin{aligned}
&
X = \{ \teta \in \accauno : \quad \hhat{\alpha}(\teta) \in H^1(\Omega) \},
\\
&
X_{\mathrm{s}} = \{ \tetas \in H^1(\GC) : \quad \hhat{\alpha}(\tetas) \in H^1(\GC) \},
\end{aligned}
\end{equation}
\par
We are now in a position to introduce our time discretization scheme for system \RCORR \eqref{PDE-regul}, \EEE postponing to Remark \ref{rmk:comm-discretiz}
below further comments on our choices.
\begin{problem}
Let $\omega>4$. Starting from the initial data
$(\tetaink, \ds \uu \tau 0, \tetasink, \ds \chi \tau 0)$ with
$\tetaink = \ds \teta \rho 0$,
$
\ds \uu \tau 0 = \ds \uu \rho 0$,
$\tetasink = \dss \rho 0$ (cf.\  \eqref{approx-initial-data}),
 and
 $\ds \chi \tau 0=\chi_0$  (with $\chi_0$ from
\eqref{cond-chi0}),
 find
\[
 \{ ( \tetak, \uuk, \dss  \tau k, \chik) \}^{K_{\tau}}_{k=1}
\subset  X \times  W_{\mathrm{D}}^{1,\omega}(\Omega;\R^3) \times X_{\mathrm{s}} \times  H^2(\GC),
\]
fulfilling
\begin{subequations}
\label{discr-syst}
\begin{itemize}
\item[-] the discrete bulk temperature equation
\begin{equation}
\label{discr-b-heat}
\begin{aligned}
&
\io \frac{\tetak - \tetakmu}{\tau} v \dd x
- \io \tetak \mathrm{div} \bigg( \frac{\uuk-\uukmu}{\tau} \bigg) v \dd x
+ \io \alpha(\tetak) \nabla \tetak \nabla v \dd x
\\
& \quad
+ \int_{\GC}  k(\chikmu) \tetak(\tetak-\dss \tau k) v \dd x
+ \int_{\GC} \nlocss {\chikmup}\, \chikmup \big( \tetak \big)^2 v \dd x
- \int_{\GC} \nlocss {\chikmup \tetask}\, \chikmup \tetak v \dd x
\\ &
=
\io \tensoretk \, \vtens \, \tensoretk v  \dd x
+ \pairing{}{H^1(\Omega)}{\hk}{v}
\end{aligned}
\end{equation}
for all test functions $v \in H^1(\Omega)$;
\item[-] the discrete momentum balance equation
\begin{equation}
\label{discr-mom-bal}
\begin{aligned}
&
\vfm \bigg( \frac{\uuk-\uukmu}{\tau}, \vv \bigg)
+ \rho \io \left| \eps\left(\frac{\uuk-\uukmu}{\tau} \right)\right|^{\omega - 2} \eps\left(\frac{\uuk-\uukmu}{\tau} \right)  \eps (\vv)  \dd x
+ \efm \Big( \uuk , \vv \Big)
\\ & \quad  + \io \tetak \mathrm{div}(\vv) \dd x
+ \int_{\GC} \chikp \uuk \vv \dd x
+\int_{\GC} \ds \zzeta \tau k \cdot \vv \dd x
+ \int_{\GC} \nlocss {\chikp}\, \chikp \uuk \vv \dd x
=
\langle \RCOMMN \FF_{\tau}^k, \EEE \vv \rangle_{\bsVD},
\\
&
\text{with } \ds \zzeta \tau k   =  \etar (\ds \uu\tau k \cdot \mathbf{n}) \mathbf{n}
\end{aligned}
\end{equation}
for all test functions $\vv \in W^{1,\omega}_{\mathrm{D}}(\Omega;\R^3)$;
\item[-] the discrete surface temperature equation
\begin{equation}
\label{discr-surf-temp}
\begin{aligned}
&
\int_{\GC} \frac{\tetask - \tetaskmu}{\tau} v \dd x
- \int_{\GC} \tetask  \frac{\lambda(\ds \chi\tau k)-\lambda(\ds \chi \tau{k-1})}{\tau}  v \dd x
+ \int_{\GC} \alpha(\tetask) \nabla \tetask \nabla v \dd x
\\
&
=
\int_{\GC} \bigg| \frac{\chik - \chikmu}{\tau}\bigg|^2 v \dd x
+ \int_{\GC} k(\chikmu)(\tetak -\tetask) \tetask v \dd x
+ \int_{\GC} \nlocss {\chikmup \tetak} \, \chikmup \tetask v \dd x
\\
& \qquad - \int_{\GC} \nlocss {\chikmup} \chikmup \big(\tetask \big)^2 v \dd x + \pairing{}{H^1(\GC)}{\lk}{v}
\end{aligned}
\end{equation}
for all test functions $v \in  H^1(\GC) $;
\item[-]
the discrete flow rule for the adhesion parameter
\begin{equation}
\label{discr-flow-rule}
\begin{aligned}
&
\frac{\chik-\chikmu}{\tau}
+\rho \left|\frac{\ds\chi\tau k - \ds \chi\tau{k-1}}\tau \right|^{\omega-2} \frac{ \ds \chi \tau k - \ds \chi \tau{k-1}}\tau
 + A\ds \chi \tau k +
 \betar(\ds \chi \tau k)
+\gamma_\nu'(\ds \chi \tau k) - \nu \ds \chi \tau{k-1}
\\
&
= - \lambda_\delta'(\chikmu) \tetask -\delta \ds \chi \tau k \tetask -
 \frac12 |\uukmu|^2\xik -\frac12 \nlocss {\chikp}\, |\uukmu|^2 \xik -\frac12 \nlocss {\chikmup |\uukmu|^2}\xik
\qquad \aein \, \GC\,,
\\
&
\text{with } \xik  \RCNEW \in \partial \varphi(\chik) \qquad \aein\, \GC\,. \EEE
\end{aligned}
\end{equation}
\end{itemize}
\end{subequations}
\end{problem}
Observe that, thanks to the request that
$\widehat{\alpha}(\ds \teta \tau k) \in H^1(\Omega)$ and $\widehat{\alpha}(\dss \tau  k) \in H^1(\GC)$, the weak \RCORR formulations \EEE for the bulk and surface
equation  with test functions in $H^1(\Omega)$ and $H^1(\GC)$, respectively, \RCORR are \EEE  appropriately posed.
Furthermore, taking into account the growth properties of $\alpha$ (cf.\ also \eqref{growth-primitives} ahead), \RCORR from $\widehat{\alpha}(\ds \teta \tau k) \in H^1(\Omega)$ \EEE we conclude that
$\ds \teta \tau k \in L^{6\mu+6} (\Omega)$. \RCORR An even higher integrability property
holds for
$\dss \tau k$, as a consequence of the fact that $\widehat{\alpha}(\dss \tau k) \in H^1(\GC)$,
taking into account that $H^1(\GC) \subset L^q(\GC)$ for all $1\leq q<\infty$.
Therefore,
starting from initial data
$(\ds \teta \tau 0, \ds \uu \tau 0, \tetasink, \ds \chi\tau 0)$ with
$\ds \teta \tau 0 \in  L^{\mu+2}(\Omega)$ and $\tetasink \in   L^{\mu+2}(\GC)$, we will  gain the same integrability property (an even higher one)  also for the discrete solutions. This information shall be used for the rigorous a priori estimates  performed on the time-discrete scheme in Section \ref{s:6}.
 \EEE
\begin{remark}
\label{rmk:comm-discretiz}
\upshape The time-discretization scheme \eqref{discr-syst}  has been carefully devised in such a way as to ensure the validity of a
 form of the total energy balance (cf.\  \eqref{discr-tot-enbal-trunc} and \eqref{discr-tot-enbal}  ahead) for the discrete solutions.
 This has motivated
 \begin{itemize}
 \item[-] the choice of the terms to be kept implicit instead of explicit;
 \item[-]
  the usage of the convex and concave decompositions of
 the functions $\gamma$ and $\lambda$ in the discrete flow rule for $\chi$, which will allow us to exploit suitable convexity/concavity inequalities
 (cf.\ \eqref{conv/conc-ineqs} ahead) that are  instrumental to the discrete total energy balance
   \RCNEW and, likewise,
 \item[-] the presence of the selections $\xik \in \partial\varphi(\chik)$ in  the terms of the discrete flow rule
  that  are coupled with the terms of the discrete 
  momentum balance featuring the positive parts $\chikp$.  \EEE
  \end{itemize}
\par
As it turns out,    scheme  \eqref{discr-syst} is fully implicit, with  all equations tightly coupled one with another. Because of this,
	it will not be possible to prove the existence of solutions to \eqref{discr-syst} by  separately solving the discrete bulk temperature equation, the
	 momentum balance equation, the
	 surface temperature equation, and  the flow rule for the adhesion parameter. Instead, to prove existence for (a suitably truncated version of)
	system  \eqref{discr-syst}  we will resort to
	a fixed-point type
	 existence result for equations featuring pseudo-monotone operators.
%
%
%
\end{remark}

\par
The main result of this section ensures the existence of solutions to scheme \eqref{discr-syst}, as well as the strict positivity of the discrete temperatures (cf.\ \eqref{discrete-positivity}).

We will also show that the solutions to system \eqref{discr-syst} comply with
the total energy inequality \eqref{discr-tot-enbal} below,
featuring the energy functional $ \RCORR \calE_\varsigma: \EEE L^1(\Omega)\times L^1(\GC) \times H_{\Dir}^1(\Omega;\R^3) \times H^1(\GC) \to \R$ defined by
\begin{equation}
\label{stored-energy-rho}
\begin{aligned}
&\calE_\varsigma(\teta,\tetas,\uu,\chi): =   \int_\Omega \teta \dd x   + \int_{\GC}\tetas \dd x
+\frac12 \efm(\uu,\uu)
\\
&
+ \int_{\GC} \RCORR \wetar \EEE (\uu {\cdot} \mathbf{n}) \dd x
+\frac12 \int_{\GC} \left( \chip|\uu|^2 {+}\chip|\uu|^2 \nlocss\chip \right) \dd x
+\int_{\GC} \left( \frac12|\nabla\chi|^2 {+}\wbetar(\chi) + \gamma(\chi)\right) \dd x,
\end{aligned}
\end{equation}
with $\RCORR \wetar \EEE$ and $\wbetar$ the Yosida approximations of the
functions $\widehat\eta$ and $\widehat\beta$.
 In fact, inequality \eqref{discr-tot-enbal} will be the starting point for the derivation of the estimates,  uniform w.r.t.\ $\tau>0$,
in Sec.\ \ref{s:6}.
\par
In the statements of all the following results, we will omit to explicitly invoke the assumptions of Theorem  \ref{thm:1}.
\begin{proposition}
\label{prop:exists-discrete}
 Let $\tau>0$, sufficiently small, be fixed.
Start from initial data
\begin{equation}
\label{previous-data}
\begin{gathered}
   (\ds \teta \tau 0, \ds \uu \tau 0, \tetasink, \ds \chi\tau 0) =
\RCORR  (\teta_{\rho}^0, \uu_\rho^0, \teta_{\mathrm{s},\rho}^0,\chi_0) \EEE
   \in L^{\mu+2}(\Omega)
   \times W_{\mathrm{D}}^{1,\omega}(\Omega;\R^3)  \times L^{\mu+2}(\GC) \times H^2(\GC)
   \\
      \text{fulfilling
\eqref{cond-teta0} and \eqref{cond-tetaso}.}
\end{gathered}
\end{equation}
Then,  for every $k \in \{1, \ldots, K_\tau\}$ and
there exists  a quadruple  $(\ds \teta \tau{k}, \ds \uu \tau{k}, \dss \tau{k}, \ds \chi\tau{k}) \in
\RCNEW X  \EEE 
\times  W_{\mathrm{D}}^{1,\omega}(\Omega;\R^3) \times
\RCNEW X_{\mathrm{s}} \EEE
\times H^2(\GC)
$, \RCNEW with an associated $\xik \in L^\infty(\GC)$ such that  $\xik \RCNEW \in \partial \varphi(\chik)$ a.e.\ in $\GC$,
 \RCORR solving \eqref{discr-syst}. \EEE
\par
Furthermore, the discrete solutions $(\ds \teta \tau k )_{k=1}^{K_\tau}$
 and $(\dss  \tau k)_{k=1}^{K_\tau}$  enjoy the following estimate
\begin{equation}
\label{Lp-est-4-recipr}
\exists\, S_0 >0 \ \ \forall\, p \in [1,\infty) \ \ \exists\, \bar{\tau}_p>0 \ \ \forall\, \tau \in (0,\bar{\tau}_p)\,  \ \ \forall\, k \in \{1, \ldots, K_\tau\} \, : \quad
\left\| \frac1{\ds \teta \tau k}\right\|_{L^p(\Omega)} + \left\| \frac1{\dss  \tau k}\right\|_{L^p(\GC)} \leq S_0.
\end{equation}
In particular,
\begin{equation}
\label{discrete-positivity}
\ds \teta \tau k >0 \quad \aein\ \Omega, \qquad \dss \tau k>0 \quad \aein \ \GC \qquad \text{for all } k \in \{1, \ldots, K_\tau\}.
\end{equation}
%
\par Finally, there holds
\begin{equation}
\label{discr-tot-enbal}
\begin{aligned}
&
\calE_\varsigma (\ds \teta \tau k,  \dss \tau k,  \ds \uu \tau k, \ds \chi  \tau k)
+\rho\tau \int_{\Omega} \left|\eps \left( \frac{\ds \uu \tau k {-} \ds \uu \tau k}{\tau}\right) \right|^\omega \dd x  + \rho \tau \int_{\GC} \left|\frac{\ds \chi \tau k - \ds \chi \tau{k-1}}{\tau}\right|^\omega \dd x
\\
& \quad
+\tau
\int_{\GC} k(\ds \chi \tau{k-1}) (\ds \teta \tau k {-} \dss \tau k)^2
\dd x
+\tau \int_{\GC\times \GC} j(x,y) \RCNEW (\ds \chi \tau {k-1} (x) )^+ \RCNEW (\ds \chi \tau {k-1} (y))^+ \EEE (\ds \teta \tau k (x){-} \dss \tau k(y))^2 \dd x \dd y
\\
&
 \leq  \calE_\varsigma(\ds \teta \tau {k-1}, \dss  \tau {k-1}, \ds \uu \tau{k-1}, \ds \chi \tau{k-1} ) + \tau \int_{\Omega} \ds h \tau k \dd x +\tau \int_{\GC} \ds \ell \tau k \dd x +
 \tau\pairing{}{\bsVD}{\ds {\mathbf{F}}\tau k}{\frac{\ds \uu {\tau} k-\ds \uu \tau {k-1}}{\tau}}\,.
 \end{aligned}
\end{equation}
\end{proposition}
\begin{remark}
\label{rmk:failure}
\upshape
Although
 the constant $S_0$ in
  \eqref{Lp-est-4-recipr} is independent of  the exponent $p$, for any fixed $p$
  \RCORR estimate \EEE
  \eqref{Lp-est-4-recipr}   in fact holds for only $\tau<\tau_p$, for a certain threshold $\tau_p$ that tends to $0$ as $p \to \infty$ \RCOMMN (cf. \eqref{taup} below). \EEE
This is the reason why, unlike in the time-continuous case (cf.\  the arguments in Section \ref{ss:3.2}), \RCORR from \EEE  the arbitrariness of $p$ in
\eqref{Lp-est-4-recipr} we  cannot deduce a uniform $L^\infty(\Omega)$-bound for the the quantities  $\tfrac1{\ds \teta \tau k}$,
which would provide a lower bound for the discrete bulk temperatures $(\ds \teta \tau k)_{k=1}^{K_\tau}$ by a strictly positive constant.
The same considerations apply to the discrete surface temperatures $(\dss  \tau k)_{k=1}^{K_\tau}$.
\par
In any case, the weaker positivity information \eqref{discrete-positivity} will be sufficient to replicate on the time-discrete level all the estimates needed to prove the existence of solutions to system \eqref{PDE-regul}.
\end{remark}

\par
We will prove Proposition \ref{prop:exists-discrete}
 by approximating system  \eqref{discr-syst} 
 via suitable truncations depending on two parameters $0<\epsilon \ll 1$
 and $M \gg 1$; we postpone to Remark \ref{rmk:explain-truncation} the motivation for such truncations, resulting in system
 \eqref{discr-syst-trunc} below.
 Next, we  will pass to the limit  in \eqref{discr-syst-trunc} as $\epsilon \down 0$, first, and then as $M\up +\infty$. Namely,
 \begin{enumerate}
 \item
 in the upcoming Section \ref{ss:4.1} we will address the existence of solutions to the  $(\epsilon,M)$-truncated system  \eqref{discr-syst-trunc};
 \item
 we shall perform the limit passage
 as $\epsilon \down 0$ in Section  \ref{ss:4.2};
 \item
  and the limit passage as $M\up +\infty$ in Section  \ref{ss:4.3},  thus concluding the proof of Prop.\ \ref{prop:exists-discrete}.
  \end{enumerate}
 \par
 In what follows, we will resort to the following \emph{discrete} Gronwall Lemma, whose proof can be found, e.g., in
 \cite[Lemma 4.5]{RossiSavare06}.
  \begin{lemma}
\label{l:discrG1/2}
Let $K_\tau \in \N$ and $ b, \, \lambda,\,  \Lambda  \in (0,+\infty)$ fulfill $1-b \geq \tfrac 1\lambda>0$; let $(a_k)_{k=1}^{K_\tau} \subset [0,+\infty)$
satisfy
\[
a_k \leq  \Lambda  + b \sum_{j=1}^k a_j \qquad \text{for all } k \in \{1,\ldots, K_\tau\}.
\]
Then,  there holds
\begin{equation}
\label{discrG1/2}
a_k\leq \lambda  \Lambda  \exp(\lambda b k) \quad \text{for all } k \in \{1,\ldots, K_\tau\}.
\end{equation}
\end{lemma}

 \subsection{Existence of solutions to the $(\epsilon,M)$-truncated discrete system}
\label{ss:4.1}

Recall the notation $(r)^+: = \max \{r,0\}$
and $(r)^-: = \max\{ -r, 0\}$ for the positive and  negative parts of a real number $r\in \R$.
Furthermore, for $\epsilon\in (0,1)$ and $M\geq 1$ we  introduce the truncation operators
\begin{equation}
\label{truncation-operators}
\begin{aligned}
&
\mathcal{T}_{\epsilon}: \R \to \R,  &&  \mathcal{T}_{\epsilon}(r): = \max \{r, \epsilon\}\,,
\\
&
\mathcal{T}_{M}: \R \to \R,  &&  \mathcal{T}_{M}(r): = \min\{ \max \{r, 0\}, M\}=
\begin{cases}
0 & \text{if } r <0,
\\
r & \text{if }0 \leq  r \leq  M,
\\
M & \text{if } r > M.
\end{cases}
\end{aligned}
\end{equation}
Accordingly, we define
\[
\alpha_{M}: \R \to (0,+\infty), \qquad \alpha_{M}(r): = \alpha \left( \mathcal{T}_{M}(r) \right)\,.
\]
It follows from \eqref{hyp-alpha} that
\begin{equation}
\label{positiv-alpha-M}
\alpha_{M}(r) \geq c_0 \qquad \text{for all } r \in \R.
\end{equation}
We consider the $(\epsilon,M)$-truncated system, consisting of
\begin{subequations}
\label{discr-syst-trunc}
\begin{itemize}
\item[-] the discrete bulk temperature equation
\begin{equation}
\label{discr-b-heat-trunc}
\begin{aligned}
&
\io \frac{\tetak - \mathcal{T}_\epsilon(\tetakmu)}{\tau} v \dd x
- \io (\tetak)^+ \mathrm{div} \bigg( \frac{\uuk-\uukmu}{\tau} \bigg) v \dd x
+ \io \alpha_{M}(\tetak) \nabla \tetak \nabla v \dd x
\\
& \quad
+ \int_{\GC} k(\chikmu)\mathcal{T}_\epsilon(\tetak)(\tetak-\dss \tau k) v \dd x
+ \int_{\GC} \nlocss {\chikmup}\, \chikmup  \tetak  \mathcal{T}_\epsilon(\tetak)   v \dd x
\\
& \quad
- \int_{\GC} \nlocss {\chikmup \tetask} \, \chikmup \mathcal{T}_\epsilon(\tetak) v \dd x
\\ &
=
\io \tensoretk \, \vtens \, \tensoretk v  \dd x
+ \pairing{}{H^1(\Omega)}{\hk}{ v}  \qquad \text{for all } v \in H^1(\Omega);
\end{aligned}
\end{equation}
\item[-] the discrete momentum balance equation
\begin{equation}
\label{discr-mom-bal-trunc}
\begin{aligned}
&
\vfm \bigg( \frac{\uuk-\uukmu}{\tau}, \vv \bigg)
+ \efm \Big( \uuk , \vv \Big)
+ \rho \io \left| \eps\left(\frac{\ds \uu \tau k -\ds \uu \tau{k-1}}{\tau}\right)\right|^{\omega - 2} \eps\left(\frac{\ds \uu \tau k -\ds \uu \tau{k-1}}{\tau}\right) \eps (\vv)  \dd x
\\
& \quad
+ \io (\tetak)^+ \mathrm{div}(\vv) \dd x
+ \int_{\GC} \chikp \uuk \vv \dd x
+\int_{\GC} \ds \zzeta \tau k \cdot  \vv \dd x 
+ \int_{\GC} \nlocss {\chikp}\, \chikp \uuk \vv \dd x
\\
&=
\langle \FF_{\tau}^k \vv \rangle_{\bsVD}, 
\\
&
 \text{ with  } \ \   \ds \zzeta \tau k  = \etar(\RCORR \ds \uu\tau k\EEE {\cdot} \mathbf{n}) \mathbf{n},
 \qquad \text{for all } \vv \in W^{1,\omega}_{\mathrm{D}}(\Omega;\R^3);
\end{aligned}
\end{equation}
\item[-] the discrete surface temperature equation
\begin{equation}
\label{discr-surf-temp-trunc}
\begin{aligned}
&
\int_{\GC} \frac{\tetask - \mathcal{T}_\epsilon(\tetaskmu)}{\tau} v \dd x
- \int_{\GC} (\tetask)^+  \frac{\lambda(\ds \chi\tau k)-\lambda(\ds \chi \tau{k-1})}{\tau}  v \dd x
+ \int_{\GC} \alpha_M(\tetask) \nabla \tetask \nabla v \dd x
\\
&
=
\int_{\GC} \bigg| \frac{\chik - \chikmu}{\tau}\bigg|^2 v \dd x
+ \int_{\GC} k(\chikmu)(\tetak -\tetask) \mathcal{T}_\epsilon(\tetask) v \dd x
+ \int_{\GC} \nlocss {\chikmup \tetak} \, (\chikmu)^+  \mathcal{T}_\epsilon(\tetask) v \dd x
\\
& \qquad - \int_{\GC} \nlocss {\chikmup}\, \chikmup \tetask \mathcal{T}_\epsilon(\tetask) v \dd x +\pairing{}{H^1(\GC)}{\lk}{v} \qquad \text{for all } v \in H^1(\GC);
\end{aligned}
\end{equation}
\item[-]
the discrete flow rule for the adhesion parameter
\begin{equation}
\label{discr-flow-rule-trunc}
\begin{aligned}
&
\frac{\chik-\chikmu}{\tau}
+\rho \left|\frac{\ds\chi\tau k - \ds \chi \tau{k-1}}\tau \right|^{\omega-2} \frac{\ds\chi\tau k - \ds \chi \tau{k-1}}\tau
 + A\ds \chi \tau k +
\betar(\ds \chi\tau k)
+\gamma_\nu'(\ds \chi \tau k) - \nu \ds \chi \tau{k-1}
\\
&
= - \lambda_\delta'(\chikmu) (\tetask)^+ -\delta \ds \chi \tau k (\tetask)^+ -
 \frac12 |\uukmu|^2\xik  -\frac12 \nlocss {\chikp}\, |\uukmu|^2 \xik -\frac12 \nlocss {\chikmup |\uukmu|^2}\, \xik
\qquad \aein \, \GC\,,
\\
&
\text{with } \xik \RCNEW \in \partial \varphi(\chik) \qquad \aein \, \GC\,. \EEE
\end{aligned}
\end{equation}
\end{itemize}
\end{subequations}
For notational simplicity, we have not highlighted the dependence of a solution
$(\ds \teta \tau k, \ds \uu \tau k, \dss \tau k, \ds \chi \tau k)$ to system  \eqref{discr-syst-trunc} on the parameters
$\epsilon $ and $M$, and we shall not do so, with the exception of the statements
of Proposition \ref{lemma:exists-discrete} and Lemma \ref{l:discr-tot-enbal-trunc} below.
\begin{proposition}
\label{lemma:exists-discrete}
 For any fixed  $\tau>0$, sufficiently small, for every
 $k \in \{1, \ldots, K_\tau\}$ and $ (\ds \teta \tau{k-1}, \ds \uu \tau{k-1}, \dss \tau{k-1}, \ds \chi\tau{k-1}) $ as in \eqref{previous-data},
there exists  a quadruple
\[ (\ds \teta {\tau,\epsilon,M} {k}, \ds \uu  {\tau,\epsilon,M}{k}, \dss  {\tau,\epsilon,M}{k}, \ds \chi {\tau,\epsilon,M}{k}) \in
H^1(\Omega) \times  W_{\mathrm{D}}^{1,\omega}(\Omega;\R^3) \times H^1(\GC) \times H^2(\GC)
\]
\RCNEW with an associated $ \RCNEW \ds \sigma {\tau,\epsilon, M}k \in \partial \RCOMMN \varphi(\ds \chi {\tau,\epsilon,M}{k}) \EEE $ a.e.\ in $\GC$, \EEE
solving \eqref{discr-syst-trunc}.
\par
Furthermore, for every  $k \in \{1, \ldots, K_\tau\}$ we have that
\begin{equation}
\label{non-negativity}
\ds \teta {\tau,\epsilon,M} {k}\geq 0 \quad \aein\ \Omega, \qquad \dss  {\tau,\epsilon,M} {k}\geq 0 \quad \aein\ \GC.
\end{equation}
Finally,
\begin{equation}
 \label{Lp-est-4-recipr-eM}
\begin{aligned}
&
\exists\, S_0 >0 \  \forall\, p \in [1,\infty) \  \exists\, \bar{\tau}_p>0 \
\forall\, \epsilon, M>0 \
\forall\, \tau \in (0,\bar{\tau}_p)\,
 \  \forall\, k \in \{1, \ldots, K_\tau\} \, :
 \\
 &
 \quad
\left\| \frac1{\mathcal{T}_\epsilon(\ds \teta {\tau, \epsilon, M} k)}\right\|_{L^p(\Omega)} + \left\| \frac1{\mathcal{T}_\epsilon(\dss  {\tau,\epsilon,M} k)}\right\|_{L^p(\GC)} \leq S_0.
\end{aligned}
\end{equation}
\end{proposition}
\begin{proof}
\underline{\emph{Step $1$: existence for system \eqref{discr-syst-trunc}:}}
We observe that a quadruple
$(\teta,\uu,\tetas,\chi) $ solves the elliptic system \eqref{discr-syst-trunc} if and only if it solves
\begin{equation}
\label{subdiffential-inclusion}
\RCNEW \partial \Psi(\ds \teta \tau k, \ds \uu \tau k, \dss \tau k, \ds \chi \tau k) + \EEE
 \mathscr{A}(\ds \teta \tau k, \ds \uu \tau k, \dss \tau k, \ds \chi \tau k)\ni 0 \qquad \text{in } \mathbf{X}^*,
\end{equation}
where
$\mathbf{X}$ is a suitable ambient space
\RCNEW $\Psi: \mathbf{X} \to [0,+\infty]$ is a (proper) convex and l.s.c.\ potential,
with subdifferential $\partial \Psi: \mathbf{X}  \rightrightarrows \mathbf{X}^* $, \EEE
 and $\mathscr{A}: \mathbf{X} \to \mathbf{X}^*$ an appropriate pseudomonotone operator. As we will see,
\RCNEW both $\Psi$ and
 $\mathscr{A}$ depend \EEE on the discrete solutions
 $ (\ds \teta \tau{k-1}, \ds \uu \tau{k-1}, \dss \tau{k-1}, \ds \chi\tau{k-1}) $ at the previous step, as well as on the parameters $\epsilon$ and $M$. However, we choose not to highlight this in their notation.
 \par
 Indeed,
let us set $\mathbf{X}: = H^1(\Omega) \times W^{1,\omega}_{\mathrm{D}}(\Omega;\R^3)  \times H^1(\GC) \times H^1(\GC) $
and define
\[
\mathscr{A}: \mathbf{X} \to \mathbf{X}^*  \quad \text{by} \quad \mathscr{A} (\teta,\uu, \tetas,\chi): = \left( \begin{array}{ll}
 \mathscr{A}_1 (\teta,\uu, \tetas,\chi)
 \\
  \mathscr{A}_2 (\teta,\uu, \tetas,\chi)
  \\
   \mathscr{A}_3 (\teta,\uu, \tetas,\chi)
   \\
    \mathscr{A}_4 (\teta,\uu, \tetas,\chi)
\end{array} \right)
\]
where
\begin{subequations}
\label{defA}
\begin{enumerate}
\item $ \mathscr{A}_1:  \mathbf{X} \to H^1(\Omega)^*$ is defined via
\begin{equation}
\label{defA_1}
 \mathscr{A}_1 (\teta,\uu, \tetas,\chi)
:= \teta- \teta^+ \mathrm{div}(\uu{-}\ds \uu \tau {k-1}) + \tau  \RCORR \mathcal{A}^M(\teta) \EEE  + \tau  B_1(\teta,\tetas) - \frac1\tau \eps(\uu) \mathbb{V}
\eps(\uu{-}2\ds \uu \tau{k-1})  -\tau  F_1
\end{equation}
with
\begin{align}
&
\label{def-AM}
\mathcal{A}_M:H^1(\Omega) \to H^1(\Omega)^*, \qquad \pairing{}{H^1(\Omega)}{\mathcal{A}_M(\teta)}{v} : =
 \int_\Omega \alpha_M(\teta) \nabla \teta \cdot \nabla v \dd x;
 \\
 &
 \label{def-B_1}
 \begin{aligned}
 &
 B_1: H^1(\Omega) \times H^1(\GC) \to H^1(\Omega)^*,
 \\
  &
  \pairing{}{H^1(\Omega)}{B_1(\teta,\tetas)}{v} : =
  \\
  & \qquad   \int_{\GC} \left( k(\chikmu)\mathcal{T}_\epsilon(\teta)(\teta-\tetas) {+}
  \nlocss {\chikmup}\, \chikmup  \teta  \mathcal{T}_\epsilon(\teta) {-} \nlocss {\chikmup \tetas}\, \chikmup \mathcal{T}_\epsilon(\teta)  \right) v \dd x,
  \end{aligned}
  \\
  & \label{def-F_1}
  F_1:= \mathcal{T}_\epsilon(\ds  \teta\tau{k-1})  + \frac1\tau \eps(\ds \uu \tau {k-1}) \mathbb{V} \eps(\ds \uu \tau {k-1}) + \ds h \tau k;
 \end{align}
\item
$ \mathscr{A}_2:  \mathbf{X} \to W^{1,\omega}(\Omega;\R^3)^*$ is defined via
\begin{equation}
\label{defA_2}
 \mathscr{A}_2 (\teta,\uu, \tetas,\chi)
:= - \mathrm{div} (\mathbb{V}\uu  +\tau \mathbb{E} \uu + \tau^{2-\omega}\rho |\eps(\uu {-} \ds \uu\tau{k-1})|^{\omega-2} \eps(\uu {-} \ds \uu\tau{k-1})+
 \tau \teta^+ \mathbb{I})
+\tau  B_2(\uu, \chi) -\tau  F_2
\end{equation}
with
\begin{align}
 &
 \label{def-B_2}
 \begin{aligned}
 &
 B_2:  W^{1,\omega}(\Omega;\R^3) \times H^1(\GC) \to  W^{1,\omega}(\Omega;\R^3)^*,
 \\
 & \quad
  \pairing{}{ W^{1,\omega}(\Omega;\R^3)}{B_2(\uu,\chi)}{\vv} : =    \int_{\GC} \chip \uu \vv \dd x
+ \int_{\GC} \nlocss {\chip}\, \chip \uu \vv \dd x + \int_{\GC} \etar(\uu \cdot \mathbf{n})\mathbf{n} \cdot \vv \dd x,
\end{aligned}
  \\
  & \label{def-F_2}
  F_2:=  \frac1\tau \mathrm{div} (\mathbb{V} \ds \uu \tau {k-1})  + \ds {\mathbf{F}} \tau k;
 \end{align}
\item $ \mathscr{A}_3:  \mathbf{X} \to H^1(\GC)^*$ is defined via
\begin{equation}
\label{defA_3}
 \mathscr{A}_3 (\teta,\uu, \tetas,\chi)
:= \tetas- (\tetas)^+  (\lambda(\chi){-}\lambda(\ds \chi \tau {k-1})) + \tau \RCORR \mathcal{A}_{\mathrm{s}}^M \EEE(\tetas)  - \tau  B_3(\teta,\tetas) - \frac1\tau \chi(\chi{-}2\ds \chi \tau{k-1})
 -\tau  F_3
\end{equation}
with
\begin{align}
&
\label{def-AMs}
\RCORR \mathcal{A}_{M,\mathrm{s}}: \EEE H^1(\GC) \to H^1(\GC)^*, \qquad \pairing{}{H^1(\GC)}{\mathcal{A}_{M, \mathrm{s}}(\tetas)}{v} : =
 \int_\Omega \alpha_M(\tetas) \nabla \tetas \cdot \nabla v \dd x;
 \\
 &
 \label{def-B_3}
 \begin{aligned}
 &
 B_3: H^1(\Omega) \times H^1(\GC) \to H^1(\GC)^*,
 \\
  &
  \pairing{}{H^1(\Omega)}{B_3(\teta,\tetas)}{v} : =
  \\
  & \qquad  \int_{\GC} \left( k(\chikmu)\mathcal{T}_\epsilon(\tetas)(\teta-\tetas) {+}
  \nlocss {\chikmup \teta} \, \chikmup    \mathcal{T}_\epsilon(\tetas) {-} \nlocss {\chikmup}\, \chikmup \tetas \mathcal{T}_\epsilon(\tetas)  \right) v \dd x,
  \end{aligned}
  \\
  & \label{def-F_3}
  F_3:= \mathcal{T}_\epsilon(\dss \tau{k-1})  + \frac1\tau |\ds \chi \tau {k-1}|^2 +\lk;
 \end{align}
\item $ \mathscr{A}_4:  \mathbf{X} \to H^1(\GC)^*$ is defined via
\label{defA_4}
\begin{equation}
\begin{aligned}
 \mathscr{A}_4 (\teta,\uu, \tetas,\chi)
:=  & \chi+\tau^{2-\omega} \rho |\chi{-}\ds \chi \tau {k-1}|^{\omega-2} (\chi{-}\ds \chi\tau{k-1}) + \tau A\chi +\tau \betar(\chi)
\\
 & \quad \RCNEW + \tau \gamma_\nu'(\chi)  + \tau \lambda_\delta'(\ds \chi \tau {k-1}) (\tetas)^+
 + \tau \delta \chi (\tetas)^+  
 - \tau F_4 
\end{aligned}
\end{equation}
with
\begin{align}
 \label{def-F_4}
 \RCNEW  F_4:= \frac1\tau \ds\chi \tau{k-1}  + \nu  \ds \chi \tau{k-1}\,.   \EEE 
 \end{align}
 \end{enumerate}
\end{subequations}
\RCNEW
The potential $\Psi: \mathbf{X} \to [0,+\infty)$ featuring in \eqref{subdiffential-inclusion} is defined by
\begin{equation}
\label{def-Psi}
\Psi(\teta, \uu, \tetas, \chi)  = \Psi(\chi) : =  \frac{\tau}2 \int_{\GC} \left(  |\ds \uu \tau{k-1}|^2 \chip{+}  \nlocss{\ds \chi \tau{k-1} |\ds \uu \tau{k-1}|^2}\, \chip {+} 
\nlocss{\chip} \, \chip |\ds \uu \tau{k-1}|^2  \right)  \dd x  \,.
\end{equation}
It can be readily checked that with $\mathscr{A}$ defined by \eqref{defA}  \RCNEW and $\Psi$ by \eqref{def-Psi}, system \EEE
\eqref{subdiffential-inclusion} yields solutions to
system \eqref{discr-syst-trunc} in which the discrete flow rule for the adhesion parameter holds as a subdifferential inclusion in $H^1(\GC)^*$. However, a comparison argument in  \eqref{discr-syst-trunc} yields a fortiori that $A \ds \chi \tau k \in L^2(\GC)$, and thus $\chi \in H^2(\GC)$ and  \eqref{discr-syst-trunc} holds a.e.\ in $\GC$.
%
Let us then show that \eqref{subdiffential-inclusion} does admit solutions.
\par
Standard arguments in the theory of quasilinear elliptic equations (cf.\ \cite[Chap.\ 2.4]{roub-NPDE})  yield that the operator $\mathscr{A}$ is pseudomonotone. The next step is to  verify that
$\mathscr{A}$ is coercive, namely that (using \RCNEW the variable \EEE $\mathbf{x}$ as a place-holder for $(\teta,\uu,\tetas,\chi)$)
\[
\lim_{\|\mathbf{x}\|_{\mathbf{X}}\to +\infty} \frac{\pairing{}{\mathbf{X}}{\mathscr{A}(\mathbf{x})}{\mathbf{x}}}{\|\mathbf{x}\|_{\mathbf{X}}} =+\infty\,.
\]
This can be done by the very same calculations that we will carry out in the proof of Lemma \ref{l:epsi-a-prio} ahead.
\par
Hence, the existence theorem in \RCORR  \cite[Cor.\ 5.17]{roub-NPDE}, \EEE 
yields that \eqref{subdiffential-inclusion}
admits a solution $(\teta,\uu,\tetas,\chi) =: (\ds \teta {\tau,\epsilon,M} {k}, \ds \uu  {\tau,\epsilon,M}{k}, \dss  {\tau,\epsilon,M}{k}, \ds \chi {\tau,\epsilon,M}{k})$.
\medskip

\par
\noindent
\underline{\emph{Step $2$: proof of the non-negativity properties \eqref{non-negativity}}}
We test the discrete bulk heat equation \eqref{discr-b-heat-trunc} by $-(\ds \teta \tau k)^{-}$ and the
discrete surface heat equation \eqref{discr-surf-temp-trunc} by $-(\dss  \tau k)^{-}$ and sum  the resulting relations. \RCORR Thus, \EEE  we get
\begin{equation}
\label{conto-for-pos}
\begin{aligned}
0  & =    \io \frac{\tetak - \mathcal{T}_\epsilon(\tetakmu)}{\tau} (-(\ds \teta  \tau k)^{-}) \dd x
+\int_{\GC} \frac{\tetask - \mathcal{T}_\epsilon(\dss \tau {k-1})}{\tau} (-(\dss   \tau {k})^{-}) \dd x
\\
&\quad
+\io (\tetak)^+ \mathrm{div} \bigg( \frac{\uuk-\uukmu}{\tau} \bigg) (\ds \teta  \tau k)^{-} \dd x
+ \int_{\GC} (\tetask)^+  \frac{\lambda(\ds \chi\tau k)-\lambda(\ds \chi \tau{k-1})}{\tau} (\dss   \tau {k})^{-}   \dd x
\\
& \quad
+ \io \alpha_{M}(\tetak) |\nabla (\tetak)^{-}|^2 \dd x
+ \int_{\GC} \alpha_M(\tetask) |\nabla (\tetask)^{-}|^2 \dd x
\\
&\quad
    + \int_{\GC} k(\chikmu)(\tetak-\dss \tau k) \left(\mathcal{T}_\epsilon(\dss \tau k) (\dss \tau k)^-{-} \mathcal{T}_\epsilon(\tetak)(\ds \teta \tau k)^-  \right) \dd x
\\
& \quad
- \int_{\GC}  (\ds \teta  \tau k)^{-}  \left(  \nlocss {\chikmup}\, \chikmup  \tetak  \mathcal{T}_\epsilon(\tetak) {-}  \nlocss {\chikmup \tetask}\, \chikmup \mathcal{T}_\epsilon(\tetak) \right)    \dd x
\\
& \quad
+ \int_{\GC}    (\dss   \tau k)^{-} \left(  \nlocss {\chikmup \tetak}\, \chikmup \mathcal{T}_\epsilon(\tetask) {-} \nlocss {\chikmup}\, \chikmup \tetask \mathcal{T}_\epsilon(\tetask)  \right) \dd x
\\
& \quad
+
\io \tensoretk \, \vtens \, \tensoretk (\ds \teta \tau k)^-   \dd x
+ \io \hk (\ds \teta \tau k)^-  \dd x+ \int_{\GC} \lk (\tetask)^-  \dd x
\\
& \quad
\RCORR + \EEE \int_{\GC} \bigg| \frac{\chik - \chikmu}{\tau}\bigg|^2 (\dss \tau k)^- \dd x
\\
&
\doteq I_1 +I_2+I_3+I_4+I_5+I_6+I_7+I_8+I_9+I_{10} +I_{11} +I_{12} +I_{13} .\EEE
\end{aligned}
\end{equation}
 First, 
we have that
$
I_1 \geq  \tfrac1\tau \| (\ds \teta \tau k)^-\|_{L^2(\Omega)}^2$  and  $ I_2 \geq  \tfrac1\tau \| (\dss  \tau k)^-\|_{L^2(\GC)}^2$.
Moreover,  since $r^+ r^- =0$ for all $r\in \R$, $I_3=I_4 =0$, whereas by \eqref{positiv-alpha-M} we find that
$I_5 \geq c_0 \| \nabla (\ds \teta \tau k)^-\|_{L^2(\Omega)}^2$ and $I_6 \geq c_0 \| \nabla (\dss  \tau k)^-\|_{L^2(\GC)}^2$.
Observing that  $\mathcal{T}_\epsilon(r) r^-=\epsilon r^-$ for all $r\in \R$, 
the function $r \mapsto  \mathcal{T}_\epsilon(r) r^{-} $ is non-increasing  and hence  $I_7 \geq 0$, while the very same arguments as in
\eqref{posult4} show that $I_8+I_9 \geq 0$.  Clearly, $I_{10}\geq 0$ and $I_{13}\geq 0$. The positivity of
$I_{11} $ is due to the fact that $\ds h \tau k \geq 0$ a.e.\ in $\Omega$ by \eqref{cond-h}; analogously we have that $I_{12} \geq 0$.
All in all, from \eqref{conto-for-pos} we gather that
\[
 \| (\ds \teta \tau k)^-\|_{L^2(\Omega)}^2 + \| (\dss  \tau k)^-\|_{L^2(\GC)}^2 \leq 0, \quad \text{whence} \quad (\ds \teta \tau k)^- =0 \ \aein\, \Omega, \quad (\dss \tau k)^-=0 \ \aein\, \GC,
\]
whence \RCORR the non-negativity properties \EEE   \eqref{non-negativity}.
\medskip

\par
\noindent
\underline{\emph{Step $3$: proof of estimate \eqref{Lp-est-4-recipr-eM}:}} Mimicking the calculations from Section \ref{ss:3.2}, for $p>2$
we test the discrete bulk heat equation \eqref{discr-b-heat-trunc} by $-(\mathcal{T}_\epsilon(\ds \teta \tau k))^{-p}$ and the
discrete surface heat equation \eqref{discr-surf-temp-trunc} by $-(\mathcal{T}_\epsilon(\dss  \tau k))^{-p}$. Summing the resulting relations, we get
\begin{equation}
\label{towards-discr-posi}
\begin{aligned}
0 = &
\io \frac{\tetak - \mathcal{T}_\epsilon(\tetakmu)}{\tau} (-(\mathcal{T}_\epsilon(\ds \teta \tau k))^{-p})  \dd x
+\int_{\GC} \frac{\tetask -\mathcal{T}_\epsilon( \tetaskmu)}{\tau} (-(\mathcal{T}_\epsilon(\dss  \tau k))^{-p})  \dd x
\\
&
\
+\io \tetak \mathrm{div} \left( \frac{\uuk-\uukmu}{\tau} \right) (\mathcal{T}_\epsilon(\ds \teta \tau k))^{-p} \dd x
+\int_{\GC} \tetask  \frac{\lambda(\ds \chi\tau k)-\lambda(\ds \chi \tau{k-1})}{\tau}    (\mathcal{T}_\epsilon(\dss  \tau k))^{-p} \dd x
\\
&
\
+  p  \io \alpha_{M}(\tetak) (\mathcal{T}_\epsilon(\tetak))^{-(1+p)} |\nabla \tetak|^2 \dd x
+  p  \int_{\GC} \alpha_{M}(\tetask) (\mathcal{T}_\epsilon(\tetask))^{-(1+p)} |\nabla \tetask|^2 \dd x
\\
&
\
+ \int_{\GC} k(\chikmu)(\tetak-\dss \tau k)  ((\mathcal{T}_\epsilon(\tetask))^{1-p} {-} (\mathcal{T}_\epsilon(\tetak))^{1-p} ) \dd x
\\
&
\
- \int_{\GC}   \mathcal{T}_\epsilon(\tetak)^{1-p}  \left(  \nlocss {\chikmup}\, \chikmup  \tetak  {-}  \nlocss {\chikmup \tetask}\, \chikmup \right)    \dd x
\\
& \
+ \int_{\GC}     \mathcal{T}_\epsilon(\dss \tau k)^{1-p} \left(  \nlocss {\chikmup \tetak}\, \chikmup  {-} \nlocss {\chikmup}\, \chikmup \tetask   \right) \dd x
\\
& \
 + \io \tensoretk \, \vtens \, \tensoretk \frac1{(\mathcal{T}_\epsilon(\ds \teta \tau k))^p}  \dd x
+ \io \hk \frac1{(\mathcal{T}_\epsilon(\ds \teta \tau k))^p} \dd x
+\int_{\GC}  \lk \frac1{(\mathcal{T}_\epsilon(\ds \teta \tau k))^p} \dd x
\\
& \
+ \int_{\GC} \bigg| \frac{\chik - \chikmu}{\tau}\bigg|^2 \frac1{(\mathcal{T}_\epsilon(\ds \teta \tau k))^p}  \dd x
\\
& \doteq I_1 +I_2+I_3+I_4+I_5+I_6+I_7+I_8+I_9+I_{10} +I_{11} +I_{12} +I_{13} \,,
\end{aligned}
\end{equation}
where we have used that $(\tetak)^+  = \tetak $ and $ (\tetask)^+ = \tetask$ by the previously obtained  \eqref{non-negativity}.
Then, we observe that, as $\tetak \leq \mathcal{T}_\epsilon(\tetak)$ a.e.\ in $\Omega$,
we have
\[
I_1 \geq \io \frac{\mathcal{T}_\epsilon(\tetak) - \mathcal{T}_\epsilon(\tetakmu)}{\tau} (-(\mathcal{T}_\epsilon(\ds \teta \tau k))^{-p})  \dd x
\geq \frac1{(p-1)\tau} \int_\Omega ( \mathcal{T}_\epsilon(\tetak)^{1-p}{-} \mathcal{T}_\epsilon(\tetakmu)^{1-p} ) \dd x,
\]
where the last estimate follows from the convexity inequality $-r^{-p}(r-s) \geq \tfrac1{p-1} (r^{1-p}{-}s^{1-p})$ for every $r,s \in (0,+\infty)$.
Analogously,
\[
I_2 \geq  \frac1{(p-1)\tau} \int_{\GC} ( \mathcal{T}_\epsilon(\tetask)^{1-p}{-} \mathcal{T}_\epsilon(\tetaskmu)^{1-p} ) \dd x.
\]
Clearly, $I_5\geq 0$ and $I_6 \geq 0$. Since the function $r \mapsto \mathcal{T}_\epsilon(r)^{1-p}$ is non-increasing, we have that $I_7 \geq 0$. The very same arguments as
for \eqref{posult4} show that $I_8+I_9 \geq 0$, and, again by \eqref{cond-h} and \eqref{cond-ell}, we have that $I_{11} \geq 0$ and $I_{12} \geq 0$.
Finally, we observe that
\[
\begin{aligned}
I_3+I_{10}  & \stackrel{(1)}{\geq} -  c_d\int_\Omega  \left| \eps  \left( \frac{\uuk-\uukmu}{\tau} \right) \right| (\mathcal{T}_\epsilon(\tetak))^{1-p} \dd x + C_{\mathrm{v}} \int_\Omega
 \left| \eps  \left( \frac{\uuk-\uukmu}{\tau} \right) \right|^2  (\mathcal{T}_\epsilon(\tetak))^{-p}
\dd x
\\
 &
  \stackrel{(2)}{\geq}
 - C \int_\Omega (\mathcal{T}_\epsilon(\tetak))^{2-p} \dd x +  \frac{C_{\mathrm{v}}}2 \int_\Omega
 \left| \eps  \left( \frac{\uuk-\uukmu}{\tau} \right) \right|^2  (\mathcal{T}_\epsilon(\tetak))^{-p} \dd x
 \\
 &
  \stackrel{(3)}{\geq}  -\frac{C}{p-1}  - \frac{C(p-2)}{p-1} \int_\Omega (\mathcal{T}_\epsilon(\tetak))^{1-p} \dd x  +  \frac{C_{\mathrm{v}}}2 \int_\Omega
 \left| \eps  \left( \frac{\uuk-\uukmu}{\tau} \right) \right|^2   (\mathcal{T}_\epsilon(\tetak))^{-p} \dd x
 \end{aligned}
\]
 where (1) is due to \eqref{divut} and to \eqref{korn},  while (2) and (3) are due to Young's inequality, arguing as in \eqref{posult2}. Analogously,
 we find that
 \[
 \begin{aligned}
 I_{4} +I_{12} \geq  -\frac C{p-1}  - \frac{C(p-2)}{p-1} \int_{\GC}(\mathcal{T}_\epsilon(\tetask))^{1-p} \dd x  +  \frac12
  \int_{\GC} \bigg| \frac{\chik - \chikmu}{\tau}\bigg|^2  (\mathcal{T}_\epsilon(\tetask))^{-p} \dd x \,.
 \end{aligned}
 \]
 Combining all of the above calculations  we arrive at
\[
\begin{aligned}
 & \frac1{(p-1)\tau} \int_\Omega \mathcal{T}_\epsilon(\tetak)^{1-p} \dd x  + \frac1{(p-1)\tau} \int_{\GC} \mathcal{T}_\epsilon(\tetask)^{1-p} \dd x
\\ & \leq \frac1{(p-1)\tau} \int_\Omega \mathcal{T}_\epsilon(\tetakmu)^{1-p} \dd x  + \frac1{(p-1)\tau} \int_{\GC} \mathcal{T}_\epsilon(\tetaskmu)^{1-p} \dd x
\\ & \quad +\frac C{p-1} + \frac{C(p-2)}{p-1} \left( \int_\Omega (\mathcal{T}_\epsilon(\tetak))^{1-p} \dd x{+}  \int_{\GC} (\mathcal{T}_\epsilon(\tetask))^{1-p} \dd x \right),
\end{aligned}
\]
whence,
multiplying the  inequality by
$(p-1)\tau$ and
summing over the index  $j$, for $j \in \{1, \ldots, k\}$ with an arbitrary $k \in \{1,\ldots, K_\tau\}$,
we arrive at the relation
\[
\mathcal{T}_k  \leq \mathcal{T}_0 + CT + \sum_{j=1}^{k} C(p-2)\tau \mathcal{T}_j
\]
with the place-holder $\mathcal{T}_k: =  \int_\Omega \mathcal{T}_\epsilon(\tetak)^{1-p} \dd x +\int_{\GC} \mathcal{T}_\epsilon(\tetask)^{1-p} \dd x$.
The discrete Gronwall Lemma \ref{l:discrG1/2}
yields,
\[
\mathcal{T}_k \leq  \frac{ \mathcal{T}_0 {+} CT}{1-C(p-2)\tau } \exp \left( \frac{C(p-2)\tau k}{1-C(p-2)\tau } \right)
\]
for $\tau \in \RCORR  (0,\bar{\tau}_p)$ \EEE with
\begin{equation}
\label{taup}
\RCORR \bar{\tau}_p= \frac1{2C(p-2)}\,. \EEE
\end{equation}
\RCORR Now, since $\tau<\bar{\tau}_p$, we have that $1-C(p-2)\tau >\tfrac12$ and thus we get the analogue of estimate \eqref{quoted-discrete}, i.e.\ \EEE
\[
\max\left\{ \left\| \frac1{ \mathcal{T}_\epsilon(\tetak)}\right\|_{L^{p-1}(\Omega)}  , \left\| \frac1{ \mathcal{T}_\epsilon(\tetask)}\right\|_{L^{p-1}(\GC)}  \right\}
\leq  \left( 2C'\right)^{1/(p-1)}   \exp \left( \frac{2CT(p-2)}{(p-1)} \right) \leq S_0
\]
\RCORR with $C'= \mathcal{T}_0+CT$ and \EEE
 $S_0 = 2C' \exp (2CT)$. Clearly, by the arbitrariness of $p>2$ we conclude \eqref{Lp-est-4-recipr-eM}.
\end{proof}

\begin{remark}
\label{rmk:explain-truncation}
\upshape
A careful perusal of the calculations in Step $3$ shows the role of the  \RCORR
positive parts $(\ds \teta  \tau k)^{+}$ and $\chikmup$, as well as of
 the  truncation operator $\mathcal{T}_\epsilon$,  in  ensuring the positivity of various integral terms that appear in \EEE the proof of estimate  \eqref{Lp-est-4-recipr-eM}.
\end{remark}
\par
We conclude this section by showing that the discrete  solutions $(\ds \teta {\tau,\epsilon,M} k,  \ds \uu {\tau,\epsilon,M} k, \dss {\tau,\epsilon,M} k, \ds \chi  {\tau,\epsilon,M} k)$ fulfill
an energy inequality, involving the stored energy functional \RCORR  $\calE_\varsigma$  from
\eqref{stored-energy-rho}, \EEE   that will play a crucial role for the limit passage as $\epsilon \down 0$.
In the proof of \eqref{discr-tot-enbal-trunc} below  we will use in a key way the convex
and concave decompositions from  \eqref{concave/convex-decomp}.
\begin{lemma}
\label{l:discr-tot-enbal-trunc}
The functions $(\ds \teta {\tau,\epsilon,M} k,  \ds \uu {\tau,\epsilon,M} k, \dss {\tau,\epsilon,M} k, \ds \chi  {\tau,\epsilon,M} k)$ fulfill
\begin{equation}
\label{discr-tot-enbal-trunc}
\begin{aligned}
&
\calE_\varsigma(\ds \teta {\tau,\epsilon,M} k,  \dss {\tau,\epsilon,M} k,  \ds \uu {\tau,\epsilon,M} k, \ds \chi  {\tau,\epsilon,M} k)
+\rho\tau \int_{\Omega} \left|\eps \left( \frac{\ds \uu {\tau,\epsilon,M}k {-}  \ds \uu \tau {k-1}}{\tau}\right) \right|^\omega \dd x  + \rho \tau \int_{\GC} \left|\frac{\ds \chi {\tau,\epsilon,M} k - \ds \chi \tau{k-1}}{\tau}\right|^\omega \dd x
\\
& \quad
+\tau
\int_{\GC} k(\ds \chi \tau{k-1}) (\ds \teta {\tau,\epsilon,M} k {-} \dss {\tau,\epsilon,M} k) (\mathcal{T}_\epsilon(\ds \teta {\tau,\epsilon,M} k) {-} \mathcal{T}_\epsilon( \dss {\tau,\epsilon,M} k))
\dd x
\\
& \quad +\tau  \iint_{\GC\times\GC} 
j(x,y)  \RCNEW (\ds \chi \tau {k-1} (x))^+ (\ds \chi \tau {k-1} (y))^+ \EEE  (\mathcal{T}_\epsilon(\ds \teta {\tau,\epsilon,M} k(x)) {-} \mathcal{T}_\epsilon( \dss {\tau,\epsilon,M} k)(y))
(\ds \teta {\tau,\epsilon,M} k (x){-} \dss {\tau,\epsilon,M} k(y)) \dd x \dd y
\\
&
 \leq  \calE_\varsigma (\mathcal{T}_\epsilon(\ds \teta \tau {k-1}), \mathcal{T}_\epsilon(\dss  \tau {k-1}), \ds \uu \tau{k-1}, \ds \chi \tau{k-1} ) + \tau \int_{\Omega} \ds h \tau k \dd x + \tau \int_{\GC} \ds \ell \tau k \dd x  +
 \tau\pairing{}{\bsVD}{\ds {\mathbf{F}}\tau k}{\frac{\ds \uu {\tau,\epsilon, M} k-\ds \uu \tau {k-1}}{\tau}}\,.
 \end{aligned}
\end{equation}
\end{lemma}
\begin{proof}
We test  \eqref{discr-b-heat-trunc} by $\tau$, \eqref{discr-mom-bal-trunc} by $(\ds \uu \tau{k}{-} \ds \uu \tau{k-1})$, \eqref{discr-surf-temp-trunc} by $\tau$, and
\eqref{discr-flow-rule-trunc} by $(\ds \chi\tau k{-} \ds \chi \tau{k-1})$, add the resulting relations, and observe the cancellation of the terms
\[
- \tau \io (\tetak)^+ \mathrm{div} \left( \tfrac{\uuk-\uukmu}{\tau} \right) \dd x, \qquad
 \tau \io \tensoretk \, \vtens \, \tensoretk  \dd x, \qquad \tau\int_{\GC} \left| \tfrac{\chik - \chikmu}{\tau}\right|^2  \dd x .
 \]
 \par
We now manipulate the remaining terms. Elementary convexity inequalities give that
\begin{equation}
\label{elementary-cvx}
\begin{aligned}
&
 \efm \left( \uuk , \ds \uu \tau{k}{-} \ds \uu \tau{k-1} \right)\geq \frac12 \efm(\ds \uu \tau k ) - \frac12 \efm (\ds \uu \tau {k-1}),
 \\
 &
\begin{aligned}
&
\begin{aligned}
&
 \int_{\GC} \chikp \uuk  (\ds \uu \tau k{-} \ds \uu \tau {k-1}) \dd x
+ \frac12 \int_{\GC} |\uukmu|^2 \xik (\ds \chi \tau  k {-} \ds \chi \tau{k-1}) \dd x
\\
&
\geq \frac12 \int_{\GC}\chikp |\ds \uu \tau k|^2 \dd x   -
\frac12 \int_{\GC}\chikp |\ds \uu \tau {k-1}|^2 \dd x
+\frac12\int_{\GC} \chikp   |\ds \uu \tau {k-1}|^2 \dd x - \frac12 \int_{\GC} \chikmup  |\ds \uu \tau {k-1}|^2 \dd x
\\
& = \frac12 \int_{\GC}\chikp |\ds \uu \tau k|^2 \dd x
-  \frac12 \int_{\GC}  \chikmup |\ds \uu \tau {k-1}|^2 \dd x,
\end{aligned}
\end{aligned}
\\
&
\int_{\GC} \ds \zzeta \tau k \cdot ( \ds \uu \tau{k}{-} \ds \uu \tau{k-1}) \dd x \geq
\int_{\GC} \widehat{\eta}_\varsigma ( \ds \uu \tau{k} {\cdot} \mathbf{n})  \dd x -\int_{\GC} \widehat{\eta}_\varsigma ( \ds \uu \tau{k-1} {\cdot} \mathbf{n})  \dd x,
 \\
 &
 \int_{\GC} \nabla \ds \chi \tau k \cdot \nabla (\ds \chi \tau k {-}\ds \chi \tau {k-1}) \dd x \geq \frac12 \int_{\GC} |\nabla \ds \chi \tau k|^2 \dd x - \frac12 \int_{\GC} |\nabla \ds \chi \tau {k-1}|^2 \dd x,
 \\
 & \int_{\GC} \betar(\ds \chi \tau k) (\ds \chi \tau k{-} \ds \chi \tau{k-1}) \dd x \geq \int_{\GC} \wbetar(\ds \chi \tau k) \dd x -  \int_{\GC} \wbetar(\ds \chi \tau {k-1}) \dd x.
 \end{aligned}
 \end{equation}
 Using that $\lambda_\delta (r) = \lambda(r) - \tfrac\delta 2r^2 $  is concave we infer that
 \begin{subequations}
\label{conv/conc-ineqs}
\begin{equation}
\begin{aligned}
&
 - \int_{\GC} (\tetask)^+  \left( \lambda(\ds \chi\tau k)-\lambda(\ds \chi \tau{k-1})  -\lambda_\delta'(\chikmu)(\ds \chi \tau k {-} \ds \chi \tau{k-1})  - \delta \ds \chi \tau k (\ds \chi \tau k {-} \ds \chi \tau{k-1}) \right) \dd x
 \\
  & \geq
 -\int_{\GC}  (\tetask)^+ \left(  \lambda(\ds \chi\tau k)-\lambda(\ds \chi \tau{k-1}) +\lambda_\delta(\ds \chi \tau {k-1}) - \lambda_\delta(\ds \chi \tau{k}) +\tfrac\delta2 |\ds \chi \tau {k-1}|^2 -\tfrac \delta2 |\ds \chi\tau{k}|^2  \right) \dd x =0\,.
 \end{aligned}
\end{equation}
  Analogously, exploiting that
   $\gamma_\nu(r) = \gamma(r) +\tfrac\nu 2 r^2$ is convex  we find that
 \begin{equation}
 \begin{aligned}
 \int_{\GC} \gamma_\nu'(\ds \chi \tau k)(\ds \chi \tau k{-} \ds \chi \tau{k-1}) - \nu  \int_{\GC}\ds \chi \tau{k-1}(\ds \chi \tau k{-} \ds \chi \tau {k-1}) \dd x
 &  \geq \int_{\GC} \left(  \gamma_\nu(\ds \chi \tau k) {-} \gamma_\nu(\ds \chi \tau {k-1})  {+}\tfrac\nu2 |\ds \chi \tau {k-1}|^2 {-}  \tfrac \nu2 |\ds \chi \tau k|^2 \right) \dd x
 \\
  & \geq \int_{\GC} \left( \gamma(\ds \chi \tau k) {-} \gamma(\ds \chi \tau{k-1}) \right) \dd x \,.
  \end{aligned}
 \end{equation}
 \end{subequations}
 \begin{subequations}
 \label{nonlocal-terms}
 As for the nonlocal terms in the discrete displacement equation, \RCORR we have \EEE
 \begin{equation}
 \label{nonlocs-d-1}
 \int_{\GC} \nlocss{\chikp}\,  \chikp  \ds \uu \tau k \cdot (\ds \uu \tau {k} {-} \ds \uu {\tau}{k-1}) \dd x \geq \frac12 \int_{\GC}\nlocss{\chikp}\, \chikp |\ds \uu \tau k|^2 \dd x - \frac12 \int_{\GC}\nlocss{\chikp}\, \chikp |\ds \uu \tau {k-1}|^2 \dd x,
 \end{equation}
 whereas the nonlocal terms in the discrete flow rule for the adhesion parameter yield
 \begin{align}
 \label{nonlocs-d-2}
 &
 \begin{aligned}
 &
 \int_{\GC} \frac12 \nlocss{\chikp}\,\xik  (\ds \chi \tau {k} {-} \ds \chi {\tau}{k-1}) |\ds \uu \tau {k-1}|^2 \dd x
 \\
 & \RCNEW \geq \EEE
 \frac12 \int_{\GC}\nlocss{\chikp}\, \chikp |\ds \uu \tau {k-1}|^2 \dd x -  \frac12 \int_{\GC}\nlocss{\chikp}\,\chikmup |\ds \uu \tau {k-1}|^2 \dd x,
 \end{aligned}
 \\
 &
  \label{nonlocs-d-3}
 \begin{aligned}
 &
\RCNEW  \frac12 \int_{\GC} \nlocss {\chikmup |\ds \uu \tau {k-1}|^2}  \, \xik \RCNEW ( \ds \chi \tau k{-} \ds \chi \tau {k-1}) \dd x
\\
 & \geq \frac12 \int_{\GC}\nlocss {\chikmup |\ds \uu \tau {k-1}|^2} \, (\chikp{-}\chikmup) \dd x
 \\
  & = \frac12 \int_{\GC}\nlocss{\chikp}\, \chikmup |\ds \uu \tau {k-1}|^2 \dd x
- \frac12 \int_{\GC} \nlocss {\chikmup}\, \chikmup |\ds \uu \tau {k-1}|^2   \dd x\,,
\end{aligned}
 \end{align}
\end{subequations}
\RCORR where we have used the symmetry properties of  $\mathcal{J}$. \EEE
  Adding \eqref{nonlocs-d-1}, \eqref{nonlocs-d-2}, and \eqref{nonlocs-d-3}, we obtain
  \[
\frac12 \int_{\GC}\nlocss{\chikp}\,  \chikp  |\ds \uu \tau k|^2 \dd x - \frac12 \int_{\GC} \nlocss {\chikmup} \, \chikmup |\ds \uu \tau {k-1}|^2
 \dd x\,.
\]
All in all, summing the terms on the right-hand sides of the inequalities in \eqref{elementary-cvx}--\eqref{nonlocal-terms} with the temperature terms we obtain
\RCORR $\calE_\varsigma (\ds \teta {\tau} k,  \dss {\tau} k,  \ds \uu {\tau} k, \ds \chi  {\tau} k)  - \calE_\varsigma(\mathcal{T}_\epsilon(\ds \teta\tau{k-1}),\mathcal{T}_\epsilon(\dss\tau{k-1}), \ds \uu \tau {k-1}, \ds \chi \tau{k-1})  $. \EEE
\par
 Furthermore,
the integrals
$\tau  \int_{\GC} k(\chikmu)\mathcal{T}_\epsilon(\tetak)(\tetak-\dss \tau k) \dd x $ and $-\tau\int_{\GC} k(\chikmu)(\tetak -\tetask) \mathcal{T}_\epsilon(\tetask)  \dd x $
combine to give the  fourth  term on the left-hand side of \eqref{discr-tot-enbal-trunc}. Repeating the same calculations as in \eqref{key-posit-conto}, we see that
\[
\begin{aligned}
&
\tau \int_{\GC}\mathcal{T}_\epsilon(\tetak) \left(  \nlocss {\chikmu} \chikmu  \tetak    {-} \nlocss {\chikmu \tetask} \chikmu   \right) \dd x
-  \tau \int_{\GC} \mathcal{T}_\epsilon(\tetask) \left(  \nlocss {\chikmu \tetak} \chikmu  {-} \nlocss {\chikmu} \chikmu \tetask  \right) \dd x
\\
 & = \tau  \iint_{\GC\times\GC} 
 j(x,y) \ds \chi \tau {k-1} (x) \ds \chi \tau {k-1} (y)  (\mathcal{T}_\epsilon(\ds \teta {\tau,\epsilon,M} k(x)) {-} \mathcal{T}_\epsilon( \dss {\tau,\epsilon,M} k)(y))
(\ds \teta {\tau,\epsilon,M} k (x){-} \dss {\tau,\epsilon,M} k(y)) \dd x \dd y \,.
\end{aligned}
\]
\par
Finally, we find
\[
\begin{aligned}
&
\rho \io \left| \eps\left(\frac{\ds \uu \tau k -\ds \uu \tau{k-1}}{\tau}\right)\right|^{\omega - 2} \eps\left(\frac{\ds \uu \tau k -\ds \uu \tau{k-1}}{\tau}\right) \eps (\ds \uu \tau k{-} \ds \uu \tau{k-1})  \dd x
= \rho \tau \io  \left| \eps\left(\frac{\ds \uu \tau k -\ds \uu \tau{k-1}}{\tau}\right)\right|^{\omega}
\dd x,
\\
&
\rho \int_{\GC} \left| \frac{\ds \chi \tau k -\ds \chi \tau{k-1}}{\tau} \right|^{\omega-2} \frac{\ds \chi \tau k - \ds \chi \tau {k-1}}\tau (\ds \chi \tau k {-}\ds \chi \tau{k-1}) \dd x = \rho\tau \int_{\GC}
\left| \frac{\ds \chi \tau k -\ds \chi \tau{k-1}}{\tau} \right|^{\omega} \dd x\,.
\end{aligned}
\]
Taking into account the above calculations, we conclude \eqref{discr-tot-enbal-trunc}.
\end{proof}

%
%
%

\subsection{Existence of solutions to the $M$-truncated discrete system}
\label{ss:4.2}
Now, we perform, for $\tau>0$, $M>0$, and $k \in \{ 1, \ldots, K_\tau \}$  fixed, the limit passage in system \eqref{discr-syst-trunc} as $\epsilon \down 0$.
In this way, we will obtain discrete solutions to a discrete system featuring only the $M$-truncation in $\alpha$.
To shorten notation, throughout this section we will abbreviate  the solution quadruple $(\ds \teta {\tau,\epsilon,M}k,\ds \uu {\tau,\epsilon,M}k,  \dss  {\tau,\epsilon,M}k, \ds \chi {\tau,\epsilon,M}k)$
with $(\ds \teta \epsilon k, \ds \uu \epsilon k, \dss \epsilon k, \ds \chi \epsilon k)$. Likewise, we will simply denote by
\RCNEW $\ds \sigma \epsilon k$ the selections in $\partial \varphi(\ds \chi \epsilon k)$, and use the notation
$\ds \zzeta \epsilon k : = \eta_\varsigma(\ds \uu {\epsilon}k {\cdot}\mathbf{n}) \mathbf{n}$. \EEE
\par
Our first result collects a series of a priori estimate on the sequence  $(\ds \teta \epsilon k, \ds \uu \epsilon k, \dss \epsilon k, \ds \chi \epsilon k, \xiek)_\epsilon$. They hold
\RCORR uniformly w.r.t.\ $\tau$ in $(0,\bar{\tau})$ for some $\bar\tau>0$ that shall be  specified
in the proof), \EEE
uniformly w.r.t.\ $\epsilon$ and, in fact, w.r.t.\ $M>0$.
\begin{lemma}
\label{l:epsi-a-prio}
Let $\tau \in (0,\bar\tau)$, for some   $\bar{\tau}>0$,  and $k \in \{ 1, \ldots, K_\tau \}$ be  fixed. There exists a constant $S_1>0$, also \emph{independent} of $M>0$,  such that the following estimate holds
\begin{equation}
\label{a-prio-est}
\sup_{\epsilon>0} \left(\| \ds\teta \epsilon k\|_{H^1(\Omega)} {+} \| \dss \epsilon k\|_{H^1(\GC)} {+} \|  \ds\uu \epsilon k\|_{W^{1,\omega}(\Omega;\R^3)}
{+} \| \ds \chi \epsilon k \|_{H^2(\GC)} \RCNEW {+} \| \xiek \RCNEW \|_{L^\infty(\GC)} \EEE \right) \leq S_1\,,
\end{equation}
in addition to estimate \eqref{Lp-est-4-recipr-eM}. 
\end{lemma}
\begin{proof}
Observe that the fourth and the fifth terms on the left-hand side of the energy inequality \eqref{discr-tot-enbal-trunc} are positive. Taking into account that the functions $(\ds \teta \tau {k-1}, \ds \uu \tau {k-1},
\dss \tau {k-1},\ds \chi \tau{k-1})$ are given, and recalling \RCNEW assumptions
\eqref{cond-data} on the problem data, \EEE  we thus infer from \eqref{discr-tot-enbal-trunc}  that
\[
\begin{aligned}
&
\calE_\varsigma (\ds \teta {\epsilon} k,  \dss {\epsilon} k,  \ds \uu {\epsilon} k, \ds \chi  {\epsilon} k)
+\rho\tau \int_{\Omega} \left|\eps \left(\tfrac{\ds \uu \epsilon k {-} \ds \uu \tau {k-1}}{\tau} \right) \right|^\omega \dd x  + \rho \tau \int_{\GC} \left|\tfrac{\ds \chi \epsilon k - \ds \chi \tau{k-1}}{\tau}\right|^\omega \dd x
 \leq  C + C \| \ds \uu \epsilon k \|_{H^1 (\Omega)}\,.
 \end{aligned}
\]
Now, the functional  \RCORR $\calE_\varsigma$ enjoys the coercivity properties  \eqref{coerc-ene}  (with the exception of the
control of the $\|\cdot\|_{L^\infty(\GC)}$-norm, and of the enforcement of the positivity, \EEE of $\chi$). Taking them into account,
we absorb the second term on the right-hand side of the above estimate into the energy term
$\calE_\varsigma (\ds \teta {\epsilon} k,  \dss {\epsilon} k,  \ds \uu {\epsilon} k, \ds \chi  {\epsilon} k)$, and thus conclude a bound for the whole
left-hand side. Again in view of \eqref{coerc-ene},
we thus conclude that
\begin{equation}
\label{right-from-enest}
\sup_{\epsilon>0}  \left(  \| \ds\teta \epsilon k\|_{L^1(\Omega)} {+} \| \dss \epsilon k\|_{L^1(\GC)} {+}
\RCNEW   \| \ds\uu \epsilon k  \|_{W^{1,\omega}(\Omega;\R^3)}
{+} \| \ds \chi \epsilon k \|_{H^1(\GC)}
\RCNEW  \right)   \leq C. \EEE
\end{equation}
\RCNEW Furthermore, by the definition of $\partial\varphi$ (cf.\
\eqref{subdiff-pos-part}) \EEE
we have that $ \| \xiek \RCNEW \|_{L^\infty(\GC)} \leq C$. \EEE
\par
Next, we test \eqref{discr-b-heat-trunc} by $\ds \teta \epsilon k$ and \eqref{discr-surf-temp-trunc} by $\dss \epsilon k$
and add the resulting relations; observe that all the positive parts $(\cdot)^+$  can be
removed, in view of \eqref{non-negativity}.
We thus obtain
\begin{equation}
\label{new-calcul-epsl}
\begin{aligned}
 & \| \ds \teta \epsilon k\|_{L^2(\Omega)}^2 +\| \dss  \epsilon k\|_{L^2(\GC)}^2
+
\tau \io \alpha_M(\ds \teta \epsilon k) |\nabla \ds \teta \epsilon k|^2 \dd x
+
\tau \int_{\GC} \alpha_M(\dss \epsilon k) |\nabla \dss \epsilon k|^2 \dd x
+ I_1 +I_2
\\ &  = I_3 +I_4 +I_5+I_6+I_7+I_8+I_9+I_{10}
 \end{aligned}
\end{equation}
with
\[
\begin{aligned}
&
I_1=\tau \int_{\GC} k(\chikmu) (\mathcal{T}_\epsilon (\ds \teta \epsilon k) \ds \teta \epsilon k {-} \mathcal{T}_\epsilon (\dss  \epsilon k) \dss \epsilon k) (\ds \teta \epsilon k {-} \dss\epsilon k)\dd x \geq 0,
\\
&
\begin{aligned}
I_2
 & =  \tau\int_{\GC}  \mathcal{T}_\epsilon(\ds \teta \epsilon k) \ds \teta \epsilon k \left(  \nlocss {\chikmup}\, \chikmup \ds \teta \epsilon k {-} \nlocss {\chikmup \dss \epsilon k}\, \chikmup  \right) \dd x
\\
& \quad - \tau \int_{\GC} \mathcal{T}_\epsilon(\dss \epsilon k)  \dss \epsilon k \left( \nlocss {\chikmup \ds \teta \epsilon k} \,\chikmup{-} \nlocss {\chikmup}\, \chikmup \dss \epsilon k \right) \dd x
\\
& \stackrel{(1)}{=} \tau  \iint_{\GC\times\GC}  j(x,y) \RCNEW ( \chikmu (x))^+( \chikmu(y))^+ \EEE ( \ds \teta \epsilon k(x) -  \dss \epsilon k(y))(\mathcal{T}_\epsilon(\ds \teta \epsilon k(x)) \ds \teta \epsilon k(x) {-}
\mathcal{T}_\epsilon(\dss \epsilon k(y))  \dss \epsilon k(y)) \dd x \dd y \geq 0 \MC,\EEE
\end{aligned}
\end{aligned}
\]
where (1) follows by the very same arguments used for \eqref{key-posit-conto}; both $I_1$ and $I_2$ are positive since the function $r\mapsto \mathcal{T}_\epsilon(r) r$ is increasing.  
As for the terms on the right-hand side, we have
\[
\begin{aligned}
&
I_3 = \int_\Omega  \mathcal{T}_\epsilon(\tetakmu) \ds \teta \epsilon k  \dd x \leq C \| \ds \teta \epsilon k\|_{L^2(\Omega)} \leq \frac18  \| \ds \teta \epsilon k\|_{L^2(\Omega)}^2 + C,
\\
&
I_4 = \int_{\GC}  \mathcal{T}_\epsilon(\tetaskmu) \dss \epsilon k  \dd x \leq C \| \dss  \epsilon k\|_{L^2(\GC)} \leq \frac18  \| \dss  \epsilon k\|_{L^2(\GC)}^2 +C,
\\
&
\begin{aligned}
I_5 = \tau \io  \mathrm{div} \bigg( \tfrac{\ds \uu \epsilon k-\uukmu}{\tau} \bigg) |\ds \teta \epsilon k|^2  \dd x  & \leq C \tau \left\| \tfrac{\ds \uu \epsilon k-\uukmu}\tau \right\|_{W^{1,4}(\Omega)} \| \ds \teta \epsilon k \|_{L^4(\Omega)}
\| \ds \teta \epsilon k \|_{L^2(\Omega)}  \\ & \stackrel{(2)}{\leq} C\tau \| \ds \teta \epsilon k \|_{L^4(\Omega)}
\| \ds \teta \epsilon k \|_{L^2(\Omega)}  \stackrel{(3)}{\leq}   C \tau^2 \| \ds \teta \epsilon k\|_{H^1(\Omega)}^2 +  \frac18  \| \ds \teta \epsilon k \|_{L^2(\Omega)}^2,
\end{aligned}
\\
&
I_6 = \tau
\io \eps\left( \tfrac{\ds \uu \epsilon k {-} \uukmu}{\tau}\right) \, \vtens \, \eps\left( \tfrac{\ds \uu \epsilon k {-} \uukmu}{\tau}\right)  \ds \teta \epsilon k   \dd x
\leq C\tau \|  \eps\left( \tfrac{\ds \uu \epsilon k {-} \uukmu}{\tau}\right)\|_{L^4(\Omega)}^2 \| \ds \teta \epsilon k \|_{L^2(\Omega)} \stackrel{(4)}{\leq}   \frac18  \| \ds \teta \epsilon k\|_{L^2(\Omega)}^2 + C,
\\
&
I_7  =  \tau \pairing{}{H^1(\Omega)}{\hk}{\ds \teta \epsilon k}    \stackrel{(5)}{\leq} C\tau   \| \ds \teta \epsilon k\|_{H^1(\Omega)} \leq \frac{\tau^2}2  \| \ds \teta \epsilon k\|_{H^1(\Omega)}^2 +C,
\\
&
I_8  =  \tau \pairing{}{H^1(\GC)}{\lk}{\dss  \epsilon k}    \stackrel{(6)}{\leq} C\tau   \| \dss \epsilon k\|_{H^1(\GC)} \leq \frac{\tau^2}2  \| \dss \epsilon k\|_{H^1(\GC)}^2 +C,
\\
&
\begin{aligned}
I_9 = \tau\int_{\GC} \tfrac{\lambda(\ds \chi\epsilon k)-\lambda(\ds \chi \tau{k-1})}\tau  |\dss \epsilon k|^2  \dd x &  \stackrel{(7)}{\leq}  C \tau  \left\| \tfrac{\ds \chi \epsilon k {-} \ds \chi \tau {k-1}}\tau\right \|_{L^4(\GC)}
\| \dss \epsilon k \|_{L^4(\GC)} \| \dss \epsilon k \|_{L^2(\GC)} \\ &   \stackrel{(8)}{\leq} \frac{C \tau^2}2  \| \dss  \epsilon k\|_{H^1(\GC)}^2 +  \frac18  \| \dss  \epsilon k \|_{L^2(\GC)}^2,
\end{aligned}
\\
&
I_{10} =\tau  \int_{\GC} \bigg| \tfrac{\ds \chi \epsilon k  - \chikmu}{\tau}\bigg|^2 \dss \epsilon k  \dd x  \leq C \tau \left\| \tfrac{\ds \chi \epsilon k - \chikmu }{\tau} \right\|_{L^4(\GC)}^2\|\dss \epsilon k\|_{L^2(\GC)}
 \stackrel{(9)}{\leq}   \frac18  \| \dss  \epsilon k\|_{L^2(\GC)}^2 + C,
\end{aligned}
\]
where (2) follows from \eqref{right-from-enest}, and in  (3)  we have to  choose $\tau>0$ small enough so that the term $ C \tau^2  \| \ds \teta \epsilon k\|_{H^1(\Omega)}^2$ can
be absorbed by the left-hand side of \eqref{new-calcul-epsl}, taking into account \eqref{positiv-alpha-M};  (4) also follows from \eqref{right-from-enest}; (5)  and (6) from \eqref{cond-h} and \eqref{cond-ell}, respectively; (7) from the Lipschitz continuity of $
\lambda$ and (8) from \eqref{right-from-enest}, just like (9). Again, in (8) we choose  $\tau>0$ small enough  in such a way as to absorb $ C \tau^2  \| \dss  \epsilon k\|_{H^1(\GC)}^2$ into the left-hand side of
 \eqref{new-calcul-epsl}.
All in all, combining the above calculations with \eqref{new-calcul-epsl} we easily conclude that
\[
\| \ds \teta \epsilon k\|_{L^2(\Omega)}^2 +\| \dss  \epsilon k\|_{L^2(\GC)}^2
+
\tau \io \alpha_M(\ds \teta \epsilon k) |\nabla \ds \teta \epsilon k|^2 \dd x
+
\tau \int_{\GC} \alpha_M(\dss \epsilon k) |\nabla \dss \epsilon k|^2 \dd x  \leq C,
\]
whence the $H^1$-bounds for $\ds \teta \epsilon k$ and $\dss \epsilon k$.
\par
Since $H^1(\GC) $ embeds continuously in $L^q(\GC)$ for every $q\in [1,\infty)$,  the term $\left|\tfrac{\ds \chi \epsilon k {-} \ds \chi \tau {k-1}}{\tau}\right|^{\omega-2}
\tfrac{\ds \chi \epsilon k {-} \ds \chi \tau {k-1}}{\tau}$ is estimated in any $L^q(\GC)$; with arguments analogous to those in the previous lines,  it is not difficult to check that
the right-hand side of  the discrete flow rule \eqref{discr-flow-rule-trunc} is estimated in $L^2(\GC)$. Hence, by comparison, $(A\ds \chi\epsilon k 
)_\epsilon$ is also estimated in
$L^2(\GC)$,
whence the estimate for $(\ds \chi \epsilon k)_\epsilon $ in $H^2(\GC)$.
%
\end{proof}
\par
We are now in a position to pass to the limit as $\epsilon \down 0$, for fixed $M>0$,  $\tau>0$, and $k \in \{1, \ldots, K_\tau \}$, in system \eqref{discr-syst-trunc}.
In what follows, for convenience we shall suppose that $\frac1\epsilon  \in \N \setminus \{0\}$,
so that, up to labelling
\RCNEW $(\ds \teta \epsilon k, \ds \uu \epsilon k, \dss \epsilon k, \ds \chi \epsilon k, \ds \sigma \epsilon k)$
   by means of the natural number $m = \frac1\epsilon$, the functions $(\ds \teta \epsilon k, \ds \uu \epsilon k, \dss \epsilon k, \ds \chi \epsilon k, \ds \sigma \epsilon k)_\epsilon$ \EEE
form a sequence.
\begin{lemma}
\label{limit-passage-as-epsilon}
There exist a (not relabeled) subsequence of  \RCNEW $(\ds \teta \epsilon k, \ds \uu \epsilon k, \dss \epsilon k, \ds \chi \epsilon k, \ds \sigma \epsilon k )_\epsilon$
and a quintuple
\[
(\ds \teta  {\tau,M} k, \ds \uu {\tau,M} k, \dss {\tau,M} k, \ds \chi {\tau,M} k, \ds \sigma {\tau,M} k)
 \in H^1(\Omega) \times W_{\mathrm{D}}^{1,\omega}(\Omega;\R^3)
\times H^1(\GC) \times  H^2(\GC) \times L^\infty(\GC)
\]
such that the following convergences hold as $\epsilon \down 0$:
\begin{equation}
\label{convs-epsilon-down-0}
\begin{aligned}
&
\ds \teta \epsilon k \weakto \ds \teta  {\tau,M} k && \text{in } H^1(\Omega),  && \dss \epsilon k \weakto  \dss {\tau,M} k && \text{in } H^1(\GC), && &&
\\
&
 \ds \uu \epsilon k \weakto \ds \uu {\tau,M} k &&\text{in } W^{1,\omega}(\Omega;\R^3), &&  \ds \chi \epsilon k \weakto  \ds \chi {\tau,M} k  && \text{in } H^2(\GC),  &&  \ds \sigma \epsilon k \weaksto \ds \sigma {\tau,M} k
 && \text{in } L^\infty(\GC),
  \end{aligned}
\end{equation}
and the functions
$
(\ds \teta  {\tau,M} k, \ds \uu {\tau,M} k, \dss {\tau,M} k, \ds \chi {\tau,M} k, \ds \sigma {\tau,M} k)$  \EEE
fulfill
\begin{equation}
\label{non-negativity-temps}
\ds \teta {\tau,M} k > 0 \quad \aein\ \Omega, \qquad \dss {\tau,M} k > 0 \quad \aein\ \GC,
\end{equation}
as well as
\begin{subequations}
\label{discr-syst-trunc-M}
\begin{itemize}
\item[-] the discrete bulk temperature equation
\begin{equation}
\label{discr-b-heat-trunc-M}
\begin{aligned}
&
\io \frac{\ds\teta{\tau,M}k - \tetakmu}{\tau} v \dd x
- \io \ds\teta{\tau,M}k \mathrm{div} \left( \tfrac{\ds \uu {\tau,M}k-\uukmu}{\tau} \right) v \dd x
+ \io \alpha_{M}(\ds\teta{\tau,M}k) \nabla \ds\teta{\tau,M}k \nabla v \dd x
\\
& \quad
+ \int_{\GC} k(\chikmu)\ds\teta{\tau,M}k(\ds\teta{\tau,M}k-\dss {\tau,M} k) v \dd x
+ \int_{\GC} \nlocss {\chikmup}\, \chikmup  (\ds\teta{\tau,M}k)^2     v \dd x
\\
& \quad
- \int_{\GC} \nlocss {\chikmup \dss{\tau,M}k}\,  \chikmup \ds\teta{\tau,M}k v \dd x
\\ &
=
\io \eps \left( \tfrac{\ds \uu {\tau,M}k-\uukmu}{\tau} \right)  \, \vtens \, \eps \left( \tfrac{\ds \uu {\tau,M}k-\uukmu}{\tau} \right)  v  \dd x
+ \pairing{}{H^1(\Omega)}{\hk}{ v}  \qquad \text{for all } v \in H^1(\Omega);
\end{aligned}
\end{equation}
\item[-] the discrete momentum balance equation \MC \eqref{discr-mom-bal}\EEE, with $\ds \zzeta{\tau, M} k =  \etar(\ds \uu {\tau,M}k \cdot \mathbf{n}) \mathbf{n} $;  
\item[-] the discrete surface temperature equation
\begin{equation}
\label{discr-surf-temp-trunc-M}
\begin{aligned}
&
\int_{\GC} \frac{\dss{\tau,M}k - \tetaskmu}{\tau} v \dd x
- \int_{\GC} \dss{\tau,M}k  \frac{\lambda(\ds \chi{\tau,M} k)-\lambda(\ds \chi \tau{k-1})}{\tau}  v \dd x
+ \int_{\GC} \alpha_M(\dss{\tau,M}k) \nabla \dss{\tau,M}k \nabla v \dd x
\\
&
=
\int_{\GC} \bigg| \frac{\ds \chi{\tau,M}k - \chikmu}{\tau}\bigg|^2 v \dd x
+ \int_{\GC} k(\chikmu)(\ds \teta{\tau,M}k -\dss{\tau,M}k) \dss{\tau,M}k v \dd x
+ \int_{\GC} \nlocss {\chikmup \ds\teta{\tau,M}k}\, \chikmup \dss{\tau,M}k v \dd x
\\
& \qquad - \int_{\GC} \nlocss {\chikmup} \chikmup (\dss{\tau,M}k)^2 v \dd x
+ \pairing{}{H^1(\GC)}{\lk}{ v}
\qquad \text{for all } v \in H^1(\GC);
\end{aligned}
\end{equation}
\item[-]
the discrete flow rule for the adhesion parameter
\eqref{discr-flow-rule} a.e.\ in $\GC$. 
\end{itemize}
Finally,
\begin{equation}
 \label{Lp-est-4-recipr-M}
\begin{aligned}
&
\exists\, S_0 >0 \  \forall\, p \in [1,\infty) \  \exists\, \bar{\tau}_p>0 \
\forall\, M>0 \
\forall\, \tau \in (0,\bar{\tau}_p)\,
 \  \forall\, k \in \{1, \ldots, K_\tau\} \, :
 \\
 &
 \quad
\left\| \frac1{\ds \teta {\tau,  M} k}\right\|_{L^p(\Omega)} + \left\| \frac1{\dss  {\tau,M} k}\right\|_{L^p(\GC)} \leq S_0,
\end{aligned}
\end{equation}
\end{subequations}
\end{lemma}
\begin{proof}
Convergences \eqref{convs-epsilon-down-0}
 are an immediate consequence of estimates \eqref{a-prio-est}.
 Besides, from \eqref{non-negativity} it immediately follows
 that
 \[
\ds \teta {\tau,M} k \geq 0 \quad \aein\ \Omega, \qquad \dss {\tau,M} k \geq 0 \quad \aein\ \GC.
 \]
   Furthermore,
 there  exists $\mathsf{E} \in L^{\omega/(\omega{-}1)}(\Omega;\R^{3\times 3})$ such that
  \begin{equation}
  \label{conv-epsi-gamma}
 \left |\eps \left(\tfrac{\ds \uu \epsilon k {-} \ds \uu \tau {k-1}}{\tau} \right) \right|^{\omega-2}\eps \left(\tfrac{\ds \uu \epsilon k {-} \ds \uu \tau {k-1}}{\tau} \right) \weakto \mathsf{E} \qquad \text{in }
  L^{\omega/(\omega-1)}(\Omega;\R^{3\times 3})\,.
\end{equation}
\RCORR Taking into account that $W^{1,\omega}(\Omega;\R^3) $ continuously embeds
(in the sense of traces) in $L^\infty(\GC)$,
we also observe that
$\RCNEW (\ds \chi {\epsilon} k)^+  \ds\uu{\epsilon} k \to (\ds \chi {\tau,M} k)^+  \ds\uu{\tau,M} k  $ and
 that, by Lemma \ref{lemmaK}, $\nlocss{(\ds \chi {\epsilon} k))^+}\, (\ds \chi {\epsilon} k)^+  \ds\uu{\epsilon} k \to \nlocss{(\ds \chi {\tau,M} k)^+}\, \ds \chi {\tau,M} k  \ds\uu{\tau,M} k $ \RCORR
 in $L^\infty(\GC)$.  \EEE
  Since $\etar$ is Lipschitz continuous, we ultimately infer that
 \[
 \ds \zzeta  \epsilon k = \etar(\ds \uu \epsilon k \cdot \mathbf{n})  \mathbf{n}  \to \etar (\ds \uu {\tau,M} k \cdot \mathbf{n})  \mathbf{n} \doteq  \ds \zzeta   {\tau,M} k
 \quad\text{ in $L^\infty(\GC)$.}
 \]
Hence,
 we readily pass to the limit in the discrete momentum balance \eqref{discr-mom-bal-trunc} and conclude that
 the quintuple
 $(\ds \teta {\tau,M} k, \ds \uu {\tau,M} k,   \dss {\tau,M} k, \ds \zzeta {\tau,M} k, \mathsf{E} )$
 fulfills
\[
\begin{aligned}
&
\vfm \left( \tfrac{ \ds \uu {\tau,M} k-\uukmu}{\tau}, \vv \right)
+ \rho \io \mathsf{E} : \eps (\vv)  \dd x
+ \efm \left( \ds \uu {\tau,M}k , \vv \right)
  + \io \ds \teta{\tau,M} k \mathrm{div}(\vv) \dd x
  + \int_{\GC} \RCNEW (\ds \chi{\tau,M} k)^+ \EEE \ds \uu{\tau,M} k\vv \dd x
\\ & \quad+\int_{\GC} \ds \zzeta {\tau,M} k\vv \dd x
+ \int_{\GC} \nlocss { \RCNEW (\ds \chi{\tau,M} k)^+} \,   \RCNEW (\ds \chi{\tau,M} k)^+ \EEE  \ds \uu{\tau,M} k \vv \dd x
=
\langle \FF^k, \vv \rangle_{\bsVD} \qquad \text{for all } \vv \in \RCORR W^{1,\omega}_{\mathrm{D}}(\Omega;\R^3)\,. \EEE
\end{aligned}
\]
Furthermore, testing
 \eqref{discr-mom-bal-trunc}  by $\tfrac{\ds \uu {\epsilon} k - \uukmu}\tau$ and passing to the limit in the equation we readily show that
 \[
 \begin{aligned}
 &
 \limsup_{\epsilon \down 0} \left( \rho  \io  \left|\eps \left( \tfrac{\ds \uu {\epsilon} k - \uukmu}\tau\right)\right|^\omega   \right)  \dd x 
 \\
 &
\leq
- \vfm \left( \tfrac{ \ds \uu {\tau,M} k-\uukmu}{\tau},\tfrac{ \ds \uu {\tau,M} k-\uukmu}{\tau}  \right)
- \efm \left( \ds \uu {\tau,M}k ,  \tfrac{ \ds \uu {\tau,M} k-\uukmu}{\tau}, \right)
- \io \ds \teta{\tau,M} k \mathrm{div}( \tfrac{ \ds \uu {\tau,M} k-\uukmu}{\tau}) \dd x
\\
& \qquad
- \int_{\GC}  \RCNEW (\ds \chi{\tau,M} k)^+ \ds \uu{\tau,M} k  \tfrac{ \ds \uu {\tau,M} k-\uukmu}{\tau} \dd x
- \int_{\GC} \nlocss { \RCNEW (\ds \chi{\tau,M} k)^+} \, \RCNEW (\ds \chi{\tau,M} k)^+ \EEE   \ds \uu{\tau,M} k  \tfrac{ \ds \uu {\tau,M} k-\uukmu}{\tau} \dd x
- \langle \FF^k,  \tfrac{ \ds \uu {\tau,M} k-\uukmu}{\tau} \rangle_{\bsVD}
\\
& \qquad
- \int_{\GC} \ds \zzeta {\tau,M} k \cdot \tfrac{\ds \uu {\tau,M} k - \uukmu}\tau \dd x
\\ &
=  \rho \io \mathsf{E}  \cdot \eps (\tfrac{ \ds \uu {\tau,M} k-\uukmu}{\tau})  \dd x
\end{aligned}
 \]
Thus, by standard results on the theory of maximal monotone operators
(cf., e.g., \ \cite[Lemma~1.3, p.~42]{Barbu}), we infer that
\[
\mathsf{E} =  \left |\eps \left(\tfrac{\ds \uu {\tau,M} k {-} \ds \uu \tau {k-1}}{\tau} \right) \right|^{\omega-2}\eps \left(\tfrac{\ds \uu{\tau,M} k {-} \ds \uu \tau {k-1}}{\tau} \right).
\]
A fortiori, we conclude that
\begin{equation}
\label{a-fortiori}
\eps \left(\tfrac{\ds \uu {\epsilon} k {-} \ds \uu \tau {k-1}}{\tau} \right) \to \eps \left(\tfrac{\ds \uu {\tau,M} k {-} \ds \uu \tau {k-1}}{\tau} \right) \quad \text{in } L^\omega(\Omega;\R^{3\times 3})\,.
\end{equation}
This completes the limit passage in  \eqref{discr-mom-bal-trunc},  leading to  the discrete momentum balance \eqref{discr-mom-bal}.
\par
We now address the limit passage in the discrete truncated bulk heat equation \eqref{discr-b-heat-trunc}.
First of all,
we observe that, as $\epsilon \down 0$,
\begin{equation}
\label{conv-Lq-Teps}
\mathcal{T}_\epsilon(\ds \teta {\epsilon}k) \to (\ds \teta{\tau,M}k)^+  = \ds \teta{\tau,M} k \qquad \text{in $L^q(\Omega) $ for every $q\in [1,6)$,}
\end{equation}
 and that the traces
of $\mathcal{T}_\epsilon(\ds \teta {\epsilon}k)$ strongly  converge to  the trace of $ \ds \teta{\tau,M} k$ in $L^q(\GC)$ for every $q\in [1,4)$. Furthermore, taking into account that
$\|\chikmup\|_{L^\infty(\GC)} \leq C \|\chikmu\|_{H^2(\GC)} \leq C $, $\|k(\chikmu)\|_{L^\infty(\GC)} \leq C$,
 that $\| \nlocss {\chikmup \dss {\epsilon}k} \|_{L^\infty(\GC)} \leq  C \|\chikmup \dss {\epsilon}k\|_{L^1(\GC)} \leq C $ by Lemma \ref{lemmaK},
 and recalling \eqref{a-fortiori},
 by a comparison in \eqref{discr-b-heat-trunc} we see that
\begin{equation}
\label{estimateAM}
  \exists\, C>0 \ \forall\, \epsilon>0 \ \sup_{v \in H^1(\Omega)} \left| \io   \alpha_{M}(\ds \teta{\epsilon}k ) \nabla \ds \teta{\epsilon} k \cdot \nabla v \dd x \right| \leq C.
  \end{equation}
 \RCORR The above estimate \EEE  can be rephrased, in terms of the operator $\mathcal{A}_M$ from \eqref{def-AM},  as
 $
 \sup_{\epsilon>0} \| \mathcal{A}_M(\ds \teta{\epsilon}k )\|_{H^1(\Omega)^*} \leq C.
 $
 We  combine  this with the facts that
 $\nabla \ds \teta{\epsilon} k  \weakto \nabla \ds \teta{\tau,M}k $ in $L^2(\Omega;\R^3)$ and that
  $\alpha_M(\ds \teta{\epsilon} k ) \to \alpha_M(\ds \teta{\tau,M}k )$ in $L^q(\Omega) $ for every $1\leq q <\infty$
 (since $\ds \teta{\epsilon} k \to \ds \teta{\tau,M} k $ a.e.\ in $
\Omega$ and the function $\alpha_M : \R \to (0,+\infty)$ is bounded).
Therefore,
$\mathcal{A}_M(\ds \teta{\epsilon}k) \weakto \mathcal{A}_M(\ds \teta{\tau,M} k)$ in $W^{1,s}(\Omega)^*$ for every $s>2$. Thanks to  \eqref{estimateAM},
we conclude that
\begin{equation}
 \label{AM-epsilon}
\mathcal{A}_M(\ds \teta{\epsilon}k) \weakto \mathcal{A}_M(\ds \teta{\tau,M} k) \qquad \text{in } H^1(\Omega)^*.
\end{equation}
We also use \RCORR the strong convergences \EEE
\begin{equation}
\label{other-conv-props}
\begin{aligned}
&
 \ds \teta{\epsilon} k \mathrm{div}\left( \tfrac{\ds \uu \epsilon k {-} \uukmu}\tau\right) \to  \ds \teta{\tau,M} k \mathrm{div}\left( \tfrac{\ds \uu {\tau,M} k {-} \uukmu}\tau\right)  && \text{in } L^q(\Omega) \text{ for all }
 q \in [1,\tfrac32),
 \\
 &
  k(\chikmu)\mathcal{T}_\epsilon(\ds \teta{\epsilon}k)(\ds \teta {\epsilon} k-\dss {\epsilon} k) \to  k(\chikmu)\ds \teta{\tau,M}k(\ds\teta {\tau,M} k-\dss {\tau,M} k)  && \text{in } L^q(\GC)  \text{ for all }
 q \in [1,2),
 \\
 &
 \nlocss {\chikmup} \,\chikmup  \ds \teta{\epsilon}k  \mathcal{T}_\epsilon(\ds \teta{\epsilon}k) \to  \nlocss {\chikmup}\, \chikmup  (\ds \teta{\tau,M}k)^2  && \text{in } L^q(\GC)  \text{ for all }
 q \in [1,2),
\\
&
\nlocss {\chikmup \dss {\epsilon}k}\, \chikmup \mathcal{T}_\epsilon(\ds \teta{\epsilon} k) \to \nlocss {\chikmup \dss {\tau,M}k}\, \chikmup \ds \teta{\tau,M} k && \text{in } L^q(\GC)  \text{ for all }
 q \in [1,4),
\\
&
\eps \left(\tfrac{\ds \uu {\epsilon} k {-} \ds \uu \tau {k-1}}{\tau} \right)  \, \vtens \, \eps \left(\tfrac{\ds \uu {\epsilon} k {-} \ds \uu \tau {k-1}}{\tau} \right)  \to
\eps \left(\tfrac{\ds \uu {\tau,M} k {-} \ds \uu \tau {k-1}}{\tau} \right)  \, \vtens \, \eps \left(\tfrac{\ds \uu {\tau,M} k {-} \ds \uu \tau {k-1}}{\tau} \right) && \text{in } L^{\omega/2}(\Omega) \MC, \EEE
\end{aligned}
\end{equation}
due to convergences \eqref{convs-epsilon-down-0}, to the properties of $\nlname$ (cf.\ Lemma \ref{lemmaK}), and to the previously proved \eqref{a-fortiori}.
Combining \eqref{AM-epsilon} and \eqref{other-conv-props} we pass to the limit in \eqref{discr-b-heat-trunc}, thus obtaining the discrete bulk heat equation \eqref{discr-b-heat-trunc-M}.
\par
The limit passage in  the discrete flow rule \eqref{discr-flow-rule-trunc} as $\epsilon \down 0$ is an easy consequence of convergences \eqref{convs-epsilon-down-0}, which in particular imply that
\RCORR $\ds\chi\eps k \to \ds \chi {\tau,M} k$ strongly in $L^\infty(\GC)$.
Hence,
 by the Lipschitz continuity of $\betar$  and $\gamma_{\nu}'$,  we conclude that
$\betar(\ds\chi\eps k) \to \betar(\ds \chi \tau k)$  and  $\gamma_\nu' (\ds\chi\eps k)\to
\gamma_\nu' (\ds \chi \tau k)$   in $L^\infty(\GC)$.
\RCNEW Furthermore, $\ds \sigma \eps k \weaksto \ds \sigma {\tau,M} k $ in $L^\infty(\GC)$ and, by the strong weak closedness (in the sense of graphs) of the maximal monotone operator (induced by) $\partial\varphi$, we readily conclude that
\[
 \ds \sigma {\tau,M} k \in \partial\varphi (  \ds \chi {\tau,M} k ) \qquad \aein\, \GC\,.
\]
  Hence,  the triple $(\dss {\tau,M}k, \ds \chi {\tau,M}k, \RCNEW \ds \sigma {\tau,M}k)$  \EEE fulfills the discrete flow rule \eqref{discr-flow-rule}.
\par
Finally, we address the passage to the limit in the discrete truncated  surface temperature equation \eqref{discr-surf-temp-trunc}: it is based on the fact that
\begin{equation}
\label{conv-Lq-Teps-tetas}
\mathcal{T}_\epsilon(\dss  {\epsilon}k) \to (\dss {\tau,M}k)^+  = \dss {\tau,M} k \qquad \text{in $L^q(\GC) $ for every $q\in [1,\infty)$,}
\end{equation}
on the convergence
\begin{equation}
 \label{AM-epsilon-tetas}
\mathcal{A}_M(\dss {\epsilon}k) \weakto \mathcal{A}_M(\dss {\tau,M} k) \qquad \text{in } H^1(\GC)^*,
\end{equation}
(which can be inferred by the same arguments as \eqref{AM-epsilon}), on the analogues of \RCORR  convergences  \eqref{other-conv-props} for the terms on the right-hand side of \eqref{discr-surf-temp-trunc}, \EEE and on the fact that
\[
\left| \tfrac{\ds \chi \epsilon k - \chikmu}{\tau} \right|^{\omega-2} \tfrac{\ds \chi \epsilon k - \chikmu}{\tau}  \to \left| \tfrac{\ds \chi {\tau,M} k - \chikmu}{\tau} \right|^{\omega-2}
\tfrac{\ds \chi {\tau,M} k - \chikmu}{\tau} \qquad \text{in } L^\infty(\GC)  
\]
as $\ds\chi \eps k \to \ds \chi {\tau,M} k$ strongly in $L^\infty(\GC)$. We thus obtain the discrete surface heat equation \eqref{discr-surf-temp}.
\par
Finally, to prove the strict positivity
\eqref{non-negativity-temps} and estimate \eqref{Lp-est-4-recipr-M},
 we combine  estimate \eqref{Lp-est-4-recipr-eM} with convergences \eqref{conv-Lq-Teps} \& \eqref{conv-Lq-Teps-tetas}.
 Let us just detail the argument for $\ds \teta {\tau,M}k$ (the positivity property for $\dss {\tau,M} k$ follows analogously). 
On the one hand, from  \eqref{Lp-est-4-recipr-eM} we infer that
\[
\int_{\Omega} \liminf_{\epsilon \down 0} \frac1{\mathcal{T}_\epsilon (\ds \teta \epsilon k)} \dd x \leq \sup_{\epsilon>0} \int_\Omega \frac1{\mathcal{T}_\epsilon (\ds \teta \epsilon k)} \dd x \leq S_0,
\]
so that  $\liminf_{\epsilon \down 0} \frac1{\mathcal{T}_\epsilon (\ds \teta \epsilon k(x))} \dd x<+\infty$ a.e.\ in $\Omega$. On the other hand, from \eqref{conv-Lq-Teps} we
have that
\[
 \frac1{\mathcal{T}_\epsilon (\ds \teta \epsilon k(x))} \longrightarrow \begin{cases} \frac1{\ds \teta {\tau,M}k (x)} & \text{if }  \ds \teta {\tau,M}k (x)>0,
 \\
 +\infty & \text{otherwise}
 \end{cases}
 \qquad \foraa\, x \in \Omega.
\]
Therefore, we conclude that $\ds \teta {\tau,M}k>0$ a.e.\ in $\Omega$, and estimate \eqref{Lp-est-4-recipr-M} follows by lower semicontinuity.
This finishes the proof.
\end{proof}
\subsection{Proof of Proposition \ref{prop:exists-discrete}}
\label{ss:4.3}
We shall carry out the proof of Proposition  \ref{prop:exists-discrete} by passing to the limit
 as $M \to +\infty$, for fixed  $\tau>0$, and $k \in \{1, \ldots, K_\tau \}$, in system \eqref{discr-syst-trunc-M}.
In what follows,  we shall  suppose that $M \in \N \setminus \{0\}$. \RCORR For simplicity, we shall drop the parameter $\tau$ \EEE
and just denote by
$(\ds \teta M k, \ds \uu M k, \dss M k, \ds \chi M k)_M$,
\RCNEW with associated selections $\ds \sigma Mk \in \partial\varphi(\ds \chi Mk)$ a.e.\ in $\GC$, \EEE
the sequence of solutions to system \eqref{discr-syst-trunc-M}.
We split the proof of the limit passage in some steps.
\par\noindent
Preliminarily, we will need the following result
(whose proof is left to the reader) collecting properties of the primitives of $\alpha$; observe that \eqref{growth-primitives} is a consequence of
  \eqref{hyp-alpha}.
\begin{lemma}
Define \[
\begin{aligned}
&
\doublehhat{\alpha} : \R^+ \to \R^+  \qquad \doublehhat{\alpha}(r): = \int_0^r \hhat{\alpha}(s) \dd s,
\end{aligned}
\]
(with $\widehat{\alpha}$ from \eqref{def-hat-alpha}).
The function $\hhat{\alpha}$ is (strictly) increasing and thus  $ \doublehhat{\alpha} $ is (strictly) convex.
Furthermore,
\begin{equation}
\label{growth-primitives}
\begin{aligned}
&
c_0 \left(r + \frac{r^{\mu+1}}{\mu+1} \right) \leq \hhat{\alpha}(r) \leq c_1 \left(r + \frac{r^{\mu+1}}{\mu+1} \right),
\\
&
c_0 \left(\frac {r^2}2 + \frac{r^{\mu+2}}{(\mu+2)(\mu+1)} \right) \leq \doublehhat{\alpha}(r) \leq c_1 \left(\frac{r^2}2 + \frac{r^{\mu+2}}{(\mu+2)(\mu+1)} \right).
\end{aligned}
\end{equation}
The functions
\begin{equation}
\label{important-for-monot-arguments}
\R^+ \ni r \mapsto \hhat{\alpha} (\mathcal{T}_M(r)) \qquad  \text{ and } \qquad   \R^+\ni r  \mapsto r \widehat{\alpha}(\mathcal{T}_M(r)) \text{  are non-decreasing.}
\end{equation}
Finally, the function
\begin{equation}
\doublehhat{\alpha}_M : \R^+ \to \R^+  \qquad \doublehhat{\alpha}_M(r): = \int_0^r \hhat{\alpha}( \mathcal{T}_M(s)) \dd s
\end{equation}
is convex, and
\begin{equation}
\label{control-hat-doublehat}
\exists\,  C_1', \, C_2'>0  \ \ \forall\, M>0 \ \ \forall\, r \in \R^+\, : \quad \widehat{\alpha} (\mathcal{T}_M(r)) \leq C_1' \doublehhat{\alpha}_M(r) + C_2'\,.
\end{equation}
An analogous estimate holds with $\widehat\alpha$  \RCORR in place of $\widehat\alpha \circ \mathcal{T}_M$ and $\doublehhat{\alpha}$ in place of $\doublehhat{\alpha}_M$. \EEE
\end{lemma}

\paragraph{\bf Proof of Proposition \ref{prop:exists-discrete}}

\noindent
\emph{Step $1$: a priori estimates.} Since the constant in estimate \eqref{a-prio-est} was independent of the parameter $M$, by virtue of convergences \eqref{convs-epsilon-down-0}
and lower semicontinuity arguments we immediately conclude that
\begin{equation}
\label{a-prio-est-M}
\sup_{M>0} \left(\| \ds\teta M k\|_{H^1(\Omega)} {+} \| \dss M k\|_{H^1(\GC)} {+} \|  \ds\uu M k\|_{W^{1,\omega}(\Omega;\R^3)}
{+} \| \ds \chi M k \|_{H^2(\GC)} 
\RCNEW {+}  \|\ds \sigma  M k\|_{L^\infty(\GC)}\EEE \right) \leq S_1\,. 
\end{equation}
\par
Next, we test \eqref{discr-b-heat-trunc-M} by $\tau \hhat{\alpha}(\mathcal{T}_M(\ds \teta{M} k))$, \eqref{discr-surf-temp-trunc-M}
by $\tau \hhat{\alpha}(\mathcal{T}_M(\dss {M} k))$ and add the resulting equations
(it is a standard matter to check that $\hhat{\alpha}(\mathcal{T}_M(\ds \teta{M} k)) \in H^1(\Omega)$ and
$\hhat{\alpha}(\mathcal{T}_M(\ds \teta{M} k)) \in H^1(\GC)$).
By convexity of $\doublehhat{\alpha}_M$, we have that
\[
\begin{aligned}
&
\io
\doublehhat{\alpha}_M (\ds\teta{M}k) \dd x - \io \doublehhat{\alpha}_M  (\tetakmu) \dd x \leq
\io (\ds\teta{M}k {-} \tetakmu)  \hhat{\alpha}(\mathcal{T}_M(\ds \teta{M} k)) \dd x,
\\
 &
\int_{\GC}
\doublehhat{\alpha}_M (\dss {M}k) \dd x - \int_{\GC} \doublehhat{\alpha}_M  (\tetaskmu) \dd x \leq
\int_{\GC} (\dss{M}k {-} \tetaskmu)  \hhat{\alpha}(\mathcal{T}_M(\dss {M} k)) \dd x\,.
\\
\end{aligned}
\]
Using that
\begin{equation}
\label{Marcus-Mizel}
\nabla ( \hhat{\alpha}(\mathcal{T}_M(\ds \teta{M} k)))  =  \alpha_{M}(\ds\teta{M}k) \nabla (\mathcal{T}_M(\ds \teta{M} k))
\end{equation}
 (cf., e.g.,
  \cite{MarcusMizel}),
 it is immediate to check that
\[
\tau \io \alpha_{M}(\ds\teta{M}k) \nabla \ds\teta{M}k \nabla ( \hat{\alpha}(\mathcal{T}_M(\ds \teta{M} k)))   \dd x
= \tau \io   \left| \nabla ( \hhat{\alpha}(\mathcal{T}_M(\ds \teta{M} k)))\right|^2 \dd x
\]
and we deal with the term $\tau \int_{\GC} \alpha_{M}(\dss{M}k) \nabla \dss {M}k \nabla ( \hhat{\alpha}(\mathcal{T}_M(\dss {M} k)))   \dd x$ analogously.
By the second monotonicity property
in \eqref{important-for-monot-arguments},
we have that
\[
\tau \int_{\GC} k(\chikmu)(\ds\teta{M}k-\dss {M} k)  (\ds\teta{M}k \hhat{\alpha}(\mathcal{T}_M(\ds \teta{M} k)){-}\dss{M}k \hhat{\alpha}(\mathcal{T}_M(\dss{M} k)))   \dd x
\geq 0,
\]
while with calculations completely analogous to those for \eqref{posult4}  we can check that
\[
\begin{aligned}
&
- \tau \int_{\GC}   \hhat{\alpha}(\mathcal{T}_M(\ds \teta{M} k))
  \left(  \nlocss {\chikmup}\, \chikmup  (\ds \teta M k)^2 {-}  \nlocss {\chikmup \dss M k} \,\chikmup \ds \teta M k \right)    \dd x
\\
& \quad
+\tau  \int_{\GC}  \hhat{\alpha}(\mathcal{T}_M(\dss {M} k)) \left(  \nlocss {\chikmu \ds \teta M k} \chikmu \dss M k {-} \nlocss {\chikmup} \chikmup (\dss M k)^2  \right) \dd x \geq 0\,.
\\
\end{aligned}
\]
Taking into account the above calculations,
  we end up with the following estimate
 \begin{equation}
 \label{initial-estimate}
 \begin{aligned}
 &
\io
\doublehhat{\alpha}_M (\ds\teta{M}k) \dd x  +\int_{\GC}
\doublehhat{\alpha}_M (\dss {M}k) \dd x  + \tau
 \| \nabla(  \hhat{\alpha}(\mathcal{T}_M(\ds \teta{M} k)))\|_{L^2(\Omega)}^2 +
  \tau \| \nabla  (\hhat{\alpha}(\mathcal{T}_M(\dss {M} k)))\|_{L^2(\GC)}^2
  \\
   &
   \leq
    \io \doublehhat{\alpha}_M  (\tetakmu) \dd x
    +  \int_{\GC} \doublehhat{\alpha}_M  (\tetaskmu) \dd x
    +  \tau \int_\Omega  \ds\teta{M}k  \hhat{\alpha}(\mathcal{T}_M(\ds \teta{M} k)) \mathrm{div} \left( \tfrac{\ds \uu {M}k-\uukmu}{\tau} \right)  \dd x
    \\
    & \quad
   + \tau \io \eps \left( \tfrac{\ds \uu {M}k-\uukmu}{\tau} \right)  \, \vtens \, \eps \left( \tfrac{\ds \uu {M}k-\uukmu}{\tau} \right)  \hhat{\alpha}(\mathcal{T}_M(\ds \teta{M} k))   \dd x
   \\
   & \quad
   + \tau  \pairing{}{H^1(\Omega)}{\hk}{ \hhat{\alpha}(\mathcal{T}_M(\ds \teta{M} k))}
   + \tau  \pairing{}{H^1(\GC)}{\lk}{ \hhat{\alpha}(\mathcal{T}_M(\dss {M} k))}
   \\
   & \quad
   +\tau \int_{\GC} \dss{M}k  \hhat{\alpha}(\mathcal{T}_M(\dss {M} k)) \frac{\lambda(\ds \chi{M} k)-\lambda(\ds \chi \tau{k-1})}{\tau}  v \dd x
   +\tau \int_{\GC} \bigg| \frac{\ds \chi{\tau,M}k - \chikmu}{\tau}\bigg|^2  \hhat{\alpha}(\mathcal{T}_M(\dss {M} k))  \dd x
   \\
   & \doteq I_1+I_2+I_3+I_4+I_5+I_6+I_7+I_8\,.
\end{aligned}
\end{equation}
Now, taking into account
\eqref{growth-primitives}, it is not difficult to check that  
\[
I_1 +I_2 \leq C \left(1+  \| \tetakmu \|_{L^{\mu+2}(\Omega)}^{\mu+2} +   \| \tetakmu \|_{L^{\mu+2}(\GC)}^{\mu+2}  \right) \leq C,
\]
since, by assumption (cf.\ \eqref{previous-data}), $\tetakmu \in L^{\mu+2}(\Omega)$ and  $\tetaskmu \in L^{\mu+2}(\GC)$.
In order to estimate the ensuing terms, we will use that
\begin{equation}
\label{Poincare-hat}
\begin{aligned}
\| \hhat{\alpha}(\mathcal{T}_M(\ds \teta{M} k))\|_{H^1(\Omega)}  & \stackrel{(1)}{\leq} C \left(\| \nabla(\hhat{\alpha}(\mathcal{T}_M(\ds \teta{M} k)) )\|_{L^2(\Omega)} +
\| \hhat{\alpha}(\mathcal{T}_M(\ds \teta{M} k))\|_{L^1(\Omega)}  \right)
\\
&  \stackrel{(2)}{\leq} C' \left(\| \nabla(\hhat{\alpha}(\mathcal{T}_M(\ds \teta{M} k)) )\|_{L^2(\Omega)} {+}
\| \doublehhat{\alpha}_M(\ds \teta{M} k)\|_{L^1(\Omega)}  {+}1 \right),
\end{aligned}
\end{equation}
where (1) follows from the Poincar\'e inequality, and (2) from \eqref{control-hat-doublehat}. Clearly, an analogous estimate holds for
$\| \hhat{\alpha}(\mathcal{T}_M(\dss {M} k))\|_{H^1(\GC)}$.
Hence
\begin{equation}
\label{est-I-3}
\begin{aligned}
|I_3|    & \leq C \tau
  \| \ds\teta{M}k \|_{L^4(\Omega)} \| \hhat{\alpha}(\mathcal{T}_M(\ds \teta{M} k))) \|_{L^2(\Omega)} \left\| \mathrm{div} \left( \tfrac{\ds \uu {M}k-\uukmu}{\tau} \right)
  \right\|_{L^4(\Omega)}
  \\
  & \stackrel{(3)}{\leq} C'\tau \| \hhat{\alpha}(\mathcal{T}_M(\ds \teta{M} k))) \|_{L^2(\Omega)}
  \\
  &  \stackrel{(4)}{\leq}  \frac{\tau}4 \| \nabla(\hhat{\alpha}(\mathcal{T}_M(\ds \teta{M} k)) )\|_{L^2(\Omega)}^2 + \overline{C}_1 \tau
  \| \doublehhat{\alpha}_M(\ds \teta{M} k)\|_{L^1(\Omega)}
  + C \MC, \EEE
 \end{aligned}
\end{equation}
where for (3) we have exploited the previously observed \RCORR estimate \EEE \eqref{a-prio-est-M}, while (4) follows from \eqref{Poincare-hat} and Young's inequality.
With analogous calculations, again taking into account \eqref{a-prio-est-M} and combining \RCORR it \EEE  with
\eqref{hyp-lambda} and  \eqref{cond-h}--\eqref{cond-ell},  we easily conclude that
\begin{equation}
\label{est-I-4}
\begin{aligned}
I_4+I_5+I_6+I_7+I_8
 \leq
  &  \frac{\tau}4 \| \nabla(\hhat{\alpha}(\mathcal{T}_M(\ds \teta{M} k)) )\|_{L^2(\Omega)}^2
    \frac{\tau}4 \| \nabla(\hhat{\alpha}(\mathcal{T}_M(\dss {M} k)) )\|_{L^2(\GC)}^2
    \\
    & \quad + \RCORR \overline{C}_2  \EEE \tau \| \doublehhat{\alpha}_M(\ds \teta{M} k)\|_{L^1(\Omega)}
  + \RCORR  \overline{C}_2 \EEE \tau \| \doublehhat{\alpha}_M(\dss{M} k)\|_{L^1(\GC)}
  + C\,.
  \end{aligned}
\end{equation}
Hence, we choose $\tau>0$ sufficiently small such that $\tau ( \overline{C}_1 {+} \overline{C}_2 )<\tfrac12$, so that the terms
with  $ \| \doublehhat{\alpha}_M(\ds \teta{M} k)\|_{L^1(\Omega)} $ and
$\| \doublehhat{\alpha}_M(\dss{M} k)\|_{L^1(\GC)} $ can be absorbed into the
right-hand side of \eqref{initial-estimate}.
Them, combining \eqref{initial-estimate} with \eqref{est-I-3} and \eqref{est-I-4}, and again taking into account \eqref{Poincare-hat}, we infer that
\begin{equation}
\label{a-prio-est-M-2}
\sup_{M>0} \left( \| \doublehhat{\alpha}_M(\ds \teta{M} k)\|_{L^1(\Omega)}{+}\| \doublehhat{\alpha}_M(\dss{M} k)\|_{L^1(\GC)} {+}
\| \hhat{\alpha}(\mathcal{T}_M(\ds \teta{M} k))\|_{H^1(\Omega)} {+}
\| \hhat{\alpha}(\mathcal{T}_M(\dss {M} k))\|_{H^1(\GC)}
 \right) \leq S_2\,.
\end{equation}
In particular, in view of \eqref{growth-primitives} we conclude that
\[
\exists\, C>0 \ \ \forall\, M>0 \, : \quad \| \ds \teta{M} k \|_{L^{\mu+2}(\Omega)} +
\RCORR \| \EEE \dss {M} k \|_{L^{\mu+2}(\GC)} \leq C\,.
\]
\emph{Step $2$: limit passage as $M\uparrow \infty$.}
In view of estimates \eqref{a-prio-est-M},
there exist a (not relabeled) subsequence and  functions $(\ds \teta \tau k, \ds \uu \tau k, \dss \tau k, \ds \chi \tau k,  \RCNEW \ds \sigma \tau k)$ \EEE  
 that
\begin{equation}
\label{convs-epsilon-down-0-1}
\begin{aligned}
&
\ds \teta M k \weakto \ds \teta  {\tau} k && \text{in } H^1(\Omega) \cap L^{\mu+2}(\Omega),  && \dss M k \weakto  \dss {\tau} k && \text{in } \RCORR H^1(\GC), \EEE 
&& &&
\\
&
 \ds \uu M k \weakto \ds \uu {\tau} k &&\text{in } W^{1,\omega}(\Omega;\R^3), &&  \ds \chi M k \weakto  \ds \chi {\tau} k  && \text{in } H^2(\GC),  && \RCNEW \ds \sigma M k \weaksto  \ds \sigma {\tau} k  \EEE  && \text{in } L^\infty(\GC).
  \end{aligned}
\end{equation}
With the very same arguments as in the proof of Lemma \ref{limit-passage-as-epsilon},
from estimate \eqref{Lp-est-4-recipr-M} we conclude that
$\ds \teta  {\tau} k>0$ a.e.\ in $\Omega$ and $\dss \tau k>0$ a.e.\ in $\GC$ and the validity of estimate  \eqref{Lp-est-4-recipr}.
\par
The limit passage in  the momentum balance and in the flow rule for the adhesion parameter in system
 \eqref{discr-syst-trunc-M} follows \RCORR by the very same arguments \EEE
 as in the proof of
  Lemma \ref{limit-passage-as-epsilon}. In this way, we conclude that the
  \RCNEW quintuple $(\ds \teta \tau k, \ds \uu \tau k, \dss \tau k, \ds \chi \tau k, \ds \sigma \tau k)$
   solve \eqref{discr-mom-bal} and \eqref{discr-flow-rule}. \EEE
  \par
 Therefore, repeating
 the arguments in the proof of   Lemma \ref{limit-passage-as-epsilon}, we pass to the limit in
the discrete bulk heat equation
\eqref{discr-b-heat-trunc-M}
and in the surface heat equations \eqref{discr-surf-temp-trunc-M}. We only detail the passage to the limit in the term featuring the operator $\mathcal{A}(\ds \teta M k): H^1(\Omega)\to H^1(\Omega)^*$  defined by
$\pairing{}{H^1(\Omega)}{\mathcal{A
}(\ds \teta Mk)}{v}: =  \int_{\Omega} \alpha_M(\ds \teta M k) \nabla \ds \teta M k \cdot \nabla v \dd x $.
On the one hand,  the sequence of operators $(\mathcal{A}(\ds \teta M k))_M \subset H^1(\Omega)^*$ is bounded, by comparison in
\eqref{discr-b-heat-trunc-M}.  On the other hand, we observe that
  $ \mathcal{T}_M(\ds \teta M k) \to \ds \teta \tau k$ a.e.\ in $\Omega$ and hence
  $\alpha(\mathcal{T}_M(\ds \teta M k)) \to \alpha(\ds \teta \tau k) $ a.e.\ in
  $\Omega$, whereas  estimates
 \eqref{a-prio-est-M-2},
 combined with the growth properties of $\widehat\alpha$,
  guarantee that $( \mathcal{T}_M(\ds \teta M k))_M$  is bounded in $L^{6(\mu+1)}(\Omega)$, so that $\alpha_M (\ds \teta M k) = \alpha(\mathcal{T}_M(\ds \teta M k)) \to \alpha(\ds \teta \tau k)$
 in $L^{6s}(\Omega)$ for every $1 \leq s \EEE  < \frac{\mu+1}{\mu}$. Therefore,  the sequence
 $(\alpha_M(\ds \teta M k) \nabla \ds \teta M k)_M$ weakly converges to $\alpha(\ds \teta \tau k) \nabla \ds \teta \tau k$ in
 $L^{3/2}(\Omega;\R^3)$. This is sufficient to conclude that
 \begin{equation}
 \label{AM-emme}
\mathcal{A}(\ds \teta M k) \weakto \mathcal{A}_M(\ds \teta{\tau} k) \qquad \text{in } H^1(\Omega)^*.
\end{equation}
 With the same argument we perform the passagge to the limit in the analogous term for $\dss \tau k$. \EEE
 All in all, we deduce
 that
$(\ds \teta \tau k, \ds \uu \tau k, \dss \tau k, \ds \chi \tau k)$ fulfill the discrete bulk and surface heat equations
 \eqref{discr-b-heat}  and \eqref{discr-surf-temp}.
\par\noindent

\emph{Step $3$: proof of
\eqref{discr-tot-enbal}.} The total energy inequality
\eqref{discr-tot-enbal} follows by repeating the very same calculations as for
\eqref{discr-tot-enbal-trunc}.
\par\noindent
This finishes the proof \RCORR of Proposition \ref{prop:exists-discrete}. \EEE
\fin

\section{Existence for the regularized system}
\label{s:6}
In this section we address the  limit passage in \RCORR the discrete \EEE
system \eqref{discr-syst} \RCORR (formulated \EEE
in terms of suitable interpolants of the discrete solutions, cf.\ \eqref{interp-syst} below)
as the time step $\tau \down 0$, and in this way  \RCORR we shall \EEE conclude the existence of (weak) solutions to the Cauchy problem for the  regularized system \eqref{PDE-regul}.
Prior to that, let us set up some notation.
\paragraph{\bf Notation and preliminaries.}
For a given $K_\tau$-uple of discrete elements $(\ds{\mathfrak{h}} \tau k)_{k=0}^{K_\tau} \subset \mathbf{B}$, with  $\mathbf{B}$ a \RCORR given \EEE   Banach space,
 we  introduce the (left-continuous and right-continuous) piecewise constant, and the piecewise linear
  interpolants $\pwc {\mathfrak{h}} \tau,\, \upwc {\mathfrak{h}} \tau,\, \pwl {\mathfrak{h}}  \tau: [0,T] \to X$ defined by
$\pwc {\mathfrak{h}} \tau(0)=  \upwc {\mathfrak{h}} \tau (0)= \pwl {\mathfrak{h}} \tau(0) := \ds {\mathfrak{h}} \tau 0$ and \RCORR by \EEE
\begin{equation}
\begin{aligned} \label{itermpM}
&
\pwc {\mathfrak{h}} \tau(t): = \ds {\mathfrak{h}}  \tau k \text{ for } t \in (\ds t\tau {k-1}, \ds t\tau k], \qquad
\upwc {\mathfrak{h}}  \tau(t): = \ds {\mathfrak{h}}  \tau {k-1} \text{ for } t \in [\ds t\tau {k-1}, \ds t\tau k),
\\
&
\pwl {\mathfrak{h}} \tau(t): = \frac{t-\ds t \tau{k-1}}{\tau}  \ds {\mathfrak{h}} \tau k
+\frac{\ds t\tau k - t }{\tau}  \ds {\mathfrak{h}} \tau {k-1}  \text{ for } t \in (\ds t\tau {k-1}, \ds t\tau k]\,.
\end{aligned}
\end{equation}
For later use, we also recall that
\begin{subequations}
 \label{regtratti1}
\begin{eqnarray}
 \label{regtratti2}
\| \pwc {\mathfrak{h}} \tau - \pwl {\mathfrak{h}} \tau \|_{L^2(0,T; \mathbf{B})}
\leq \| \pwc {\mathfrak{h}} \tau - \upwc {\mathfrak{h}} \tau \|_{L^2(0,T; \mathbf{B})}
&\leq&  \tau \| \partial_t \pwl {\mathfrak{h}} \tau \|_{L^2(0,T; \mathbf{B})}, \\ \label{regtratti3}
\| \pwc {\mathfrak{h}} \tau - \pwl {\mathfrak{h}} \tau \|_{L^{\infty}(0,T; \mathbf{B})}
\leq  \| \pwc {\mathfrak{h}} \tau - \upwc {\mathfrak{h}} \tau \|_{L^\infty(0,T; \mathbf{B})}
&\leq&  \sqrt{\tau} \| \partial_t \pwl {\mathfrak{h}} \tau \|_{L^2(0,T; \mathbf{B})},
\end{eqnarray}
\end{subequations}
as well as  the well-known discrete by-part integration formula
\begin{equation}
\label{discrete-by-parts}
\sum_{i=1}^{j} \tau \pairing{}{\mathbf{B}}{\ds {\mathfrak{l}}\tau i}{\tfrac{\ds {\mathfrak{h}}\tau i{-}\ds {\mathfrak{h}}\tau {i-1}}\tau}
=  \pairing{}{\mathbf{B}}{\ds {\mathfrak{l}}\tau j}{\ds {\mathfrak{h}}\tau j}  - \pairing{}{\mathbf{B}}{\ds {\mathfrak{l}}\tau 0}{\ds {\mathfrak{h}}\tau 0}
 -\sum_{i=1}^{j} \tau \pairing{}{\mathbf{B}}{\tfrac{\ds {\mathfrak{l}}\tau k{-}\ds {\mathfrak{l}}\tau {k-1}}\tau}{\ds {\mathfrak{h}}\tau i}
\end{equation}
for all $K_\tau$-uples $(\ds {\mathfrak{h}}\tau k)_{k=1}^{K_\tau} \subset \mathbf{B}$,
$(\ds {\mathfrak{l}}\tau k)_{k=1}^{K_\tau} \subset \mathbf{B}^*$. 
\paragraph{\bf Approximate solutions.}
With  ${\mathfrak{h}} \in \RCNEW \{\teta, \uu, \tetas, \chi, \sigma, \mathbf{F}, h, \ell \}$, \EEE
we thus obtain the piecewise constant and linear interpolants of the discrete solution quadruples
$(\ds \teta \tau k, \ds \uu \tau k, \dss \tau k, \ds \chi \tau k, \RCNEW \ds \sigma \tau k )$ \EEE
and of the discrete data  $(\mathbf{F}_{\tau}^k,\hk, \lk)$. 
Finally, we also introduce the piecewise constant interpolants $\pwc {\mathsf{t}}\tau:[0,T] \to [0,T]$ and
 $\upwc {\mathsf{t}}\tau: [0,T] \to [0,T]$ associated with the partition $\mathscr{P}_\tau = (\ds t\tau k)_{k=1}^{N_\tau}$ and defined by $\pwc {\mathsf{t}}\tau(0) = \upwc {\mathsf{t}}\tau(0):=0$, $\pwc {\mathsf{t}}\tau(T) = \upwc {\mathsf{t}}\tau(T):=T$, and
 \[
 \pwc {\mathsf{t}}\tau(t) = \ds t\tau k \quad \text{for } t \in (\ds t\tau{k-1}, \ds t \tau k], \qquad  \pwc {\mathsf{t}}\tau(t) = \ds t\tau {k-1} \quad \text{for } t \in [\ds t\tau{k-1}, \ds t \tau k).
 \]
 \par
In terms of the above interpolants, the discrete system \eqref{discr-syst} rephrases as
\begin{subequations}
\label{interp-syst}
\begin{align}
&
\label{interp-b-heat}
\begin{aligned}
&
 \io  \pwl \teta \tau' (t) v \dd x
- \io \pwc\teta\tau (t) \mathrm{div} (\pwl \uu \tau'(t)) v \dd x
+ \io \alpha(\pwc \teta\tau (t)) \nabla \pwc\teta\tau(t) \nabla v \dd x
\\
& \quad
+ \int_{\GC} k(\upwc \chi\tau(t))\pwc\teta\tau(t)(\pwc\teta\tau(t)-\pwcs \teta \tau (t)) v \dd x
+ \int_{\GC} \nlocss {\RCNEW(\upwc\chi\tau (t))^+}\,  \RCNEW(\upwc\chi\tau (t) )^+\EEE \left( \pwc\teta\tau (t) \right)^2 v \dd x
\\
& \quad
- \int_{\GC} \nlocss {( \upwc\chi\tau (t) )^+  \pwcs \teta\tau(t)}\, \RCNEW( \upwc\chi\tau (t) )^+\EEE \pwc\teta\tau (t) v \dd x
\\ &
=
\io \eps \left( \pwl \uu\tau'(t)\right) \, \vtens \,  \eps \left( \pwl \uu\tau'(t)\right) v  \dd x
+ \pairing{}{H^1(\Omega)}{\pwc h\tau(t)}{v}  \qquad \text{for all } v \in H^1(\Omega),
\end{aligned}
\\
&
\label{interp-mom-bal}
\begin{aligned}
&
\vfm \left( \pwl \uu\tau'(t), \vv \right)
+ \rho \io \left| \eps \left( \pwl \uu\tau'(t)\right) \right|^{\omega - 2} \eps \left( \pwl \uu\tau'(t)\right)  \eps (\vv)  \dd x
+ \efm \left( \pwc\uu\tau(t) , \vv \right)
\\ & \quad  + \io \pwc\teta\tau(t) \mathrm{div}(\vv) \dd x
+ \int_{\GC} \RCNEW(\pwc\chi\tau(t) )^+\pwc\uu\tau(t) \vv \dd x
+\int_{\GC} \pwc\zzeta\tau(t) \cdot \vv \dd x
+ \int_{\GC} \nlocss {\RCNEW(\pwc\chi\tau(t))^+}\, \RCNEW( \pwc\chi\tau(t))^+ \EEE \pwc\uu\tau(t) \vv \dd x
\\
&
=
\langle  \pwc{\mathbf{F}}\tau(t), \vv \rangle_{\bsVD} \qquad \text{for all } \vv \in W^{1,\omega}_{\mathrm{D}}(\Omega;\R^3),
\\
&
\text{with } \pwc\zzeta\tau(t)  :=  \etar(\pwc\uu \tau(t)\cdot \mathbf{n})  \mathbf{n},
\end{aligned}
\\
&
\label{interp-surf-temp}
\begin{aligned}
&
\int_{\GC} \pwls \teta\tau'(t)  v \dd x
- \int_{\GC} \pwcs \teta \tau (t)  \frac{\lambda(\pwc \chi\tau(t))-\lambda(\upwc \chi\tau(t))}{\tau}  v \dd x
+ \int_{\GC} \alpha(\pwcs \teta \tau(t)) \nabla \pwcs\teta\tau(t) \nabla v \dd x
\\
&
=
\int_{\GC} \left| \pwl\chi\tau'(t)\right|^2 v \dd x
+ \int_{\GC} k(\upwc\chi\tau(t))(\pwc\teta\tau(t) -\pwcs\teta\tau(t)) \pwcs \teta\tau(t) v \dd x
+ \int_{\GC} \nlocss {\RCNEW(\upwc\chi\tau(t))^+ \EEE  \pwc\teta\tau(t)}\, \RCNEW(\upwc\chi\tau(t) )^+ \EEE \pwcs\teta\tau(t) v \dd x
\\
& \qquad - \int_{\GC} \nlocss {\RCNEW(\upwc\chi\tau(t))^+}\,\RCNEW( \upwc\chi\tau(t) )^+ \EEE \big(\pwcs\teta\tau(t) \big)^2 v \dd x
+ \pairing{}{H^1(\GC)}{\pwc\ell\tau}{v}
\qquad \text{for all }  v \in  H^1(\GC).
\end{aligned}
\\
&
\label{interp-flow-rule}
\begin{aligned}
&
\pwl\chi\tau'(t)
+\rho \left|\pwl\chi\tau'(t)  \right|^{\omega-2} \pwl\chi\tau'(t)
 + A\pwc\chi\tau(t)  +
\betar(\pwc\chi\tau(t))
+\gamma_\nu'(\pwc\chi\tau(t)) - \nu \upwc\chi\tau(t)
\\
&
= - \lambda_\delta'(\upwc\chi\tau(t)) \pwcs \teta\tau(t) -\delta \pwc\chi\tau(t) \pwcs \teta\tau(t)
\\
& \quad -
 \frac12 |\upwc\uu\tau(t)|^2 \RCNEW \pwc \sigma \tau(t)  -\frac12 \nlocss {\RCNEW(\pwc\chi\tau(t))^+} \, |\upwc\uu\tau(t)|^2  \pwc \sigma \tau(t) -\frac12 \nlocss {(\upwc\chi\tau(t))^+ |\upwc\uu\tau(t)|^2} \, \pwc \sigma \tau(t)
\qquad \aein \, \GC\,
\\
&
\RCNEW \text{with } \pwc \sigma\tau(t) \in \partial\varphi(\pwc\chi\tau(t)) \qquad \aein\, \GC
\end{aligned}
\end{align}
\end{subequations}
\EEE for almost all $t\in (0,T)$,
supplemented with the Cauchy data $(\teta_\rho^0,\uu_\rho^0, \dss \rho 0, \chi_0 ) $ as in \eqref{approx-initial-data} and \eqref{cond-chi0}.
 It is in system \eqref{interp-syst} that we shall pass to the limit as $\tau \down 0$, thus proving the existence of solutions to
(the weak formulation of) the Cauchy problem for system   \eqref{PDE-regul}. In the following
results we shall omit to specify the standing assumptions on the problem and on the Cauchy data.
\subsection{A priori estimates}
The following result collects a series of a priori estimates on the approximate solutions that
will be at the basis of the limit passage procedure as $\tau \down 0$ performed in Sec.\ \ref{ss:5.2}
and leading to the  proof of the existence of solutions to system \eqref{PDE-regul}.
In view of the further \RCORR  limit passages \ as $\rho \down 0 $  and $\varsigma \down 0$ carried out in
Sec.\ \ref{s:7}, \EEE  we shall distinguish the estimates that hold
\emph{uniformly} w.r.t. $\tau$ \emph{and}  $\rho,\, \varsigma$, from those that are not uniform w.r.t.\ $\rho,\, \varsigma$.
In the statement of Lemma \ref{l:a-prio-tau} below we will use the notation
\[
\mathrm{Var}_{\mathbf{B}} (\mathfrak{h}; [0,T]): = \sup \left\{ \sum_{i=1}^{m} \|  \mathfrak{h}(\sigma_i) {-} \mathfrak{h}(\sigma_{i-1})\|_{\mathbf{B}}\, : \ (\sigma_i)_{i=1}^m \text{ partition of } [0,T] \right\}
\]
for the total variation of a function  $\mathfrak{h}:[0,T]\to \mathbf{B}$.
\begin{lemma}[A priori estimates]
\label{l:a-prio-tau}
There exists   a constant $S>0$ 
such that the following estimates hold for every $\tau>0$ and $\rho, \, \varsigma>0$:
\begin{subequations}
\label{a-prio-tau}
\begin{align}
&
\label{teta-a-prio-tau}
\| \pwc \teta \tau\|_{L^\infty(0,T;L^1(\Omega))} + \| \pwl \teta \tau\|_{L^\infty(0,T;L^1(\Omega))}  \leq S, 
\\
&
\label{u-a-prio-tau}
\| \pwc \uu \tau \|_{L^\infty(0,T;H^1(\Omega;\R^3))}  %
+\| \pwl \uu \tau \|_{L^\infty(0,T;H^1(\Omega;\R^3))}  \leq S,
\\
&
\label{tetas-a-prio-tau}
\| \pwcs \teta \tau\|_{L^\infty(0,T;L^1(\GC))} + \| \pwls \teta \tau\|_{L^\infty(0,T;L^1(\GC))} \   \leq S, %
  \\
  &
  \label{chi-a-prio-tau}
  \| \pwc \chi \tau \|_{L^\infty(0,T;H^1(\GC))} 
  +\| \pwl \chi \tau \|_{L^\infty(0,T;H^1(\GC))} 
  \leq S,
  \\
  &
  \label{upsilon-a-prio-tau}
\RCNEW  \| \pwc \sigma \tau \|_{L^\infty(0,T;L^\infty(\GC))} \leq S\,.    \EEE
\end{align}
\end{subequations}
Furthermore,  for every $\rho >0$ there exists a constant $S_{\rho}>0$ such that for every $\tau>0$ and $\varsigma>0:$
\begin{subequations}
\label{a-prio-tau-rho}
\begin{align}
&
\label{teta-a-prio-tau-rho}
 \| \pwc\teta\tau\|_{L^2(0,T;H^1(\Omega))\cap L^\infty(0,T;L^2(\Omega))}
+ \| \widehat{\alpha}(\pwc\teta\tau)\|_{L^2(0,T;H^1(\Omega))}+
 \| \pwl\teta\tau\|_{L^2(0,T;H^1(\Omega)) \cap H^1(0,T;H^1(\Omega)^*)}
  \leq  S_{\rho}, 
  \\
  &
\label{VAR-teta-a-prio-tau}
\mathrm{Var}_{H^1(\Omega)^*}(\pwc \teta \tau; [0,T])
  \leq  S_{\rho}, 
  \\
  &
\label{u-a-prio-tau-rho}
\| \pwc \uu \tau \|_{L^\infty(0,T;W^{1,\omega}(\Omega;\R^3))} +
 \| \pwl \uu \tau \|_{W^{1,\omega}(0,T;W^{1,\omega}(\Omega;\R^3))}\leq  S_{\rho}, 
\\
&
\label{tetas-a-prio-tau-rho}
\begin{aligned}
&
\| \pwcs \teta \tau\|_{L^2(0,T;H^1(\GC))\cap L^\infty(0,T;L^2(\GC))} + \| \widehat{\alpha}(\pwcs \teta\tau)\|_{L^2(0,T;H^1(\GC))}
\\
& \quad
+ \| \pwls\teta\tau\|_{L^2(0,T;H^1(\GC)) \cap H^1(0,T;H^1(\GC)^*)}
  \leq  S_{\rho}, 
  \end{aligned}
  \\
    &
\label{VAR-tetas-a-prio-tau}
\mathrm{Var}_{H^1(\GC)^*}(\pwcs \teta \tau; [0,T])
  \leq  S_{\rho}, 
  \\
    &
    \label{chi-a-prio-tau-rho}
     \| \pwl \chi \tau \|_{W^{1,\omega}(0,T;L^\omega(\GC))}  S_{\rho}. 
     \end{align}
\end{subequations}
 Moreover, for every $\rho, \varsigma >0$ there exists a constant $S_{\rho, \varsigma}>0$ such that for every $\tau>0$:
\begin{equation}
\label{chi-a-prio-tau-rhoH2}
    \| A\pwc \chi \tau \|_{L^{\omega/(\omega{-}1)}(\GC\times(0,T))}  \leq S_{\rho,\varsigma}.
\end{equation}
Finally,
\begin{equation}
\label{est-recipr-interp}
 \exists S_0>0 \ \forall\, p \in [1,\infty) \ \ \exists\, \bar{\tau}_p>0 \ \  \forall\, \tau \in (0,\bar{\tau}_p)\,  \ \ \forall\,\rho,\, \varsigma>0 : \quad
\left\| \frac1{\pwc \teta \tau}\right\|_{L^\infty(0,T;L^p(\Omega))} + \left\| \frac1{\pwcs \teta \tau}\right\|_{L^\infty(0,T;L^p(\GC))} \leq S_0.
\end{equation}
\end{lemma}
\begin{proof}
\emph{Step $1$: energy estimates.}
We sum the discrete total energy inequality \eqref{discr-tot-enbal} over the index $k \in \{1, \ldots, J\}$, for every $J  \in \{1, \ldots, K_\tau\}$.
Applying the discrete by-part integration formula \eqref{discrete-by-parts} to the term
$\sum_{k=1}^{J}\tau\pairing{}{\bsVD}{\ds {\mathbf{F}}\tau k}{\frac{\ds \uu {\tau} k-\ds \uu \tau {k-1}}{\tau}}$ we obtain for all $0 \leq s < t \leq T$
\begin{equation}
\label{discr-tot-enbal-interp}
\begin{aligned}
&
\calE_\varsigma (\pwc\teta\tau(t),  \pwcs \teta \tau(t),  \pwc \uu \tau (t), \pwc\chi\tau(t))
+\rho\int_{\upwc \sft \tau(s)}^{\pwc \sft \tau(t)} \int_{\Omega}
 \left|\eps (\pwl \uu\tau')\right|^\omega \dd x \dd r   + \rho  \int_{\upwc \sft \tau(s)}^{\pwc \sft \tau(t)} \int_{\GC} \left| \pwl \chi\tau' \right|^\omega \dd x \dd r
\\
& \quad
+ \int_{\upwc \sft \tau(s)}^{\pwc \sft \tau(t)}
\int_{\GC} k(\upwc \chi\tau) (\pwc \teta \tau {-} \pwcs \teta \tau)^2  \dd x \dd r
\dd x
+ \int_{\upwc \sft \tau(s)}^{\pwc \sft \tau(t)}  \iint_{\GC\times\GC}  j(x,y) \RCNEW (\upwc\chi\tau (x))^+ (\upwc\chi\tau  (y))^+ \EEE (\pwc\teta\tau (x){-} \pwcs \teta \tau(y))^2 \dd x \dd y  \dd r
\\
&
 \leq \calE_\varsigma (\upwc\teta\tau(s),  \upwcs \teta \tau(s),  \upwc \uu \tau (s), \upwc\chi\tau(s))  +
\RCORR  \int_{\upwc \sft \tau(s)}^{\pwc \sft \tau(t)} \EEE  \int_{\Omega} \pwc h \tau \dd x \dd r
 +
\RCORR  \int_{\upwc \sft \tau(s)}^{\pwc \sft \tau(t)} \EEE  \int_{\GC} \pwc \ell \tau \dd x \dd r
 \\
 & \quad
\RCORR + \EEE   \pairing{}{\bsVD}{\pwc {\mathbf{F}}\tau (\pwc \sft \tau(t))}{\pwc \uu \tau (\pwc \sft \tau(t))} - \pairing{}{\bsVD}{\pwc {\mathbf{F}}\tau (\upwc \sft \tau(s))}{\upwc \uu \tau (\pwc \sft \tau(s))}
-\int_{\upwc \sft \tau(s)}^{\pwc \sft \tau(t)}  \pairing{}{\bsVD}{\pwl {\mathbf{F}}\tau'}{\pwc \uu \tau} \dd r\,.
 \end{aligned}
\end{equation}
We now take $s=0$ in \eqref{discr-tot-enbal-interp}, and observe that $\calE_\varsigma(\upwc\teta\tau(0),  \upwcs \teta \tau(0),  \upwc \uu \tau (0), \upwc\chi\tau(0)) =
\calE_\varsigma(\teta_\rho^0,\dss \rho 0,
\uu_\rho^0,\chi_0) \leq C$ thanks to  \eqref{cond-chi0} and \eqref{approx-initial-data}.
For the second and third integrals on the right-hand side, we use that $\| \pwc h \tau \|_{L^1(0,T;L^1(\Omega))} \leq C$ by \eqref{cond-h} and that
$\| \pwc \ell \tau \|_{L^1(0,T;L^1(\GC))} \leq C$ by \eqref{cond-ell}, respectively,  while we deal with the last three terms by
 mimicking the calculations from \eqref{E1}. Namely, they can be controlled by the left-hand side of \eqref{discr-tot-enbal-interp} thanks to the coercivity estimate
 \eqref{coerc-ene}. All in all,  as in Section \ref{sss:3.3.1} we conclude that
 \begin{equation}
 \label{energy-bound-rho}
 \sup_{t \in (0,T)} \left| \calE_\varsigma (\pwc\teta\tau(t),  \pwcs \teta \tau(t),  \pwc \uu \tau (t), \pwc\chi\tau(t))
\right| \leq C.
 \end{equation}
 An analogue of the coercivity estimate   \eqref{coerc-ene} holds for the functional $\calE_\varsigma$: the only difference is that, since
  $\calE_\varsigma$ features
  $\widehat\betar$ in place of $ \widehat\beta$, it does no longer control the $L^\infty(\GC)$-norm of $\chi$.
  However, it is not difficult to see that
  $
  \int_{\GC}\widehat\betar(\chi) \dd x \geq c \|\chi\|_{L^2(\GC)}^2 -C.
  $
   Therefore,  from \eqref{energy-bound-rho}
we deduce
estimates \eqref{teta-a-prio-tau}, \eqref{u-a-prio-tau}, \eqref{tetas-a-prio-tau}, and  \eqref{chi-a-prio-tau}.
\par
 Furthermore, also taking into account the positivity of the fourth and fifth terms on the
 left-hand side of \eqref{discr-tot-enbal-interp}, we deduce a bound for the second and third summands,
 which gives \eqref{u-a-prio-tau-rho} and \eqref{chi-a-prio-tau-rho}.
 \par
 \RCNEW Finally, estimate \eqref{upsilon-a-prio-tau} follows from the fact that $\pwc \sigma\tau \in \partial\varphi(\pwc\chi\tau)$ a.e.\ in $\GC \times (0,T)$, cf.\ \eqref{subdiff-pos-part}.  \EEE
 \medskip

\emph{Step $2$: Estimates for the temperature variables.}
We test \eqref{discr-b-heat} by $\tau \ds \teta \tau k$, \eqref{discr-surf-temp} by $\tau\dss \tau k$, add the resulting relations and
 sum over the index $k \in \{1, \ldots, J\}$ for an arbitrary $J \in \{1,\ldots, K_\tau\}$.
 Let $t \in (0,T]$ satisfy $(J-1) \tau \leq t \leq J\tau $: we obtain
 \[
 \begin{aligned}
 &
 \sum_{k=1}^J \tau  \int_\Omega \left( \frac{\ds \teta \tau k-\ds \teta \tau{k-1}}{\tau} \right) \ds \teta\tau k \dd x \geq \frac12 \| \pwc \teta \tau(t)\|_{L^2(\Omega)}^2 -
  \frac12 \| \teta_\rho^0 \|_{L^2(\Omega)}^2,
  \\
 &
 \sum_{k=1}^J \tau  \int_{\GC} \left( \frac{\dss  \tau k-\dss \tau{k-1}}{\tau} \right) \dss\tau k \dd x \geq \frac12 \| \pwcs \teta \tau(t)\|_{L^2(\GC)}^2 -
 \frac12 \| \dss \rho 0 \|_{L^2(\GC)}^2.
\end{aligned}
 \]
 Mimicking the calculations in Lemma \ref{l:epsi-a-prio} (cf.\ \eqref{new-calcul-epsl})
  we obtain
\begin{equation}
\label{new-calcul-cont}
\begin{aligned}
 &\frac12 \| \pwc \teta \tau(t)\|_{L^2(\Omega)}^2 +\frac12 \| \pwcs \teta \tau(t)\|_{L^2(\GC)}^2
+
c_0 \int_0^{\pwc \sft \tau(t)} \io |\nabla \pwc \teta \tau|^2 \dd x  \dd r
+
c_0 \int_0^{\pwc \sft \tau(t)} \int_{\GC}|\nabla \pwcs \teta \tau|^2 \dd x  \dd r
+ I_1 +I_2
\\ &  \leq  I_3 +I_4+I_5+I_6+I_7+I_8+I_9+I_{10}
 \end{aligned}
\end{equation}
with
\[
\begin{aligned}
&
I_1= \int_0^{\pwc \sft \tau(t)}\int_{\GC} k(\upwc \chi \tau) ((\pwc \teta \tau)^2  {-} (\pwcs \teta \tau)^2) (\pwc \teta \tau{-} \pwcs \teta \tau)\dd x \dd r  \stackrel{(1)}{\geq} 0,
\\
&
\begin{aligned}
I_2
 & =   \int_0^{\pwc \sft \tau(t)} \int_{\GC}   (\pwc \teta \tau)^2\left(  \nlocss {\RCNEW (\upwc \chi \tau)^+}\, \RCNEW (\upwc \chi \tau)^+ \EEE \pwc\teta \tau  {-} \nlocss {\RCNEW(\upwc \chi \tau)^+ \EEE \pwcs \teta \tau}\,  \RCNEW(\upwc \chi \tau)^+  \right) \dd x  \dd  r
 \\
 & \quad
 -  \int_0^{\pwc \sft \tau(t)} \int_{\GC} ( \pwcs \teta \tau)^2 \left( \nlocss {\RCNEW(\upwc \chi \tau)^+ \EEE \pwc\teta\tau}\, \RCNEW (\upwc \chi \tau)^+\, {-} \nlocss {\RCNEW(\upwc \chi \tau)^+}\, \RCNEW (\upwc \chi \tau)^+  \EEE \pwcs \teta \tau \right) \dd x \MC \dd  r \EEE
\stackrel{(2)}{\geq} 0
\end{aligned}
\end{aligned}
\]
where
for
 (1) we have used the estimate
 $(a^2-b^2)(a-b) = (a+b)(a-b)^2 \geq 0 $ for all $a,\,b \in [0,\infty)$ and for (2)
 the very same \RCORR monotonicity \EEE arguments used for \eqref{key-posit-conto}, based on the fact  that $[0,+\infty) \ni r\mapsto r^2$
 is non-decreasing.
As for the terms on the right-hand side of \eqref{new-calcul-cont},  we have
\begin{align}
&
I_3 =  \frac12 \| \teta_\rho^0 \|_{L^2(\Omega)}^2 \leq C, \qquad I_4= \frac12 \| \dss \rho 0 \|_{L^2(\GC)}^2  \leq C,
\nonumber
\\
&
\nonumber
\begin{aligned}
I_5  & =  \int_0^{\pwc \sft \tau(t)} \int_\Omega  \mathrm{div} \left(\pwl \uu \tau' \right) |\pwc\teta\tau|^2  \dd x \dd r
\\
&
\leq C  \int_0^{\pwc \sft \tau(t)}  \left\| \pwl \uu \tau'  \right\|_{W^{1,4}(\Omega)} \| \pwc\teta\tau \|_{L^4(\Omega)}
\|\pwc \teta \tau \|_{L^2(\Omega)} \dd r   \\
& \stackrel{(3)}{\leq}
 C  \int_0^{\pwc \sft \tau(t)}  \left\| \pwl \uu \tau'  \right\|_{W^{1,4}(\Omega)} \| \nabla \pwc\teta\tau \|_{L^2(\Omega;\R^3)}
\|\pwc \teta \tau \|_{L^2(\Omega)} \dd r +  \int_0^{\pwc \sft \tau(t)}  \left\| \pwl \uu \tau'  \right\|_{W^{1,4}(\Omega)} \|  \pwc\teta\tau \|_{L^1(\Omega;\R^3)}
\|\pwc \teta \tau \|_{L^2(\Omega)} \dd r
\\
& \stackrel{(4)}{\leq}  \frac{c_0}4  \int_0^{\pwc \sft \tau(t)}  \| \nabla \pwc\teta\tau \|_{L^2(\Omega;\R^3)}^2 \MC \dd  r \EEE + C
 \int_0^{\pwc \sft \tau(t)}  \left\| \pwl \uu \tau'  \right\|_{W^{1,4}(\Omega)}^2 \|\pwc \teta \tau \|_{L^2(\Omega)}^2  \dd r + C
\end{aligned}
\nonumber
\intertext{where (3), with $c_0$ the constant from \eqref{hyp-alpha}, follows from the continuous
 embedding $H^1(\Omega) \subset L^4(\Omega)$ and from the Poincar\'e inequality,  and (4) from the previously proved
\eqref{teta-a-prio-tau},}
&
\nonumber
\begin{aligned}
I_6 &=  \int_0^{\pwc \sft \tau(t)}\int_\Omega
 \eps\left( \pwl \uu\tau'\right) \, \vtens \, \eps\left( \pwl \uu \tau'\right)  \pwc \teta \tau    \dd x  \dd  r
 \leq C \int_0^{\pwc \sft \tau(t)} \| \eps\left( \pwl \uu\tau'\right) \|_{L^4(\Omega)}^2 \|  \pwc \teta \tau   \|_{L^2(\Omega)} \dd r
 \\
  &   \leq  \frac12 \int_0^{\pwc \sft \tau(t)}
 \|  \pwc \teta \tau   \|_{L^2(\Omega)}^2 \dd r + C \int_0^{\pwc \sft \tau(t)}\| \pwl \uu \tau'\|_{W^{1,4}(\Omega;\R^3)}^4 \dd r
 \end{aligned}
\\
&
\nonumber
\begin{aligned}
I_7   & =  \int_0^{\pwc \sft \tau(t)} \pairing{}{H^1(\Omega)}{\pwc h\tau}{\pwc \teta \tau}  \dd r   \\ &   \stackrel{(5)}{\leq}
 \frac{c_0}4  \int_0^{\pwc \sft \tau(t)}  \| \nabla \pwc\teta\tau \|_{L^2(\Omega)}^2 \dd r  + C  \int_0^{\pwc \sft \tau(t)} \| \pwc h\tau \|_{H^1(\Omega)^*}^2 \dd r
 +C  \int_0^{\pwc \sft \tau(t)} \| \pwc h\tau \|_{H^1(\Omega)^*} \| \pwc \teta \tau\|_{L^1(\Omega)} \dd r
 \\
 & \RCORR \leq  \frac{c_0}4  \int_0^{\pwc \sft \tau(t)}  \| \nabla \pwc\teta\tau \|_{L^2(\Omega)}^2 \dd r  + C  \EEE
 \end{aligned}
\intertext{where (5) again follows from Poincar\'e inequality, and analogously}
&
\nonumber
\begin{aligned}
\nonumber
I_8   & =  \int_0^{\pwc \sft \tau(t)} \pairing{}{H^1(\GC)}{\pwc \ell\tau}{\pwcs \teta \tau}  \dd r   \\ &\leq
 \frac{c_0}4  \int_0^{\pwc \sft \tau(t)}  \| \nabla \pwcs \teta \tau \|_{L^2(\GC)}^2 \dd r  + C  \int_0^{\pwc \sft \tau(t)} \| \pwc \ell\tau \|_{H^1(\GC)^*}^2 \dd r
 +C  \int_0^{\pwc \sft \tau(t)} \| \pwc \ell\tau \|_{H^1(\GC)^*} \| \pwcs \teta \tau\|_{L^1(\GC)} \dd r
 \nonumber
 \\
 \RCORR &  \leq
 \frac{c_0}4  \int_0^{\pwc \sft \tau(t)}  \| \nabla \pwcs \teta \tau \|_{L^2(\GC)}^2 \dd r  + C  \EEE
 \end{aligned}
\nonumber
\\
&
\begin{aligned}
\nonumber
I_9  & =   \int_0^{\pwc \sft \tau(t)} \int_{\GC} \tfrac{\lambda(\pwc\chi\tau)-\lambda(\upwc\chi\tau)}\tau  (\pwcs\teta\tau)^2  \dd x
\dd r
\\  &  \stackrel{(6)}{\leq}   C \MC \int_0^{\pwc \sft \tau(t)} \EEE\| \pwl\chi\tau' \|_{L^4(\GC)}
\| \pwcs \teta \tau \|_{L^4(\GC)} \| \pwcs \teta \tau \|_{L^2(\GC)} \dd r \\ &   \stackrel{(7)}{\leq}
 \frac{c_0}4  \int_0^{\pwc \sft \tau(t)}  \| \nabla \pwcs\teta\tau \|_{L^2(\GC;\R^2)}^2 \MC \dd  r \EEE + C
 \int_0^{\pwc \sft \tau(t)}  \left\| \pwl \chi \tau'  \right\|_{L^{4}(\GC)}^2 \|\pwcs \teta \tau \|_{L^2(\GC)}^2  \dd r + C
\end{aligned}
\intertext{where (6) follows from the Lipschitz continuity of $\lambda$, and (7) from the same arguments as for inequality (3), \RCORR and, finally,}
&
\nonumber
I_{10} =\int_0^{\pwc \sft \tau(t)}  \int_{\GC} (\pwl\chi\tau')^2   \pwcs\teta\tau \dd x
\dd r  \leq	  \frac12 \int_0^{\pwc \sft \tau(t)}
 \|  \pwcs \teta \tau   \|_{L^2(\GC)}^2 \dd r + C \int_0^{\pwc \sft \tau(t)}\| \pwl \chi \tau'\|_{L^{4}(\GC)}^4 \dd r \,.
\end{align}
All in all,  we arrive at
\[
\begin{aligned}
&
\frac12 \| \pwc \teta \tau(t)\|_{L^2(\Omega)}^2 +\frac12 \| \pwcs \teta \tau(t)\|_{L^2(\GC)}^2
+
\frac{c_0}2 \int_0^{\pwc \sft \tau(t)} \MC \io \EEE  |\nabla \pwc \teta \tau|^2 \dd x  \dd r
+
\frac{c_0}2 \int_0^{\pwc \sft \tau(t)} \int_{\GC} 
\RCORR |\nabla \pwcs \teta \tau|^2 \EEE \dd x  \dd r
\\
 & \leq C
+
C
 \int_0^{\pwc \sft \tau(t)}  (\left\| \pwl \uu \tau'  \right\|_{W^{1,4}(\Omega;\R^3)}^2{+}1)  \|\pwc \teta \tau \|_{L^2(\Omega)}^2  \dd r
+
C
 \int_0^{\pwc \sft \tau(t)}  \left(\left\| \pwl \chi \tau'  \right\|_{L^{4}(\GC)}^2{+}1\right)  \|\pwcs \teta \tau \|_{L^2(\GC)}^2  \dd r\,.
 \end{aligned}
\]
Applying the discrete Gronwall Lemma \ref{l:discrG1/2}
and  taking into account the previously proved \eqref{u-a-prio-tau-rho} and  \eqref{chi-a-prio-tau-rho},
we conclude that
\begin{equation} \label{GiovStep2}
\| \pwc\teta\tau\|_{L^2(0,T;H^1(\Omega))\cap L^\infty(0,T;L^2(\Omega))} + \| \pwcs\teta\tau\|_{L^2(0,T;H^1(\GC))\cap L^\infty(0,T;L^2(\GC))} \leq C_\rho,
\end{equation}
with the constant $C_\rho$ depending on the parameter $\rho$.
 \medskip

 \noindent
\emph{Step $3$: Further  estimates for the temperature variables.}
We test \eqref{discr-b-heat} by $\tau \widehat\alpha( \ds \teta \tau k)$, \eqref{discr-surf-temp} by $\tau\widehat{\alpha}(\dss \tau k)$, add the resulting relations and sum over the index $k \in \{1, \ldots, J\}$ for an arbitrary $J \in \{1,\ldots, K_\tau\}$.
Let $t \in (0,T]$ satisfy $(J-1) \tau \leq t \leq J\tau $.
By  the convexity and positivity  of $\doublehhat{\alpha}$ (cf.\  \eqref{growth-primitives}),
 we have
\begin{eqnarray} \nonumber
\sum_{k=1}^J \int_\Omega ( \ds \teta \tau k-\ds \teta \tau{k-1} ) \widehat{\alpha}(\ds \teta \tau k ) \dd x
& \geq &
\sum_{k=1}^J \Big( \| \doublehhat{\alpha} (\ds \teta \tau k ) \|_{L^1(\Omega)} - \| \doublehhat{\alpha}  (\tetakmu) \|_{L^1(\Omega)}  \Big) \\
\label{step3stiminizM}
&= & \| \doublehhat{\alpha} (\pwc \teta \tau(t)) \|_{L^1(\Omega)} - \| \doublehhat{\alpha}  (\teta_\rho^0 ) \|_{L^1(\Omega)} ,
\end{eqnarray}
and, analogously,
\begin{eqnarray} \nonumber
\sum_{k=1}^J \int_{\GC} ( \dss \tau k-\dss \tau{k-1} ) \widehat{\alpha}(\dss \tau k ) \dd x
&\geq&
 \| \doublehhat{\alpha} (\pwcs \teta \tau(t)) \|_{L^1(\GC)} - \| \doublehhat{\alpha}  (\dss \rho 0  ) \|_{L^1(\GC)} .
\end{eqnarray}
Since the function $ \R^+ \ni r\mapsto r \widehat{\alpha}(r)$ is increasing,
we have that
\begin{equation}
\int_0^{\pwc \sft \tau(t)} \int_{\GC} k(\upwc \chi \tau) (\pwc \teta \tau  {-} \pwcs \teta \tau)
(\pwc \teta \tau \widehat{\alpha} (\pwc \teta \tau) {-} \pwcs \teta \tau \widehat{\alpha} (\pwcs \teta \tau ) )\dd x \dd r  \geq 0,
\end{equation}
while,
using that $\widehat\alpha$ is (strictly) increasing and
mimicking the calculations in  \eqref{posult4}, we observe that
\begin{equation} \label{step3stimfinalM}
\begin{aligned}
&
\int_0^{\pwc \sft \tau(t)} \int_{\GC}  \widehat{\alpha} (\pwc\teta \tau )  \RCNEW (\upwc \chi \tau)^+  \pwc\teta \tau
\Big(  \nlocss {\RCNEW(\upwc \chi \tau)^+}\, \pwc\teta \tau - \nlocss {\RCNEW(\upwc \chi \tau)^+ \EEE \pwcs \teta \tau} \Big)
 \dd x \dd r
\\
& \quad
 - \int_0^{\pwc \sft \tau(t)} \int_{\GC}
  \widehat{\alpha} (\pwcs \teta \tau )  (\upwc \chi \tau)^+ \pwcs \teta \tau
\Big(\nlocss {\RCNEW(\upwc \chi \tau)^+\EEE \pwc\teta\tau}\, - \nlocss {(\upwc \chi \tau)^+} \, \pwcs \teta \tau \Big) \dd x \dd r \geq 0.
 \end{aligned}
\end{equation}

Combining \eqref{step3stiminizM}--\eqref{step3stimfinalM}
and observing that
 $\nabla \widehat{\alpha} (\ds \teta \tau k) = \alpha(\ds \teta \tau k) \nabla \ds \teta \tau k$ and
$\nabla \widehat{\alpha} (\dss \tau k) = \alpha(\dss \tau k) \nabla \dss \tau k$,
 we arrive at
\begin{equation}
 \label{stp3Itot}
 \begin{aligned}
 &
\| \doublehhat{\alpha} (\pwc \teta \tau(t)) \|_{L^1(\Omega)}
+ \| \doublehhat{\alpha} (\pwcs \teta \tau(t)) \|_{L^1(\GC)}
+ \int_0^{\pwc \sft \tau(t)} \| \nabla \widehat{\alpha} (\pwc \teta \tau) \|_{L^2(\Omega)}^2 \dd r
+ \int_0^{\pwc \sft \tau(t)} \| \nabla \widehat{\alpha} (\pwcs \teta \tau ) \|_{L^2(\GC)}^2  \dd r
\\
&
\leq
C \Big(1+ \| \doublehhat{\alpha}  (\teta_\rho^0 ) \|_{L^1(\Omega)} +\| \doublehhat{\alpha}  (\dss \rho 0  ) \|_{L^1(\GC)} \Big)
+  \int_0^{\pwc \sft \tau(t)} \io  \pwc\teta\tau \mathrm{div} \left(\pwl \uu \tau' \right)  \widehat{\alpha} (\pwc \teta \tau)  \dd x \dd r
\\
&
\quad
+ \int_0^{\pwc \sft \tau(t)} \pairing{}{H^1(\Omega)}{\pwc h\tau}{ \widehat{\alpha} (\pwc \teta \tau) }  \dd r
+ \int_0^{\pwc \sft \tau(t)} \pairing{}{H^1(\GC)}{\pwc \ell\tau}{ \widehat{\alpha} (\pwcs \teta \tau) }  \dd r
\\
&
\quad
+  \int_0^{\pwc \sft \tau(t)} \int_\Omega \eps\left( \pwl \uu\tau'\right) \, \vtens \, \eps\left( \pwl \uu \tau'\right)
\widehat{\alpha} (\pwc \teta \tau) \dd x \dd r
+\int_0^{\pwc \sft \tau(t)}  \int_{\GC} (\pwl\chi\tau')^2   \widehat{\alpha} (\pwcs \teta \tau) \dd x \dd r
\\
& \quad
+ \int_0^{\pwc \sft \tau(t)} \int_{\GC} \tfrac{\lambda(\pwc\chi\tau)-\lambda(\upwc\chi\tau)}\tau \pwcs\teta\tau
\widehat{\alpha} (\pwcs \teta \tau) \dd x \dd r
\doteq I_1 + I_2 + I_3 + I_4 + I_5+I_6+I_7.
\end{aligned}
\end{equation}
Thanks to \eqref{approx-initial-data}  and \eqref{growth-primitives}, we have 
\begin{equation} \label{stp3I1}
I_1 \leq C'' \left(1+  \| \teta_\rho^0 \|_{L^{\mu+2}(\Omega)}^{\mu+2} +   \| \dss \rho 0 \|_{L^{\mu+2}(\GC)}^{\mu+2}  \right) \leq C.
\end{equation}
With calculations analogous to those of \eqref{Poincare-hat} we infer that
\begin{equation}
\label{simhatatratti}
\begin{aligned}
\| \hhat{\alpha}(\pwc \teta \tau) \|_{H^1(\Omega)}
&\leq C \left(\| \nabla(\hhat{\alpha}(\pwc \teta \tau) )\|_{L^2(\Omega)} +
\| \hhat{\alpha}(\pwc \teta \tau)\|_{L^1(\Omega)}  \right) \\
&\leq C' \left(\| \nabla(\hhat{\alpha}(\pwc \teta \tau))\|_{L^2(\Omega)} +
\| \doublehhat{\alpha}(\pwc \teta \tau)\|_{L^1(\Omega)}  + 1 \right).
\end{aligned}
\end{equation}
By virtue of \RCORR estimate \EEE \eqref{GiovStep2}, we have that $(\pwc\teta\tau)_\tau$ is \RCORR bounded in   \EEE $   L^{\infty}(0,T;L^1(\Omega)) \cap L^2(0,T;L^6(\Omega))$ and, a fortiori, in $L^4(0,T;L^{12/7}(\Omega))$ by interpolation. Therefore,
\begin{equation}
 \label{stp3I2}
 \begin{aligned}
I_2
&\leq   \int_0^{\pwc \sft \tau(t)} \io  \pwc\teta\tau |\mathrm{div} \left(\pwl \uu \tau' \right)|
\widehat{\alpha} (\pwc \teta \tau)  \dd x \dd r
\\
&\stackrel{(1)}{\leq} \int_0^{\pwc \sft \tau(t)}   \| \pwc\teta\tau \|_{L^{12/7}(\Omega)} \| \mathrm{div}
\left(\pwl \uu \tau' \right) \|_{L^{4}(\Omega)} \| \widehat{\alpha} (\pwc \teta \tau)\|_{L^6(\Omega)}
\dd r
\\
&
\stackrel{(2)}{\leq}  C \int_0^{\pwc \sft \tau(t)}  \| \pwc\teta\tau \|_{L^{12/7}(\Omega)} \| \mathrm{div}
\left(\pwl \uu \tau' \right) \|_{L^{4}(\Omega)} \left( \| \doublehhat{\alpha}(\pwc \teta \tau)\|_{L^1(\Omega)}  + 1 \right) \dd r
\\
 & \qquad + C \int_0^{\pwc \sft \tau(t)}  \| \pwc\teta\tau \|_{L^{12/7}(\Omega)}^2 \| \mathrm{div}
\left(\pwl \uu \tau' \right) \|_{L^{4}(\Omega)}^2 \dd r +
 \frac{1}{4}  \int_0^{\pwc \sft \tau(t)} \| \nabla \widehat{\alpha} (\pwc \teta \tau) \|_{L^2(\Omega)}^2 \dd r
\end{aligned}
\end{equation}
with (1) due to the  H\"older inequality and (2) to \eqref{simhatatratti} and the Young inequality.
Analogously,  we  have 
\begin{align}
\label{stp3I3}
&
\begin{aligned}
I_5  & \leq C  \int_0^{\pwc \sft \tau(t)} \int_\Omega |\eps\left( \pwl \uu\tau'\right) |^2
\widehat{\alpha} (\pwc \teta \tau) \dd x \dd r
\\
 & \leq  C \int_0^{\pwc \sft \tau(t)} \| \eps\left( \pwl \uu\tau'\right) \|^2_{L^4(\Omega)}
 \left(\| \nabla(\hhat{\alpha}(\pwc \teta \tau))\|_{L^2(\Omega)} +
\| \doublehhat{\alpha}(\pwc \teta \tau)\|_{L^1(\Omega)}  + 1 \right) \dd r
\\
&\leq
\frac{1}{4}  \int_0^{\pwc \sft \tau(t)} \| \nabla \widehat{\alpha} (\pwc \teta \tau) \|_{L^2(\Omega)}^2 \dd r
+  C \int_0^{\pwc \sft \tau(t)}  \| \eps\left( \pwl \uu\tau'\right) \|^2_{L^4(\Omega)}  \left(   \| \doublehhat{\alpha} (\pwc \teta \tau) \|_{L^1(\Omega)} +1 \right) \dd r
\\
& \qquad + C \int_0^{\pwc \sft \tau(t)} \| \eps\left( \pwl \uu\tau'\right) \|^4_{L^4(\Omega)}
 \dd r \,,
\end{aligned}
\\
&
\begin{aligned}
&
I_6 = \int_0^{\pwc \sft \tau(t)}  \int_{\GC} (\pwl\chi\tau')^2   \widehat{\alpha} (\pwcs \teta \tau) \dd x \dd r
\\
&
\leq
\frac{1}{4}  \int_0^{\pwc \sft \tau(t)} \| \nabla \widehat{\alpha} (\pwcs \teta \tau) \|_{L^2(\GC)}^2 \dd r
+  C \int_0^{\pwc \sft \tau(t)}  \|  \pwl \chi\tau' \|^2_{L^4(\GC)}  \left(   \| \doublehhat{\alpha} (\pwcs \teta \tau) \|_{L^1(\GC)} +1 \right) \dd r
\\
& \qquad + C \int_0^{\pwc \sft \tau(t)} \|  \pwl \chi\tau' \|^4_{\RCORR L^4(\GC)\EEE} \dd r.
\end{aligned}
\end{align}
Furthermore,
again observing that
$(\pwcs \teta \tau)_\tau$ is bounded in $L^4 (0,T;L^{12/7}(\GC))$
\RCORR (a higher integrability estimate actually holds) \EEE
 by interpolation and arguing as for
\eqref{stp3I2}, we find that
\begin{equation}
\begin{aligned}
&
I_7 \leq C \int_0^{\pwc \sft \tau(t)}  \int_{\GC} \frac{|\pwc\chi\tau - \upwc\chi\tau|}{\tau} \pwcs \teta \tau
    \widehat{\alpha} (\pwcs \teta \tau) \dd x \dd r
    \\
    & \leq
    C \int_0^{\pwc \sft \tau(t)}    \|  \pwl \chi\tau' \|_{L^4(\GC)} \|  \pwcs \teta \tau \|_{L^{12/7}(\GC)}
    \|  \widehat{\alpha} (\pwcs \teta \tau) \|_{H^1(\GC)}  \dd r
    \\
    &
    \leq  C \int_0^{\pwc \sft \tau(t)}  \| \pwcs\teta\tau \|_{L^{12/7}(\GC)} \|
\pwl \chi\tau'  \|_{L^{4}(\GC)} \left( \| \doublehhat{\alpha}(\pwcs \teta \tau)\|_{L^1(\GC)}  + 1 \right) \dd r
\\
 & \qquad + C \int_0^{\pwc \sft \tau(t)}  \| \RCORR \pwcs\teta\tau \EEE  \|_{L^{12/7}(\GC)}^2 \|
\pwl \chi \tau'  \|_{L^{4}(\GC)}^2 \dd r +
 \frac{1}{4}  \int_0^{\pwc \sft \tau(t)} \| \nabla \widehat{\alpha} (\pwcs \teta \tau) \|_{L^2(\GC)}^2 \dd r\,.
\end{aligned}
\end{equation}
Finally, we observe that
\begin{equation} \label{stp3I456}
I_3   \leq  \int_0^{\pwc \sft \tau(t)}  \| \pwc h \tau \|_{H^1(\Omega)^*}  \left( \| \doublehhat{\alpha}(\RCORR \pwc \teta \tau)\EEE\|_{L^1(\GC)}  + 1 \right) \dd r
+ \int_0^{\pwc \sft \tau(t)}  \| \pwc h \tau \|_{H^1(\Omega)^*}^2 \dd r +
 \frac{1}{4}  \int_0^{\pwc \sft \tau(t)} \| \nabla \widehat{\alpha} (\pwc \teta \tau) \|_{L^2(\Omega)}^2 \dd r\,,
\end{equation}
and we estimate in the very same way the term $I_4$.
All in all, from  \eqref{stp3I1} and  \eqref{stp3I2}--\eqref{stp3I456}, also
taking into account the previously obtained estimates,
 we conclude
\begin{equation}
\begin{aligned}
&
\| \doublehhat{\alpha} (\pwc \teta \tau(t)) \|_{L^1(\Omega)}
+ \| \doublehhat{\alpha} (\pwcs \teta \tau(t)) \|_{L^1(\GC)}
+ \frac{1}{4}  \int_0^{\pwc \sft \tau(t)} \| \nabla \widehat{\alpha} (\pwc \teta \tau) \|_{L^2(\Omega)}^2 \dd r
+  \frac{1}{4}  \int_0^{\pwc \sft \tau(t)} \| \nabla \widehat{\alpha} (\pwcs \teta \tau ) \|_{L^2(\GC)}^2  \dd r
\\
&
\leq
C +
C \int_0^{\pwc \sft \tau(t)} m_\tau \left( \| \doublehhat{\alpha} (\pwc \teta \tau) \|_{L^1(\Omega)}
+ \| \doublehhat{\alpha} (\pwcs \teta \tau) \|_{L^1(\GC)} \right) \dd r\,,
\end{aligned}
\end{equation}
with $
(m_\tau )_\tau$ a sequence bounded in $L^1(0,T)$.
Therefore, applying the discrete Gronwall Lemma \ref{l:discrG1/2}, we conclude that
\begin{equation} \label{alfacappellinofin}
\| \widehat{\alpha}(\pwc \teta\tau)\|_{L^2(0,T;H^1(\Omega))} + \| \widehat{\alpha}(\pwcs \teta\tau)\|_{L^2(0,T;H^1(\GC))} + \| \doublehhat{\alpha} (\pwc \teta \tau) \|_{L^\infty(0,T;L^1(\Omega))}
+ \| \doublehhat{\alpha} (\pwcs \teta \tau) \|_{L^\infty(0,T;L^1(\GC))}
\leq C_\rho.
\end{equation}
\noindent
\emph{Step $4$: Comparison estimates.}
Taking into account  estimates \eqref{alfacappellinofin}  and the previously found bounds,
by comparison in the heat equations  \eqref{interp-b-heat} and
\eqref{interp-surf-temp},
 respectively, we infer that
\begin{equation}
\| \pwl \teta\tau\|_{H^1(0,T;H^1(\Omega)^*)} +
\| \pwls\teta\tau\|_{H^1(0,T;H^1(\GC)^*)} \leq C_{\rho}\,,
\end{equation}
so that  \eqref{teta-a-prio-tau-rho} and \eqref{tetas-a-prio-tau-rho} immediately follow.
\par The total variation estimate \eqref{VAR-teta-a-prio-tau} then immediately ensues, taking into account that
\[
\mathrm{Var}_{H^1(\Omega)^*}(\pwc \teta \tau; [0,T])   \leq \| \pwl \teta\tau'\|_{L^1(0,T;H^1(\Omega)^*)};
\]
 the very same arguments also yield \eqref{VAR-tetas-a-prio-tau}.
\par
Moreover,  taking into account the previous estimates,
the Lipschitz continuity of
$\betar$,
$\gamma'$
and $\lambda$,  by comparison in the flow rule  \eqref{interp-flow-rule} we find  estimate \eqref{chi-a-prio-tau-rhoH2}, with a constant also depending on $\varsigma$.
 Finally, estimate \eqref{est-recipr-interp} follows from \eqref{Lp-est-4-recipr}. 
\end{proof}
\subsection{Limit passage as $\tau \down 0$}
\label{ss:5.2}
 Lemma  \ref{l:a-prio-tau-compact} ahead   fixes the compactness properties of the sequences of approximate solutions for which the estimates  of Lemma \ref{l:a-prio-tau} hold.
\par
The most delicate point is the proof  of the relative compactness, a.e.\ in $\Omega \times (0,T) $ and
a.e.\ in $\GC \times (0,T)$, of the families of functions
$(\pwc \teta{\tau})_\tau$ and $(\pwcs \teta\tau)_{\tau}$; from this information we can indeed infer the compactness, a.e.\  $\Omega \times (0,T) $ and
a.e.\ in $\GC \times (0,T)$, of the sequences $\left(\frac1{\pwc \teta\tau} \right)_\tau$ and
$\left(\frac1{\pwcs \teta\tau} \right)_\tau$, which, combined with  estimates \eqref{est-recipr-interp}, ultimately yields \eqref{che-fatica-pos} below.
In fact, for proving the pointwise (in space and time) convergence of $(\pwc \teta{\tau})_\tau$ and $(\pwcs \teta\tau)_{\tau}$ we shall resort to the following Helly-type compactness result, which we quote from
\cite{RocRos}  in a slightly simplified version.
In the statement we will use  the following space
\[
\begin{aligned}
\mathrm{B} ([0,T];\mathbf{Y}^*):=\{{\bf h}:[0,T] \to \mathbf{Y}^* : \hbox { measurable and such that }
{\bf h}(t) \hbox { is defined at every } t\in [0,T]\}. 
\end{aligned}
\]
\begin{theorem}[Theorem A.5, \cite{RocRos}]
\label{th:mie-theil}
Let $\mathbf{V}$ and $\mathbf{Y}$ be two (separable) reflexive Banach spaces  such that
$\mathbf{V} \subset \mathbf{Y}^*$ continuously.
Let $(\mathfrak{h}_n)_n \subset L^p
(0,T;\mathbf{V}) \cap \mathrm{B} ([0,T];\mathbf{Y}^*)$ be  bounded  in $L^p
(0,T;\mathbf{V}) $ and suppose in addition  that
\begin{align}
\label{ell-n-0}
&
\text{$(\mathfrak{h}_n(0))_n\subset \mathbf{Y}^*$ is bounded},
\\
&
\label{BV-bound}
\exists\, C>0 \  \ \  \forall\, n \in \N\, : \quad
 \mathrm{Var}_{\mathbf{Y}^*}(\mathfrak{h}_n; [0,T] )  \leq C.
\end{align}

Then, there exists  a subsequence
 $(\mathfrak{h}_{n_k})_k$ of $(\mathfrak{h}_n)_n$
 and a function $\mathfrak{h} \in L^p (0,T;\mathbf{V}) \cap L^\infty (0,T; \mathbf{Y}^*) $ 
 such that  as $k\to \infty$
 \begin{align}
 \label{weak-LpB}
 &
 \mathfrak{h}_{n_k} \weaksto \mathfrak{h} \quad \text{ in } L^p (0,T;\mathbf{V}) \cap L^\infty (0,T;\mathbf{Y}^*),
 \\
\label{weak-ptw-B}
&
\mathfrak{h}_{n_k}(t) \weakto \mathfrak{h}(t) \quad \text{ in } \mathbf{V}\quad \foraa\, t \in (0,T).
\end{align}
\end{theorem}

\begin{lemma}[Compactness results]
\label{l:a-prio-tau-compact}
Let $\rho,\, \varsigma >0$ be fixed.
%
For any  sequence
$(\tau_k)_k \subset (0,+\infty)$ with $\tau_k \downarrow 0$ as $k \rightarrow + \infty$ there exist a (not relabeled) subsequence
and a \RCNEW quintuple $(\teta,  \uu, \tetas, \chi, \sigma)$ \EEE  with
\[
\begin{cases}
\teta \in L^2(0,T;H^1(\Omega)) \cap \mathrm{C}^0_{\mathrm{weak}}([0,T];L^2(\Omega)) \cap H^1(0,T;H^1(\Omega)^*),
\\
\uu \in W^{1,\omega}(0,T;W^{1,\omega}(\Omega;\R^3)),
\\
\tetas \in L^2(0,T;H^1(\GC)) \cap \mathrm{C}^0_{\mathrm{weak}}([0,T];L^2(\GC)) \cap H^1(0,T;H^1(\GC)^*),
\\
\chi \in  \mathrm{C}^0_{\mathrm{weak}}([0,T];H^1(\GC)) \cap  W^{1,\omega}(0,T;L^\omega(\GC)), \quad A\chi \in \RCORR  L^{\omega/(\omega-1)}(\GC{\times}(0,T)), \EEE
\\
\RCNEW \sigma \in L^\infty(\GC{\times}(0,T)), \EEE
\end{cases}
\]
 such that the following weak and strong converges hold as $k\to\infty$
\begin{subequations}
\label{convergences-tau}
\begin{align}
&
\label{cv-t-teta-1}
\pwl \teta {\tau_k}  \weaksto \teta && \text{in } L^2(0,T;H^1(\Omega)) \cap L^\infty(0,T;L^2(\Omega)) \cap H^1(0,T;H^1(\Omega)^*),
\\
&
\label{cv-t-teta-pwc-1}
\pwc \teta {\tau_k}  \weaksto \teta && \text{in } L^2(0,T;H^1(\Omega)) \cap L^\infty(0,T;L^2(\Omega)),
\\
&
\label{cv-t-teta-2}
\pwl \teta {\tau_k} \to \teta && \text{in }  L^2(0,T;L^{6-\epsilon}(\Omega))  
\quad \text{for all } \epsilon\in (0,5]
\\
&
 \label{also-useful-later}
 \pwc \teta {\tau_k}(t)  \weakto \teta(t)  && \text{in } H^1(\Omega) \  \foraa\, t \in (0,T),
 \\
&
\label{cv-t-teta-pwc-ptw}
\pwc \teta {\tau_k} \to \teta && \aein   \ \Omega \times (0,T),
\\
&
\label{cv-t-tetas-1}
\pwls \teta {\tau_k}  \weaksto \tetas && \text{in } L^2(0,T;H^1(\GC)) \cap L^\infty(0,T;L^2(\GC)) \cap H^1(0,T;H^1(\GC)^*),
\\
&
\label{cv-t-tetas-pwc-1}
\pwcs \teta {\tau_k}  \weaksto \tetas && \text{in } L^2(0,T;H^1(\GC)) \cap L^\infty(0,T;L^2(\GC)),
\\
&
\label{cv-t-tetas-2}
\pwls \teta {\tau_k} \to \tetas && \text{in }  L^2(0,T;L^q(\GC))
 \text{ for all }  1\leq q<\infty,
\\
&
 \label{also-useful-later-tetas}
 \pwcs \teta {\tau_k}(t)  \weakto \tetas(t) && \text{in } H^1(\GC) \  \foraa\, t \in (0,T),
 \\
&
\label{cv-t-tetas-pwc-ptw}
\pwcs \teta {\tau_k} \to \tetas && \aein \ \GC \times (0,T),
\\
&
\label{cv-u-1}
\pwl \uu {\tau_k}  \weakto  \uu	&& \text{in }  W^{1,\omega}(0,T;W^{1,\omega}(\Omega;\R^3)),
\\
&
\label{cv-u-2}
\pwl \uu {\tau_k}  \to  \uu	&& \text{in } \mathrm{C}^0([0,T];\mathrm{C}^0(\overline\Omega;\R^3)),
\\
&
\label{cv-u-pwc-1-weak}
\pwc \uu {\tau_k},\, \upwc \uu {\tau_k}  \weaksto  \uu	&& \text{in }  L^\infty(0,T;W^{1,\omega}(\Omega;\R^3)),
\\
&
\label{cv-u-pwc-1}
\pwc \uu {\tau_k},\,  \upwc \uu {\tau_k} \to  \uu	&& \text{in }  L^\infty(0,T;\mathrm{C}^0(\overline\Omega;\R^3)),
\\
&
\label{cv-chi-1}
\pwl \chi {\tau_k}  \weaksto  \chi	&& \text{in }  L^\infty(0,T;H^1(\GC)) \cap  W^{1,\omega}(0,T;L^\omega(\GC)),
\\
&
\label{cv-chi-2}
\pwl \chi {\tau_k}  \to  \chi	&& \text{in } \mathrm{C}^0 ([0,T];L^q(\GC)) \text{ for all }  1\leq q<\infty.
\\
&
\label{cv-chi-pwc-1}
\pwc \chi {\tau_k}  \weaksto  \chi	&& \text{in } L^\infty (0,T;H^1(\GC)),
\\
&
\label{cv-chi-pwc-2}
\pwc \chi {\tau_k} , \, \upwc \chi {\tau_k}  \ \to  \chi	&& \text{in } L^\infty (0,T;L^q(\GC)) \text{ for all }  1\leq q<\infty,
\\
&
\label{upsilon-cv}
\RCNEW \pwc \sigma {\tau_k}  \weaksto  \sigma \EEE	&& \text{in } L^\infty (\GC{\times}(0,T))\,. 
\end{align}
\end{subequations}
Furthermore,
\begin{equation}
\label{che-fatica-pos}
\left\|\frac1{\teta}\right\|_{L^\infty(\Omega\times (0,T))} + \left\|\frac1{\tetas}\right\|_{L^\infty(\GC\times (0,T))} \leq S_0,
\end{equation}
with the constant $S_0$ from \eqref{Lp-est-4-recipr}.  Therefore, the functions $\teta$ and $\tetas$ enjoy the  positivity properties  \RCORR \eqref{teta-strict-pos}  with constants \EEE $\bar\teta$ and $\bar\tetas$ independent of $\rho$ and $\varsigma$.
\end{lemma}
\begin{proof}
Convergences \eqref{cv-t-teta-1}, \eqref{cv-t-tetas-1}, \eqref{cv-u-1}, \eqref{cv-chi-1}, \RCNEW and \eqref{upsilon-cv} \EEE  follow from estimates
\eqref{a-prio-tau} and \eqref{a-prio-tau-rho} by weak  and weak$^*$ compactness arguments.
In view of \eqref{chi-a-prio-tau-rhoH2}, we also have
\begin{equation}
\label{AchiLomega}
A\pwc \chi{\tau_k} \weakto A\chi \qquad \text{in } \RCORR L^{\omega/(\omega-1)}(\GC{\times}(0,T)). \EEE
\end{equation}
As for  \eqref{cv-t-teta-pwc-1},
combining
estimates \eqref{regtratti1}
and the fact that the sequence $(\pwl \teta{\tau_k}')_k$
 is bounded in $L^2(0,T;H^1(\Omega)^*)$, we conclude that
 $\|\pwl \teta{\tau_k}{-} \pwc \teta{\tau_k} \|_{L^\infty(0,T;H^1(\Omega)^*)} \to 0 $ as $\tau_k\down 0$.
This identifies $\teta$ as the weak$^*$ limit of
 $(\pwc \teta{\tau_k})_k$ in $L^2(0,T;H^1(\Omega)) \cap L^\infty(0,T;L^2(\Omega))$.
 With the same argument
   we also  infer  convergence \eqref{cv-t-tetas-pwc-1}.
   Clearly, we have that $\teta \geq 0$ a.e.\ in $\Omega \times (0,T)$ and $\tetas \geq 0$ a.e.\ in $\GC \times (0,T)$.
   Analogously,
    \eqref{cv-u-pwc-1-weak} and \eqref{cv-u-pwc-1} (\eqref{cv-chi-pwc-1} \&  \eqref{cv-chi-pwc-2}, respectively) shall follow from \eqref{cv-u-2} (\eqref{cv-chi-2}, resp.).
\par
For \eqref{cv-t-teta-2} we
apply, e.g., the compactness result \cite[Cor.\ 4]{Simon87}, which ensures that
$(\pwl \teta{\tau_k})_k$ is relatively compact in  \RCORR $L^2(0,T;X)  \EEE \cap \mathrm{C}^0([0,T];Y)$ \EEE for any Banach spaces
$X$ and $Y$ such that
$H^1(\Omega) \Subset X \subset L^2(\Omega)$ and
$L^2(\Omega) \Subset Y \subset H^1(\Omega)^*$. In the same way, \eqref{cv-t-tetas-2} follows, recalling that $H^1(\GC) \Subset L^q(\GC)$ for every $1\leq q <\infty$.
The strong convergence \eqref{cv-u-2} can be deduced by the same result, taking into account that
$W^{1,\omega}(\Omega;\R^3) \Subset \mathrm{C}^0(\overline\Omega)$ since  $\omega > 4$,  by the Rellich-Kondrachov Theorem.
  Analogously, we have
 \eqref{cv-chi-2}. 
 \par
 Finally,  applying to the sequences
 $(\pwc \teta{\tau_k})_k$ and $(\pwl \teta{\tau_k})_k$ the compactness Theorem \ref{th:mie-theil} (also recalling estimate \eqref{VAR-teta-a-prio-tau}), we infer the \emph{pointwise-in-time} convergences
 whence, in particular, \eqref{cv-t-teta-pwc-ptw} (since
  $H^1(\Omega) \Subset L^p(\Omega)$  for all $1\leq p <6$).
We recover \eqref{cv-t-tetas-pwc-ptw} in the very same way.
\par
Therefore,
\[
\frac1{\pwc \teta {\tau_k}(x,t)} \to L(x,t): = \begin{cases}
\frac1{\teta(x,t)} & \text{if } \teta(x,t)>0,
\\
+\infty & \text{otherwise} \qquad \foraa\, (x,t) \in \Omega\times (0,T)
\end{cases}
\]
In turn, combining the Fatou Lemma with estimate   \eqref{est-recipr-interp}  for, e.g.,  $p=2$ we infer that
\[
\int_{\Omega\times (0,T)} L^2(x,t) \dd x  \dd t \leq\liminf_{k\to \infty} \int_0^T \int_{\Omega} \frac1{\pwc \teta {\tau_k}^2  (x,t)} \dd x  \dd t \leq S_0 T,
\]
so that $L (x,t)<\infty$ for a.a.\ $(x,t)\in \Omega \times (0,T)$.
Hence, $\teta >0 $ a.e.\ in $\Omega\times (0,T)$ and $L = \frac1{\teta}$.  A fortiori, again  in view of   \eqref{est-recipr-interp}  and the Fatou Lemma   we conclude that
\[
\int_{t_0-r}^{t_0+r} \left\| \frac1{\teta(s)}\right\|_{L^p(\Omega)} \dd s \leq \liminf_{k\to\infty}
\int_{t_0-r}^{t_0+r} \left\| \frac1{\pwc\teta{\tau_k}(s)}\right\|_{L^p(\Omega)} \dd s \leq 2 S_0 r
\]
for every $t_0 \in (0,T)$ and $r \in (t_0,T-t_0)$.  In particular, picking  a Lebesgue point for $\|\frac1{\teta(\cdot)}\|_{L^p(\Omega)}$ we gather that
\[
 \left\| \frac1{\teta(t_0)}\right\|_{L^p(\Omega)} \leq S_0 \qquad \foraa\, \RCORR  t_0 \in (0,T) \EEE   \text{ and for all } p \in [1,\infty),
\]
whence estimate \eqref{che-fatica-pos} for  $\tfrac1\teta$. 
We conclude the estimate for  $\left\|\tfrac1{\tetas}\right\|_{L^\infty(\GC\times (0,T))}$ in the very same way.
  This finishes the proof.
\end{proof}

\par
We are now in a position to prove our existence result for the Cauchy problem for (a weak formulation of) the regularized system
\eqref{PDE-regul}.
\begin{theorem}
\label{thm:exist-regul}
Assume
\eqref{ass:domain}--\eqref{hyp-k} and \eqref{hyp-alpha}--\eqref{cond-data}.
Let $\rho,  \varsigma >0 $ be fixed. Then, for any quadruple
$(\teta_\rho^0,\uu_\rho^0, \dss \rho 0, \chi_0) $ as in \eqref{approx-initial-data} and \eqref{cond-chi0},
there exists a \RCNEW quintuple $(\teta,\uu,\tetas,\chi,\sigma)$, \EEE with
\[
\begin{aligned}
&
\teta \in L^2(0,T;H^1(\Omega)) \cap \mathrm{C}_{\mathrm{weak}}^0([0,T];L^2(\Omega)) \cap H^1(0,T;H^1(\Omega))^* \quad \text{ and } \quad  \widehat\alpha(\teta) \in L^2(0,T;H^1(\Omega)),
\\
&
\uu \in W^{1,\omega}(0,T;W^{1,\omega}(\Omega;\R^3)),
\\
&
\tetas \in L^2(0,T;H^1(\GC)) \cap \mathrm{C}_{\mathrm{weak}}^0([0,T];L^2(\GC)) \cap H^1(0,T;H^1(\GC))^* \quad \text{ and } \quad  \RCOMMN \widehat\alpha(\tetas) \EEE \in L^2(0,T;H^1(\GC)),
\\
&
\chi \in  L^2 (0,T;H^2(\Omega)) \cap L^\infty (0,T;H^1(\GC)) \cap  W^{1,\omega}(0,T; L^\omega(\Omega)),
\\
& \RCNEW \sigma \in L^\infty(\GC{\times}(0,T)), \EEE
\end{aligned}
\]
fulfilling the initial conditions
\begin{equation}
\label{initial-rho}
\teta(0) = \teta_\rho^0 \ \aein\ \Omega, \ \tetas(0) = \dss \rho 0 \ \aein\ \GC, \ \uu(0) = \uu_\rho^0 \ \aein\  \Omega, \ \chi(0)=\chi_0 \ \aein\  \GC,
\end{equation}
and the weak formulation of system \eqref{PDE-regul}, consisting of
\begin{subequations}
\label{weak-368}
\begin{enumerate}
\item the weak formulation 
of the bulk heat equation for almost all $t\in (0,T)$
\begin{equation}
 \label{form_debole_theta-chip}
\begin{aligned}
  &  \pairing{}{H^1(\Omega)}{\theta_t}{v} 
- \io \theta\mathrm{div}(\uu_t) v \dd x
+ \io \ab(\theta) \nabla\theta \nabla v \dd x
+ \int_{\GC} k(\chi) \theta (\theta - \thetas) v \dd x
\\
& \qquad
+ \int_{\GC} \nlocss{\chip}\, \chip \theta^2 v \dd x
- \int_{\GC} \nlocss{\chip \thetas} \,\chip \theta v \dd x
 = \io \tensoret \mathbb{V} \tensoret v  \dd x + \io h v \dd x
\end{aligned}
\end{equation}
for all $v\in H^1(\Omega)$;
\item the weak formulation of the displacement equation for almost all $t\in (0,T)$
\begin{equation}
\label{weak-U-rho}
\begin{aligned}
&
\vfm(\uu_t,\mathbf{v}) +\rho \int_\Omega |\eps(\uu_t)|^{\omega-2} \eps(\uu_t) \eps(\vv) \dd x
+ \efm(\uu,\mathbf{v})
+\int_{\Omega} \teta \mathrm{div}(\vv) \dd x +
 \int_{\GC} \chip \uu \mathbf{v} \dd x
 \\
  & \qquad +
 \int_{\GC} \zzeta \cdot \mathbf{v} \dd x
+  \int_{\GC} \chip \nlocss{\chip}\,  \uu \mathbf{v} \dd x =
 \pairing{}{\bsVD}{\mathbf{F}}{\vv}
\end{aligned}
\end{equation}
 for all $ \vv \in  W^{1,\omega}_{\mathrm{D}}(\Omega;\R^3)$, 
 with  $\zzeta(t) = \etar(\uu(t) \cdot \mathbf{n}) \mathbf{n}$  $\foraa\, t \in (0,T)$;
 \item the weak formulation 
 of the surface heat equation
 for almost all $t\in (0,T)$
 \begin{equation} \label{form_debole_thetas-chip}
\begin{aligned}
&
 \pairing{}{H^1(\GC)}{\partial_t {\thetas}}{w} 
- \int_{\GC} \thetas \lambda'(\chi) \chi_t w \dd x
+ \int_{\GC} \as(\thetas) \nabla\thetas \nabla w  \dd x
\\
& =
\int_{\GC} \ell w \dd x +
  \int_{\GC} |\chi_t|^2 w \dd x
+ \int_{\GC}  k(\chi)\thetas(\theta - \thetas) w \dd x
+ \int_{\GC} \nlocss{\chip\theta}\, \chip \thetas w \dd x
- \int_{\GC} \nlocss{\chip}\, \chip \thetasq  w \dd x
 \end{aligned}
\end{equation}
   with test functions $w\in H^1(\GC)$,  a.e.\ in $(0,T)$;
\item the flow rule for the adhesion parameter
\begin{equation}
\label{ptw-flow-rule}
\begin{aligned}
&
 \chi_t + \rho |\chi_t|^{\omega-2}\chi_t  +A\chi +\betar(\chi)   +
		\gamma'(\chi)+\lambda'(\chi)\thetas
		\\ &
		 \quad  = \RCNEW -\frac 1 2\vert{\uu }\vert^2 \sigma-\frac 12\nlocss{\chip}\,|\uu|^2\sigma - \frac 1 2\nlocss{\chip |\uu|^2} \,\sigma \quad  \aein\  \GC
		\times (0,T)
		\\
		& \RCNEW \text{with } \sigma \in \partial\varphi(\chi) \quad \aein\ \GC\,.
		\end{aligned}
\end{equation}
\end{enumerate}
\EEE
Furthermore, estimate \eqref{che-fatica-pos} holds and the quadruple $(\teta,\uu, \tetas,\chi)$ satisfies the total energy balance (with the stored energy $\calE_\varsigma$ from \eqref{stored-energy-rho})
\begin{equation}
\label{total-enbal-rho-cont}
\begin{aligned}
& \calE_\varsigma (\teta(t),\tetas(t),\uu(t),\chi(t))
+\rho \int_s^t \int_\Omega |\eps(\uu_t)|^\omega \dd x \dd r + \rho \int_s^t \int_{\GC} |\chi_t|^\omega \dd x \dd r
 \\
 &
 \qquad
\RCORR  + \EEE \int_s^t \int_{\GC} k(\chi)(\teta{-}\tetas)^2 \dd x \dd r
\int_s^t \int_{\GC} \kr(x,y)\RCNEW (\chi(x))^+(\chi(y))^+ \EEE (\teta(x){-}\tetas(y))^2 \dd x \dd y  \dd r
\\
&
=
\calE_\varsigma(\teta(s),\tetas(s),\uu(s),\chi(s)) + \int_s^t \int_\Omega h \dd x \dd r + \int_s^t \int_{\GC} \ell \dd x \dd r +\int_s^t \pairing{}{\bsVD}{\mathbf{F}}{\uu_t} \dd r
\end{aligned}
\end{equation}
\end{subequations}
for all $0\leq s \leq t \leq T$.
\end{theorem}
\begin{proof}
 We shall pass to the limit in system \eqref{interp-syst}
 relying on convergences \eqref{convergences-tau} for the approximate solutions.
\paragraph{\bf Step $1$: limit passage in the  momentum balance:}
First of all, we  focus on the limit passage in equation \eqref{interp-mom-bal}, which we integrate in time.
Thanks to convergences  \eqref{cv-t-teta-pwc-1},  \eqref{cv-u-1},  and \eqref{cv-u-pwc-1-weak} we pass to the limit in the first, third and fourth integral terms on the left-hand side of \eqref{interp-mom-bal}. As for the second term, we observe that there exists $\textbf{E} \in L^{\omega/(\omega-1)}(\Omega\times (0,T); \mathbb{R}^{3 \times 3})$ such that
\begin{equation}
\label{weak-limit-E}
 \left|\eps \left( \pwl \uu{\tau_k}'\right) \right|^{\omega-2}\eps \left( \pwl \uu{\tau_k}'\right) \weakto \mathbf{E} \quad \text{in } L^{\omega/(\omega-1)}(\Omega\times (0,T); \mathbb{R}^{3 \times 3})
\end{equation}
as $\tau_k \down 0$.
We also use that
\begin{subequations}
\label{further-cv-4mombal}
\begin{equation}
\begin{cases}
\RCNEW (\pwc\chi{\tau_k})^+ \EEE \pwc\uu{\tau_k} \to  (\chi)^+ \uu,
\\
\nlocss {\RCNEW (\pwc\chi{\tau_k})^+}\, ( \pwc\chi{\tau_k} )^+ \EEE  \pwc\uu{\tau_k}  \to  \nlocss {(\chi)^+}  (\chi)^+\uu
\end{cases}
\qquad \text{in } L^\infty(0,T;L^q(\GC)) \quad \text{for all } 1\leq q <\infty
\end{equation}
as $\tau_k\down 0$
thanks to convergences \eqref{cv-u-pwc-1}, \eqref{cv-chi-pwc-2}, and Lemma \ref{lemmaK}.
By the Lipschitz continuity of $\etar$ we readily have that
\begin{equation}
\pwc \zzeta{\tau_k} \to \zzeta: = \etar(\uu \cdot \mathbf{n})\mathbf{n} \qquad \text{ in } \mathrm{C}^0(0,T;\mathrm{C}^0(\overlineGC)).
\end{equation}
\end{subequations}
Finally, we use that 
\begin{equation}
\label{conv-tau-F}
\pwc{\mathbf{F}}{\tau_k} \rightarrow \mathbf{F}	 	\quad \quad \textrm{in $L^2(0,T; H^1(\Omega, \mathbb{R}^3)^*)$}.
\end{equation}
All in all, we conclude that
\begin{equation}
\label{temporary=mombal}
\begin{aligned}
&
\int_s^t \left( \vfm(\uu_t,\mathbf{v}) +\rho\, \mathbf{E} : \eps(\vv) \dd x
+ \efm(\uu,\mathbf{v})
+\int_{\Omega} \teta \mathrm{div}(\vv) \dd x +
 \int_{\GC} \RCNEW (\chi)^+ \EEE \uu \mathbf{v} \dd x \right) \dd r
 \\
  & \qquad +
\int_s^t  \left( \int_{\GC} \zzeta \cdot \mathbf{v} \dd x
+  \int_{\GC}  \nlocss{\RCNEW (\chi)^+}\, \RCNEW (\chi)^+ \EEE \uu \mathbf{v} \dd x \right) \dd r  =
 \int_s^t \pairing{}{\bsVD}{\mathbf{F}}{\vv}  \dd r
\end{aligned}
\end{equation}
for every test function $ \vv \in  W^{1,\omega}_{\mathrm{D}}(\Omega;\R^3)$ 
and every $0 \leq s \leq t \leq T$.
Then,
testing \eqref{interp-mom-bal} by $\pwl \uu{\tau_k}'$
we infer that for every $(s,t)\subset (0,T)$ there holds
\begin{equation} \label{mom-tau1}
\begin{aligned}
&
\limsup_{\tau_k \searrow 0}
\rho \int_s^t \io \left| \eps \left( \pwl \uu{\tau_k}'\right) \right|^{\omega} \dd x \dd r
\\
&
\leq
\limsup_{k\to\infty} \Big( -\int_s^t \left( \vfm( \pwl \uu{\tau_k}',  \pwl \uu{\tau_k}')  + e(\pwc\uu{\tau_k},\pwl \uu{\tau_k}') +\int_\Omega \pwc\teta{\tau_k}\mathrm{div}(\pwl \uu {\tau_k}')  \dd x 
 +
 \int_{\GC} \RCNEW(\pwc\chi{\tau_k})^+ \EEE   \pwc\uu{\tau_k} \pwl \uu{\tau_k}' \dd x \right) \dd r
 \\
 & \qquad \qquad \quad
 -\int_s^t \left( \int_{\GC} \langle \pwc\zzeta{\tau_k} \cdot  \pwl \uu{\tau_k}' \dd x
+ \int_{\GC} \nlocss {\RCNEW (\pwc\chi{\tau_k})^+}\, \RCNEW (\pwc\chi{\tau_k})^+ \EEE \pwc\uu{\tau_k}  \pwl \uu{\tau_k}' \dd x  - \langle  \pwc{\mathbf{F}}{\tau_k}(t),   \pwl \uu{\tau_k}'  \rangle_{\bsVD}  \right) \dd r    \Big)
\\
&
\stackrel{(1)}{\leq}
- \int_s^t
 \left( \vfm(\uu_t,\uu_t)
+ \efm(\uu,\uu_t)
+\int_{\Omega} \teta \mathrm{div}(\uu_t) \dd x +
 \int_{\GC} \RCNEW (\chi)^+ \EEE \uu \uu_t \dd x \right) \dd r
 \\
  & \qquad \qquad \quad
   -
\int_s^t  \left( \int_{\GC} \zzeta \cdot \uu_t \dd x
+  \int_{\GC}  \nlocss{\RCNEW (\chi)^+}\,  \RCNEW (\chi)^+ \EEE \uu \uu_t \dd x
- \pairing{}{\bsVD}{\mathbf{F}}{\uu_t} \right)  \dd r
\\
&
\stackrel{(2)}{=}
\rho\int_s^t  \mathbf{E} : \eps(\uu_t) \dd r\,,
\end{aligned}
\end{equation}
where (1) ensues from convergences \eqref{convergences-tau} (which, in particular, yield that $\pwc \teta{\tau_k}\to\teta$ in $L^2(0,T;L^2(\Omega))$, for instance) and \eqref{further-cv-4mombal}, while (2) follows from the previously obtained \eqref{temporary=mombal}.
Hence,
\cite[Lemma~1.3, p.~42]{Barbu} yields that
\begin{equation}
\label{mom-tau2}
\textbf{E} = \left| \eps \left( \uu' \right) \right|^{\omega-2} \eps \left(  \uu' \right) \  \aein \ \Omega \times (0,T)
\quad \text{ and } \quad    \eps ( \pwl \uu{\tau_k}' )\to \eps ( \uu_t) \ \text{ strongly in } L^{\omega}(\Omega{\times}(0,T);\mathbb{R}^{3 \times 3})\,,
\end{equation}
and, since the interval $(s,t)$ in
\eqref{temporary=mombal} is chosen arbitrarily,
 we thus conclude the momentum balance equation \eqref{weak-U-rho}.
We remark for later use that \eqref{mom-tau2} yields   that
\begin{equation} \label{mom-tau23}
\eps ( \pwl \uu{\tau_k}' (t)) \vtens \eps ( \pwl \uu{\tau_k}'(t) )
\to  \eps (\RCORR \uu_t(t) ) \vtens \eps(  \uu_t (t)) \EEE
\quad \textrm{ strongly in $L^2(\Omega)$ for a.e. $t \in (0,T)$}.
\end{equation}
\paragraph{\bf Step $2$: limit passage in the flow rule \eqref{interp-flow-rule}:}
We now address the limit passage in the approximate flow rule 
\eqref{interp-flow-rule}, integrated on a generic interval $(s,t) \subset (0,T)$. We use that
there exists $\varLambda \in L^{\omega/(\omega-1)}(\GC{\times}(0,T))$ such  that
\[
| \pwl \chi{\tau_k}'|^{\omega-2} \pwl \chi{\tau_k'} \weakto \varLambda \quad \text{ in } L^{\omega/(\omega-1)}(\GC{\times}(0,T)),
\]
and that $\betar(\pwc\chi{\tau_k})\to \betar(\chi)$ and $\gamma_\nu'
(\pwc\chi{\tau_k}) \to \gamma_\nu'(\chi)$ in $L^\infty(0,T;L^q(\GC))$ for every $1\leq q<\infty$ by the Lipschitz continuity of $\betar$ and $\gamma_\nu'$.
Also taking into account
convergence \eqref{cv-chi-pwc-2} for the right-continuous piecewise constant interpolants $(\pwc\chi{\tau_k})_k$, we carry out the limit passage for the terms on the left-hand side of \eqref{interp-flow-rule}. As for the right-hand side,  we use that $\lambda_\delta'$ is Lipschitz continuous and that, for instance,
$\pwcs \teta{\tau_k}\to \tetas $ in $L^2(0,T;L^2(\GC))$, so that
\[
 - \lambda_\delta'(\upwc\chi{\tau_k}) \pwcs \teta{\tau_k} -\delta \pwc\chi{\tau_k} \pwcs \teta{\tau_k}
 \to -\lambda'(\chi) \tetas \qquad \text{in } L^2(0,T;L^{\omega/(\omega-1)}(\GC{\times}(0,T)).
\]
We  combine \RCNEW \eqref{cv-chi-pwc-1},    \eqref{cv-chi-pwc-2}, and \eqref{upsilon-cv} \EEE  yielding that
\[
\begin{cases}
-\frac12 |\upwc \uu{\tau_k}|^2 \RCNEW \pwc \sigma{\tau_k} \weaksto -\frac12 |\uu|^2 \sigma,
\\
-\frac12 \nlocss{\RCNEW (\pwc\chi{\tau_k})^+}\,  |\upwc \uu{\tau_k}|^2 \RCNEW \pwc \sigma{\tau_k} \weaksto -\frac12 \nlocss{(\chi)^+}|\uu|^2 \sigma,
\\
-\frac12 \nlocss{\RCNEW (\upwc\chi{\tau_k})^+ |\upwc\uu{\tau_k}|^2}\, \pwc\sigma{\tau_k} \weaksto - \frac12 \nlocss{(\chi)^+ |\uu|^2}\, \sigma
\end{cases}
\qquad \text{ in $L^\infty(\GC {\times}(0,T))$}
\]
also in view of Lemma \ref{lemmaK}.
All in all, also recalling \eqref{AchiLomega} we take the limit of \eqref{interp-flow-rule}
and obtain that
\[
\begin{aligned}
&
\int_s^t \left( \int_{\GC}\chi_t v \dd x +\rho \int_{\GC}\varLambda  v \dd x +
 \int_{\GC} A\chi  v \dd x +\int_{\GC}(\betar(\chi){+}\gamma'(\chi)  + \lambda'(\chi)\thetas) v \dd x  \right) \dd r
\\
&
=  -\frac12 \EEE \int_s^t \left(  \int_{\GC} \left(|\uu|^2 \sigma +  \nlocss{(\chi)^+} |\uu|^2 \sigma + \nlocss{(\chi)^+|\uu|^2} \sigma \right)   v   \dd x \right)\dd r
\end{aligned}
\]
for every $v\in L^{\omega}(\GC)$ and every sub-interval $[s,t]\subset (0,T)$.
\RCNEW By the strong-weak closedness  in the sense of graphs
of (the maximal monotone operator induced by) $\partial\varphi$,
we have that $\sigma\in \partial\varphi(\chi)$ a.e.\ in $\GC\times (0,T)$. \EEE
 In order to identify the weak limit
$\varLambda$, we proceed as for the weak limit $\mathbf{E}$ from \eqref{weak-limit-E} and conclude that
that
\begin{equation}
\label{identification-varLambda}
\varLambda = |\chi_t|^{\omega-2}\chi_t \text{ and } \pwl\chi{\tau_k}' \to \chi_t \text{ strongly in } L^\omega(\GC{\times}(0,T))\,.
\end{equation}
In particular,
\begin{equation}
\label{ciliegina-chi}
| \pwl\chi{\tau_k}' (t)|^2 \to |\chi_t(t)|^2 \quad \text{strongly in } L^2(\GC) \ \foraa\, t \in (0,T)\,.
\end{equation}
From \eqref{identification-varLambda} we deduce the validity of the flow rule for the adhesion parameter integrated along an arbitrary time interval, and with arbitrary test functions $v\in L^\omega(\GC)$. Then, \eqref{ptw-flow-rule} ensues.
\paragraph{\bf Step $3$: limit passage in the  bulk heat equation:}
We are now in a position to perform the limit passage in
 \eqref{interp-b-heat}, integrated in time. For this, we need to refine the convergences available for the
 sequences $(\pwc\teta{\tau_k})$ and $(\pwcs\teta{\tau_k})$.
 \begin{enumerate}
 \item
 In order to pass to the limit in the elliptic operator, we will use that
 $\nabla\widehat{\alpha}(\pwc\teta{\tau_k}) =\alpha(\pwc\teta{\tau_k}) \nabla \pwc\teta{\tau_k}$ a.e.\  in $\Omega\times (0,T)$.
 First of all,
 we notice that, by  \eqref{teta-a-prio-tau-rho}
there exists $\phi \in L^2(0,T;H^1(\Omega))$ 
such that, along a not relabeled subsequence,
\[
\begin{aligned}
&
\widehat{\alpha}(\pwc\teta{\tau_k}) \rightharpoonup \phi 	 	&&  \textrm{in } L^2(0,T;H^1(\Omega)).
\end{aligned}
\]
On the other hand, $\widehat{\alpha}(\pwc\teta{\tau_k})\to \widehat{\alpha}(\teta)$ a.e.\ in $\Omega \times (0,T)$ thanks to
\eqref{cv-t-teta-pwc-ptw}.
We combine this with the fact that $(\pwc \teta{\tau_k})_k$ is bounded in
$L^\infty(0,T;L^{\mu+2}(\Omega))$, thanks to \eqref{alfacappellinofin}  and the growth properties of $\doublehhat\alpha$, to deduce that
$\widehat{\alpha}(\pwc\teta{\tau_k})\weaksto \widehat\alpha(\teta)$ in $L^\infty(0,T;L^{(\mu+2)/(\mu+1)}(\Omega))$
(by the growth properties of $\widehat\alpha$). Therefore, we ultimately conclude
\RCORR that $\phi = \widehat{\alpha}(\teta)$, so that  \EEE
\begin{equation}
\label{alpha-teta-cv}
\widehat{\alpha}(\pwc\teta{\tau_k}) \rightharpoonup \widehat\alpha(\teta) \quad \text{in }  	L^2(0,T;H^1(\Omega)).
\end{equation}
\item
In order to
identify the elliptic operator
featuring in the bulk heat equation, we argue in a similar way as we did in the proof of
Proposition  \ref{prop:exists-discrete}. Indeed,
from the fact that $( \widehat{\alpha}(\pwc\teta {\tau_k}))_k$ is bounded in $L^2(0,T;H^1(\Omega))$
and from the growth properties of $\widehat\alpha$  we deduce that
$(\pwc\teta{\tau_k})_k$ is bounded in
$L^{2(\mu+1)}(0,T;L^{6(\mu+1)}(\Omega))$, with $\mu>1$. Taking into account \eqref{cv-t-teta-pwc-ptw}  we deduce, a fortiori, that
\begin{equation}
\label{enhanced-teta-k}
\pwc \teta{\tau_k} \to \teta \quad \text{in } L^4(\Omega{\times}(0,T)),
\end{equation}
as well as $\alpha (\pwc \teta{\tau_k})\to \alpha (\teta)$ in
$L^{2s}(0,T;L^{6s}(\Omega))$ for every $1\leq s <1+\frac1{\mu}$. This is enough to pass to the limit in
the relation $\int_\Omega \nabla\widehat{\alpha}(\pwc \teta{\tau_k}) \cdot \nabla v \dd x = \int_\Omega \alpha(\pwc \teta{\tau_k}) \nabla \pwc \teta{\tau_k} \cdot \nabla v \dd x$ for every $v\in H^1(\Omega)$ by suitably adapting the arguments developed at the end of the proof of Prop.\  \ref{prop:exists-discrete}.
\item It follows from  \eqref{also-useful-later}, combined with the fact that $H^1(\Omega) \Subset L^{p}(\GC)$
(in the sense of traces) for every $1\leq p<4$,
 that
\[
\pwc\teta{\tau_k} \to \teta \qquad \aein\, \GC\times (0,T).
\]
In turn, from the fact that  $( \widehat{\alpha}(\pwc\teta {\tau_k}))_k$ is bounded
in $L^2(0,T;L^4(\GC))$ we gather that
 $(\pwc\teta {\tau_k})_k$ is bounded in $L^{2(\mu+1)}(0,T;L^{4(\mu+1)}(\GC))$.  Combining this with
the above pointwise convergence
  we immediately infer that
 \begin{equation}
 \label{L4-cv-teta-1}
 \pwc\teta{\tau_k} \to \teta  \qquad \text{in } L^{4+\epsilon}(\GC{\times}(0,T)) \qquad \text{for all $\epsilon \in (0,2\mu-2)$}
 \end{equation}
 (so that $4+\epsilon<2\mu+2$).
 \item Analogously, combining \eqref{also-useful-later-tetas} with the estimate for
 $( \widehat{\alpha}(\pwcs\teta {\tau_k}))_k$
 is bounded in $L^2(0,T;H^1(\GC))$, which continuously embeds into $L^2(0,T;L^q(\GC))$ for all $1\leq q<\infty$, we deduce, for instance,  that
  \begin{equation}
 \label{L4-cv-teta}
 \pwcs\teta{\tau_k} \to \tetas  \qquad \text{in } L^{4+\epsilon}(\GC{\times}(0,T)) \qquad \text{for all $\epsilon \in (0,2\mu-2)$}.
 \end{equation}
\end{enumerate}
\par
In view of the enhanced convergences \eqref{enhanced-teta-k}--\eqref{L4-cv-teta}, we infer that
\begin{equation}
\label{minute-convergences}
\begin{aligned}
&
\pwc\teta{\tau_k} \mathrm{div}(\pwl \uu{\tau_k}') \weakto \teta \mathrm{div}(\uu') && \text{in } L^{2}(\Omega{\times}(0,T)),
\\
&
k(\pwc \chi{\tau_k}) \pwc\teta{\tau_k} (\pwc\teta{\tau_k}- \pwcs \teta{\tau_k}) \to k(\chi)\teta(\teta-\tetas) && \text{in } L^2(\GC{\times}(0,T)),
\\
&
 \nlocss {\RCNEW (\upwc\chi{\tau_k})^+}\,  \RCNEW( \upwc\chi{\tau_k})^+   \pwc\teta{\tau_k}^2 \to  \nlocss{(\chi)^+} \, (\chi)^+ \teta^2  &&  \RCORR \text{in } \EEE L^2(\GC{\times}(0,T)),
 \\
  &
 \nlocss {\RCNEW(\upwc\chi{\tau_k})^+  \pwcs \teta{\tau_k}}\, (\upwc\chi{\tau_k})^+ \pwc\teta{\tau_k} \to
  \nlocss{\chip\tetas}\,\chip\teta &&  \RCORR \text{in } \EEE L^2(\GC{\times}(0,T)),
\end{aligned}
\end{equation}
where we have also used that $k(\upwc \chi{\tau_k}) \to k(\chi)$ in $L^q(\GC{\times}(0,T))$ for all $1\leq q<\infty$ thanks to \eqref{cv-chi-pwc-2} and the polynomial growth of $k$, that $ \nlocss {\upwc\chi{\tau_k}}  \to \nlocss \chi$ in $L^\infty(\GC{\times}(0,T))$ by Lemma \ref{lemmaK}, as well as that
$ \nlocss {\upwc\chi{\tau_k}\pwcs \teta{\tau_k}}  \to \nlocss {\chi\teta}$ in $L^4(\GC{\times}(0,T))$.
\par
Also recalling \eqref{mom-tau23} and the fact that
 \begin{equation} \label{conv-tau-h}
\pwc h{\tau_k} \rightarrow h	 	\quad \quad \textrm{in $L^2(0,T;H^1(\Omega)^*)$}
\end{equation}
we conclude the limit passage in \eqref{interp-b-heat}. This yields
 the weak formulation \RCORR \eqref{form_debole_theta-chip}  \EEE of the bulk heat equation,  with test functions $v\in H^1(\Omega)$,  a.e.\ in $(0,T)$.
\paragraph{\bf Step $4$: limit passage in the surface heat equation:} For the limit passage in
\eqref{interp-surf-temp} we use that
\[
\widehat{\alpha}(\pwcs\teta{\tau_k}) \rightharpoonup \widehat\alpha(\tetas) \quad \text{in }  	L^2(0,T;H^1(\GC))
\]
(which can be shown by the very same arguments as for \eqref{alpha-teta-cv}).
Arguing as we did for the bulk heat equation, we identify the elliptic operator featuring in the limiting surface heat equation.
 Furthermore, we observe that
\begin{equation} \label{lambda-dev-prima-tau}
\begin{aligned}
\frac{\lambda(\pwc \chi{\tau_k})-\lambda(\upwc \chi{\tau_k})}{{\tau_k}}
&=\frac1{\tau_k} \int_0^1 \frac{\dd}{\dd r} f(r) \dd r  && \text{with } f(r) = \lambda(\upwc \chi{\tau_k}+r(\pwc \chi{\tau_k}{-}\upwc \chi{\tau_k}))
\\
& = \int_0^1 \lambda'(\upwc \chi{\tau_k}+r(\pwc \chi{\tau_k}{-}\upwc \chi{\tau_k})) \frac1{\tau_k}  (\pwc \chi{\tau_k}{-}\upwc \chi{\tau_k}) \dd r &&
\\
& \longrightarrow \int_0^1 \lambda'(\chi) \chi_t \dd r =  \lambda'(\chi) \chi_t  && \text{strongly  in } L^2(\GC{\times}(0,T)),
\end{aligned}
\end{equation}
thanks to the Lipschitz continuity of $\lambda'$ combined with  convergences \eqref{cv-chi-pwc-2} and \eqref{ciliegina-chi}. The latter convergence also allows us to pass to the limit in the first term on the right-hand side of \eqref{interp-surf-temp}; the limit passage in the second, third, and fourth terms follows by the same arguments leading to \eqref{minute-convergences}. Finally, we observe that
\[
\pwc\ell{\tau_k}\to \ell
\quad \text{in $L^2(0,T;H^1(\GC)^*)$}.
\]
All in all, we deduce that the triple $(\teta, \tetas, \chi)$ fulfills the  weak formulation
\RCORR \eqref{form_debole_thetas-chip} \EEE of the surface heat equation,  with test functions $w\in H^1(\GC)$,  a.e.\ in $(0,T)$.
\par
Finally, \eqref{total-enbal-rho-cont} follows from testing the weak formulation \RCORR  \eqref{form_debole_theta-chip} \EEE of the bulk heat equation by $1$, the weak momentum balance \eqref{weak-U-rho} by $\uu_t$, the weak surface heat equation  \RCORR \eqref{form_debole_thetas-chip} \EEE by $1$, \RCNEW  the flow rule \eqref{ptw-flow-rule} by $\chi_t$,  \EEE adding the resulting relations, and integrating \RCORR them \EEE  over the generic interval $[s,t]\subset [0,T]$.
This concludes the proof.
\end{proof}

\section{Proof of Theorem \ref{thm:1}}
\label{s:7}
In order to prove Theorem \ref{thm:1}, we will perform a double limit passage in
 system \eqref{PDE-regul} (more precisely, in its weak formulation that was specified in Theorem \ref{thm:exist-regul}). We shall first pass to the limit as $\rho \down 0$, with
 the parameter
 $\varsigma>0$ fixed, and then as $\varsigma \down 0$.
Let us thus consider a family $(\teta_{\rho,\varsigma},\uu_{\rho,\varsigma},\zzeta_{\rho,\varsigma}, \tetasrhos,\chi_{\rho,\varsigma},  \uprs,  \xi_{\rho,\varsigma})_{\rho,\varsigma}$,
with
\[
\zzetars: = \etar(\uu_{\rho,\varsigma} {\cdot}\mathbf{n})\mathbf{n},\qquad
\xirs=\betar(\chi_{\rho,\varsigma}),
\]
of weak solutions to the Cauchy problem for the approximate   system
 \eqref{PDE-regul}; the first result of this Section collects all the a priori estimates, uniform w.r.t.\ $\rho$
 \emph{and} $\varsigma$,
  on which our compactness arguments shall rely. As we will see, these estimates can be obtained by replicating
 the formal estimates carried out in Section \ref{s:3-calculations} on the level of system  \eqref{PDE-regul}.
 \begin{proposition}
There holds for every $\rho,\, \varsigma>0$
 \begin{equation}
 \label{uniform-strict-positivity}
\tetars \geq \frac1{S_0} >0 \quad \aein\, \Omega\times (0,T), \qquad \tetasrhos \geq \frac1{S_0} >0 \quad \aein\, \GC\times (0,T),
\end{equation}
with $S_0$ from \eqref{che-fatica-pos}. Furthermore,
  exists a constant $\overline{S}>0$ such that the following estimates hold
 for all $\rho,\,\varsigma>0$: 
\begin{subequations}
\label{estimates-unif-rho}
\begin{align}
&
\label{e:teta-rho}
\| \tetars \|_{L^{2}(0,T;H^1(\Omega)) \cap L^{\infty}(0,T;L^1(\Omega)) \cap W^{1,1}(0,T; W^{1,3+\epsilon}(\Omega)^*)} \leq \overline{S},  
\\
&
\label{e:teta-nu}
\| (\tetars)^{(\mu+\nu)/2} \|_{\RCORR L^{2}(0,T;H^1(\Omega))} \leq \overline{S},  
\\
&
\label{e:tetas-rho}
 \| \tetasrhos \|_{L^{2}(0,T;H^1(\GC)) \cap L^{\infty}(0,T;L^1(\GC)) \cap W^{1,1}(0,T; W^{1,2+\epsilon}(\GC)^*)}
\leq \overline{S}, 
\\
&
\label{e:tetas-nu}
\| (\tetasrhos)^{(\mu+\nu)/2} \|_{\RCORR L^{2}(0,T;H^1(\GC))} \leq \overline{S},  
\\
&
\label{e:u-rho}
\| \uurs  \|_{H^{1}(0,T;H_{\GD}^1(\Omega;\R^3))} + \rho^{1/\omega} \| \eps(\partial_t \uurs)\|_{L^{\omega}(\Omega{\times}(0,T);\R^{3\times 3})} \leq \overline{S}, &
\\
&
\label{e:zeta-rho}
\|\zzetars \|_{L^{\omega/(\omega-1)}(0,T;W^{1,\omega}(\Omega;\R^3)^*)}  \leq \overline{S}, 
\\
&
\label{e:chi-rho}
 \| \chirs \|_{L^{\infty}(0,T;\VC) \cap \RCORR H^1(0,T;L^2(\GC))} +   \rho^{1/\omega} \| \partial_t \chirs \|_{L^{\omega}(\Omega{\times}(0,T))} \leq \overline{S}, 
 \\
 &
 \label{e:xi-rho}
 \| A\chirs + \xirs \|_{L^{\omega/(\omega-1)}(\GC{\times}(0,T))} \leq \overline{S},
 \\
 & \label{e:ups-rho}
 \RCNEW \| \uprs\|_{L^\infty(\GC{\times}(0,T))} \leq \overline{S}. \EEE
 \end{align}
\end{subequations}
\end{proposition}
\begin{proof}
The positivity property \eqref{uniform-strict-positivity} clearly follows from estimate \eqref{che-fatica-pos}.
The bounds for $(\tetars)_{\rho,\varsigma}$ in $L^\infty(0,T;L^1(\Omega))$, for $(\tetasrhos)_{\rho,\varsigma}$ in
$L^\infty(0,T;L^1(\GC)$, for $(\rho^{1/\omega}  \eps(\partial_t \uurs ))_{\rho,\varsigma}$ in
$L^{\omega}(\Omega{\times}(0,T);\R^{3\times 3})$,  for $(\chirs)_{\rho,\varsigma}$ in $L^\infty(0,T;H^1(\GC))$, 
and for $(\rho^{1/\omega} \partial_t \chirs)_{\rho,\varsigma}$ in
$L^{\omega}(\GC{\times}(0,T))$ follow from the total energy balance \eqref{total-enbal-rho-cont}, arguing in the very same way as in Section \ref{sss:3.3.1}. \RCNEW Estimate \eqref{e:ups-rho} simply follows from the fact that $\uprs \in \partial\varphi(\chirs)$ a.e.\ in  $\GC{\times}(0,T)$. 
\par
We then proceed to the Second a priori estimate (cf.\ Sec.\ \ref{sss:3.3.2}) and test the weak formulations \eqref{form_debole_theta} and \eqref{form_debole_thetas} of the heat equations
\begin{compactenum}
\item
by
$\tetars^{\nu-1}$ and $\tetasrhos^{\nu-1}$, with $\nu = 2-\mu$, in the case $\mu \in (1,2)$;
\item by
$-\tetars^{-1}$ and $-\tetasrhos^{-1}$ in the case $\mu =2$;
\item by $-\tetars^{-q}$ and $-\tetasrhos^{-q}$, with
$q=\mu-1$, in the case $\mu>2$.
\end{compactenum}
Observe that in all of the above cases
the test functions are admissible (namely, they belong to $H^1(\Omega)$ and $H^1(\GC)$, respectively), thanks to \eqref{uniform-strict-positivity}, combined with the fact that $\tetars \in L^2(0,T;H^1(\Omega))$ and
$\tetasrhos \in L^2(0,T;H^1(\GC))$, respectively.
We then add the resulting relations, integrate in time, and perform, in the three cases $\mu \in (1,2)$, $\mu=2$,
and $\mu>2$, the very same calculations as in Section \ref{sss:3.3.2}. In this way, we conclude
that $\| \tetars \|_{L^{2}(0,T;H^1(\Omega))} \leq C$ and $\| \tetasrhos \|_{L^{2}(0,T;H^1(\GC))}\leq C$.
These estimates are enhanced  to \eqref{e:teta-nu} and \eqref{e:tetas-nu}  by repeating the calculations from Section \ref{sss:3.3.3}.
\par
In order to replicate the  Fourth a priori estimate \RCORR  from Section \ref{sss:3.3.4}, \EEE  we   subtract from the total energy balance
 \eqref{total-enbal-rho-cont} the bulk and surface heat equations tested by $1$ and integrated in time. This leads to  the analogue of the \emph{mechanical energy inequality} \eqref{mechanical-energy-bala}, additionally featuring the integrals
 $\rho \int_0^t \RCORR \int_\Omega \EEE  |\eps(\partial_t\uurs)|^\omega \dd x \dd r$ and  $\rho \int_0^t  \RCOMMN
 \int_{\GC} \EEE   |\partial_t\chirs|^\omega \dd x \dd r$ on the left-hand side. Repeating the very same calculations as in Sec.\ \ref{sss:3.3.4}, we conclude the estimates for $\| \uurs  \|_{H^{1}(0,T;H_{\GD}^1(\Omega;\R^3))} $ and $\|\chirs \|_{H^1(0,T;L^2(\GC))}$.
 \par
 Relying on the Sixth a priori estimate  \RCORR (cf.\ Sec.\ \ref{sss:3.3.6}), \EEE
 which yields the bounds \eqref{genialata-gc} and \eqref{genialata-gc-2} for the sequences $(\tetars)_{\rho,\varsigma}$
 and $(\tetasrhos)_{\rho,\varsigma}$,
 we are in a position to rigorously render the calculations for Seventh a priori estimate, cf.\ Sec.\ \ref{sss:3.3.7}.
 Thus, we deduce the bounds for $(\partial_t \tetars)_{\rho,\varsigma} \subset L^1(0,T; W^{1,3+\epsilon}(\Omega)^*)$ and
$(\partial_t \tetasrhos)_{\rho,\varsigma} \subset L^1(0,T; W^{1,2+\epsilon}(\GC)^*)$  for every $\epsilon>0$.
\par
\RCNEW Finally, estimates \eqref{e:zeta-rho} and \eqref{e:xi-rho} follow from a comparison in the momentum balance equation and in the flow rule for the adhesion parameter. \EEE
\end{proof}
\par
We shall now prove \underline{\textbf{Theorem \ref{thm:1}}} in two main steps,  carried out in the ensuing Sections \ref{ss:6.1} and \ref{ss:6.2}. More precisely,
\begin{enumerate}
\item
in Sec.\  \ref{ss:6.1} we  will pass to the limit in system  \eqref{PDE-regul} as $\rho\down 0$ and $\varsigma>0$ is kept fixed; in this way, we shall prove the existence of \emph{weak energy} solutions (in the sense of
Definition \ref{def:weak-sol}) of system   \eqref{PDE-regul},  in which $\rho$ is set equal to $0$;
\item in Sec.\  \ref{ss:6.2} we will finally perform the limit passage as $\varsigma \down 0$, thus concluding the proof of   Thm.\ \ref{thm:1}.
\end{enumerate}
\subsection{Limit passage as $\rho \down 0$, for fixed $\varsigma>0$}
\label{ss:6.1}
Since the parameter $\varsigma>0$ is kept fixed, we shall not highlight the dependence
on $\varsigma$ of the  solutions  
to system \eqref{PDE-regul} and just denote them
by $(\teta_\rho,\uu_\rho,\teta_{\mathrm{s},\rho},\chi_\rho, \RRRN \sigma_\rho\EEE)$.
\par
 Let $(\rho_j)_j \subset (0,+\infty)$  be a null sequence and, correspondingly, let
 $(\tetaj,\uuj,\tetasj,\chij,  \upj)_j$
 be a sequence of solutions to system \eqref{PDE-regul}, formulated as in the statement of Thm.\ \ref{thm:exist-regul} and supplemented by the initial conditions \eqref{initial-rho}, with sequences $(\tetaj^0)_j$, $(\tetasj^0)_j$   and $(\uuj^0)_j$ 
 of initial data fulfilling \eqref{approx-initial-data}; set $\zzetaj: = \RCORR \eta_{\varsigma}(\uuj \cdot \mathbf{n} ) \EEE \mathbf{n}$, $\xij: = \beta_{\varsigma}(\chij)$, \RCNEW $\upj: = \RCNEW \sigma_{\rho_j,\varsigma}$. \EEE
 In what follows we will show that, up to a subsequence, the  quintuples  $(\tetaj,\uuj,\tetasj,\chij, \upj)_j $ 
 converge to a `weak energy solution' $(\teta,\uu,\tetas,\chi,\RCNEW \sigma)$
   \EEE  to the Cauchy problem for system
 \eqref{PDE-regul}, in which $\rho=0$. Namely, we will prove that  $(\teta,\uu,\tetas,\chi)$
\begin{itemize}
\item[-]
 enjoy the integrability and regularity properties \eqref{reg-entropic}, and the positivity property \eqref{teta-strict-pos};
\item[-]
 fulfill the weak formulations  \RCNEW \eqref{form_debole_theta-chip} and \eqref{form_debole_thetas-chip}
 \EEE
  of the bulk and surface heat equations, \RCORR with test functions $v\in W^{1,3+\epsilon}(\Omega)$
  and $w\in W^{1,2+\epsilon}(\GC)$, respectively,  for every $\epsilon>0$; \EEE
 \item[-]
 fulfill the weak formulation of the displacement equation \eqref{weak-displ},  with $\zzeta \in \mathrm{C}^0([0,T]; L^4(\GC;\R^3))$ given by $\zzeta = \eta_{\varsigma}(\uu \cdot \mathbf{n})  \mathbf{n}$;
 \item[-]
 fulfill the pointwise formulation
 \eqref{ptw-flow-rule}
of the flow rule
in which $\rho$ is set equal to $0$.
\end{itemize}
We shall
 split the argument into some steps.
\paragraph{\bf Step $1.0$: compactness.}  There exist a (not relabeled) subsequence and a \RCNEW quintuple $(\teta,\uu,\tetas,\chi,\RCNEW \sigma)$, \EEE  with
\begin{equation}
\label{regularity-quadruple}
\begin{aligned}
&
\teta \in L^2(0,T;H^1(\Omega)) \cap L^\infty(0,T;L^1(\Omega)),  && \uu \in    H^1(0,T;\bsVD),
\\
&
\tetas \in
L^2(0,T;H^1(\GC)) \cap L^\infty(0,T;L^1(\GC)),  &&
\chi \in L^\infty(0,T;H^1(\GC)) {\cap}H^1(0,T;L^2(\GC)),
\\ &   \RCNEW\sigma \in L^\infty(\GC{\times}(0,T)),  && \EEE
\end{aligned}
\end{equation}
 such that the following weak and strong convergences hold as $j\to\infty$:
\begin{subequations}
\label{convs-tetaj}
\begin{align}
\label{w-teta-j}
&
\tetaj \weaksto \teta && \text{in } L^2(0,T;H^1(\Omega)) \cap L^\infty (0,T;W^{1,3+\epsilon}(\Omega)^*) &&  \text{for all } \epsilon>0,
\\
& \label{w-teta-j-ptw}
\tetaj(t) \weakto \teta(t) && \text{in } H^1(\Omega) && \foraa\, t \in (0,T),
\\
&
\label{s-teta-j}
\tetaj \to \teta && \text{in } L^2(0,T;L^p(\Omega)) \cap L^q(0,T;L^1(\Omega)) && \text{for all } p \in [1,6) \text{ and } q \in [1,\infty),
\\
\label{w-tetas-j}
&
\tetasj \weaksto \tetas && \text{in } L^2(0,T;H^1(\GC)) \cap L^\infty (0,T;W^{1,2+\epsilon}(\GC)^*) &&  \text{for all } \epsilon>0,
\\
& \label{w-tetas-j-ptw}
\tetasj(t) \weakto \tetas(t) && \text{in } H^1(\GC) && \foraa\, t \in (0,T),
\\
&
\label{s-tetas-j}
\tetasj \to \tetas && \text{in } L^2(0,T;L^q(\GC)) \cap L^q(0,T;L^1(\Omega)) && \text{for all } q \in [1,\infty),
\\
\label{w-uu-j}
&
\uuj \weakto \uu && \text{in } H^1(0,T;\bsVD) &&
\\
\label{s-uu-0}
&
\uuj \to \uu && \text{in }  \mathrm{C}^0([0,T];H^{1-\epsilon}(\Omega;\R^3)) && \text{for all } \epsilon \in (0,1) 
\\
\label{s-uu-j}
&
\rho_j^{\upsilon} \uuj \to 0 && \text{in } W^{1,\omega}(0,T; W^{1,\omega}(\Omega;\R^3)) && \text{for all } \upsilon >\frac1{\omega},
\\
&
\label{w-chi-j}
\chij \weaksto \chi && \text{in } L^\infty(0,T;H^1(\GC)) \cap H^1(0,T;L^2(\GC)),
\\
&
\label{s-chi-j}
\chij \to \chi && \text{in } \mathrm{C}^0(0,T;L^q(\GC))&& \text{for all } q \in [1,\infty),
\\
&
\label{rho-chi-j}
\rho_j^{\upsilon} \chij \to 0 && \text{in } W^{1,\omega}(0,T;L^\omega(\GC))&& \text{for all } \upsilon>\frac1\omega,
\\
& \label{upsilon-j}
\upj \weaksto \sigma && \text{in } L^\infty(\GC{\times}(0,T))\,. &&
\end{align}
\end{subequations}
Indeed, convergences \eqref{w-teta-j}, \eqref{w-tetas-j}, \eqref{w-uu-j},  \eqref{w-chi-j},  \RCNEW
and \eqref{upsilon-j} \EEE
 immediately follow from estimates
 \eqref{estimates-unif-rho} via weak compactness arguments.
Convergence
\eqref{s-uu-j} is a straightforward consequence of the second of \eqref{e:u-rho} also in view of Korn's inequality.
 Analogously, \eqref{rho-chi-j} follows from estimate \eqref{e:chi-rho}.
  Arguing as in the proof of Lemma \ref{l:a-prio-tau-compact} and resorting to the aforementioned results
  from \cite{Simon87} we  deduce the strong convergences \eqref{s-teta-j}, \eqref{s-tetas-j},
   \eqref{s-uu-0},
   and \eqref{s-chi-j}.
  Likewise, the pointwise convergences \eqref{w-teta-j-ptw} and \eqref{w-tetas-j-ptw} ensue from combining estimates  \eqref{estimates-unif-rho}  with Theorem \ref{th:mie-theil}.
  \par
  Combining the  estimates for $(\tetaj)_j$ and $(\tetasj)_j$ in $L^\infty (0,T;L^1(\Omega))$ and $L^\infty(0,T;L^1(\GC))$ with the pointwise convergences \eqref{w-teta-j-ptw} and \eqref{w-tetas-j-ptw}
  we immediately deduce that $\teta \in L^\infty (0,T;L^1(\Omega))$ and $\tetas \in L^\infty(0,T;L^1(\GC))$.
 \par
 Clearly, from the strong convergences \eqref{s-teta-j} and \eqref{s-tetas-j} and the strict positivity properties \eqref{uniform-strict-positivity} we conclude that the limiting temperatures
 $\teta$ and $\tetas$ also fulfill
  \begin{equation}
 \label{uniform-strict-positivity-bis}
\teta \geq \frac1{S_0} >0 \quad \aein\, \Omega\times (0,T), \qquad \tetas \geq \frac1{S_0} >0 \quad \aein\, \GC\times (0,T).
\end{equation}
  \par
  Furthermore, from convergences
  \eqref{s-uu-0} (which in particular yields, for the traces $(\uuj)_j$,
  the convergence $\uu_j \to \uu $ in $\mathrm{C}^0([0,T];L^q(\GC;\R^3))$ for every $1\leq q<4$ since
  $\bsVD \Subset L^q(\GC;\R^3)$ in the sense of traces),
   and \eqref{s-chi-j} we derive, taking into account that the functions $\eta_{\varsigma} :\R \to \R$
  and $\beta_{\varsigma}:\R \to \R$ are Lipschitz continuous, that
  \begin{subequations}
  \label{convs-max-mon}
  \begin{align}
  &
\label{w-zeta-j}
\zzetaj \to \zzeta := \eta_\varsigma (\uu \cdot \mathbf{n})\mathbf{n} && \text{in } \
\mathrm{C}^0([0,T];L^q(\GC;\R^3)), &&   \text{ for all } q \in [1,4),
\\
\label{weak-xi-j}
& \xij  \to \xi : = \beta_{\varsigma}(\chi) && \text{in }  \RCORR \mathrm{C}^{0}([0,T] ;L^q(\GC)) \EEE && \text{ for all } q \in [1,\infty).
\end{align}
\end{subequations}
  \paragraph{\bf Step $1.1$: limit passage in the momentum balance.}
  We integrate \RCORR the weak formulation \eqref{weak-U-rho}   of the momentum balance \EEE
  over an arbitrary time interval $(s,t)$ and pass to the limit as $j\to\infty$ in   \eqref{weak-U-rho}. We handle the first, second, third, fourth  and sixth integrals on the left-hand side by resorting to convergences \eqref{w-teta-j},  \eqref{w-uu-j}, \eqref{s-uu-j},  and \eqref{w-zeta-j}.
  For the remaining terms, we use that
  \begin{equation}
  \label{remaing-terms}
  \begin{cases}
  \chijp \uuj \to \chip \uu,
  \\
   \chijp\uuj \nlocss{\chijp} \to \chip \uu \nlocss{\chip}
   \end{cases}
    \text{ in } \mathrm{C}^0([0,T]; L^q(\GC))  \quad \text{for all } 1 \leq q<4
  \end{equation}
  which follow from the strong convergences \eqref{s-uu-0} and \eqref{s-chi-j}, also by Lemma \ref{lemmaK}.
  All in all, we conclude that the quadruple $(\teta,\uu,\chi,\zzeta)$, with $\zzeta =  \eta_\varsigma (\uu \cdot \mathbf{n})\mathbf{n}$ fulfill
  \begin{equation}
  \label{temp-momen-balance}
\begin{aligned}
&
\int_s^t \left( \vfm(\uu_t,\mathbf{v})
+ \efm(\uu,\mathbf{v})
+\int_{\Omega} \teta \mathrm{div}(\vv) \dd x +
 \int_{\GC} \chip\uu \mathbf{v} \dd x \right) \dd r
 \\
  & \qquad +
\int_s^t \left( \int_{\GC}\zzeta \cdot \mathbf{v} \dd x
+  \int_{\GC} \chip\uu \nlocss{\chip} \mathbf{v} \dd x \right) \dd r  =
\int_s^t  \pairing{}{\bsVD}{\mathbf{F}}{\vv} \dd r
\end{aligned}
\end{equation}
for all $\vv \in W^{1,\omega}(\Omega;\R^3) \cap \in \bsVD$,
which translates into a relation holding at almost all $t\in (0,T)$ by the arbitrariness of the interval $(s,t)$. Furthermore, taking into account the integrability properties of $\uu$, $\teta$ and $\chi$, it is immediate to see that \eqref{temp-momen-balance} extends to all test functions $\vv \in \bsVD$. Therefore, we have proved  \eqref{weak-displ}
 (where $\chi$ is replaced by $(\chi)^+$). 
\par
Lastly, in view
of the limit passage in the bulk heat equation,
let us improve the weak convergence $\eps(\partial_t \uuj)\weakto \eps(\uu_t)$
in $L^2(0,T;L^2(\Omega;\R^{3\times 3}))$ to
  a strong one. To this end,
  we revert to  \eqref{weak-U-rho}, test it by $\partial_t \uuj$ and integrate it in time.
  Passing to the limit as $j\to\infty$ we find that
  \[
  \begin{aligned}
  &
  \limsup_{j\to\infty}\left(  \int_s^t  \vfm(\partial_t \uuj,\partial_t \uuj) \dd r +  \rho_j\int_s^t \int_\Omega |\eps(\partial_t \uuj)|^{\omega} \dd x \dd r \right)      \\ &
  \leq -
  \liminf_{j\to\infty} \int_s^t  \efm(\uuj,\partial_t \uuj) \dd r
   -\liminf_{j\to\infty} \int_s^t \int_{\Omega} \tetaj \mathrm{div}(\partial_t \uuj) \dd x \dd r
  \\
  & \qquad
   -\liminf_{j\to\infty} \int_s^t \int_{\GC} \chijp \uuj \cdot \partial_t \uuj  \dd x \dd r
  - \liminf_{j\to\infty} \int_s^t\int_{\GC} \zzetaj \cdot \partial_t \uuj \dd x \dd r
    \\
  &
  \qquad
  -  \liminf_{j\to\infty} \int_s^t\int_{\GC} \chijp \nlocss{\chijp}\, \uuj \cdot \partial_t \uuj  \dd x \dd r
  + \lim_{j\to\infty} \int_s^t  \RCORR  \pairing{}{\bsVD}{\mathbf{F}}{\partial_t\uuj} \EEE  \dd r
  \\
  &
\RCORR  \stackrel{(1)}{\leq}  \EEE -  \int_s^t \left(
 \efm(\uu,\mathbf{u}_t)
+\int_{\Omega} \teta \mathrm{div}(\uu_t) \dd x +
 \int_{\GC} \chip\uu \mathbf{u}_t \dd x \right) \dd r
 \\
  & \qquad -
\int_s^t \left(
\int_{\GC} \zzeta \cdot \uu_t \dd x
  \int_{\GC} \chip\uu \nlocss{\chip}\, \mathbf{u}_t \dd x +
\RCORR \pairing{}{\bsVD}{\mathbf{F}}{\uu_t} \EEE \right) \dd r\stackrel{\eqref{temp-momen-balance}}
= \int_s^t    \vfm(\uu_t,\uu_t)  \dd x \dd r\,.
  \end{aligned}
  \]
  For (1), we have used that
\[
\begin{aligned}
  \liminf_{j\to\infty} \int_s^t  \efm(\uuj,\partial_t \uuj) \dd r  & = \liminf_{j\to\infty} \left(  \frac12 \efm(\uuj(t), \uuj(t))  -  \frac12 \efm(\uuj(s),\uuj(s))\right)
  \\
  &  \geq \frac12 \efm(\uu(t),\uu(t)) - \frac12 \efm(\uu(s),\uu(s))
   =  \int_s^t  \efm(\uu,\mathbf{u}_t) \dd r
  \end{aligned}
\]
thanks to \eqref{s-uu-0}, as well as  the strong convergences \eqref{s-teta-j} \RCORR and
\eqref{w-zeta-j}. \EEE
All in all, we conclude that
\[
 \int_s^t  \vfm(\partial_t \uuj,\partial_t \uuj) \dd r  \to  \int_s^t    \vfm(\uu_t,\uu_t)  \dd x \dd r,
\]
which, combined with \eqref{s-uu-0},  immediately gives the desired strong convergence
\begin{equation}
\label{s-uu-j-bis}
\uuj \to \uu \qquad \text{in } H^1(0,T; \bsVD) \qquad \text{as } j \to \infty.
\end{equation}
  \paragraph{\bf Step $1.2$: limit passage in the flow rule.}
We take the limit  as $j\to\infty$ of \eqref{ptw-flow-rule} integrated on an arbitrary time interval $(s,t)\subset (0,T)$. For the left-hand side we use convergences \eqref{w-chi-j}, \eqref{rho-chi-j}, \eqref{s-chi-j}
(which also yields  strong convergences for the terms $\gamma'(\chij)$ and $\lambda'(\chij)$ by the Lipschitz continuity of $\gamma'$ and $\lambda$), \eqref{s-tetas-j}, and \eqref{weak-xi-j}.
We also use that, in view of  estimate  \eqref{e:xi-rho} and  the previously observed
\eqref{weak-xi-j},  there holds
\begin{equation}
\label{weak-Achi}
A\chij\weakto A\chi \text{ in } L^{\omega/(\omega-1)} (\GC{\times}(0,T)) \qquad \text{as } j\to\infty.
\end{equation}
 As for the right-hand side,
we use that
\[
\begin{cases}
-\frac12 |\uuj|^2 \upj  \RCNEW \weaksto -\frac12 |\uu|^2 \sigma & \text{ in } L^\infty (0,T; L^q(\GC)) \text{ for all } 1 \leq q<2,
\\
-\frac12 \nlocss{\chijp}\,  |\uuj|^2 \upj  \RCNEW \weaksto -\frac12 \nlocss{\chip}\, |\uu|^2 \sigma & \text{ in } L^\infty
(0,T; L^q(\GC)) \text{ for all } 1 \leq q<2,
\\
-\frac12 \nlocss{\chijp |\uuj|^2} \upj \RCNEW \weaksto - \frac12 \nlocss{\chip |\uu|^2}\, \sigma &
\text{ in  } L^\infty(\GC {\times}(0,T))
\end{cases}
\qquad
\]
\EEE
also in view of Lemma \ref{lemmaK}.
All in all,
we conclude the validity of \eqref{ptw-flow-rule} with $\rho=0$.
\RCNEW Again by the strong weak closedness of the graph of the (operator induced by) $\partial\varphi$, we
have
$\sigma\in \partial\varphi(\chi)$ a.e.\ in $\GC\times (0,T)$. \EEE
 A comparison argument in \eqref{ptw-flow-rule}  immediately yields that
$A\chi \in L^2(\GC{\times}(0,T))$, so that we ultimately infer that $\chi \in L^2(0,T;H^2(\GC))$.
\par
Lastly,  in view of the limit passage in the surface heat equation, let us enhance the weak convergence of $\partial_t \chij$ to a strong one. \RCNEW With this aim, we test \eqref{weak-U-rho} by
$\partial_t \uuj$,
\eqref{ptw-flow-rule}
by $\partial_t\chij$, add the resulting relations, and integrate in time (cf.\ \eqref{mechanical-energy-bala}).
Taking  the limit as $j\to\infty$
 we have
\begin{equation}
\label{611}
\begin{aligned}
&
\limsup_{j\to\infty}
\left(  \int_0^t  \int_{\GC} |\partial_t\chij|^2 \dd x \dd r+
   \rho_j\int_0^t \int_{\GC} |\partial_t \chij|^{\omega} \dd x \dd r \right)
   \\
   &
  \leq
  -\lim_{j\to\infty}
  \Big(
   \int_0^t \left(  \vfm(\partial_t\uuj,\partial_t\uuj) {+} \int_\Omega \tetaj \mathrm{div}(\partial_t \uuj) \dd x \right) \dd r   + \frac12\efm(\uuj(t),\uuj(t)) +\int_{\GC} \widehat{\eta}_{\varsigma} (\uuj(t){\cdot} \mathbf{n}) \dd x
\\
& \quad \quad  \quad \quad\quad
    +\frac12\int_{\GC} \chijpt \RCNEW |\uuj(t)|^2 \dd x + \frac12 \int_{\GC}\chijpt \RCNEW |\uuj(t)|^2  \nlocss{\chijpt}  \RCNEW \dd x  \Big)
 \\
&
\quad
 +
 \lim_{j\to\infty} \int_0^t   \pairing{}{\bsVD}{\mathbf{F}}{\partial_t\uuj}   \dd r
  -\liminf_{j\to\infty}
 \left( \int_{\GC} \left( \frac12 |\nabla \chij(t)|^2 + \widehat{\beta}_\varsigma (\chij(t)) + \gamma (\chij(t)) \right) \dd x \right)
 \\
 &
\quad
+\lim_{j\to\infty} \Big(
\frac12\efm(\uu_0^{\rho_j},\uu_0^{\rho_j}) +  \int_{\GC} \widehat{\eta}_{\varsigma}(\uu_0^{\rho_j}{\cdot}\mathbf{n}) \dd x   +\frac12\int_{\GC} \RCNEW (\chi_0)^+ |\uu_0^{\rho_j}|^2 \dd x + \frac12 \int_{\GC} (\chi_0)^+ |\uu_0^{\rho_j}|^2  \nlocss{(\chi_0)^+}\, \dd x  %
\\
& \quad \quad  \quad \quad\quad  +
\RCNEW  \int_{\GC} \big( \frac12 |\nabla \chi_0|^2 + \widehat{\beta}_\varsigma (\chi_0) +\gamma(\chi_0) \big) \dd x \Big)
- \liminf_{j\to\infty} \int_0^t \int_{\GC}\lambda'(\chij) \tetasj \partial_t\chij \dd x \dd  r
\\
&
 \stackrel{(1)}{\leq}
-   \int_0^t \left(  \vfm(\uu_t,\uu_t) {+}\int_\Omega \teta \mathrm{div}(\uu_t) \dd x  \right) \dd r  + \frac12\efm(\uu(t),\uu(t))
+\int_{\GC} \widehat{\eta}_{\varsigma} (\uu(t){\cdot} \mathbf{n}) \dd x
\\
& \quad   +\frac12\int_{\GC} \chipt \RCNEW |\uu(t)|^2 \dd x + \frac12 \int_{\GC}\chipt \RCNEW |\uu(t)|^2  \nlocss{\chipt}  \RCNEW \dd x   +\int_0^t   \pairing{}{\bsVD}{\mathbf{F}}{\uu_t}   \dd r
 \\
&
\quad
- \int_{\GC} \left( \frac12 |\nabla \chi(t)|^2 + \widehat{\beta}_\varsigma (\chi(t)) + \gamma (\chi(t)) \right) \dd x
+
\frac12\efm(\uu_0,\uu_0) +\int_{\GC} \widehat{\eta}_\varsigma(\uu_0 {\cdot}\mathbf{n}) \dd x
 \\
 &
\quad
+\frac12\int_{\GC} \RCNEW (\chi_0)^+ |\uu_0|^2 \dd x + \frac12 \int_{\GC} (\chi_0)^+ |\uu_0|^2  \nlocss{(\chi_0)^+}\dd x
+ \int_{\GC} \big( \frac12 |\nabla \chi_0|^2 + \widehat{\beta}_\varsigma (\chi_0) +\gamma(\chi_0) \big) \dd x \Big)
\\
& \quad
- \int_0^t \int_{\GC}\lambda'(\chi) \tetas \partial_t\chi \dd x \dd  r
\stackrel{(2)}{=} \int_0^t  \int_{\GC} | \chi_t|^2 \dd x \dd r
 \end{aligned}
\end{equation}
%
%
%
%
  where for (1)  we have used the previously found convergences properties, while  (2)  follows from testing
  the weak momentum balance  \eqref{temp-momen-balance} 
  by $\uu_t$, the flow rule \eqref{ptw-flow-rule}  by $\chi_t$,  adding the resulting relations and integrating in time.
  \EEE
All in all, from the  above chain of inequalities we infer
  \begin{equation}
\label{s-chi-j-bis}
\chij \to \chi \qquad \text{in } H^1(0,T; L^2(\GC)) \qquad \text{as } j \to \infty.
\end{equation}
  \paragraph{\bf Step $1.3$: limit passage in the bulk heat equation.}
  We shall pass to the limit in \RCNEW \eqref{form_debole_theta-chip}  \EEE with test functions $v\in W^{1,3+\epsilon}(\Omega)$, for  an arbitrary $\epsilon>0$. 
In analogy with \eqref{stim6-mic}, we  rewrite the bulk heat  equation by  grouping  its terms  in the following way:
\begin{equation}
\label{later-comparison-j}
\partial_t \tetaj(t) = \mathcal{L}_{1,j}(t)  - \mathcal{A}(\tetaj (t))  +  \mathcal{L}_{2,j}(t)  \text{ in } W^{1,3+\epsilon}(\Omega)^* \qquad
\foraa\, t \in (0,T),
\end{equation}
with (omitting the $t$-dependence of the operators below to simplify notation)
\[
\begin{cases}
 \mathcal{L}_{1,j} := \tetaj \mathrm{div}(\partial_t \uuj) +  \eps(\partial_t \uuj) \mathbb{V} \eps(\partial_t \uuj) + h \in L^1(\Omega)
 \\
 \mathcal{L}_{2,j} \in   W^{1,3+\epsilon}(\Omega)^* \text{ defined by }
 \\
 \qquad
  \pairing{}{W^{1,3+\epsilon}(\Omega)}{ \mathcal{L}_{2,j} }{v}: =  \int_{\GC} \left(   \nlocss{\chijp \tetasj}\, \chijp \tetaj  - k(\chij) \tetaj (\tetaj - \tetasj) - \nlocss{\chijp} \,\chijp \tetaj^2 \right) v \dd x,
 \\
 \mathcal{A}(\tetaj) \in  W^{1,3+\epsilon}(\Omega)^* \text{ defined by }
 \\
 \qquad  \pairing{}{W^{1,3+\epsilon}(\Omega)}{\mathcal{A}(\tetaj)}{v} : = \int_{\Omega} \alpha(\tetaj)\nabla\tetaj \cdot \nabla v \dd x = \int_{\Omega} \nabla(\widehat{\alpha}(\tetaj)) \cdot \nabla v \dd x
\end{cases}
\]
with $\widehat\alpha$ from \eqref{def-hat-alpha}.
\par
Now, it follows from
\eqref{s-uu-j-bis} and  \eqref{s-teta-j} 
that
\begin{equation}
\label{L1j}
 \mathcal{L}_{1,j} \to  \mathcal{L}_{1} \text{ in } L^1(0,T;L^1(\Omega)) \quad \text{with }
  \mathcal{L}_{1}(t) := \teta(t) \mathrm{div}(\uu_t(t)) +  \eps(\uu_t(t)) \mathbb{V} \eps(\uu_t(t)) + h(t)
\end{equation}
for a.a.\ $t\in (0,T)$.
\par
Let us now address the limit of the operators $ \mathcal{A}(\tetaj)$:
for this, we rely on the interpolation estimate \eqref{genialata-gc} 
implying, for all $\mu>1$ and $0<\nu<1$, that the sequence
 \begin{equation}
 \label{al-pelo}
 (\tetaj)_j \text{ is bounded in } L^{\mu-\nu+2} (0,T;L^{3(\mu-\nu+2)/(7-6\nu)}(\Omega)).
 \end{equation}
Choosing $\tfrac{2}{3}<\nu<1$, we have that $\mu-\nu+2>\mu+1$ and $\tfrac{3(\mu-\nu+2)}{7-6\nu}>\mu+1$ and hence, from \eqref{al-pelo}, we infer the estimate
 \begin{equation}
 \label{questo-ci-salva}
\sup_j \|\tetaj^{\mu+1}\|_{L^{1+\delta}(\Omega{\times}(0,T))} \leq C \qquad \text{for some } \delta>0.
 \end{equation}
Now, we use \eqref{questo-ci-salva} to settle the compactness properties of the sequence $(\widehat{\alpha}(\tetaj))_j$.
First of all, it follows from \eqref{w-teta-j-ptw} that
$\tetaj \to \teta$ a.e.\ in $\Omega\times (0,T)$ and hence
$\widehat{\alpha}(\tetaj) \to \widehat{\alpha}(\teta)$ a.e.\ in $\Omega\times (0,T)$.
Combining  this information with the fact that 
  $|\widehat{\alpha}(\tetaj) |\leq C(\tetaj^{\mu+1}+1)$ (cf.\ \eqref{growth-primitives}) and with estimate \eqref{questo-ci-salva},
\EEE
  we ultimately infer that
  \[
  \widehat{\alpha}(\tetaj)  \to \widehat{\alpha}  (\teta) \qquad \text{ in } L^1(\Omega{\times}(0,T)).
  \]
  Therefore, $\nabla   (\widehat{\alpha}(\tetaj) ) \to \nabla(\widehat{\alpha}(\teta)) $ in the sense of distributions on $\Omega\times (0,T)$.
  \par
Now, we need to improve the convergence properties of the sequence
  $(\nabla   (\widehat{\alpha}(\tetaj) ) )_j$.
  We again interpolate estimates \eqref{e:teta-rho} and \eqref{e:teta-nu}  and (cf.\ \eqref{genialata-gc}) deduce that
 \begin{equation}
\label{finer-teta-1}
 \sup_j \|\tetaj \|_{L^{r}(0,T;L^{s}(\Omega)}  \leq C \quad \text{for some } r > \mu-\nu+2 \text{ and } s <\frac{3(\mu-\nu+2)}{7-6\nu}\,.
\end{equation}
Therefore,
the \RCORR  sequence \EEE
  \[
  (\tetaj^{(\mu-\nu+2)/2})_j \text{ is bounded in } L^{2r/(\mu-\nu+2 )} (0,T;L^{2s/(\mu-\nu+2 )}(\Omega)).
  \]
In turn, mimicking the calculations from Sec.\ \ref{sss:3.3.7} we find that
  \begin{equation}
  \label{careful-holder}
  \begin{aligned}
  &
  \left| \pairing{}{W^{1,3+\epsilon(\Omega)}}{ \mathcal{A}(\tetaj)(t) }{v}\right|
  \\
  &  \leq
  C  \| \nabla \tetaj(t)\|_{L^2(\Omega)} \| \nabla v \|_{L^2(\Omega)} + C\| \tetaj^{(\mu-\nu+2)/2}(t) \|_{L^{2s/(\mu-\nu+2 )}(\Omega)}
\| \nabla (\tetaj^{(\mu+\nu)/2}(t) )\|_{L^2(\Omega)} \| \nabla v \|_{L^{3+\epsilon}(\Omega)}
\end{aligned}
\end{equation}
where we have applied H\"older's inequality, choosing  $\nu\in (\tfrac23,1)$  such that
\[
\frac{\mu-\nu+2}{2s} + \frac12 + \frac1{3+\epsilon} =1.
\]
Hence, from \eqref{finer-teta-1}  and \eqref{careful-holder}
we deduce that the sequence
\begin{equation}
\label{ci-serve-anche-questa}
(\mathcal{A}(\tetaj))_j \text{ is bounded in $L^{1+\delta}(0,T;W^{1,3+\epsilon}(\Omega)^*)$ for some $\delta>0.$}
\end{equation}
Therefore,
\begin{equation}
\label{convergences-A-tetaj}
\begin{aligned}
&
\mathcal{A}(\tetaj) \weakto \mathcal{A}(\teta) \quad \text{in } L^1(0,T; W^{1,3+\epsilon}(\Omega)^*), \qquad \text{with }
\\
&  \pairing{}{W^{1,3+\epsilon}(\Omega)}{\mathcal{A}(\teta(t))}{v}: = \int_\Omega \nabla (\widehat{\alpha}(\teta(t))) \cdot \nabla v \dd x  
\RCORR \quad \foraa\, t \in (0,T). \EEE
\end{aligned}
\end{equation}
\par
Finally, in order to take the limit of the operators $(\mathcal{L}_{2,j})_j$
 we need to refine the convergences available for the traces of $(\tetaj)_j$. Indeed, taking into account that
the sequence $(\tetaj^{(\mu+\nu)/2})_j $ is bounded in $L^2(0,T;H^1(\Omega))$ for every $\nu \in (0,1)$ and that, a fortiori, its traces are bounded in $L^2(0,T;L^4(\GC))$, we infer that
$(\tetaj)_j$ is bounded in $L^{\mu+\nu}(0,T;L^{2(\mu+\nu)}(\GC))$ for every $\nu\in (0,1)$. Since $\mu>1$,  we may choose
 $\nu\in (\tfrac23,1)$  such that $\mu +\nu>2$. Thus, from this estimate we improve the weak convergence of $(\tetaj)_j $ in $L^2(0,T;L^4(\GC))$ to a strong convergence, i.e.\
\begin{equation}
\label{e'-pure-forte}
\tetaj \to \teta \quad \text{in } L^2(0,T;L^4(\GC)).
\end{equation}
Therefore, in view of   \eqref{s-chi-j} we find that
$\chijp \tetaj \to \chip \teta$ in $L^2(0,T; L^q(\GC))$ for every $q\in [1,4)$.
\RCORR We now use \EEE
 that $\chijp\tetasj \to \chip \tetas$ in $L^2(0,T;L^s(\GC))$ for every $s \in [1,\infty)$ thanks to \eqref{s-tetas-j}  so that, by Lemma \eqref{lemmaK},
 $\nlocss{\chij\tetasj} \to \nlocss{\chi\tetas}$ in $L^2(0,T;L^\infty(\GC))$. \RCORR Hence  we have \EEE that, as $j\to\infty$,
\[
  \nlocss{\chijp \tetasj}\, \chijp \tetaj   \to \nlocss{\chip\tetas}\,\chip\teta \qquad \text{in } L^1(0,T; L^q(\GC)) \quad \text{ for all } q\in [1,4).
\]
In order to pass to the limit in the  second contribution to $\mathcal{L}_{2,j}$, we recall that $k(\chij)\to k(\chi)$ in $L^\infty(0,T;L^q(\GC))$ for every $q\in [1,\infty)$ by \eqref{s-chi-j}
and the polynomial growth of $k$. Hence,  in view of \eqref{e'-pure-forte} and \eqref{s-tetas-j} we have that
\[
 k(\chij) \tetaj (\tetaj - \tetasj)
 \to k(\chi)\teta(\teta-\tetas) \qquad \text{in } L^1(0,T; L^1(\GC)).
 \]
Analogously, we find that
\[
\nlocss{\chijp}  \chijp  \tetaj^2 \to \nlocss{\chip}\,  (\chi)^+  \teta^2  \qquad \text{in } L^1(0,T; L^1(\GC)).
\]
All in all, we have  that
\begin{equation}
\label{L2j}
\begin{aligned}
&
 \mathcal{L}_{2,j} \weakto  \mathcal{L}_{2} \qquad \text{ in } L^1(0,T;W^{1,3+\epsilon}(\Omega)^*) \qquad \text{with }
 \\
 &
\pairing{}{W^{1,3+\epsilon}(\Omega)}{\mathcal{L}_{2}(t)}{v} :=\int_{\GC} \left(   \nlocss{(\chi(t))^+ \tetas(t)}  (\chi(t))^+ \teta(t)  - k(\chi(t)) \teta(t) (\teta(t) - \tetas(t)) - \nlocss{ (\chi(t))^+} (\chi(t))^+\teta(t)^2 \right)  v \dd x
\end{aligned}
\end{equation}
for a.a.\ $t\in (0,T)$.
\par
Combining \eqref{later-comparison-j} with \eqref{L1j}, \eqref{convergences-A-tetaj}, and \eqref{L2j} we ultimately conclude, \RCORR by comparison in the bulk heat equation, \EEE  that, a fortiori, $\teta \in  W^{1,1}(0,T; W^{1,3+\epsilon}(\Omega)^*)$ for every $\epsilon>0$ and
\begin{equation}
\label{also-tetaj}
\partial_t\tetaj \weakto  \teta_t \quad \text{in } L^{1}(0,T; W^{1,3+\epsilon}(\Omega)^*)\,.
\end{equation}
This concludes the limit passage in the bulk heat equation \RCNEW \eqref{form_debole_theta-chip}. \EEE
%

  \begin{remark}
  \upshape
  \label{rmk:n-esimo-problema}
  We have not succeeded in showing that the elliptic operator $\mathcal{A}(\teta)$ from
  \eqref{convergences-A-tetaj} satisfies
$\pairing{}{W^{1,3+\epsilon}(\Omega)}{\mathcal{A}(\teta)}{v}=\int_\Omega \RCOMM \alpha(\teta) \EEE \nabla \teta  \cdot \nabla v \dd x$ for every $v\in W^{1,3+\epsilon}(\Omega)$. Indeed,  from \eqref{questo-ci-salva} we are just in a position to infer that $\alpha(\tetaj) \to \alpha(\teta) $ in $L^{1+\zeta}(\Omega{\times} (0,T))$ for some $\zeta>0$, which is not sufficient to identify the weak limit of the sequence $(\alpha(\tetaj)\nabla\tetaj)_j$
in any $L^p$ space. \RCORR Thus, we are not in a position
to pass \EEE  to the limit  in the relation  $\pairing{}{W^{1,3+\epsilon}(\Omega)}{\mathcal{A}(\tetaj)}{v} : = \int_{\Omega} \alpha(\tetaj)\nabla\tetaj \cdot \nabla v \dd x$.
  \end{remark}
    \paragraph{\bf Step $1.4$: limit passage in the surface heat equation.}
    We pass to the limit in \RCNEW \eqref{form_debole_thetas-chip}, \EEE written for test functions
    $w\in  W^{1,2+\epsilon}(\GC)$ for all $\epsilon>0$
     as
     \begin{equation}
\label{later-comparison-js}
\partial_t \tetasj(t) = \mathcal{F}_{j}(t) - \mathcal{A}_{\mathrm{s}}(\tetasj(t)) \qquad \text{in } W^{1,2+\epsilon}(\GC)^*
\foraa \  t \in (0,T),
\end{equation}
with
\begin{equation}
\begin{aligned}
&
\mathcal{F}_j:= \tetasj \lambda'(\chij) \partial_t\chij +\ell + |\partial_t\chij|^2
+ k(\chij)\tetasj(\tetaj - \tetasj) + \nlocss{\chijp\tetaj} \,\chijp \tetasj
\\ & \qquad \qquad
- \nlocss{\chijp}\, \chijp \tetasj^2,
\\
&
 \mathcal{A}_{\mathrm{s}}(\tetasj) \in  W^{1,2+\epsilon}(\GC)^* \text{ defined by }
 \\
 &
 \qquad  \pairing{}{W^{1,2+\epsilon}(\GC)}{\mathcal{A}_{\mathrm{s}}(\tetasj)}{w} : = \int_{\GC} \alpha(\tetasj)\nabla\tetasj \cdot \nabla w \dd x = \RCOMMN \int_{\GC} \EEE \nabla(\widehat{\alpha}(\tetasj)) \cdot \nabla w \dd x
 \,.
\end{aligned}
\end{equation}
Taking into account convergences \eqref{convs-tetaj}, \eqref{s-chi-j-bis}, and \eqref{e'-pure-forte}, the Lipschitz continuity of $\lambda$ and the polynomial growth of $k$, it is easy to show that
\begin{equation}
\label{strongL1}
\begin{aligned}
&
\mathcal{F}_j \to \mathcal{F} \quad \text{in } L^1(0,T;L^1(\GC))
\\
 & \text{with } \mathcal{F}: = \tetas\lambda'(\chi) \chi_t +\ell + |\chi_t|^2
+ k(\chi)\tetas(\teta - \tetas) + \nlocss{\chip\teta} \,\chip \tetas
- \nlocss{\chip}\, \chip \thetasq \,.
\end{aligned}
\end{equation}
In order to pass to the limit in the elliptic operators $( \mathcal{A}_{\mathrm{s}}(\tetasj))_j$ we adapt the very same arguments for the operators  \RCORR $(\mathcal{A}(\teta_{\rho_j}))_j$, \RF cf.\
\eqref{al-pelo}--\eqref{questo-ci-salva}. 
 \EEE Namely,
on the one hand,
 arguing by interpolation we
deduce from the bound for \RCORR  $(\tetasj)_j \subset L^{\mu+\nu}(0,T;L^q(\GC)) \cap L^\infty(0,T;L^1(\GC))$ \EEE that  $\widehat{\alpha}(\tetasj) \to \widehat{\alpha}(\teta)$ in $L^1(0,T;L^1(\GC))$, and thus $\nabla\widehat{\alpha}(\tetasj)  \to \nabla \widehat{\alpha}(\tetas)  $ in the sense of distributions on $(0,T)\times \GC$. On the other hand,   relying on  the estimates in Sec.\ \ref{sss:3.3.8} in the same way as we have done in Step 1.3,
 we show that the sequence
$
(\mathcal{A}_{\mathrm{s}}(\tetasj))_j $ is bounded in $L^{1+\delta}(0,T;W^{1,2+\epsilon}(\GC)^*)$ for some $\delta>0$, so that
\begin{equation}
\label{convergences-As-tetaj}
\begin{aligned}
&
\mathcal{A}_{\mathrm{s}}(\tetasj) \weakto \mathcal{A}_{\mathrm{s}}(\tetas) \quad \text{in } L^1(0,T; W^{1,2+\epsilon}(\GC)^*), \qquad \text{with }
\\
&  \pairing{}{W^{1,2+\epsilon}(\GC)}{\mathcal{A}_{\mathrm{s}}(\tetas(t))}{w}: = \int_{\GC} \nabla (\widehat{\alpha}(\tetas(t))) \cdot \nabla w \dd x  
\RCORR \quad \foraa\, t \in (0,T). \EEE
\end{aligned}
\end{equation}
By comparison in  \eqref{later-comparison-js} and convergences \eqref{strongL1}, \eqref{convergences-As-tetaj}  we deduce that, a fortiori,
$\tetas\in W^{1,1}(0,T;W^{1,2+\epsilon}(\GC)^*)$ for every $\epsilon>0$, and
 \[
 \partial_t \tetasj \weakto \partial_t \tetas \quad \text{ in $L^1(0,T;W^{1,2+\epsilon}(\GC)^*)$.}
 \]
Hence, we pass to the limit in the surface heat equation  \RCNEW \eqref{form_debole_thetas-chip}.  \EEE
\par
We have thus shown that the quintuple  \RCNEW $(\teta,\uu,\tetas,\chi,\sigma)$ \EEE is a  `weak energy' solution    to the Cauchy problem for system
 \eqref{PDE-regul}  with  $\rho=0$.
% ...
\subsection{Limit passage as $\varsigma\down 0$ and conclusion of the proof of Theorem \ref{thm:1}}
\label{ss:6.2}
We shall only sketch the argument for the limit passage, as it is completely analogous to that carried out in Section \ref{ss:6.1} up to the identification of this maximal monotone operators in the momentum balance and in the flow rule for the adhesion parameter.
\par
Let $(\tetan,\uun,\tetasn,\chin,\zzetan,\xin,\upin)_n$ be a sequence of  weak energy solutions to  the Cauchy problem for system
 \eqref{PDE-regul}, in which  $\rho=0$  and $\varsigma =\varsigma_n$ with  $\varsigma_n\down 0$ as $n\to\infty$; we have set
 $\zzetan := \eta_{\varsigma_n}(\uun \cdot \mathbf{n}) \mathbf{n}$ and $\xin := \beta_{\varsigma_n}(\chin)$.
We suppose that for every $n\in \N$ the seventuple $(\tetan,\uun,\tetasn,\chin,\zzetan,\xin,\upin)$ has been obtained by the limiting procedure described in Sec.\ \ref{ss:6.1}, so that, by lower semicontinuity arguments, estimates \eqref{estimates-unif-rho} hold for the sequence  $(\tetan,\uun,\tetasn,\chin,\zzetan,\xin,\upin)_n$, uniformly w.r.t.\ $n$. Therefore, there exists a \RCNEW quintuple
$(\teta,\uu,\tetas,\chi,\sigma)$ \EEE as in \eqref{regularity-quadruple}  such that convergences \eqref{convs-tetaj} hold, as $n\to\infty$,  along a not relabeled subsequence.
Then, the limiting temperatures $\teta$ and $\tetas$ enjoy the positivity properties \eqref{uniform-strict-positivity-bis}.
\par
In turn, we are in a position to improve
estimates \eqref{e:zeta-rho} and \eqref{e:xi-rho}
for the sequences
$(\zzetan)_n$
and $(\xin)_n$. Indeed,
a comparison argument in the momentum balance \eqref{weak-displ} shows that the sequence $(\zzetan)_n$ is bounded in $L^2(0,T;\bsY^*)$. Analogously, by comparison in the pointwise flow rule \eqref{ptw-flow-rule}  (cf.\ Sec.\ \ref{sss:3.3.5})
we infer that the sequence $(A\chin {+}\xin)_n$ is bounded in $L^2(0,T;L^2(\GC))$ and then,
\RCORR a fortiori,  we easily \EEE deduce that the sequence  \RCORR $(\xin)_n$ \EEE  is bounded in $L^2(0,T;L^2(\GC))$.
Hence, there exist $\zzeta \in L^2(0,T;\bsY^*)$ and $\xi \in L^2(0,T;L^2(\GC))$ such that, up to a subsequence,
there holds
\begin{equation}
\label{cv-xin-zetan}
\begin{aligned}
&
\zzetan\weakto \zzeta  && \text{in }  L^2(0,T;\bsY^*),
\\
& \xin \weakto \xi && \text{in } L^2(0,T;L^2(\GC).
\end{aligned}
\end{equation}
Finally, since $(\chin)_n$ is bounded in $L^2(0,T;H^2(\GC))$, we ultimately have that
\begin{equation}
\label{cv-chin}
\chin \weakto \chi \quad \text{in } L^2(0,T;H^2(\GC)).
\end{equation}
\par
Let us now outline the argument for the limit passage  in the weak formulation of
system \eqref{PDE-regul}.
\paragraph{\bf Step $2.1$: limit passage in the momentum balance.}
Thanks to convergences \eqref{w-teta-j} and \eqref{cv-xin-zetan},
with the very same arguments as in Sec.\ \ref{ss:6.1} we conclude that the quadruple $(\teta,\uu,\chi,\zzeta)$ fulfills \eqref{temp-momen-balance}
for every $\vv \in \bsVD$,
namely the weak formulation \eqref{weak-displ} of the momentum balance. It remains to show that
$\zzeta(t) \in \eeta(\uu(t))$ in $\bsY^*$ for almost all $t\in (0,T)$. With this aim,
%
%
%
%
%
we
test \eqref{weak-displ} by $\uun$ and integrate it in time. Passing to the limit as $n\to\infty$ we find that
  \[
  \begin{aligned}
  &
  \limsup_{n\to\infty} \int_s^t    \int_{\GC}\zzetan \cdot \uun \dd x
   \dd r
   \\ &
  \leq -\liminf_{n\to\infty} \int_s^t  \vfm(\partial_t \uun,\uun) \dd r
  -\liminf_{n\to\infty} \int_s^t  \efm(\uun,\uun) \dd r
    -\liminf_{n\to\infty} \int_s^t \int_{\Omega} \tetan \mathrm{div}(\uun) \dd x \dd r
   \\
  & \quad
  -\liminf_{n\to\infty} \int_s^t \int_{\GC} \chinp|\uun|^2 \dd x \dd r
  -  \liminf_{n\to\infty} \int_s^t\int_{\GC} \chinp \nlocss{\chinp}\,  |\uun|^2 \dd x \dd r
  \\
  &
  + \lim_{n\to\infty} \int_s^t   \pairing{}{\bsVD}{\mathbf{F}}{\uun} \dd r
  \\
  &
  \leq -  \int_s^t \left( \vfm(\uu_t,\uu)
+ \efm(\uu,\mathbf{u})
+\int_{\Omega} \teta \mathrm{div}(\uu) \dd x +
 \int_{\GC} \chip\uu \mathbf{u} \dd x \right) \dd r
 \\
  & \qquad -
\int_s^t \left(  \int_{\GC} \chip\uu \nlocss{\chip}\, \mathbf{u} \dd x +
\pairing{}{\bsVD}{\mathbf{F}}{\uu} \right) \dd r\stackrel{\eqref{weak-displ}}
= \int_s^t \langle \zzeta, \mathbf{u}\rangle_{\bsY} \dd r
  \end{aligned}
  \]
  where we have used that
  \[
  \begin{aligned}
  \liminf_{n\to\infty} \int_s^t  \vfm(\partial_t \uun,\uun) \dd r  &  =   \liminf_{n\to\infty} \left( \frac12 \vfm(\uun(t),\uun(t))  - \frac12 \vfm(\uun(s),\uun(s)) \right)
 \\ &  \geq \frac12 \vfm(\uu(t),\uu(t))  - \frac12 \vfm(\uu(s),\uu(s))
  =  \int_s^t  \vfm(\partial_t \uu,\uu) \dd r
  \end{aligned}
  \]
  by the chain rule and convergence \eqref{s-uu-0}, and that
  \[
  \begin{aligned}
  &
   \lim_{n\to\infty}  \int_s^t \int_{\GC} \chinp|\uun|^2 \dd x \dd r  =  \int_s^t \int_{\GC}\chip |\uu|^2 \dd x \dd r,
  \\
&  \lim_{n\to\infty}  \int_s^t\int_{\GC} \chinp \nlocss{(\chin)^+} |\uun|^2 \dd x \dd r  =  \int_s^t\int_{\GC} \chip \nlocss{(\chi)^+} |\uu|^2 \dd x \dd r
\end{aligned}
  \]
  by well-known lower semicontinuity results. Therefore, we conclude that for every $\vv \in \bsVD$ such that $\widehat{\eta}(\vv \cdot \mathbf{n}) \in L^1(\GC)$ there holds
  \[
  \begin{aligned}
  \int_s^t\left( \widehat{\balpha}(\vv){-} \widehat{\balpha}(\uu) \right) \dd r
 &   = \int_s^t \int_{\GC} (\widehat{\eta}(\vv {\cdot} \mathbf{n}) {-} \widehat{\eta}(\uu {\cdot} \mathbf{n})  ) \dd x \dd r
  \\ & \stackrel{(1)}{\geq} \limsup_{n \to\infty}  \int_s^t \int_{\GC} (\widehat{\eta}_{\varsigma_n}(\vv {\cdot} \mathbf{n}){-}\widehat{\eta}_{\varsigma_n}(\uun {\cdot} \mathbf{n})) \dd x  \dd r
 \\ &    \stackrel{(2)}{\geq} \limsup_{n \to\infty}  \int_s^t \int_{\GC} \eta_{\varsigma_n}(\uun {\cdot} \mathbf{n})  \mathbf{n} \cdot (\vv{-}\uun) \dd x \dd r
 \\ &    \stackrel{(3)}{\geq} \int_s^t \langle \zzeta, \vv{-}\uu \rangle_{\bsY} \dd r,
    \end{aligned}
  \]
  which \RCORR yields \EEE  the desired \eqref{zeta}. We have thus shown that the quadruple $(\teta,\uu,\chi,\zzeta)$ fulfills
  \eqref{weak-U} (where $\chi$ is, \RF momentarily, \EEE replaced by $(\chi)^+$).
  \paragraph{\bf Step $2.2$: limit passage in the flow rule.}
 With convergences \eqref{w-teta-j} and \eqref{cv-xin-zetan} and the arguments developed for the limit passage in system \eqref{PDE-regul} as $\rho_j \down 0$  we show that the quintuple
 $(\uu,\tetas,\chi,\xi,\sigma)$ fulfills \eqref{weak-chi-sigma}.
 Combining the weak convergence and the strong convergence of $(\xin)_n$
 and $ (\chin)_n$
  in $L^2(0,T;L^2(\GC))$    we infer that
  \[
  \lim_{n\to\infty}\int_s^t \int_{\GC}\xin \chin \dd x \dd r =   \int_s^t \int_{\GC}\xi \chi \dd x \dd r \qquad \text{for all } (s,t) \subset (0,T),
  \]
whence we deduce that $\xi \in \beta(\chi)$ a.e.\ in $\Omega\times (0,T)$ 
\RCNEW so that, in particular,
\[
\chi \geq 0 \qquad \aein\, \GC \times (0,T)\,.
\]
All in all,  the quintuple
 $(\uu,\tetas,\chi,\xi,\sigma)$ fulfills  \eqref{weak-chi-sigma}.
   \paragraph{\bf Steps $2.3$ \& $2.4$: limit passage in the bulk and surface equations} These limit procedures can be performed by the very same arguments as in Steps $1.3$ and  $1.4$. We thus obtain the
   weak formulations of the bulk and surface equations \eqref{weak_theta}  and \eqref{weak_thetas}.
   \paragraph{\bf Conclusion of the proof.}
   We have shown that the seventuple $(\teta,\uu,\tetas,\chi,\zzeta,\xi,\sigma)$
   \begin{compactenum}
   \item  enjoy the regularity, integrability, and positivity properties \eqref{reg-entropic}, \eqref{teta-pos}, and  \eqref{reg-selections};
   \item
    fulfill   the Cauchy conditions \eqref{init-cond} as a trivial consequences of convergences \eqref{convs-tetaj}; \item fulfill the weak formulation of system \eqref{PDE-true} consisting of \eqref{weak_theta}--\eqref{weak-displ}
    and
    \eqref{weak-chi-sigma}.
    \end{compactenum}
   \RCNEW The total energy balance \eqref{total-enbal} follows
   by testing \eqref{weak_theta} by $1$,  \eqref{weak_thetas} by $1$, \eqref{weak-displ} by $\uu_t$,
    \eqref{weak-chi-sigma} by $\chi_t$, and carrying out the same calculations
   as in Sec.\ \ref{ss:3.1}.  \EEE
   \par
   This finishes the proof
 of Theorem \ref{thm:1}.
\QED

\end{document}